\documentclass[11pt,reqno]{amsart}
\usepackage{amsmath,amsthm,amsfonts,amssymb,mathrsfs,bm,graphicx,stmaryrd}
\usepackage{mathtools}
\usepackage{dsfont}
\usepackage{multicol}
\usepackage[colorlinks=true,linkcolor=blue]{hyperref}
\usepackage{xcolor}

\usepackage[letterpaper,hmargin=1.0in,vmargin=1.0in]{geometry}
\parskip 	\smallskipamount

\newtheorem{theorem}{Theorem}[section]
\newtheorem{lemma}[theorem]{Lemma}
\newtheorem{corollary}[theorem]{Corollary}
\newtheorem{proposition}[theorem]{Proposition}

\newtheorem{maintheorem}{Theorem}

\def\N{\mathbb{N}}
\def\P{\mathbb{P}}
\def\Z{\mathbb{Z}}
\def\R{\mathbb{R}}

\def\E{\mathbb{E}}
\newcommand{\cA}{\mathcal{A}}
\newcommand{\cB}{\mathcal{B}}
\newcommand{\cC}{\mathcal{C}}
\newcommand{\cD}{\mathcal{D}}
\newcommand{\cE}{\mathcal{E}}
\newcommand{\cF}{\mathcal{F}}

\newcommand{\cH}{\mathcal{H}}
\newcommand{\cI}{\mathcal{I}}
\newcommand{\cJ}{\mathcal{J}}
\newcommand{\cK}{\mathcal{K}}
\newcommand{\cL}{\mathcal{L}}
\newcommand{\cM}{\mathcal{M}}

\newcommand{\cP}{\mathcal{P}}
\newcommand{\cQ}{\mathcal{Q}}
\newcommand{\cR}{\mathcal{R}}
\newcommand{\cS}{\mathcal{S}}
\newcommand{\cT}{\mathcal{T}}
\newcommand{\cU}{\mathcal{U}}
\newcommand{\cV}{\mathcal{V}}
\newcommand{\cW}{\mathcal{W}}
\newcommand{\cX}{\mathcal{X}}

\newcommand{\cZ}{\mathcal{Z}}

\newcommand{\origin}{\mathbf{0}}

\newcommand{\uk}{\underline{k}}

\newcommand{\Var}{\mathrm{Var}}

\newcommand{\fK}{{\mathfrak{K}}}

\newcommand{\rX}{{\mathcal{X}}}
\newcommand{\hrX}{\widehat{{\mathcal{X}}}}
\newcommand{\hV}{\widehat{V}}
\newcommand{\somega}{\omega_\star}
\newcommand{\comega}{\omega_\circ}
\newcommand{\bomega}{\underline{\boldsymbol{\omega}}}

\begin{document}

\title[Rotationally invariant FPP]{Rotationally invariant first passage percolation: Breaking the $n/\log n$ Variance Barrier}

\author{Riddhipratim Basu}
\address{Riddhipratim Basu, International Centre for Theoretical Sciences, Tata Institute of Fundamental Research, Bangalore, India} 
\email{rbasu@icts.res.in}

\author{Vladas Sidoravicius}
\address{Vladas Sidoravicius, Courant Institute of Mathematical Sciences, New York and NYU-ECNU Institute of Mathematical Sciences at NYU Shanghai}
\email{vs1138@nyu.edu}

\author{Allan Sly}
\address{Allan Sly, Department of Mathematics, Princeton University, Princeton, NJ, USA}
\email{allansly@princeton.edu}

\date{}
\begin{abstract}
    For first passage percolation (FPP) on Euclidean lattices $\Z^d$ with $d\ge 2$, it is expected that the variance of the first passage time between two points grows sublinearly in the distance with a universal exponent strictly smaller than $1$. Following Kesten's $O(n)$ upper bound \cite{Kes93} on the variance, Benjamini, Kalai and Schramm \cite{BKS04} used hypercontractivity to obtain an improvement of a factor of $\log n$ when passage times take two values with equal probability. This was later extended to more general classes of passage time distributions. However, unlike in exactly solvable planar models in last passage percolation where the variance is known to be $\Theta(n^{2/3})$,  the best known upper bound for the variance of passage times has remained $O(n/\log n)$ in all non-trivial variants of FPP. For a class of rotationally invariant Riemannian FPP on the plane, we show that the variance is $O(n^{1-\varepsilon})$ for some $\varepsilon>0$. Our argument uses fluctuation estimates for passage times and geodesics derived in \cite{BSS23} together with a multi-scale argument to establish that the geodesic exhibits disorder chaos, i.e., upon resampling a small fraction of the underlying randomness, the updated geodesic has on average a small overlap with the original one; this, established at a large number of scales, leads to a polynomial improvement of the variance bound. 
\end{abstract}
\maketitle

\tableofcontents

\section{Introduction}
First passage percolation (FPP) is a well-known model for shortest distances in a random environment, typically on a graph, where the graph distance is perturbed by assigning independent and identically distributed non-negative random lengths to each edge. It was introduced first by Hammersley and  Welsh~\cite{HW65} on Euclidean lattices $\Z^d$ and has been one of the most extensively studied models in probability theory over the last sixty years. Despite initial progress on the first order behaviour using subadditivity \cite{HW65,K73,R73} that culminated in the celebrated Cox-Durrett shape theorem  \cite{CD81}, and important recent progresses such as sub-linear variance of passage times in many models~\cite{BKS04,BR08,DHS13} (see the excellent monograph~\cite{ADH15} for a comprehensive description of the state of the field), understanding finer properties of the metric and the geodesics remain largely out of reach.

A major difficulty in studying the lattice FPP is that the limit shape remains poorly understood. Beyond basic properties like convexity, compactness, and the symmetries inherited from the lattice, virtually nothing about the limit shape is known for general edge length (passage time) distributions. Although it is believed that, under mild conditions, the limit shape is strictly convex with a uniformly curved boundary, there are still no examples of canonical lattice FPP for which this has been verified. A program dating back to the nineties, initiated by Newman and co-authors \cite{New95,NP95,LNP96}, investigated various properties of FPP metric and geodesics, either conditional on some unverified (but believed to be true) assumptions on the limit shape, or for some non-lattice model with additional symmetries which make the limit shape explicit. To this end a number of \emph{rotationally invariant} FPP models have been defined and studied by various authors \cite{VW90,VW92, VW93, HN97, HN98, HN01, LW10, LW14}. The rotational symmetry implies that the limit shape must be a scalar multiple of the Euclidean ball, and many stronger results (e.g.\ upper bound of transversal fluctuations of geodesics, improved lower bound for variance in the planar case, existence of semi-infinite geodesics) have been proved for those models. Nevertheless, until recently, even for rotationally invariant models, finer results such as the scaling relation between the longitudinal and transversal fluctuation exponents (KPZ relation) were known only under further (unverified) assumptions such as exponential concentration of passage times at the standard deviation scales \cite{Cha11, Ale20, Ale21}.    

 In \cite{BSS23}, we initiated a program to study a class of rotationally invariant FPP models which covers all of the well-known rotationally invariant models of FPP. For such models, we established stretched exponential concentration for passage times at the standard deviation scale and as a consequence gave a first unconditional proof of the KPZ scaling relation between the scaling exponents for passage times fluctuation and transversal fluctuation of geodesics for those models. In this paper, we further this program by focusing on improving the upper bound of passage time fluctuations in one of the classes of models considered in \cite{BSS23}, namely Riemannian FPP. 

 It is believed that for standard first passage percolation in $d\ge 2$ dimensions, there exists a universal fluctuation exponent $\chi=\chi(d)\in [0,1/2)$ such that under mild conditions on the passage time distribution, the fluctuations of the passage times between two points at distance $n$ scales as $n^{\chi}$ (in two dimensions, from the theory of Kardar-Parisi-Zhang (KPZ) universality class \cite{KPZ86} it is expected that $\chi=1/3$). Towards this, Kesten \cite{Kes93} used a Poinc\'are inequality argument to show that the variance (of passage times between two points at distance $n$) is $O(n)$. The next major breakthrough came in early 2000s when Benjamini, Kalai and Schramm \cite{BKS04} used hypercontractivity to show that when the passage time distribution takes two values with equal probability then the upper bound on the variance can be improved by a factor of $\log n$. There has been a large number of subsequent results (more details later) over the last 20 years weakening the hypothesis of \cite{BKS04}, and similar improvements were obtained in many related models, but until now there were no further improvements on the variance bound of $O(\frac{n}{\log n})$. 

In this paper we show that for a class of rotationaly invariant models on the plane, called Riemannian FPP, the variance is upper bounded by $n^{1-\varepsilon}$ for some $\varepsilon>0$, thus providing a first proof of $\chi <1/2 $ (provided the exponent exists) beyond a few exactly solvable models of planar last passage percolation (where $\chi=1/3$ is known). Riemannian FPP was introduced by Lagatta and Wehr~\cite{LW10}. In this model, one considers the random environment to be a smooth positive isotropic random field on the plane.  For nice paths one defines the length of a path by integrating the field along the path and the first passage time between two points is defined by taking the infimum of lengths of paths joining the points. We will further assume that the field is bounded away from $0$ and $\infty$ and satisfies the FKG inequality, i.e., increasing functions of the field are positively correlated (a formal definition and construction of such fields are given later).  Let $X_n$ denote the passage time from the origin to $(n,0)$. The following theorem is the main result of this paper. 

 \begin{maintheorem}
     \label{t:theorem}
     For a class of Riemannian FPP on $\R^2$ as defined below, there exists $\varepsilon, n_0>0$ such that for all  $n\ge n_0$
     $$\Var (X_n) \le n^{1-\varepsilon}.$$
 \end{maintheorem}

We shall make some remarks on the choice of the model and role of the specific assumptions in Theorem \ref{t:theorem} and potential extensions in Section \ref{s:extension}.

\subsection{Riemannian FPP}
We now formally define the Riemannian FPP model. Fix a radially symmetric (i.e., $K(x)$ is depends on $x$ only through $|x|$), nonnegative, smooth $(C^{\infty})$, bounded kernel $K:\R^2\to\R$ that vanishes outside the unit ball.  We will take $\omega$ to be either a Gaussian white noise or a homogeneous Poisson point Process on $\R^2$. We then write.
\[
\Xi(x):=\int K(x-y) \omega(d y).
\]
In the case where $\omega$ is Gaussian white noise this means the stochastic integral
\[
\Xi(x)=\int K(x-y) dB(y)
\]
where $dB(y)$ is two dimensional Gaussian white noise. In the case that $\omega$ is a Poisson point process this means
\[
\Xi(x)=\sum_{y\in \Pi} K(x-y)
\]
where $\Pi$ is the set of points in the Poisson point Process. To ensure that we have a bounded, positive field we fix a monotone increasing smooth function  $\psi:\R\to (d_1,d_2)$ such that $\psi(\R)$ is supported on $[d_1,d_2]$ for some $0<d_1<d_2<\infty$. Then we set $\Psi(x)=\psi(\Xi(x))$ to be the underlying environment which the paths traverse. Observe that $\Psi$ satisfies the conditions set forth above, it is a smooth random field that is invariant under the symmetries of $\R^2$ (owing to the isotropic nature of Gaussian white noise and Poisson process on $\R^2$ and the fact that $K$ is chosen to be radially symmetric), is $2$-dependent, as $K$ is supported on the unit ball, and satisfies the FKG inequality. That $\Psi$ satisfies the FKG inequality is standard in the case where $\omega$ is a Poisson point process, see e.g.\ \cite[Lemma 2.1]{J84}. For the case of the Gaussian white noise, notice that, for $x_1,x_2\in \R^2$, we have $\mbox{Cov}(\Xi(x_1),\Xi(x_2))=\int{K(x_1-y)K(x_2-y)}dy \ge 0$ as $K$ is assumed to be non-negative. Therefore, by Pitt's Theorem \cite{Pit82}, (all finite dimensional marginals of) $\Xi$ satisfies the FKG inequality and so does $\Psi$ since $\psi$ is monotone increasing (see, e.g., \cite[Proposition 2]{Bar05}).

For any path $\gamma: [0,1]\mapsto \R^2$ such that $\gamma$ is piecewise $C^1$, we define the passage time of $\gamma$ by \[
X_{\gamma}:= \int_{\gamma} \Psi(x) dx =  \int_{0}^{1} \Psi(\gamma(t))|\dot{\gamma}(t)| dt.
\]
This definition can be extended to all bounded variation paths by taking suitable limits. Notice that for two paths $\gamma_1$ and $\gamma_2$ such that the end point of $\gamma_1$ is the starting point of $\gamma_2$ the concatenated path $\gamma$ satisfies
$$X_{\gamma}=X_{\gamma_1}+X_{\gamma_2}.$$
Finally, we define the first passage time  
$$X_{uv}:=\inf_{\gamma: \gamma(0)=u, \gamma(1)=v} X_{\gamma}$$
by taking infimum over all such paths from $u$ to $v$. This defines a random Riemannian metric on $\R^2$. 

The following properties are basic (see e.g.\ \cite{LW10,LW14} where a more general class of random fields were considered): there exists $\mu\in (0,\infty)$ such that for any $u\in \R^2$,

$$\lim_{|v|\to \infty} \frac{X_{u,u+v}}{|v|}=\mu $$
almost surely and in $L^1$. One can also upgrade this to a \emph{shape theorem}. 

Also for $u,v\in \R^2$, there is a geodesic attaining the infimum in the definition of $X_{uv}$ \cite{LW14}. We expect that the geodesic between two fixed points to be almost surely unique but this will not be required in our arguments. 

\textbf{We shall henceforth scale the field so that $\mu=1$}. It is easy to see that this does not lead to any loss of generality. Denoting $X_{(0,0),(n,0)}$ by $X_n$, it was shown in \cite{BSS23} that there exists $C,\theta'>0$ such that
$$ \P\left(\dfrac{|X_n-n|}{\rm{SD} (X_n)}\ge x\right)\le \exp(1-Cx^{\theta'})$$
for all $x>0$ where $\mbox{SD}(X_n)$ denotes the standard deviation of $X_n$. It was also shown that the geodesics between two points at distance $n$ have transversal fluctuations at the scale $\sqrt{n \rm{SD}(X_n)}$. Further results we shall need from \cite{BSS23} will be recalled later. 

\subsection{Related literature and main ideas}
As mentioned above, the study of fluctuations in first passage percolation goes back a long time. The classically studied variant of FPP considers putting i.i.d.\ passage times on the nearest neighbour edges of $\Z^d$. The classical shape theorem (see e.g.\ \cite{CD81}) shows that under mild conditions on the passage time distribution, first passage times between two vertices at distance $n$ in $\Z^d$ has fluctuations $o(n)$. The first non-trivial upper bound on the variance of passage times was obtained by Kesten in \cite{Kes93} who showed that if the passage times on the edges have finite second moment, then the variance of the passage times between two points at distance $n$ is $O(n)$. Kesten's argument to bound the variance uses a variant of the Efron-Stein inequality. His method of looking at the Doob martingale for the first passage time also gives exponential concentration of passage times as scale $\sqrt{n}$. Kesten's martingale argument is also the starting point of our proof and we shall describe this below in some detail. 

In \cite[Remark 1]{Kes93} Kesten remarks \begin{quote}
    \emph{The next problem one should attack now is to show that \ldots (\mbox{$X_n$}) behaves ``subdiffusively"; that is \ldots ($\Var(X_n)$) ~$\le n ^{1-\varepsilon}$ for some $\varepsilon>0$.}
\end{quote}
 Unfortunately, this question turned out to be very difficult and remains open to date. 

 Kesten's argument can also be interpreted as a Poincar\'e inequality on the product space (for the Markov semi-group where each edge weight is independently resampled at rate $1$; see \cite{Chabook}). It is by now well-known that one way to improve upon the Poincar\'e inequality for the product spaces is to use hypercontractivity and the Talagrand's $L^1-L^2$ inequality or the log-Sobolev inequality. Indeed, essentially the only known improvements to Kesten's upper bound of the variance in FPP so far have been established using this method. The first breakthrough was by Benjamini, Kalai and Schramm in \cite{BKS04} where it was shown that if the passage times take two values $a,b\in (0,\infty)$ with probability $1/2$ each then one has $\Var (X_n)=O(\frac{n}{\log n})$. A number of follow up works (see e.g.\ \cite{BR08, DHS13}) successively weakened the hypothesis on the passage time distribution (see \cite{ADH15} for a history of the developments) and the current best known result from \cite{DHS13} requires $2+$ log moments of the passage time distribution and that the size of the atom at $0$ (if any) is less than the bond percolation critical probability to conclude that $\Var (X_n)=O(\frac{n}{\log n})$. This argument uses a version of the log-Sobolev inequality. Exponential concentration for the passage time at the scale $\sqrt{\frac{n}{\log n}}$ was established in~\cite{DHS14} (one can also establish Gaussian concentration at scale $\sqrt{n}$ using Talagrand's inequality; see~\cite{ADH15} for more details). There has however not been any further improvements on the upper bound on the variance. 

 As mentioned before, it is expected that under mild conditions the passage times fluctuations are expected to diverge as $n^{\chi}$ for some universal exponent $\chi(d)$; it is also expected that $\chi>0$ in low dimensions, and $\chi(2)=\frac{1}{3}$. There are some special cases where the fluctuations are known to be of constant order, e.g.\ along a direction within the percolation cone when the passage time distribution has a large atom at the positive infimum of its support; see \cite{NP95} for a discussion of such cases. Excluding these cases, 
a lower bound of order $\log n$ is known in dimension 2 under mild conditions, \cite{NP95} (see also \cite{Z08}).

 Since the work \cite{BKS04}, the method of using hypercontractivity and Talagrand's method or log-Sobolev inequality to prove improved variance upper bounds (compared to the bound obtained by Poincar\'e inequality/ Efron-Stein inequality) has found applications in different variants of first passage percolation as well as in many related models: last passage percolation, spin glasses, directed polymers, random surface growth to name only a few. While discussing all these developments is beyond the scope of our articles; a representative sample of references exploring this line of work for the interested reader is given here: \cite{BT15, Chabook, AZ23,CN19,G12,Dem24,DG24,CL24}. However, in all of these cases, this method gives an improvement of a factor of $\log n$ over the Efron-Stein bound whereas in most of the cases the actual order of the variance is expected to a smaller by a polynomial factor. 

We also note that to obtain the logarithmic improvement for the variance upper bound in FPP on $\Z^d$, it is also necessary to show that influence of most of the individual edges (i.e., $\P(e\in \gamma)$ where $\gamma$ is the geodesic) is small; this is non-trivial. Benjamini-Kalai-Schramm \cite{BKS04} circumvented this by an averaging trick, and most of the follow up works used versions of the same trick as well. Only very recently it was shown in \cite{DEP23} that with high probability the number of edges with large influence is indeed small; this however did not lead to any improvement on the upper bound of the variance. 

\subsubsection{Idea of Proof}
We shall give a high level overview of the ideas that go into the proof of Theorem \ref{t:theorem}; a more detailed outline of the argument will be provided in the next section. As mentioned above, the starting point of our argument is Kesten's proof \cite{Kes93}, therefore we start by taking a more detailed look at the same.

\noindent
\textbf{A short recap of Kesten's argument.} Consider any enumeration of the set of edges $e_1,e_2,\ldots$. Denoting by $X_n$ the passage time between the two points at distance $n$, consider the Doob Martingale $M_{i}=\E[X_n\mid \cF_i]$ where $\cF_i$ is the $\sigma$-algebra generated by the edge weights on $e_1,\ldots, e_i$. Then

$$ \Var (X_n)=\E \left[\sum_{i} \E[(M_{i}-M_{i-1})^2\mid \cF_{i-1}] \right].$$

One can think of the $i$-th term in the above sum as the contribution from resampling the passage time on the edge $e_i$. Let us denote $\Delta_i:=M_i-M_{i-1}$ and focus on the term $\Delta_i^2$. Writing the passage times of the edges as $\mathbf{X}=\{X(e_i)\}$; let $\mathbf{X}^i$ denote the environment where $X(e_i)$ is replaced by an independent copy $X'(e_i)$ and let $X^i_n$ denotes the passage time from $\origin$ to $(n,0)$ computed in the environment $\mathbf{X}^i$. Clearly, 

$$\Delta_i^2 = (\widetilde{\E}(X_n-X_n^i))^2$$
where $\widetilde{\E}$ averages over $X'(e_i), X(e_{i+1}), \ldots$. By exchangeability of the environments, it suffices to only consider the case  when $X'(e)\ge X(e)$ (so that $X_n\ne X_n^i$ only if $e_i$ is on the geodesic $\gamma$ (for the purpose of this exposition let us assume the geodesic is almost surely unique, the argument works even without this assumption) in the environment $\mathbf{X}$ and in that case $ X'(e_i)-X(e_i)\ge X_n^i- X_n\ge 0$), so our task reduces to upper bounding 
$$(\widetilde{\E}[(X'(e_i)-X(e_i))I_i)])^2$$
where $I_i$ denotes the indicator that the edge $e_i\in \gamma$ (where $\widetilde{\E}$ now denotes the expectation just over $X'(e)$). Next one uses the Cauchy-Schwarz inequality a to upper bound this by 
\begin{equation}
    \label{e:kestenloss}
    \widetilde{\E}[(X'(e_i)-X(e_i))I_i)^2]\widetilde{\E}[I^2_i]
\end{equation}
At this point the indicator in the first term is upper bounded by $1$, and the noticing that $\widetilde{\E}{[I_i]}=\P(e_i\in \gamma\mid \cF_i)$ eventually one gets the bound 

$$\E[\Delta_i^2\mid \cF_{i-1}]\le C\P[e_i\in \gamma\mid \cF_{i-1}]$$
for some $C>0$. Summing over all $i$, finally we get 
$$\Var (X_n) \le C \E[\#\{e:e\in \gamma\}].$$
It is a classical fact (see e.g.\ \cite{Kes86}) that the expected number of edges on the geodesic is $O(n)$, and this completes the proof. 

Kesten already comments (see \cite[Remark 1]{Kes93}) that to get improved upper bounds on the variance one should try and improve on the step after the Cauchy-Schwarz inequality in \eqref{e:kestenloss} where we upper bound the indicator $I_i$ by $1$. Indeed, as explained above, we expect that the influence of the individual edges are small, so upper bounding these indicators by 1 leads to a loss. To get a better bound for the variance one might want to try and come up with better bounds for 
$$\sum_i \P[e_i\in \gamma \mid \cF_{i-1}]^2.$$
Indeed, that a better bound on the above leads to an improved variance bound can be formalised using the principle that \emph{superconcentration and chaos are equivalent}; we explain these notions now.

\noindent
\textbf{Superconcentration and Chaos.}
Notice that
$$\P[e_i\in \gamma \mid \cF_{i-1}]^2=\P[e_i\in \gamma\cap \gamma'\mid \cF_{i-1}]$$
where $\gamma'$ is the geodesic in the environment where the weights of the edges $e_i,e_{i+1},\ldots$ are each independently resampled. One therefore would expect that if the expected intersection of the geodesics before and after resampling a small fraction of edge weights is $o(n)$, this will lead to a $o(n)$ upper bound on the variance. This intuition was formalised by Chatterjee \cite{Cha08,Chabook} where he showed that \emph{superconcentration} (the phenomenon that one gets a variance upper bound which is of a smaller order than the Poincar\'e inequality bound) is equivalent to \emph{disorder chaos} for the geodesic (in our context it refers to the phenomenon where resampling a small fraction of the randomness leads to the geodesic after resampling having a microscopic overlap on average with the geodesic before resampling; Chatterjee's result is more general). More precisely we have the following in the context of first passage percolation: the statements 

\begin{quote}
    \emph{For every $\varepsilon>0$ we have $\Var X_n \le \varepsilon n$ for all $n$ large}
\end{quote}

and 

\begin{quote}
    \emph{For every $\varepsilon,\delta>0$ we have 
    $\E[\#\{e:e\in \gamma \cap \gamma^{\varepsilon}\}]\le \delta n$ for all $n$ large
    where $\gamma^{\varepsilon}$ denotes the geodesic after resampling each edge weight independently with probability $\varepsilon$}
\end{quote}
are equivalent. Strictly speaking, Chatterjee did not work with the independent resampling semigroup in \cite{Chabook} but the above statement essentially follows from his arguments; see also \cite{ADS23} for more details on this connection in the context of FPP. 

In most examples in \cite{Chabook} and other results in the literature, one usually proves superconcentration (by using Talagrand's $L^1-L^2$ inequality, say) and derives chaotic behaviour as a consequence of the above equivalence. Our strategy is to go in the other direction, we first establish the chaotic nature of the geodesic at different scales and then get an improved variance estimate from there via a multi-scale argument.

The road map, roughly, is as follows. We fix large a $M$, and then show that the following holds for all $n$ sufficiently large depending on $M$. Let $W_n$ denote the typical transversal fluctuation scale at distance $n$ (more details about this in the following section). We then tile the plane by $n\times W_n$ rectangles. We show by a percolation argument (this is the most difficult part of the proof) that for any $\kappa,\epsilon>0$ when we resample the randomness (Gaussian White Noise or Poisson process) in each of these rectangles independently with probability $\kappa$ then the expected number of these rectangles that intersect the geodesics both before and after the resampling is $\le \epsilon M$ (see Proposition \ref{p:chaos}). A block version of the \emph{chaos implies superconcentration} principle will then imply that 

$$\Var (X_{Mn}) \le \frac{M}{2} \Var X_n.$$

Repeating the same argument at exponentially growing scales gives for all $k\ge 1$

{$$\Var (X_{M^k}) \le C\left(\frac{M}{2}\right)^k = C(M^k)^{1-\frac{\log 2}{\log M}}$$
which completes the proof.}

As mentioned above, the technical core of our argument is proving the block version of the chaos statement described above. This requires a multi-scale argument which in turn requires some complicated geometric constructions as well as several estimates established in \cite{BSS23}. We shall start by recalling these results and giving a more detailed outline of the technical steps of our argument in the next section.

\subsection{Role of Assumptions and possible extensions}
\label{s:extension}

It is natural to ask whether our general framework could establish improved variance bounds under more general assumptions. We make some comments here regarding to what extent it might be possible.

As already mentioned, this paper crucially depends on \cite{BSS23} where several critical estimates were proved for a class of general rotationally invariant planar FPP models including the Riemannian FPP and our choice of model is primarily dictated by this. However, results of \cite{BSS23} were proved under a class of hypotheses which are valid for a number of other well-known rotationally invariant FPP models, where the underlying graph is constructed based on a Possion point process. These models include the Howard-Newman model, distances in supercritical random geometric graph and Voronoi FPP (see \cite{BSS23} for the detailed definitions of these models). However, beyond the assumptions of \cite{BSS23} this paper also critically requires the model to satisfy the FKG inequality, which is used many times throughout the percolation arguments used to establish the chaotic nature of geodesics. As defined in \cite{BSS23}, the Howard-Newman model and the model of distances in random geometric graph do not satisfy the FKG inequality, mainly because the definitions there were making a natural discrete model artificially into a continuum model for the sake of a unified treatment. One can, however, define minor variants of those models that indeed satisfy the FKG inequality and we expect that under minor modifications our arguments will prove improved variance bounds for those models as well. It might also be possible to prove our result under the general assumptions of \cite{BSS23} together with the FKG inequality. The Voronoi FPP, however, fails the FKG inequality in a more serious manner and we do not expect the current arguments to extend to include that model. Certain parts of our estimates, e.g. the ones in Section \ref{s:1.2proof} do not depend on the FKG inequality, and can be made to work for all rotationally invariant models.  

We also expect our results to hold in all dimensions $d\ge 2$, with the model definitions extended to $\R^d$ in an obvious way. Despite the fact that \cite{BSS23} uses planarity in several crucial points, it requires planarity only to show that the variance cannot grow too slowly i.e., it grows at least polynomially.  Furthermore, it does not use other consequences of planarity such as ordering of geodesics. Since in this paper we are only concerned with upper bounding the variance, failing to have a lower bound on the variance is not an obstacle.  As a result, one may suitably modify the arguments to establish sublinear variance bounds in higher dimensions as well. 

It might also be possible to relax the assumption of rotational invariance and prove our results for lattice models with the additional assumption that the limit shape is uniformly curved. Again, \cite{BSS23} uses rotational invariance crucially, but as we are not attempting to prove concentration at the correct standard deviation scale as there, and only at the scale $n^{1/2-\varepsilon}$, it might be possible to (non-trivially) adjust our arguments to cover this case under the same general framework. 

Due to the already substantial length of this manuscript, we have refrained from pursuing details of any of the above possible extensions here. They might be taken up elsewhere in the future.

\subsection*{Acknowledgments}Riddhipratim Basu is supported by a MATRICS grant (MTR/2021/000093) from ANRF (formerly SERB), Govt.~of India, DAE project no.~RTI4019 via ICTS, and the Infosys Foundation via the Infosys-Chandrasekharan Virtual Centre for Random Geometry of TIFR. Allan Sly is supported by a Simons Investigator grant.

\section{Preliminaries and an overview of the argument}
\label{s:prelim}

In this Section we shall set up basic notation, recall relevant estimates from \cite{BSS23} and provide a detailed sketch of the proof of Theorem \ref{t:theorem}. Recall that we are working with rotationally invariant Riemannian FPP model under the hypothesis as described in the previous section scaled such that $\lim_{n\to \infty} n^{-1}\E X_n =1$.

\subsection{Relevant results from \cite{BSS23}}

\subsubsection{Concentration for passage times}
We start with recalling is the following result from \cite{BSS23} where we established stretched exponential concentration bounds for passage times in Riemannian FPP at the standard deviation scale. 

\begin{theorem}[{\cite[Theorem 1, Theorem 3, (6), Lemma 7.3]{BSS23}}]
\label{t:all}
There exist constants $\alpha_*,\theta>0$ and an increasing sequence $\{Q_n\}_{n\in \R_+}$
such that for all $n\ge 1$ and all $x>0$
$$ \P\left( \frac{|X_n-n|}{Q_n}\ge x\right) \le \exp(1-x^{\theta})$$ and the sequence $Q_n$ satisfies the following properties:
\begin{itemize}
    \item $Q_n=\Theta (\sqrt{\Var (X_n)})$.
    \item $n^{\alpha_*} \le Q_n =O(\sqrt{n})$.
    \item  For all $n'>n$ sufficiently large and all $\delta>0$
      $$ \left(\frac{n'}{n}\right)^{\alpha_*} Q_n \le Q_{n'} \le \left(\frac{n'}{n}\right)^{1/2+\delta}Q_n.$$
\end{itemize}
Further, the non-negative sequence $A_n:=\E X_n-n$ satisfies 
$$A_n=\Theta (Q_n).$$
\end{theorem}

Lemma 7.3 in \cite{BSS23} states the upper bound on $Q_{n'}/Q_{n}$ with the exponent $3/4$ instead of $\frac{1}{2}+\delta$; however it was commented immediately following the proof of that Lemma that any exponent larger then $1/2$ will suffice for the proof. In fact, we shall show (see Proposition \ref{p:uij.bounds} and Corollary \ref{c:kesten}) that $\frac{Q_{n'}}{Q_{n}}\le C\sqrt{\frac{n'}{n}}$ for some $C>0$. 

The proof in \cite{BSS23} required an involved definition of $Q_n$, and we continue to use the same definition and notation here. However, for the purposes of this paper we could just take $Q_n$ to be the standard deviation of $X_n$. 

To enable the multi-scale argument one needs to update the point-to-point concentration bounds in the above theorem to passage times across blocks of suitably chosen size. From the KPZ relation (see e.g.\ \cite{NP95,Cha11}) between the fluctuation exponent of the passage times and the \emph{transversal fluctuation} of geodesics  it is expected (proved also in \cite{BSS23}, see below) that the scale of transversal fluctuation at distance $n$ is given by 
\begin{equation}
    \label{e:tfdefn}
    W_n= \sqrt{nQ_n}. 
\end{equation}
Notice that Theorem \ref{t:all} implies that 
$$ n^{1/2+\alpha_*/2}\le W_n =O(n^{3/4})$$ 
and for $M,n$ large 
\begin{equation}\label{eq:basic.W.bounds}
   M^{1/2+\alpha/2}W_n \le W_{Mn}\le M^{7/8}W_{n}. 
\end{equation}
Due to the self-similar nature of the predicted scaling limit of planar FPP models, it is natural to consider rectangles and  parallelograms of size $n\times W_n$ and as in \cite{BSS23} these boxes will be the building blocks of our renormalization argument. 

Let us introduce some notations that we shall need throughout this paper. For $x,y,y'\in \R$ with $y\le y'$, Let $\ell_{x,y,y'}$ denote line segment $\{x\}\times [y,y']$. Following \cite{BSS23} we also define canonical parallelograms $\cP_{i,k,k',n,W_n}$ whose left and right sides are 
$\ell_{(i-1)n,kW_n (k+1)W_n}$ and $\ell_{in,k'W_n (k'+1)W_n}$. When $n$ (and $W_n$) is clear from the context, we shall denote these parallelograms by $\cP_{i,k,k'}$. For such a parallelogram $\cP$ we shall denote its right and left sides by $L_{\cP}$ and $R_{\cP}$ respectively. For the special case of rectangles with $k=k'=j-1$ we shall define 
$$\Lambda_{ij}=[(i=1)n,in]\times[(j-1)W_n,jW_n].$$
These will be the building blocks of our multi-scale arguments.

\subsubsection{Transversal fluctuation estimates}
The next result from \cite{BSS23} shows that across an $n\times W_n$ parallelogram whose slope is not too high; the exponential concentration result from Theorem \ref{t:all} still holds.

\begin{proposition}[{\cite[Lemma 4.9]{BSS23}}]
    \label{p:para}
    For $\theta$ as in Theorem \ref{t:all}, there exists $C>0$ such that for $0\le k \le n/W_n$ and $z>0$ we have for all $n$ sufficiently large

    $$\P\left(\max_{u\in L_{\cP_{1,0,k}}} \max_{v\in R_{\cP_{1,0,k}}} |X_{uv}-|u-v||\ge zQ_n \right) \le  \exp(1-Cz^{\theta}).$$
\end{proposition}

Notice that by Pythagoras' theorem, for $u,v$, as in the above result $|u-v|=n+\Theta(k^2)Q_n$ for large values of $k$. This fact is useful for us at multiple points of the argument so we record this explicitly. 

\begin{lemma}[{\cite[Lemma 3.4]{BSS23}}]
\label{l:paraexplicit}
In the set-up of Proposition \ref{p:para} there exists $C>0$ such that for $0\le k \le n/W_n$ and $z>0$ we have for all $n$ sufficiently large
$$\P\left(\min_{u\in L_{\cP_{1,0,k}}} \min_{v\in R_{\cP_{1,0,k}}} X_{uv}- n-\frac{k^2}{32}Q_n\le -zQ_n \right) \le  \exp(1-Cz^{\theta}).$$
\end{lemma}

The two other estimates that we shall need are about transversal fluctuations of geodesics as alluded to above. We need some more notation.

For $v_1\in\ell_{0,0,W_n},v_2\in\ell_{n,0,W_n}$
\[
\Upsilon_{n,v_1,v_2,z}=\bigg\{\gamma':\gamma'(0)=v_1,\gamma'(1)=v_2, \ \sup_t\inf_{x\in[0,n]} |\gamma'(t) - (x,0)|\geq zW_n\bigg\}
\]
denote the set of paths that travel at least distance $zW_n$ from the line segment joining $(0,0)$ and $(n,0)$.   The set $\Upsilon_{n,\origin,(n,0),z}$  contains the paths that have transversal fluctuations at least $zW_n$. 

We have the following theorem which says that for geodesics across $\cP_{1,0,0}$ the maximal transversal fluctuation is of the order $W_n$. 

\begin{theorem}[{\cite[Lemma 5.3, Theorem 5.4]{BSS23}}]
    \label{t:tfold}
    There exists $D, \theta_0>0$ such that for any $n\geq 1$ we have that for all $z\geq 0$,
\[
\P\left[\exists v_1\in\ell_{0,0,W_n},v_2\in\ell_{n,0,W_n}: \gamma_{v_1v_2}\in \Upsilon_{n,v_1,v_2,z}  \right]\leq \exp(1-Dz^{\theta_0}).
\]
\end{theorem}

The final result we need from \cite{BSS23} is about local transversal fluctuations. It says that the fluctuation of the geodesics at distance $n$ from the endpoints for geodesics between two points at distance $n'\gg n$ is still typically of the order $W_n$. 

For integers $k,x,y$ with $k\geq 0$ and $x+2^k < x'$ define
\begin{align*}
\Xi^{(n),R}_{x,x',y,k,z}:=
\Big\{\gamma':&\gamma'(0)\in\ell_{xn,yW_n,(y+1)W_n},\gamma'(1)\in\ell_{x'n,(y-\lceil (x'-x)^{{89/100}}\rceil)W_n,(y+\lceil (x'-x)^{{89/100}}\rceil)W_n},\\ &\sup\{|w|:((x+2^k)n,yW_n+w) \in \gamma'\}\geq z2^{9k/10}W_n\Big\}.
\end{align*}

We have the following lemma.

\begin{lemma}[{\cite[Corollary 8.3]{BSS23}}]
\label{c:local.trans}
There exist absolute constants $D,z_0, \theta_0$ such that for all $z\geq z_0$ and all $k\in [1,\lfloor \log_2 (x-x')-1\rfloor]$,
\begin{equation}\label{eq:local.trans.proof.corol}
\P\Big[\inf_{\gamma' \in \Xi^{(n),R}_{x,x',y,k,z}} X_{\gamma'}-X_{\gamma'(0),\gamma'(1)} \geq \frac{{z^{1/3}}}{4000} Q_{n}\Big] \leq \exp(1-Dz^{\theta_0}).
\end{equation}
\end{lemma}

Observe that this lemma implies that with large probability all the geodesics starting from $\ell_{xn,yW_n,(y+1)W_n}$ and ending at $\ell_{x'n,(y-\lceil (x'-x)^{{89/100}}\rceil)W_n,(y+\lceil (x'-x)^{{89/100}}\rceil)W_n}$ has local transversal fluctuation at location $(x+1)n$ of the order of $W_n$. By symmetry, a similar result holds for local transversal fluctuations to the left of the right endpoint of geodesics.

\noindent

\subsubsection{Constrained passage times}
Note that Theorem \ref{t:tfold} says that the typical transversal fluctuation of a geodesic of length $n$ is of order $W_n$. One might therefore expect that analogues of Proposition \ref{p:para} should remain true if we consider only paths contained in an appropriate $n\times W_n$ rectangle. This bound on constrained passage times was not proved in \cite{BSS23} as it was not needed there, but is not too hard to prove from Proposition \ref{p:para} and Theorem \ref{t:tfold} (In fact, similar estimates have been proved and been useful in related models before; see e.g.\ \cite[Proposition 12.2]{BSS14}. That result is actually slightly stronger and simultaneously controls passage time from all points on the left to all point on the right of an $r\times W_r$). Bounds on constrained passage times will be needed here, and we shall prove the following theorem in Appendix \ref{s:proxy}; we expect this to be useful and of independent interest. 
\begin{proposition}
    \label{p:constraine}
    There exist $C,\theta_6>0$ such that for all $r$ sufficiently large and all $z\ge 0$ we have 
    $$\P\left(\inf_{\substack{\gamma'(0)=(0,0)\\ \gamma'(1)=(r,0)\\ \gamma'\subset [0,r]\times[0,W_r]}} X_{\gamma'}\ge r+zQ_r\right)\le \exp(1-Cz^{\theta_6}).$$
\end{proposition}

\subsection{Conforming paths and modified distances}
One of the technical inconveniences in the analysis of FPP models is that paths are allowed to backtrack. To circumvent the resulting complications, we shall define a modified distance function between two points by looking only at paths which are not allowed to backtrack to a vertical column (of width $n$) to the left after entering a column to the right. Also, the passage times across different columns are calculated based on independent randomness. Let us now make things precise. 

The modified distance will depend on a parameter $n$  which will be taken to be large depending on other parameters but will be fixed, and also on a parameter $\beta$ which will be taken to be sufficiently small (not depending on $n$ and $M$; see below). To reduce notational overhead we shall suppress the $n$ and $\beta$-dependence in the modified model. 

For $i\in \Z$, let $\Lambda_i$ denote the column
$$\Lambda_i=[(i-1)n,in]\times \R.$$
At this point we need to introduce the probability space we shall work with. While the notation below adds some complexity, it ensures that the weights of  a path contained in different columns are independent. Since the random field used to compute the weights of paths is $2$-dependent this is not quite true; therefore we enlarge the probability space.

For $\omega$ considered as above (i.e., $\omega$ a rate $1$ Poisson point process on $\R^2$, or a Gaussian White Noise on $\R^2$) and for $i\in \Z$, let $\omega^{\Lambda_i}$ denote a random field on $\R^2$ which is equal to $\omega$ in distribution and and such that $\omega^{\Lambda_i}=\omega$ on $\Lambda_i$ and $\omega=\omega_i$ on $\R^2\setminus \Lambda_i$ where $\omega_i$'s are independent copies of $\omega$. We shall work with the $\Z\times \R^2$ valued field 
$$\omega_*=\{\omega^{\Lambda_i}:i\in \Z\}.$$
The point of introducing this field is that for paths contained in the column $\Lambda_i$ we shall compute its weight in the field $\omega^{\Lambda_i}$. For any path $\gamma$ contained in the column $\Lambda_i$, $X^{\Lambda_i}_{\gamma}$ shall denote its weight computed in the field $\omega^{\Lambda_i}$ (i.e., $X_\gamma(\omega^{\Lambda_i})=X^{\Lambda_i}_{\gamma}(\omega)$). 
Observe that this definition ensures that for paths $\gamma_i$ and $\gamma_j$ contained in $\Lambda_i$ and $\Lambda_j$ respectively, $X^{\Lambda_i}_{\gamma_i}$ and $X^{\Lambda_j}_{\gamma_j}$ are independent if $i\ne j$. For $u,v\in \Lambda_i$, we call a path $\gamma$ from $\gamma(0)=u$ to $\gamma(1)=v$ to be a \textbf{conforming path} if $\gamma\subset \Lambda_i$. We define now the \textbf{restricted distance} $\rX_{uv}$ for points $u,v\in \Lambda_i$ by 
\[
\rX_{uv} = \inf_{\substack{\gamma \subset \Lambda_i\\ \gamma(0)=u, \gamma(1)=v}} X_{\gamma}^{\Lambda_i},
\]
i.e., the infimum is taken over all conforming paths. Note that we do not define the restricted distance between pairs of points $(in,y),(in,y')$ (we will call such a pair a \textbf{vertical boundary pair}) since pair is both in $\Lambda_i$ and $\Lambda_{i+1}$.

We now extend this notion to all pairs of points $u,v\in [0,Mn]\times [-n^{\beta}W_n, n^{\beta}W_n]$. Here $\beta$ is a parameter of our construction and is chosen sufficiently small (such that $n^{\beta}W_n\ll n$ and satisfying several other conditions, but chosen independently of $n$).

For a non-vertical boundary pair $u=(u_1,u_2)\in \Lambda_{i_1}, v=(v_1,v_2) \in \Lambda_{i_2}$ with   $i_1< i_2$ in two separate columns and $|u_2|,|v_2|\le n^{\beta}W_n$, a path $\gamma$ from $u$ to $v$ is called {\bf conforming} if there exists a sequence of times $t_{1},\ldots,t_{i_2-i_1}$ such that 

\begin{itemize}
    \item $t_0=0,t_{i_2-i_1+1}=1$.
    \item For $1\leq i \leq i_2-i_1$ we have $\gamma(t_i)=(in,y_i)$ with $|y_i|\leq n^\beta W_n$.
    \item For $0\leq i \leq i_2-i_1$ we have $\gamma([t_i,t_{i+1}]) \subset \Lambda_i$.
\end{itemize}
The weight of a conforming path is
\[
\rX_\gamma = \inf_{\{t_i\}} \sum_{i} X_{\gamma([t_i,t_{i+1}])}^{\Lambda_i}
\]
where the infimum is over choices of $t_i$ satisfying the conforming property.  The restricted distance $\rX_{uv}$ denotes the infimum over all conforming paths.  There may be multiple paths which achieve the infimum (it is easy to see that the infimum is attained), we will select the topmost path as the canonical choice (by planarity, whenever we have two such paths such that one is not above the other, by  appropriately switching from one path to the other at crossing points, we can construct a weight minimising path that lies above both, and hence the topmost path exists), this will be denoted $\gamma_{uv}$ and called the \emph{conforming geodesic} or simply the geodesic between $u$ and $v$.  When there are multiple choices of $t_i$ that achieve the infimum in the topmost path we will fix the topmost values of $t_i$ to be the \textbf{canonical choice}. For any general conforming path, we shall define the points $t_i$ to be its \emph{canonical} intersection with the line $\{x=in\}$. From now on we shall whenever we refer to \emph{the} point where a geodesic intersects a column boundary, we shall always be referring to the canonical choice of the geodesic and its canonical intersection with the column boundaries.

{
Even with the canonical choice, it is still a downside to the definition of a conforming path that the conforming geodesic may intersect the line $x=in$ for  a segment of positive measure.  We say a path is \emph{strongly conforming} if it intersects each line $x=in$ in (at most) one point.  The following lemma shows that any conforming path can be approximated with arbitrary accuracy be strongly conforming paths, this will be useful in the later constructions and analysis.}

\begin{lemma}\label{l:strongly.conforming}

{
Let $\gamma(t)=(\gamma_1(t),\gamma_2(t))$ be a conforming path such that $\gamma_1(0)<\gamma_1(1)$.  For any $\epsilon>0$ there exists a strongly conforming path $\gamma'(t)=(\gamma'_1(t),\gamma'_2(t))$ such that $\gamma(0)=\gamma'(0),\gamma(1)=\gamma'(1)$ and $\gamma_2(t)=\gamma'_2(t)$ for all $t\in[0,1]$ and
\[
\cX_\gamma\geq \cX_{\gamma'}-\epsilon.
\]
}
\end{lemma}

The proof of this lemma is given in Appendix \ref{s:proxy}.

For the rest of this paper we shall work with the modified distance $\rX_{uv}$. The following lemma guarantees $\rX_{uv}$ is a sufficiently good proxy for $X_{uv}$ so that it suffices to prove the chaos bounds described in the previous section for the restricted distance $\rX_{uv}$.

\begin{lemma}
\label{l:proxy}
There exists $\epsilon>0$ such that for all integers $M\geq 1$ and all $n\geq n(M)$ we have that
\begin{align*}
\P\Big[\sup_{u,v \in [0,nM]\times[-\frac12 n^\beta W_n,\frac12 n^\beta W_n]} |X_{uv}-\rX_{uv}| \geq n^{-\epsilon} Q_n\Big] \leq \exp(-n^{\theta_1}),\\
{\P\Big[\sup_{u,v \in [0,nM]\times[- n^\beta W_n, n^\beta W_n]} X_{uv}-\rX_{uv} \geq n^{-\epsilon} Q_n\Big] \leq \exp(-n^{\theta_1}),}
\end{align*}
where the supremum is taken over all pairs that are not vertical boundary pairs.
\end{lemma} 
Since $cr\le X_r,\rX_r \leq Cr$ deterministically for some $0<c<C<\infty$, it follows that the versions of Theorem \ref{t:all}, Proposition \ref{p:para}, and Lemma \ref{l:paraexplicit} hold at all length scales $r$ between $n$ and $Mn$, with possible different exponents in the stretched exponential tails. As these results will be extensively quoted throughout the remainder of the paper we shall now state them explicitly.

\begin{proposition}
    \label{p:paraestimateconforming}
    There exist $\theta_2>0, C>0$ such that for all $n\le r\le Mn$ that are integer multiples of $n$ and all integers $i$ with $0\le ir\le Mn-r$ we have
    \begin{enumerate}
        \item[(i)] For $j_1<j_2$ with $|j_1W_r|, |j_2W_r|\le \frac{1}{2}n^{\beta}W_n$ and for $z>0$
        $$\P\left(\max_{u\in \ell_{ir, j_1W_r, (j_1+1)W_r}} \max_{v\in \ell_{(i+1)r, (j_2-1)W_r, j_2W_r}} |\rX_{uv}-|u-v||\ge zQ_r \right) \le  \exp(1-Cz^{\theta_2}).$$
        \item[(ii)] For $j_1<j_2$ with $|j_1W_r|, |j_2W_r|\le n^{\beta}W_n$ and for $z>0$
        $$\P\left(\max_{u\in \ell_{ir, j_1W_r, (j_1+1)W_r}} \max_{v\in \ell_{(i+1)r, (j_2-1)W_r, j_2W_r}} \rX_{uv}- |u-v|\le -zQ_r \right) \le  \exp(1-Cz^{\theta_2}).$$
    \end{enumerate}

\end{proposition}

{We shall also need to control fluctuations of the conforming geodesics. To this end we shall need the following versions of Theorem \ref{t:tfold} and Lemma \ref{c:local.trans} for conforming geodesics. Recall the notation $\Upsilon_{n,v_1,v_2,z}$ from Theorem \ref{t:tfold}. For $w\in \R$, $v_1\in \ell_{0,w,W_n}$, $v_2\in \ell_{Mn,w,W_n}$ let $\Upsilon^{w}_{Mn,v_1,v_2,z}$ denote the set of paths 
\[
\Upsilon^{w}_{Mn,v_1,v_2,z}=\bigg\{\gamma':\gamma'(0)=v_1,\gamma'(1)=v_2, \ \sup_t\inf_{x\in[0,Mn]} |\gamma'(t) - (x,w)|\geq zW_{Mn}\bigg\}.
\]
Let $\gamma_{v_1,v_2}$ denote the conforming geodesic from $v_1,v_2$. We have the following analogue of Theorem \ref{t:tfold} for conforming geodesic. 
\begin{lemma}
    \label{l:proxytrans:intro}
    There exists $D>0, \epsilon>0, \theta_2>0, z_0>0$ such that for all integers $M\geq 1$ and all $n\geq n(M)$ and $z\in [z_0,n^{\epsilon}]$, we have for all $w\in [-\frac{1}{2}n^{\beta}W_n,\frac{1}{2}n^{\beta}W_n]$ we have 
    $$\P(\exists v_1\in \ell_{0,w,W_n}, v_2\in \ell_{Mn,w,W_n} : \gamma_{v_1,v_2}\in \Upsilon^{w}_{Mn,v_1,v_2,z})\le \exp(1-Dz^{\theta_2}).$$
\end{lemma}}

Notice that by changing the constants if necessary we can assume that $\theta_2$ in Lemma \ref{l:proxytrans:intro} is the same as $\theta_2$ in Proposition \ref{p:paraestimateconforming}.

Next we shall need a version of the local transversal fluctuation result Lemma \ref{c:local.trans}. This result shows that the the fluctuations of the conforming geodesic $\gamma$ from $(0,0)$ to $(Mn,0)$ at a distance $r$ from either endpoint is of the order $W_r$. Let us define the events 

$$A^{R}_{r,z}=\{\exists (r,s)\in \gamma: |s|\ge zW_r\};$$
$$A^{L}_{r,z}=\{\exists (Mn-r,s)\in \gamma: |s|\ge zW_r\}.$$
We have the following lemma. 

\begin{lemma}
\label{l:localtransproxy:intro}
     There exist $\epsilon>0, \theta_2>0, z_0>0$ such that for all integers $M\geq 1$ and all $n\geq n(M)$ and $z\in [z_0,n^{\epsilon}]$ and $r=kn$ where $1\le k \le M^{1/100}$ we have that
     $$\P(A^R_{r,z}), \P(A^{L}_{r,z})\le  \exp(-z^{\theta_2}).$$
\end{lemma}

We shall also need a stronger version of this local transversal fluctuation estimate which will consider the local transversal fluctuation estimate near the right end point of an initial segment of $\gamma$ and the left endpoint of a terminal segment of $\gamma$; see Lemma \ref{c:localproxytrans}.

Observe that Lemma \ref{l:proxytrans:intro} and Lemma \ref{l:localtransproxy:intro} are not immediate from Lemma \ref{l:proxy} and the corresponding transversal fluctuations results in the original model from \cite{BSS23}. Proofs of these results will be provided in Appendix \ref{s:proxytrans}. 

{We shall also need the constrained passage time estimate for the restricted distance. The following analogue of Proposition \ref{p:constraine} will be proved in Appendix \ref{s:proxytrans}.
\begin{proposition}
    \label{p:constrainerX}
    There exist $C,\theta_6>0$ such that for all $r$ sufficiently large and all $z\ge 0$ we have 
    $$\P\left(\inf_{\substack{\gamma'(0)=(0,0)\\ \gamma'(1)=(r,0)\\ \gamma'\subset [0,r]\times[0,W_r]}} \rX_{\gamma'}\ge r+zQ_r\right)\le \exp(1-Cz^{\theta_6}).$$
\end{proposition}
}

\subsection{Completing the proof of Theorem \ref{t:theorem}}
The main result we need for the restricted distance is the following theorem.

\begin{theorem}\label{t:varianceM.growth}
There exists a constant $M_0$ such that for $M\geq M_0$ and all sufficiently large $n$ (depending on $M$),
\[
\Var(\rX_{Mn}) \leq \frac12 M\Var(\rX_{n}).
\]
\end{theorem}

The rest of the paper is devoted to the proof of Theorem \ref{t:varianceM.growth}. Before proceeding to describe how this is achieved we quickly show how this implies Theorem \ref{t:theorem} as sketched in the previous section.

The following result is immediate from Theorem \ref{t:varianceM.growth} and Lemma \ref{l:proxy}. 

\begin{theorem}\label{t:varianceM.growthoriginal}
There exists a constant $M_0$ such that for $M\geq M_0$ and all sufficiently large $n$,
\[
\Var(X_{Mn}) \leq \frac{3M}{5}\Var(X_{n}).
\]
\end{theorem}

\begin{proof}
    Since for $r=n, Mn$ by Cauchy-Schwarz inequality we have

    \begin{eqnarray*}
       |\Var(X_{r})- \Var (\rX_r)|  &\le & \Var(X_r-\rX_r)+2|\mbox{Cov}(X_r,X_r-\rX_r)|\\
       &\le &\Var(X_r-\rX_r)+2\sqrt{\Var(X_r)\Var(X_r-\rX_r)}
    \end{eqnarray*}
    Since for these two values of $r$, we deterministically have $0\le X_r, \rX_r \le Cr$ for some $C>0$ we have by Lemma \ref{l:proxy}
    $$\Var(X_r-\rX_r)\le \E(X_r-\rX_r)^2 \le n^{-2\epsilon}Q_n^2+C^2M^2n^2\exp(-n^{\theta_1})\le n^{-\epsilon}Q_n^2$$
    by taking $n$ sufficiently large. From Theorem \ref{t:all} it also follows that 
    $\Var(X_r)\le C^2M^2Q_n^2$ for some $C>0$ and hence 
    $\sqrt{\Var(X_r)\Var(X_r-\rX_r)}\le CMn^{-\epsilon/2}Q_n^2\le n^{-\epsilon/4}Q_n^2$ by choosing $n$ sufficiently large. Combining these estimates we get that for $r=n,Mn$
    $$|\Var(X_{r})- \Var (\rX_r)|\le n^{-\epsilon/8}Q_n^2$$
    for all $n$ sufficiently large. The result now follows from Theorem \ref{t:varianceM.growth} and the properties of $Q_n$ listed in Theorem \ref{t:all}.
\end{proof}

We can now complete the proof of Theorem \ref{t:theorem}.

\begin{proof}[Proof of Theorem \ref{t:theorem}]
    Fix $M$ such that the conclusion of Theorem \ref{t:varianceM.growthoriginal} holds for all $n\ge n_0$. It follows from Theorem \ref{t:varianceM.growthoriginal} that for all $k\in \N$ we have 
    
    {    $$\Var X_{M^kn_0} \le \left(\frac{3M}5\right)^k \Var X_{n_0} = M^{k(1-\frac{\log \frac53}{\log M})} \Var X_{n_0} \le C(M^k n_0)^{1-\frac{\log \frac53}{\log M}}$$}
for some $C>0$ where the final inequality comes from $\Var X_n \le Cn$ (Kesten's bound). Therefore, Theorem \ref{t:theorem} holds for all $n=M^kn_0$. For $M^{k}n_0< n < M^{k+1} n_0$ the result then follows from properties of $Q_n$ listed in Theorem \ref{t:all} (and the fact that $M$ is fixed). 
\end{proof}

\subsection{Sketch of the Chaos argument}
The rest of the paper is devoted to the proof of Theorem \ref{t:varianceM.growth} (we also provide the proof of Lemma \ref{l:proxy} and several other auxiliary estimates stated above). For the remainder of this paper (except while we are proving Lemmas \ref{l:proxy}, \ref{l:proxytrans:intro}, \ref{l:localtransproxy:intro}) \textbf{we shall only work with conforming paths and restricted distances and therefore we shall not mention these qualifiers explicitly from now on}. As alluded to above, we shall prove Theorem \ref{t:varianceM.growth} by the \emph{chaos implies superconcentration} principle. We shall show that for any arbitrarily small $\kappa>0$, once the randomness in a $\kappa$ fraction of the $n\times W_n$ boxes $\Lambda_{ij}$ are replaced by an independent copy, then the expected number of such blocks that the geodesic from $(0,0)$ to $(Mn,0)$ passes through both before and after the resampling is $o(M)$. To make a precise statement we need to introduce some notation.  

Recall that we are working with an enhanced random field $\omega_*$ over $\Z\times\R^2$.  Let us also define a second independent copy $\comega$ over $\Z\times\R^2$.  We will be interested in the effect of resampling part of the field $\somega$ and replacing it with $\comega$.  Combined we will write $\bomega = (\somega,\comega)$ as the pair of fields on $\Z\times\R^2$. 

For each $(i,j)\in\Z^2$ we let $T_{ij}$ be IID random variables chosen uniformly in $[0,1]$.  We will partition $\R^2$ into blocks of size $n\times W_n$, we denote $\Lambda_{ij}^+ = \{i\}\times \R\times [(j-1)W_n,jW_n]$. For $t\in (0,1)$, let us define the random field $\omega_t$ by 

\begin{equation}\label{eq:omega.interpolation}
\omega_t(\Lambda_{ij}^+) = \begin{cases}
\somega(\Lambda_{ij}^+) &\hbox{if } T_{ij}\leq 1-t,\\
\comega(\Lambda_{ij}^+) &\hbox{if } T_{ij}> 1-t,
\end{cases} 
\end{equation}
which corresponds to replacing blocks $\Lambda_{ij}^+$ each independently with probability $t$. For some fixed $\kappa$ to be determined later we let $\cX'$ denote the restricted distance with respect to the field $\omega_\kappa$. 

{Note that the field $\omega_{*}$ on $\Lambda_{ij}^+$ only affects the restricted distance within $\Lambda_i$ and so only $\{i\} \times [(i-1)n-1,in+1] \times [(j-1)W_n,jW_n])$  is actually relevant.  Changing the field in $\Lambda_{ij}^+$ only affects paths that pass through the block $[(i-1)n,in]\times [(j-1)W_n-1,jW_n+1]$.}

We shall estimate the variance of $\rX_{Mn}$ by revealing the field block by block and considering the associated Doob martingale. We will let $\cF_t$ be the filtration 
\[
\cF_t=\sigma\bigg(\Big\{ T_{ij},\bomega(\Lambda_{ij}^+):T_{ij}\leq t\Big\}\bigg).
\]
We let $\cU_{ij}$ denote the event that the geodesic (in the $\rX$ distance) $\gamma$ from $\origin$  to $(Mn,0)$ passes within distance 1 of $\Lambda_{ij}:=[(i-1)n,in]\times [(j-1)W_n, jW_n]$ within $\Lambda_i$, i.e.,
\[
\cU_{ij} = \Big\{ d(\gamma\cap \Lambda_i,\Lambda_{ij}) \leq 1 \Big \}.
\]
The value of $\rX_\gamma$ is only affected by the field in $\Lambda_{ij}^+$ if $\gamma$ passes through $[(i-1)n,in]\times [(j-1)W_n-1,jW_n+1]$.  So on the event $\cU_{ij}^c$, replacing $\somega(\Lambda_{ij}^+)$  with $\comega(\Lambda_{ij}^+)$ won't change the value of $\rX_\gamma$.

In order to analyze the Doob martingale $\E[\rX_{Mn}\mid \cF_{t}]$ it will be useful to consider the field with only $\somega(\Lambda_{ij}^+)$ updated.  We define the field $\somega^{ij}$ by
\[
\somega^{ij} (\Lambda_{ij}^+) = \comega(\Lambda_{ij}^+), \qquad \somega^{ij} ((\Lambda_{ij}^+)^c) = \somega((\Lambda_{ij}^+)^c),
\]
and let $\rX^{ij}$ be conforming distances with respect to $\somega^{ij}$.  Note that on the event $\cU_{ij}^c$ we have that
\begin{equation}\label{eq:resampling.decreasing}
\rX^{ij}_{Mn} - \rX_{Mn} \leq 0
\end{equation}
since the optimal conforming path for $\somega$ does not come within distance $1$ of $\Lambda_{ij}$ and so $$\rX_{Mn}=\rX_{\gamma}=\rX^{ij}_{\gamma}\geq \rX^{ij}_{Mn}.$$

{From this point onward we shall work with some fixed but large integer $M$ (that will actually be taken to be a power of a power of $2$ depending on several other parameters to be defined later) and some sufficiently large $n$ depending on all parameters as well as $M$. In all our subsequent estimates, whenever we mention some constants it would be understood that the constants will be independent of $n$ and $M$ but they could depend on the other parameters. This will not be explicitly mentioned each time.}

\subsubsection{Proof of Theorem \ref{t:varianceM.growth}: Chaos Bounds}
In order to establish the variance bound on $\rX_{Mn}$ in Theorem \ref{t:varianceM.growth} we will need two key propositions. The first one is the following. 
\begin{proposition}\label{p:uij.bounds}
There exist constants $C_1,C_2, \theta_3>0$ such that
\begin{align*}
\E\Bigg[ \sum_{i=1}^M \sum_{j=-M}^M  I(\cU_{ij})(\rX^{ij}_{Mn} - \rX_{Mn})^2 \Bigg] &\leq C_1 M \Var(\rX_{n}), \\
\P\Bigg[ \sum_{i=1}^M \sum_{j=-M}^M  I(\cU_{ij})((\rX^{ij}_{Mn} - \rX_{Mn})^+)^4  \geq (C_1 M + z) Q_n^4 \Bigg] &\leq \exp(1-C_2 z^{\theta_3}),\\
\P\bigg[\sum_{i=1}^M \Big(\sum_{j=-M}^M  I(\cU_{ij})\Big)^2 \geq C_1 M + z \bigg] &\leq \exp(1-C_2 z^{\theta_3}).
\end{align*}
\end{proposition}
An immediate corollary is a block version of Kesten's bound, which shows that in going from scale $n$ to scale $Mn$, the variance grows by a factor of $O(M)$.
\begin{corollary}
\label{c:kesten}
For some $C>0$ we have that
\[
\Var(\rX_{Mn}) \leq CM Q_n^2.
\]
\end{corollary}
\begin{proof}
By the Efron-Stein Inequality,
\begin{align*}
\Var(\rX_{Mn}) &\leq   \frac12  \sum_{i=1}^M \sum_{j=-\infty}^\infty  \E[(\rX^{ij}_{Mn} - \rX_{Mn})^2]\\
&= \sum_{i=1}^M \sum_{j=-\infty}^\infty \E[((\rX^{ij}_{Mn} - \rX_{Mn})^+)^2]\\
&= \sum_{i=1}^M \sum_{j=-\infty}^\infty  \E[I(\cU_{ij})((\rX^{ij}_{Mn}-\rX_{Mn})^+ )^2]\\
&\leq CMQ_n^2 + \sum_{i=1}^M \sum_{|j|>M}  \E[I(\cU_{ij})((\rX^{ij}_{Mn}-\rX_{Mn})^+ )^2]
\end{align*}
where the first equality holds by exchangeability of $\rX$ and $\rX^{ij}$, the second equality by~\eqref{eq:resampling.decreasing} and the last inequality follows by the first estimate in Proposition~\ref{p:uij.bounds}.  By Lemma \ref{l:proxytrans} and the fact that $W_{Mn}\ll M^{9/10}W_n$ we have that for $|j|\geq M$,
\begin{equation}\label{eq:trans.Uij}
\P[\cU_{ij}] \leq C\exp(-|j M^{-9/10}|^{\epsilon_1})
\end{equation}
and so
\begin{align}\label{eq:var.decomp.big.trans}
\sum_{i=1}^M \sum_{|j|>M}  \E[I(\cU_{ij})((\rX^{ij}_{Mn}-\rX_{Mn})^+ )^2] 
&\leq  \sum_{i=1}^M \sum_{|j|>M}  \P[I(\cU_{ij})]^{1/2} \Big(16\E(\rX_{Mn}-\E \rX_{Mn})^4]\Big)^{1/2}\nonumber\\
&\leq  \sum_{|j|>M}  C'\exp(-\frac12|j M^{-9/10}|^{\epsilon_1}) Q_n^2 \leq C'' M^{-1} Q_n^2
\end{align}
where the first inequality is by Cauchy-Schwartz and the fact that $(x+y)^4 \leq 8x^4+8y^4$ and the second is that $\E[\E(\rX_{Mn}-\E \rX_{Mn})^4]=O(Q_{Mn}^4)=O(M^3 Q_n)$ by the conditions on $Q_n$ in Theorem~\ref{t:all}. This completes the proof.
\end{proof}

The proof of Proposition \ref{p:uij.bounds}, given in Section \ref{s:1.2proof}, is based on a stretched exponential polymer estimate; see Proposition \ref{p:perc1}. This is more technically complicated than a similar estimate developed in \cite[Proposition 4.1]{BSS23}.

Our ultimate goal will be to show that in fact the growth of the variance is sublinear (i.e., the order $M$ term in Corollary \ref{c:kesten} can be upgraded to $\epsilon M$ for arbitrarily small $\epsilon$) and so we need the following estimate which says that the location of the path is chaotic. 
We show that even for times close to 1, there is still great uncertainty in the location of the path.
\begin{proposition}\label{p:chaos}
For any $\kappa>0,\epsilon>0$ there exist constants $M_0$ such that for $M\geq M_0$, and all $n\ge n_0(M)$
\[
\E\Bigg[ \sum_{i=1}^M \sum_{j=-M}^M  \P[\cU_{ij} | \cF_{1-\kappa}]^2 \Bigg] \leq \epsilon M.
\]
\end{proposition}

We next prove Theorem \ref{t:varianceM.growth} assuming Proposition \ref{p:uij.bounds} and Proposition \ref{p:chaos}.

\begin{proof}[Proof of Theorem \ref{t:varianceM.growth}]
With $C_1$ as in  Proposition~\ref{p:uij.bounds}, set $\kappa=\frac1{\lceil 4C_1\rceil}$ so that
\begin{equation}\label{eq:simple.var.decomp}
\kappa\E\Bigg[ \sum_{i=1}^M \sum_{j=-M}^M  I(\cU_{ij})(\rX^{ij}_{Mn} - \rX_{Mn})^2 \Bigg] \leq \frac14 M \Var(\rX_{n}).
\end{equation}
First let us write
\begin{align*}
\cF_t^{-ij}&=\sigma\bigg(\Big\{ T_{i'j'},\bomega(\Lambda_{i'j'}):T_{i'j'}\leq t(i',j'), \neq (i,j)\Big\}\bigg),\\
\cF_t^{+ij} &= \sigma\bigg( \cF_t^{-ij} \cup \Big \{\bomega(\Lambda_{ij})\Big\}\bigg),
\end{align*}
which corresponds to $\cF_t$ with $\bomega(\Lambda_{ij})$ removed and added respectively and $T_{ij}$ removed. Setting
\[
H^{ij}(t)=\E[\cX_{Mn} \mid \cF_t^{+ij}] - \E[\cX_{Mn} \mid \cF_t^{-ij}]
\]
which is the change in the conditional expectation if we added information of block $(i,j)$ at time $t$.
Now the Doob martingale $\E[\cX_{Mn} \mid \cF_t]$ is a jump process which changes at times $\{T_{ij}\}$ with jumps $H^{ij}(T_{ij})$.
We can write its total variance as the expected squared sum of the jumps so,
\begin{align}
\Var(\rX_{Mn}) 
& =\sum_{i=1}^M \sum_{j=-\infty}^{\infty} \E\Big[\big(H^{ij}(T_{ij})\big)^2 \Big].
\end{align}
Now let
\[
\widetilde{G}^{ij}(t)=\E[(\rX^{ij}_{Mn}-\rX_{Mn}) \mid \cF_{t}^{+ij}]
\]
and
\[
G^{ij}(t)=\E[I(\cU_{ij})(\rX^{ij}_{Mn}-\rX_{Mn})^+ \mid \cF_{t}^{+ij}].
\]
Observe now that 
 \[
  \widetilde{G}^{ij}(t) = (\E[\rX^{ij}_{Mn} \mid \cF_{t}^{+ij}] - \E[\cX_{Mn} \mid \cF_t^{-ij}])-(\E[\rX_{Mn} \mid \cF_{t}^{+ij}] - \E[\cX_{Mn} \mid \cF_t^{-ij}])
 \]
and notice that the right hand side is the difference of two random variables with the same law as $H^{ij}(t)$ which are conditionally independent given $\cF_t^{-ij}$. It therefore follows that 
$$ \E[(\widetilde{G}^{ij}(t))^2]= 2\E [H^{ij}(t)]^2.$$
Notice next that by exchangeability of $\rX^{ij}_{Mn}$ and $\rX_{Mn}$ it also follows that 

\begin{align*}
\E[(\widetilde{G}^{ij}(t))^2]
&= \E\left[ \big(\E[(\rX^{ij}_{Mn}-\rX_{Mn})^{+}\mid \cF_{t}^{+ij}]-\E[(\rX^{ij}_{Mn}-\rX_{Mn})^{-}\mid \cF_{t}^{+ij}]\big)^2\right]\\
&\leq 2\E\left[ \big(\E[(\rX^{ij}_{Mn}-\rX_{Mn})^{+}\mid \cF_{t}^{+ij}]\big)^2+\big(\E[(\rX^{ij}_{Mn}-\rX_{Mn})^{-}\mid \cF_{t}^{+ij}]\big)^2\right]\\
&= 4\E\left[ \big(\E[(\rX^{ij}_{Mn}-\rX_{Mn})^{+}\mid \cF_{t}^{+ij}]\big)^2\right]\\
&= 4\E\left[ \big(\E[I(\cU_{ij})(\rX^{ij}_{Mn}-\rX_{Mn})^{+}\mid \cF_{t}^{+ij}]\big)^2\right]=4\E[(G^{ij}(t))^2]
\end{align*}
Using \eqref{eq:resampling.decreasing},  it follows that
\[
\E[(G^{ij}(t))^2] \geq 2\E[(H^{ij}(t))^2].
\]
Since  $G^{ij}(t)$ and $H^{ij}(t)$ are independent of $T_{ij}$ we also have that
\[
\E[(G^{ij}(T_{ij}))^2] \geq 2\E[(H^{ij}(T_{ij}))^2].
\]
Hence,
\begin{align}\label{eq:var.decompA}
\Var(\rX_{Mn}) 
& \leq 2\sum_{i=1}^M \sum_{j=-\infty}^{\infty} \E\Big[\big(G^{ij}(T_{ij})\big)^2 \Big]\nonumber\\
& \leq2\sum_{i=1}^M \sum_{j=-M}^{M} \E\Big[\big(G^{ij}(T_{ij})\big)^2 \Big]+ 2\sum_{i=1}^M \sum_{|j|>M} \E\Big[I(\cU_{ij})((\rX^{ij}_{Mn}-\rX_{Mn})^+)^2  \Big]\nonumber\\
& \leq 2\sum_{i=1}^M \sum_{j=-M}^M \E\Big[\big(G^{ij}(T_{ij})\big)^2 \Big] + 2C'' M^{-1} Q_n^2
\end{align}
where the first inequality follows from Jensen's Inequality and the second is by equation~\eqref{eq:var.decomp.big.trans}. Now,
\begin{align}\label{eq:cond.exp.time.comp}
\E\Big[\big(G^{ij}(T_{ij})\big)^2 \Big]
&=\int_0^1 \E\Big[\big(G^{ij}(t)\big)^2 \Big] dt\leq \kappa\E\Big[\big(G^{ij}(1)\big)^2 \Big] + (1-\kappa)\E\Big[\big(G^{ij}(1-\kappa)\big)^2 \Big]\nonumber\\
&= \kappa \E\Big[\big(I(\cU_{ij})(\rX^{ij}_{Mn}-\rX_{Mn})^+ \big)^2 \Big] \nonumber\\
&\qquad +(1-\kappa)\E\Big[\E[I(\cU_{ij})(\rX^{ij}_{Mn}-\rX_{Mn})^+ \mid \cF_{1-\kappa}^{+ij}] ^2 \Big],
\end{align}
where the first equality follows from the fact that $T_{ij}$ is independent of $G^{ij}(t)$ and the first inequality is due to $G^{ij}(t)$ being a martingale so its second moment is increasing.  Also, for any random variable $Y$ measurable with respect to $\bomega$,
{\begin{align*}
\E\Big[\big(\E[Y \mid \cF_{1-\kappa}] \big)^2 \Big] &
= (1-\kappa)\E\Big[\big(\E[Y \mid \cF_{1-\kappa}^{+ij}] \big)^2 \Big] + \kappa \E\Big[\big(\E[Y \mid \cF_{1-\kappa}^{-ij}] \big)^2 \Big]
\end{align*}}
and so
\begin{align}\label{eq:cond.exp.remove.ij}
(1-\kappa)\E\Big[\big(\E[I(\cU_{ij})(\rX^{ij}_{Mn}-\rX_{Mn})^+ \mid \cF_{1-\kappa}^{+ij}] \big)^2 \Big] &
\leq \E\Big[\big(\E[I(\cU_{ij})(\rX^{ij}_{Mn}-\rX_{Mn})^+ \mid \cF_{1-\kappa}] \big)^2 \Big].
\end{align}
Hence by equations~\eqref{eq:simple.var.decomp}, \eqref{eq:var.decompA}, \eqref{eq:cond.exp.time.comp} and \eqref{eq:cond.exp.remove.ij} we have that,
\begin{align}\label{eq:var.decompB}
\Var(\rX_{Mn}) \leq \frac14 M \Var(\rX_{n}) + C'' M^{-1} Q_n^2 + \sum_{i=1}^M \sum_{j=-M}^M \E\Big[\big(\E[I(\cU_{ij})(\rX^{ij}_{Mn}-\rX_{Mn})^+ \mid \cF_{1-\kappa}] \big)^2 \Big]  .
\end{align}
Now
\begin{align}\label{eq:var.decompC}
&\sum_{i=1}^M \sum_{j=-M}^M \E\Big[\big(\E[I(\cU_{ij})(\rX^{ij}_{Mn}-\rX_{Mn})^+ \mid \cF_{1-\kappa}] \big)^2 \Big]\nonumber\\
&\qquad\leq \sum_{i=1}^M \sum_{j=-M}^M \E\Big[\E[I(\cU_{ij})\big((\rX^{ij}_{Mn}-\rX_{Mn})^+ \big)^2 \mid \cF_{1-\kappa}]\E[I(\cU_{ij}) \mid \cF_{1-\kappa}] \Big]\nonumber\\
&\qquad = \sum_{i=1}^M \sum_{j=-M}^M \E\Big[I(\cU_{ij})\big((\rX^{ij}_{Mn}-\rX_{Mn})^+ \big)^2 \P[\cU_{ij} \mid \cF_{1-\kappa}]  \Big]\nonumber\\
&\qquad \leq  \E\Bigg[\bigg(\sum_{i=1}^M \sum_{j=-M}^M I(\cU_{ij})\big((\rX^{ij}_{Mn}-\rX_{Mn})^+ \big)^4\bigg)^{1/2}  \bigg(\sum_{i=1}^M \sum_{j=-M}^M \P[\cU_{ij} \mid \cF_{1-\kappa}] ^2\bigg)^{1/2}  \Bigg]\nonumber\\
&\qquad \leq  \E\Bigg[\bigg(\sum_{i=1}^M \sum_{j=-M}^M I(\cU_{ij})\big((\rX^{ij}_{Mn}-\rX_{Mn})^+ \big)^4\bigg)  \bigg(\sum_{i=1}^M \sum_{j=-M}^M \P[\cU_{ij} \mid \cF_{1-\kappa}] ^2\bigg)  \Bigg]^{1/2}
\end{align}
where the first inequality is by Cauchy-Schwartz, the equality is by the tower property of conditional expectation, the second inequality is another application of Cauchy-Schwartz and the third is by Jensen's inequality.
By the tail bound from Proposition~\ref{p:uij.bounds} we have for some $C_3>0$,
\begin{equation}
\E\Bigg[\bigg(\bigg(\sum_{i=1}^M \sum_{j=-M}^M I(\cU_{ij})\big((\rX^{ij}_{Mn}-\rX_{Mn})^+ \big)^4\bigg) - C_3 M Q_n^4\bigg)^+  \Bigg] \leq M^{-3} Q_n^4
\end{equation}
and so
\begin{align}\label{eq:var.decompD}
&\E\Bigg[\bigg(\sum_{i=1}^M \sum_{j=-M}^M I(\cU_{ij})\big((\rX^{ij}_{Mn}-\rX_{Mn})^+ \big)^4\bigg)  \bigg(\sum_{i=1}^M \sum_{j=-M}^M \P[\cU_{ij} \mid \cF_{1-\kappa}] ^2\bigg)  \Bigg]\nonumber\\
&\qquad \leq \E\Bigg[C_3 M Q_n^4 \bigg(\sum_{i=1}^M \sum_{j=-M}^M \P[\cU_{ij} \mid \cF_{1-\kappa}] ^2\bigg)  \Bigg]\nonumber\\
&\qquad \qquad + \E\Bigg[\bigg(\bigg(\sum_{i=1}^M \sum_{j=-M}^M I(\cU_{ij})\big((\rX^{ij}_{Mn}-\rX_{Mn})^+ \big)^4\bigg) - C_3 M Q_n^4\bigg)^+  3M^2 \Bigg]\nonumber\\
&\qquad \leq C_3 M Q_n^4 \E\bigg[ \sum_{i=1}^M \sum_{j=-M}^M \P[\cU_{ij} \mid \cF_{1-\kappa}] ^2\bigg] + 3M^{-1} Q_n^4
\end{align}
where for the first inequality we have used in the second term that $\P(\cU_{i,j}\mid \cF_{1-\kappa})\le 1$ and there are at most $(2M+1)M$ terms in the sum. Combining equations~\eqref{eq:var.decompB}, \eqref{eq:var.decompC} and \eqref{eq:var.decompD} we have that
\begin{align*}
\Var(\rX_{Mn}) 
\leq \frac14 M \Var(\rX_{n}) + C'' M^{-1} Q_n^2 + \Bigg( C_3 M Q_n^4 \E\bigg[ \sum_{i=1}^M \sum_{j=-M}^M \P[\cU_{ij} \mid \cF_{1-\kappa}] ^2\bigg] + 3M^{-1} Q_n^4 \Bigg)^{1/2}.
\end{align*}
By Proposition~\ref{p:chaos} (applied for $\epsilon$ sufficiently small depending on $C_3$), for all large enough $M$ we have that
\[
\Var(\rX_{Mn}) 
\leq \frac12 M \Var(\rX_{n})
\]
completing the proof.
\end{proof}

It now remains to prove the two key estimates: Proposition \ref{p:uij.bounds} and Proposition \ref{p:chaos}. In the remainder of this subsection we shall give a high level overview of the proof of Proposition \ref{p:chaos} which is the most technically challenging part of this paper. 

\subsubsection{Argument for Proposition \ref{p:chaos}: multi-scale construction}
Recall the random field $\omega_{\kappa}$ constructed in \eqref{eq:omega.interpolation} from $\omega_*$ by replacing the randomness in each block $\Lambda^{+}_{ij}$ independently with probability $\kappa$ by the corresponding randomness in $\comega$. Recall also that we denoted by $\rX'$ the restricted distance functions with respect to the field $\omega_{\kappa}$. We shall, many times through the course of this paper, consider pairs of events which are the same except that one is defined for passage times $\rX$ and the other using the passage times $\rX'$. In such cases, if we denote the former event by $\cA$, we shall use $\cA'$ to denote the latter event. Recall the event $\cU_{ij}$. Let us consider the corresponding event $\cU'_{ij}$, i.e., 
\[
\cU'_{ij} = \Big\{ d(\gamma'\cap \Lambda_i,\Lambda_{ij}) \leq 1 \Big \}
\]
where $\gamma'$ is the conforming geodesic for the restricted distance $\rX'_{Mn}$. Observe now that conditional on $\cF_{1-\kappa}$, $\cU_{ij}$ and $\cU'_{ij}$ are conditionally i.i.d.\, and therefore 

$$\P[\cU_{ij}\cap \cU'_{ij}]=\E[\P[\cU_{ij}\mid \cF_{1-\kappa}]^2].$$
It follows that the left hand side in the inequality in the statement of Proposition \ref{p:chaos} equals 
$$\sum_{i=1}^{M}\sum_{j=-M}^{M} \P[\cU_{ij}\cap\cU'_{ij}].$$
It, therefore, follows that proving Proposition \ref{p:chaos} is equivalent to showing a block version of \emph{disorder chaos} for the conforming geodesic. To show this, we shall treat the values of $i$ close to $0$ and $M$ differently from the values of $i$ in the bulk. It is not particularly difficult to show that (see Lemma \ref{l:chaosbasic2}) given $\epsilon,\kappa>0$ and $M$ sufficiently large, we have for all $n$ sufficiently large 
\begin{equation}
    \label{eq:edge1}
\sum_{i=1}^{2M^{99/100}}\sum_{j} \P[\cU_{ij}\cap\cU'_{ij}]\le \epsilon M/4;
\end{equation}

\begin{equation}
    \label{eq:edge2}
\sum_{i=M-2M^{99/100}}^{M}\sum_{j} \P[\cU_{ij}\cap\cU'_{ij}]\le \epsilon M/4.
\end{equation}
The most difficult part of the construction is to deal with the bulk values of $i$ between $2M^{99/100}$ and $M-2M^{99/100}$. To do this we need a multi-scale argument and look at the paths at many different intermediate length scales $1\ll r \ll M$. 

\noindent
\textbf{Parameters of the multi-scale construction.}
In the set-up of Proposition \ref{p:chaos}, let us fix $\epsilon>0, \kappa\in (0,1)$. Without loss of generality we shall assume $\kappa^{-1}$ is an integer. We shall also choose $M$ sufficiently large such that $\log_2\log_2 M$ is an integer and set 

\begin{equation}
    \label{eq:scale1}
    \Phi=2^{(\log_2\log_2 M)^{5}}.
\end{equation}
We will shall work at a series of scales 

\begin{equation}
    \label{eq:scale2}
   r_\ell = \Phi^\ell n= 2^{\ell(\log_2 \log_2 M)^5}n 
\end{equation}
for $\ell=0,1,\ldots, \ell_{\max}=\frac1{100}\lfloor \frac{\log_2 M}{(\log_2 \log_2 M)^5} \rfloor$. Later, for each length scale $r_{\ell}$, we will define a collection of events, each of which implicitly depends on $n,M,\ell$ but for notational simplicity we will suppress this most of the times. To denote the position at which the geodesic enters the $i$th column at scale $r_\ell$ we set 

\begin{equation}
    \label{eq:jdefn}
J_i^{n,M,\ell}=\lfloor y_{i\Phi^\ell}/W_{r_\ell} \rfloor,
\end{equation}
where $y_i$ is the (canonical) intersection of the geodesic $\gamma$ and the line $x=in$.

\begin{center}
\begin{figure}[htbp!]
\includegraphics[width=5in]{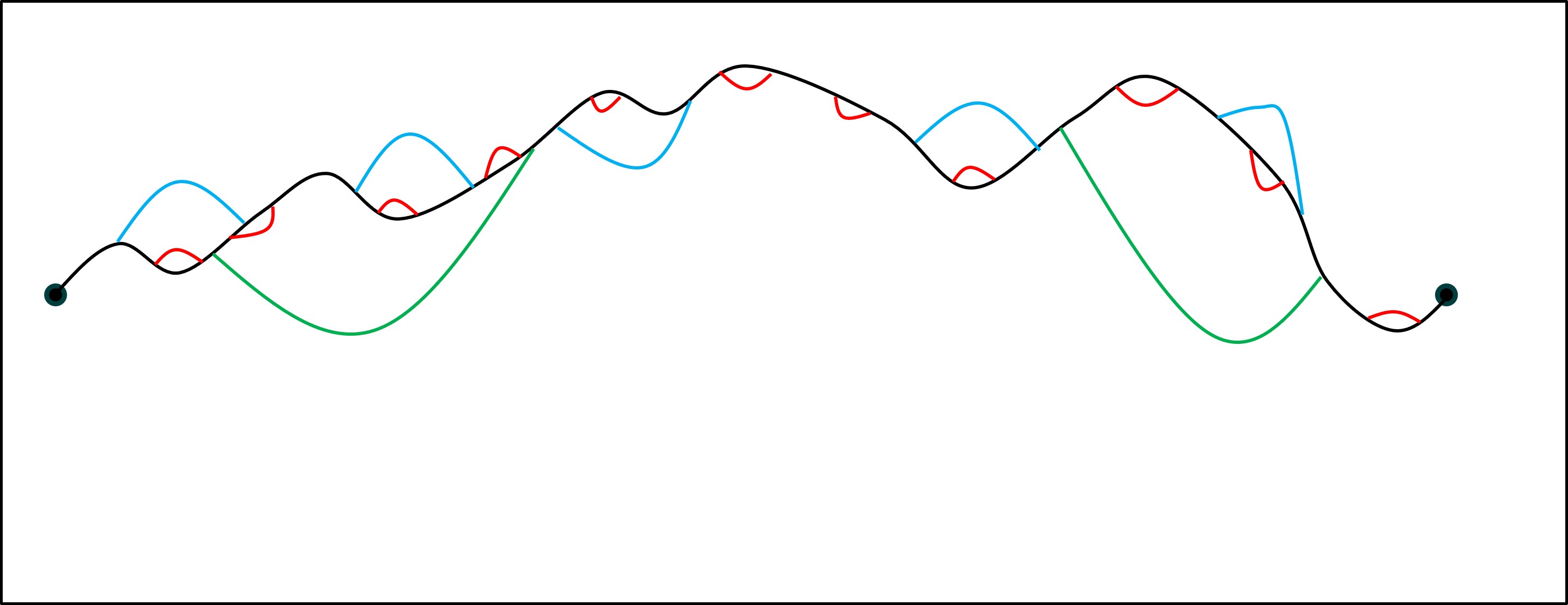}
\caption{Construction of alternative paths at different scales. Proposition \ref{p:chaos} states that once a small fraction of the $n\times W_n$ blocks have their randomness resamples, the expected number of blocks that intersect the geodesic from $\bf 0$ to $(Mn,0)$ both before and after resampling is $o(M)$; this essentially boils down to showing that on most vertical columns of width $n$, the geodesics before and after resampling are separated. To show the latter fact, we prove Theorem \ref{t:Pplusintro} which shows that at every given length scale with large probability there are a constant fraction ($\delta_0$) of locations where the geodesics before and after resampling are separated. In the figure, the black path shows the original geodesic whereas the red, green, and blue paths show alternative paths at different scales which are better then the original geodesic (or any other path passing close to the original geodesic at that location) after the resampling, thus guaranteeing the geodesics before and after are separated at that location. Starting from the largest scale $\ell_{\max}$, (which corresponds to columns of width $n\Phi^{\ell_{\max}}\approx M^{1/100}n$) we find a constant fraction of columns where geodesics (before and after resampling) are separated; then we find another constant fraction of columns at then next scale, and so on, eventually showing that geodesics are separated at most columns since $\ell_{\max}$ was chosen appropriately large.}
\label{f:multiscale}
\end{figure}
\end{center}

Fix $0\le \ell \le \ell_{\max}$. For all $i$ such that $2M^{99/100}n\le ir_{\ell} \le (M-2M^{99/100})n$ and $j$ such that $jW_{r_{\ell}}\in [-2MW_{n},2MW_{n}]$ we shall define an event $\cP_{i,j}^{n,M,\ell}$. The definition of the event $\cP_{i,j}$ spans several pages and is not suitable to elaborate on here.  The key idea is that the event $\cP_{i,j}$ creates two different highways such that if the optimal path passes through the region then it will take one highway before the resampling and a different highway, well separated from the first, after resampling. As a consequence, on the event $\cP_{i,J_{i}}^{n,M,\ell}\cap \{|J_{i}^{n,M,\ell}|\le M^{8/10}\}$ we have that 
$$\sum_{i': ir_{\ell}\le i'n \le (i+1)r_{\ell}}\sum_{j} I(\cU_{i'j}\cap \cU'_{i'j})=0,$$
that is, this event ensures that on the $i$-th column at scale $r_{\ell}$ (which corresponds to the columns indexed by $i'$ satisfying $ir_{\ell}\le i'n \le (i+1)r_{\ell}$ at scale $r_0=n$) the geodesic $\gamma$ before the resampling (i.e., with respect to the distance function $\rX$) and the geodesic $\gamma'$ after the resampling (i.e., with respect to the distance function $\rX'$) do not pass within distance $1$ of the same blocks $\Lambda_{ij}$. The main technical estimate regarding these events is that for each $\ell \le \ell_{\max}$, the event that $P^{n,M,\ell}_{i,J^{n,M,\ell}_i}$ occurs at a positive fraction of consecutive locations everywhere in the bulk with large probability (i.e., the complement of this event has probability going to $0$ as a large negative power of $M$). We now state this result. 

\begin{theorem}\label{t:Pplusintro}
There exists $\delta_0$ and $M_0$ such that for all $M\geq M_0$ and all $n$ sufficiently large and $\ell\le \ell_{\max}$ and $2M^{99/100}n\leq  ir_\ell \leq (M-2M^{99/100})n-\Phi$, 
\[
\P\Bigg[\sum_{i'=i}^{i+\Phi-1} I(\cP_{i',J^{n,M,\ell}_i}^{n,M,\ell}, |J^{n,M,\ell}_{i'}| \leq M^{8/10}) \leq \delta_0\Phi \Bigg] \leq M^{-90}.
\]
\end{theorem}
The number $\delta_0$ will be a function of several other parameters involved in the construction of the event $\cP_{i,j}$ and a more precise version of the above result is stated in Theorem \ref{t:Pplus} later. Given Theorem \ref{t:Pplusintro}, completing the proof of Proposition \ref{p:chaos} is not difficult; by taking a union bound over all values of $\ell\le \ell_{\max}$ we get that the geodesics before and after resampling can pass through the same blocks only at $(1-\delta_0)^{\ell_{\max}}$ fraction of columns at length scale $n$; see Figure \ref{f:multiscale}. By local transversal fluctuation estimates (Lemma \ref{l:localtransproxy:intro}), it is also easy to show that a geodesic is unlikely to pass through too many blocks in a single column. Combining all these, and using that $\ell_{\max}$ is chosen suitably large one gets (see Lemma \ref{l:chaosbasic}) that 
\[
\sum_{i=1}^M \sum_{j=-M}^M  \P[\cU_{ij},\cU_{ij}'] \leq \exp\Big(-\frac{\log M}{(\log_2 \log_2 M)^6}\Big) M.
\]
This completes the proof of Proposition \ref{p:chaos} as discussed above. 

Proof of Theorem \ref{t:Pplusintro} is the most technically challenging part of this paper. Without going into further details about the definition of $\cP_{i,j}$, let us just mention that $\cP_{i,j}$ consists of several different parts. The sub-events can roughly be divided into three categories. First, the likely events; we use a percolation argument to show that it is highly likely that these likely events occur at most locations along the geodesic (see Lemma \ref{l:Pminus.perc}). The second part consists of monotone events, which even though unlikely can be shown to occur at a constant fraction of locations along the geodesic by using the FKG inequality. The third type of event deals with the existence of an alternative path away from the geodesic in the environment before the resampling which becomes the geodesic after updating the randomness in each block independently with probability $\kappa$; this event is neither likely nor monotone and needs to be handled by a delicate resampling analysis. Combining these different ideas and analysis we eventually establish Theorem \ref{t:Pplusintro}.

\subsection{Outline of the rest of the paper}
The rest of this paper is primarily devoted to the proofs of Proposition \ref{p:uij.bounds} and Proposition \ref{p:chaos} (as well as the proofs of some auxiliary results such as Lemma \ref{l:proxy}) and is organised as follows. We first start with the proof of Proposition \ref{p:uij.bounds} in Section \ref{s:1.2proof}. The arguments here depend on a general polymer estimate Proposition \ref{p:perc1} and its consequences proved later in Appendix \ref{s:perc}. We next move on to the proof of Proposition \ref{p:chaos}. Following the strategy outlined above we first define a number of events at different length scales and state their probability bounds in Section \ref{s:events}. The two main results of this section are Theorem \ref{t:Pplus} (which is the more precise version of Theorem \ref{t:Pplusintro} stated above) and Lemma \ref{l:separatedfirst} which states that on the event $\cP_{i,J_i}$, the geodesics before and after the resampling are separated at location~$i$. Assuming these two results, the proof of Proposition \ref{p:chaos} is completed in Section \ref{s:chaos.estimate}. Section \ref{s:separation} is purely deterministic and furnishes the proof of Lemma \ref{l:separatedfirst}. The next two sections deal with the proof of Theorem \ref{t:Pplus}. Section \ref{s:likely} deals with the likely events in the definition of $\cP_{i,j}$ while the delicate analysis of the other events is carried out via a resampling argument in Section \ref{s:glauber}.
Section \ref{s:letter} provides the probability estimates  for the events defined in Section \ref{s:events}. The last three sections are appendices that provide the proofs of several auxiliary estimates. Proofs of Lemma \ref{l:proxy} showing that restricted passage times approximate the passage times well  
is provided in Section \ref{s:proxy}. This proof requires Proposition \ref{p:constraine} whose proof is also provided in the same section. Section \ref{s:proxytrans} is devoted to results about the transversal fluctuations of conforming geodesics and proves Lemma \ref{l:proxytrans:intro}, Lemma \ref{l:localtransproxy:intro}, and Proposition \ref{p:constrainerX}. Finally, Section \ref{s:perc} provides the proof of the polymer estimate Proposition \ref{p:perc1} and its several consequences.

{\textbf{Interdependence of Sections.}
The finals three sections of the paper (Appendices \ref{s:proxy}, \ref{s:proxytrans} and \ref{s:perc}) are self-contained in the sense that they only depend on results from \cite{BSS23} and can be read linearly, they do not require any results from Sections \ref{s:1.2proof} to \ref{s:letter}, while these seven sections use results from the final three sections (as well as other results from \cite{BSS23}). 
}

\bigskip

\section{Proof of Proposition \ref{p:uij.bounds}}

\label{s:1.2proof}
This section is devoted to the proof of the first of the two remaining major estimates, namely Proposition \ref{p:uij.bounds}, which also has the shorter proof. Recall that Proposition \ref{p:uij.bounds} has three parts. All three parts involve tail estimates of summing certain quantities over the blocks $(i,j)$ such that $\cU_{ij}$ holds, that is blocks within distance $1$ of the {conforming geodesic} from $\bf0$ to $(Mn,0)$. Proofs of all three estimates are via an abstract stretched exponential polymer estimate which shows that a polymer moving in field of weights with stretched exponential tail also has stretched exponential tail. We state this result first.

\subsection{Stretched exponential polymer estimate}

Denote a set of length $M$ {integer sequence} by
\[
\mathfrak{K}_M=\{(k_0,\ldots,k_M)\in\Z^{M+1}: k_0=0\}.
\]
and set  
\[
\tau_2(\uk):=\sum_{i=1}^M |k_i-k_{i-1}|^2.
\]
We have the following general maximization estimate over sums indexed by paths.

\begin{proposition}
\label{p:perc1}
For $1\le i\le M, k,k'\in \Z$, let $\mathcal{V}_{i,k,k'}$ be random variables such that for some $0<\xi<1$ and for some $0<\delta<1/100$, we have
\[
\P[\mathcal{V}_{i,k,k'} \geq z]\leq C_1\exp\Big(-C_2 \big(z/(1+|k-k'|^{1+\delta})\big)^{\xi}\Big)
\]

Assume further that the collections of variables $\mathfrak{Z}_i=\{\mathcal{V}_{i,k',k'}\}_{k,k'}$ are independent for different~$i$. Then there exist $C_3,C_4,C_5,C_6$ depending on $C_1,C_2$ and $\xi,\delta$ but not depending on $M$ such that for all $R\geq 1$ and $z>0$ we have
\[
\P\left[\max_{\substack{\uk\in \mathfrak{K}_{M}\\ \tau_2(\uk)\leq R M}} \sum_{i=1}^M \mathcal{V}_{i,k_{i-1}, k_{i}} \geq (C_3+ C_4R^{3/4})M + z \right] \leq C_5\exp\left(-C_6 z^{\xi/4}\right ).
\]
\end{proposition}

The proof of Proposition \ref{p:perc1} is provided in Appendix \ref{s:perc}. The way we apply the above proposition will be the following. For the conforming geodesic $\gamma$ (we shall keep using the notation $\gamma$ for this geodesic) from $\bf0$ to $(Mn,0)$ we define $$J_{i}=\lfloor y_i/W_n \rfloor$$ where $(in,y_i)$ is the {canonical} point where the geodesic $\gamma$ intersects the line $\{x=in\}$ (note that this is the special case $\ell=0$ of the notation $J^{n,M,\ell}_{i}$ introduced earlier). For a general conforming path $g$ from $\bf0$ to $(Mn,0)$ we shall define $\uk=\uk(g)$ by setting $k_i=\lfloor\frac{y_i}{W_n}\rfloor$ where $(in, y_i)$ is the canonical point where $g$ intersects the line $\{x=in\}$. We shall bound sums of quantities along locations on the geodesic by summing over the locations given by $\uk$ and maximizing over a large class of $\uk$ such that it is unlikely that $\uk(\gamma)$ is not contained in this class. 

To implement this strategy we shall also need to bound 
$$\tau_2(\gamma)=\sum_{i=1}^{M} |J_i-J_{i-1}|^2$$
which measures the fluctuation of the geodesic. The following result will also be proved in Section~\ref{s:perc} using Proposition \ref{p:perc1}. 

\begin{proposition}
    \label{p:tau2perc}
    There exist $C_7,c, \theta_4>0$ such that for all $z\ge 0$, we have 
    $$\P\left(\sum_{i=1}^{M} (J_{i}-J_{i-1})^2\ge C_7M+z\right)\le \exp(1-cz^{\theta_4}).$$
\end{proposition}

An estimate quite similar to Proposition \ref{p:perc1} was established in \cite[Proposition 4.1]{BSS23} but with a couple of crucial differences. In \cite[Proposition 4.1]{BSS23}, the tails of $\mathcal{V}_{i,k,k'}$ did not depend of $|k-k'|$ whereas in Proposition \ref{p:perc1} the tail estimate gets worse as $|k-k'|$ grows. This necessitates us having to maximize over a restricted class of $\uk$ with bounded $\tau_2$ fluctuation whereas in \cite[Proposition 4.1]{BSS23} the maximum was over the class of $\uk$ with $$\tau_1(\uk):=\sum_{i} |k_i-k_{i-1}| \le RM.$$
Since 
$$M \left(\sum_{i} (J_{i}-J_{i-1})^2 \right)\ge \left(\sum_{i} |J_{i}-J_{i-1}| \right)^2$$ Proposition \ref{p:tau2perc} implies that there exists $R>0$ such that for all $z>0$ we have 
\begin{equation}
    \label{e:tau1est}
    \P\left(\sum_{i} |J_{i}-J_{i-1}|\ge RM+z\right)\le \exp(1-cz^{\theta_4}).
\end{equation}
As the reader will observe below, \cite[Proposition 4.1]{BSS23} will suffice for some of our estimates in the proof of Proposition~\ref{p:uij.bounds} below but the stronger Proposition~\ref{p:perc1} (and its consequence Proposition \ref{p:tau2perc}) will be required in a number of other estimates.

\subsection{Proof of the second estimate}

{Out of the three estimates in Proposition \ref{p:uij.bounds}, the second one is the most complicated, so we shall start with that. The other two estimates will follow from similar but somewhat simpler arguments.}

Observe that every block $(i,j)$ for $j$ between $J_{i-1}$ and $J_i$ must have $I(\cU_{ij})=1$. In addition there can be also be blocks outside this interval with $I(\cU_{ij})=1$. The latter blocks will be referred to as \emph{overhangs} and for all the estimates we shall bound the contribution of the blocks between $J_{i-1}$ and $J_{i}$ and the overhang blocks separately; see Figure \ref{f:multiscaleP}.
\begin{center}
\begin{figure}[htbp!]
\includegraphics[width=6in]{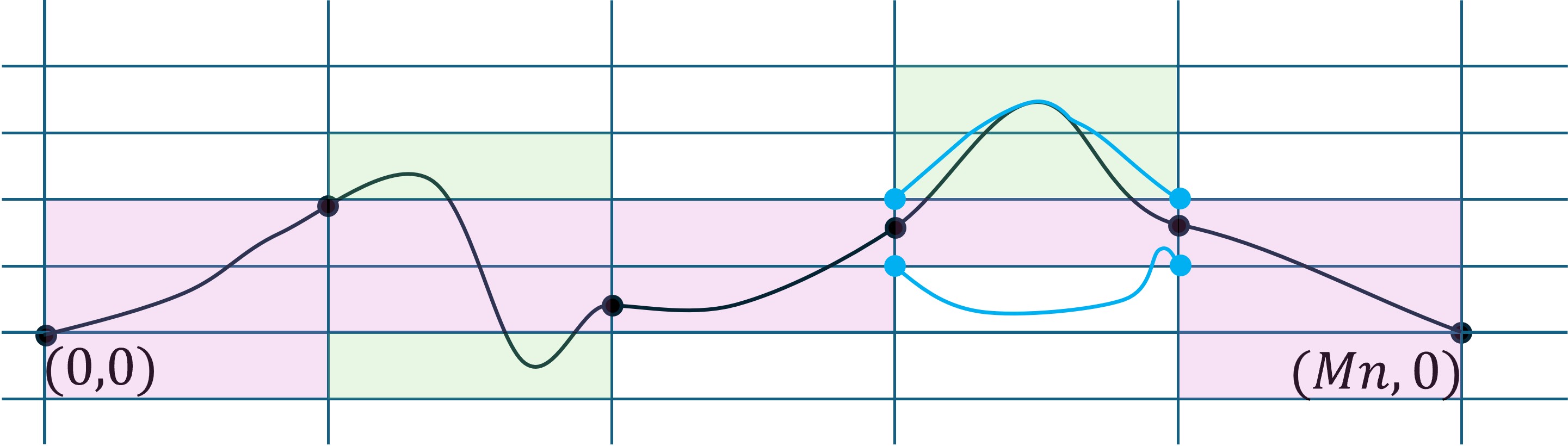}
\caption{Building blocks of the multiscale estimate: the plane is divided into boxes $\Lambda_{i,j}=[(i-1)n,in]\times[(j-1)W_n,jW_n]$. The event that the geodesic comes within distance $1$ of $\Lambda_{ij}$ is denoted $\cU_{ij}$. Proposition \ref{p:uij.bounds} has various estimates showing certain weighted sums of $I(\cU_{ij})$ cannot be too large. To this end we denote the $j$-index of the block where the geodesic $\gamma$ depicted in the picture enters the column $\Lambda_i$ by $J_{i-1}$. Therefore all the blocks $\Lambda_{ij}$ with $j$ between $J_{i-1}$ and $J_{i}$ will witness $\cU_{ij}$ (these blocks are marked in pink in the figure). However, there  may also other other blocks in $\Lambda_i$ where $\cU_{ij}$ holds. These blocks, called overhang blocks, are marked in green in the figure. The contributions from these two types of blocks are controlled separately by observing that the segment of the geodesics within a column remains sandwiched between the two blue paths marked in the figure. The number of overhang blocks can then be controlled just by looking at the transversal fluctuations of the blue paths.}
\label{f:multiscaleP}
\end{figure}
\end{center}

For convenience of notation let us define 
$$Y_{i,j}=((\rX^{ij}_{Mn}-\rX_{Mn})^+)^{4} I(-M\le j\le M).$$
Recall that we want to show  
$$\P\left(\sum_{i=1}^{M}\sum_{j\in \Z} I(\cU_{ij})Y_{i,j} \ge (CM+z)Q_n^4\right)\le \exp(1-cz^{\theta_3})$$
for some $\theta_3>0$. Notice that trivially

$$\sum_{i=1}^{M}\sum_{j\in \Z} I(\cU_{ij})Y_{i,j}= O(M^{6}n^{4})$$
and therefore, by adjusting the choice of $\theta_3$ if needed and choosing $n\ge M$, it suffices to prove the tail bound for values of $z\le n^{\delta}$ for some $\delta>0$ (where $\delta$ will be chosen sufficiently small compared to $\beta$). From now on we shall work with this choice of $z$ only. 

Notice next that by Lemma \ref{l:proxytrans:intro} and by our choice of $z$ it follows that with probability at least $1-\exp(-z^{\theta_3})$ (for $\theta_3$ sufficiently small) we have $|J_i|\le n^{\delta}$ for all $i$. We shall therefore work on this event from now on. 

To control the overhang blocks, we make the following definitions: for $1\le i\le M$ and $j\in \Z$, consider points $u_{i,j}=(in,(j+1)W_n)$. Define $S^{+}_{i,j}$ to be the maximum index $j'$ such that the geodesic $\gamma_{ij}$ from $u_{i-1,j+1}$ to $u_{i,j+1}$ comes within distance 1 of $\Lambda_{i,j'}$. Similarly, let $S^{-}_{i,j}$ denote the minimum index $j'$ such that the geodesic from $u_{i-1,j}$ to $u_{i,j}$ comes within distance 1 of $\Lambda_{i,j'}$. 

Notice that by planarity and ordering of geodesics $\gamma$ must lie above $\gamma_{i,J_{i-1}\wedge J_{i}}$ and below $\gamma_{i,J_{i-1}\vee J_{i}}$. It therefore follows that the set $\{j: I(\cU_{ij}=1)\}$ is contained in the interval $[S^{-}_{i, J_{i-1}\wedge J_i}, S^{+}_{i,J_{i-1}\vee J_{i}}]$. We shall therefore divide this set into three parts by considering its intersections with $[J_{i-1}\wedge J_i, J_{i-1}\vee J_i]$, $[S^{-}_{i, J_{i-1}\wedge J_i},J_{i-1}\wedge J_i]$ and $[J_{i-1}\vee J_i, S^{+}_{i,J_{i-1}\vee J_{i}}]$ where the last two are the overhang blocks. The following lemma, which is an immediate consequence of Lemma \ref{l:proxytrans}, will be used to control the sums over overhang blocks. 

\begin{lemma}
    \label{l:overhang}
    There exists $\theta'>0$ such that for each $i\in [1,M]$ and $k\in [-n^{\delta},n^{\delta}]$ and for all $z$ sufficiently large we have 
    $$\P(S^{-}_{i,k}\le k-z)\le \exp(1-z^{\theta'});$$
    $$\P(S^{+}_{i,k}\ge k+z)\le \exp(1-z^{\theta'}).$$
\end{lemma}

We shall divide the sum $\sum_{j\in \Z} I(\cU_{ij})Y_{i,j}$ into three parts. Observe that

$$\sum_{j\in \Z} I(\cU_{ij})Y_{i,j}\le \sum_{j=S^{-}_{i,J_{i-1}\wedge J_i}}^{J_{i-1}\wedge J_i-1}Y_{i,j}+ \sum_{j=J_{i-1}\wedge J_i}^{J_{i-1}\vee J_i}Y_{i,j}+ \sum_{j=J_{i-1}\vee J_i+1}^{S^{+}_{i,J_{i-1}\vee J_i}}Y_{i,j}.$$

For brevity of notation let us denote the three terms in the right hand side above by $A_{i}, B_{i}$ and $C_{i}$ respectively. The second inequality in Proposition \ref{p:uij.bounds} follows from the next three lemmas which control  $\sum A_{i}, \sum B_{i}$ and $\sum C_{i}$ respectively.  

\begin{lemma}
    \label{l:uij1}
    There exist constants $C,c, \theta_3>0$ such that for $z\in (0,n^{\delta})$ we have 
$$\P\left(\sum_{i} A_{i}\ge (CM+z)Q_n^{4}\right)\le \exp(1-cz^{\theta_3}).$$
\end{lemma}

\begin{lemma}
    \label{l:uij2}
    There exist constants $C,c, \theta_3>0$ such that for $z\in  (0,n^{\delta})$ we have 
$$\P\left(\sum_{i} B_{i}\ge (CM+z)Q_n^{4}\right)\le \exp(1-cz^{\theta_3}).$$
\end{lemma}

\begin{lemma}
    \label{l:uij3}
    There exist constants $C,c, \theta_3$ such that for $z\in (0,n^{\delta})$ we have 
$$\P\left(\sum_{i} C_{i}\ge (CM+z)Q_n^4\right)\le \exp(1-cz^{\theta_3}).$$
\end{lemma}

The proofs of Lemma \ref{l:uij1} and Lemma \ref{l:uij3} are essentially identical, so we shall only provide a proof for the first one. For this we shall need to bound the individual summands in $A_i$. We first make the following notation. 

For $i\in [1,M]$, and $k,k'\in \Z$ let us define 
$$Z_{i,k,k'}= \sum_{j=S^-_{i,k}}^{k-1} \max_{u\in \ell_{(i-1)n,kW_{n},(k+1)W_n},\atop v\in \ell_{in,k'W_{n},(k'+1)W_n}} ((\rX^{ij}_{uv}-\rX_{uv})^+)^{4} I(-M\le j \le M)$$
if $k\le k'$ and 
$$Z_{i,k,k'}= \sum_{j=S^-_{i,k'}}^{k'-1} \max_{u\in \ell_{(i-1)n,kW_{n},(k+1)W_n},\atop v\in \ell_{in,k'W_{n},(k'+1)W_n}} ((\rX^{ij}_{uv}-\rX_{uv})^+)^{4}I(-M\le j \le M)$$
if $k>k'$. The next lemma controls the tails of $Z_{i,k,k'}$ which will be necessary to apply Proposition~\ref{p:perc1} later.

\begin{lemma}
    \label{l:uij1aux}
    There exist $c, \theta'>0$ such that for $i\in [1,M], k,k'\in \Z$ with $|k|\vee |k'|\le n^{\delta}$ and for all  $z\ge 0$ 
    $$\P(Z_{i,k,k'}\ge zQ_n^4)\le \exp(1-cz^{\theta'}).$$
\end{lemma}

\begin{proof}
    Let us only consider the case $k\le k'$. The proof in the other case is identical. 
    
    Let $0<\alpha'<1$ be fixed. Observe that 
    $$\P(Z_{i,k,k'}\ge zQ_n^4)\le \P\left(\sum_{j=k-z^{\alpha'}}^{k-1} \max_{u\in \ell_{(i-1)n,kW_{n},(k+1)W_n}, \atop v\in \ell_{in,k'W_{n},(k'+1)W_n}} (\rX^{ij}_{uv}-\rX_{uv})^{4}I(-M\le j\le M) >zQ_n^{4}\right)+\P(S^-_{i,k}\le k-z^{\alpha'}).$$
    Using Lemma \ref{l:overhang}, it suffices to bound only the first term. Notice now that by Proposition \ref{p:paraestimateconforming}, for each {$j\in [k-z^{\alpha'},k-1]\cap [-M,M]$} (here we use the a priori bound on $k,k'$) we have for for some $c>0$ and for all $z'>0$
    $$\P\left(\max_{u\in \ell_{(i-1)n,kW_{n},(k+1)W_n}, \atop v\in \ell_{in,k'W_{n},(k'+1)W_n}} \Big|\rX^{ij}_{uv}-\mu|u-v|\Big| \ge z'Q_n\right)\le \exp(1-c(z')^{\theta_2});$$
    $$\P\left(\max_{u\in \ell_{(i-1)n,kW_{n},(k+1)W_n}, \atop v\in \ell_{in,k'W_{n},(k'+1)W_n}} \Big|\rX_{uv}-\mu|u-v|\Big| \ge z'Q_n\right)\le \exp(1-c(z')^{\theta_2}).$$
    It therefore follows that 
    $$\P\left(\max_{u\in \ell_{(i-1)n,kW_{n},(k+1)W_n}, \atop v\in \ell_{in,k'W_{n},(k'+1)W_n}} (\rX^{ij}_{uv}-\rX_{uv})^{4} >z'Q_n^4 \right) \le \exp(1-c(z')^{\theta_2/4})$$
    for all $z\ge 0$ and some $c>0$. 
    Finally, using the above and a union bound we get
    $$\P\left(\sum_{j=k-z^{\alpha'}}^{k-1} \max_{u\in \ell_{(i-1)n,kW_{n},(k+1)W_n}, \atop v\in \ell_{in,k'W_{n},(k'+1)W_n}} (\rX^{ij}_{uv}-\rX_{uv})^{4}I(-M\le j\le M) >zQ_n^{4} \right)\le z^{\alpha}\exp(1-c(z^{(1-\alpha)\theta_2/4})).$$
    It follows from Lemma \ref{l:overhang} that for some $\theta'$
    $$\P(Z_{i,k,k'}\ge zQ_n^4)\le z^{\alpha}\exp(1-c(z^{(1-\alpha)\theta'}))+ \exp(1-z^{\theta'}).$$
     The proof of the lemma is completed by choosing $\theta_3$ sufficiently small. 
\end{proof}

We are now ready to prove Lemma \ref{l:uij1}.

\begin{proof}[Proof of Lemma \ref{l:uij1}]
Notice first that on the event 
$\{J_{i-1}=k_{i-1}\}\cap \{J_{i}=k_{i}\}$ there must exist $u\in  \ell_{(i-1)n,k_{i-1}W_{n},(k_{i-1}+1)W_n}$ and $v\in \ell_{in,k_iW_{n},(k_i+1)W_n}$ such that $u,v\in \gamma$. It follows that 
$$ \rX^{ij}_{Mn}-\rX_{Mn} \le \rX^{ij}_{uv}-\rX_{uv}.$$
It therefore follows that for all $i,j$, on the event above
$$Y_{i,j}\le \max_{u\in \ell_{(i-1)n,k_{i-1}W_{n},(k_{i-1}+1)W_n} \atop v\in \ell_{in,k_iW_{n},(k_i+1)W_n}} ((\rX^{ij}_{uv}-\rX_{uv})^+)^{4} I(-M\le j \le M).$$ Summing now over $j$ from $j\in [S^{-}_{i,J_{i-1}\wedge J_{i}},J_{i-1}\wedge J_{i}-1]$, it follows that 
$$A_{i}\le Z_{i,J_{i-1},J_{i}}.$$
It therefore suffices to show that 
$$\P\left(\sum_{i=1}^{M} Z_{i,J_{i-1},J_{i}} \ge (CM+z)Q_n^4 \right)\le \exp(1-cz^{\theta_3}).$$

Clearly, by Proposition \ref{p:tau2perc} we can choose $R$ sufficiently large such that it suffices to show that for $\theta_3$ sufficiently small 
$$\P\left(\max_{\uk \in \mathfrak{K}_{M}, \tau_2(\uk)\le RM+z } \sum_{i} Z_{i,k_{i-1},k_i} \ge (CM+z)Q^4_{n}\right) \le \exp(1-cz^{\theta_3}).$$

Recall that we only need to prove the lemma for $z\le n^{\delta}$. Notice that this a priori upper bound on $z$ implies that (for $\delta$ sufficiently small) that for each $\uk \in \mathfrak{K}_{M}$ with $\tau_2(k)\le RM+z$ we have $\max |k_i|\le n^{\delta}$ and hence $Z_{i,k_{i-1},k_{i}}$ satisfy the conclusion of Lemma \ref{l:uij1aux} (for $n$ sufficiently large), which further implies that $Q_n^{-4}Z_{i,k,k'}$ satisfy the weaker tail estimates in the hypothesis of Proposition \ref{p:perc1}. Observe also that the family of random variables $Z_{i,k,k'}$ are independent across different $i$. Applying Proposition \ref{p:perc1} to the random variables $Q_n^{-4}Z_{i,k,k'}$ we get for some $C,C'c, \theta_3>0$

$$\P\left(\max_{\uk \in \mathfrak{K}_{M}\atop \tau_2(\uk)\le RM+z} \sum_{i} Z_{i,k_{i-1},k_i} \ge ((C+C'(R+z/M)^{3/4})M+ z)Q^4_{n}\right) \le \exp(1-cz^{\theta'}).$$
Observe now that 
$$((C+C'(R+z/M)^{3/4})M+ z)\le C'_1M+C'_2z$$ for some $C'_1,C'_2>0$ (depending on $C,C',R$ but not depending on $M$ or $z$) therefore we get that 
$$\P\left(\max_{\uk \in \mathfrak{K}_{M}\atop \tau_2(\uk)\le RM+z} \sum_{i} Z_{i,k_{i-1},k_i} \ge (CM+z)Q^4_{n}\right) \le \exp(1-c(z/C'_2)^{\theta_3}),$$
as required. This completes the proof of the lemma. 
\end{proof}

We now move towards the proof of Lemma \ref{l:uij2}. The argument is similar to the proof of Lemma \ref{l:uij1}. For $i\in [1,M]$, and $k,k'\in \Z$ let us define 
$$\widetilde{Z}_{i,k,k'}= \sum_{j=k\wedge k'}^{k\vee k'} \max_{u\in \ell_{(i-1)n,kW_{n},(k+1)W_n},\atop v\in \ell_{in,k'W_{n},(k'+1)W_n}} ((\rX^{ij}_{uv}-\rX_{uv})^+)^{4} I(-M\le j \le M).$$

Similar to Lemma \ref{l:uij1aux}, we have the following lemma to bound the tails of $\widetilde{Z}_{i,k,k'}$.

\begin{lemma}
    \label{l:uij2aux}
    Let $\epsilon>0$ be fixed but arbitrarily small. There exist $c,c', \theta_3>0$ such that for $i\in [1,M], k,k'\in \Z$ with $|k|\vee |k'|\le n^{\delta}$ and for all  $z\ge 0$ 
    $$\P(\widetilde{Z}_{i,k,k'}\ge zQ_n^4)\le (1+|k-k'|)\exp \left(1-c'(\frac{z}{1+|k-k'|})^{\theta_3}\right)\wedge 1\le \exp \left(1-c(\frac{z}{1+|k-k'|^{1+\epsilon}})^{\theta_3}\right).$$
\end{lemma}

\begin{proof}
     Observe that the term $(1+|k-k'|)\exp\left(1-c(\frac{z}{1+|k-k'|^{1+\epsilon}})^{\theta_3}\right)$ is smaller than $1$ only if {$z\gg 1+|k-k'|^{1+\epsilon}$}. The second inequality in the statement follows from this. Therefore it suffices to prove only the first inequality. 

    By definition of $\widetilde{Z}_{i,k,k'}$, and arguing as in the Proof of Lemma \ref{l:uij1aux} it follows that 
    $$\P(\widetilde{Z}_{i,k,k'}\ge zQ_n^4) \le (1+|k-k'|) \max_{j\in [k\wedge k', k\vee k']} \P\left(  \max_{u\in \ell_{(i-1)n,kW_{n},(k+1)W_n},\atop v\in \ell_{in,k'W_{n},(k'+1)W_n}} ((\rX^{ij}_{uv}-\rX_{uv})^+)^{4} \ge \frac{z Q^4_n}{1+|k-k'|}\right).$$
    Therefore it suffices to prove that there exists $c',\theta_3>0$ such that for all $z'>0$, and for all $k,k'$ as in the statement of the lemma and $j\in [k\wedge k',k\vee k']$ we have
 \begin{equation}
     \label{e:uij21}
     \P\left(  \max_{u\in \ell_{(i-1)n,kW_{n},(k+1)W_n},\atop v\in \ell_{in,k'W_{n},(k'+1)W_n}} ((\rX^{ij}_{uv}-\rX_{uv})^+)^{4} \ge z'Q_n^4\right) \le \exp(1-c'(z')^{\theta_3}).
 \end{equation}
    
    As in the proof of Lemma \ref{l:uij1aux}, we have from Proposition \ref{p:paraestimateconforming} that for all $i,k,k',j$ as above we have for some $c'>0$ and $\theta_2$ as in Proposition \ref{p:paraestimateconforming}
     $$\P\left(\max_{u\in \ell_{(i-1)n,kW_{n},(k+1)W_n}, \atop v\in \ell_{in,k'W_{n},(k'+1)W_n}} |\rX^{ij}_{uv}-\mu|u-v| \ge z'Q_n\right)\le \exp(1-c(z')^{\theta_2});$$
    $$\P\left(\max_{u\in \ell_{(i-1)n,kW_{n},(k+1)W_n}, \atop v\in \ell_{in,k'W_{n},(k'+1)W_n}} |\rX_{uv}-\mu|u-v| \ge z'Q_n\right)\le \exp(1-c(z')^{\theta_2}).$$
    We therefore get the \eqref{e:uij21} with $\theta_3=\theta_2/4$. This completes the proof of the lemma. 
\end{proof}

Using the above lemma, the proof of Lemma \ref{l:uij2} is similar to the proof of Lemma \ref{l:uij1}.

\begin{proof}[Proof of Lemma \ref{l:uij2}]
Arguing as in the proof of Lemma \ref{l:uij1}, we get that for all $i,j$, on the event $\{J_{i-1}=k_i\}\cap \{J_i=k_i\}$ we have 
$$Y_{i,j}\le \max_{u\in \ell_{(i-1)n,k_{i-1}W_{n},(k_{i-1}+1)W_n} \atop v\in \ell_{in,k_iW_{n},(k_i+1)W_n}} ((\rX^{ij}_{uv}-\rX_{uv})^+)^{4} I(-M\le j \le M).$$
 Summing now over $j\in [J_{i-1}\wedge J_{i}, J_{i-1}\vee J_{i}]$, it follows that 
$$B_{i}\le \widetilde{Z}_{i,J_{i-1},J_{i}}.$$
For $R$ as in Proposition \ref{p:tau2perc} it follows therefore that it suffices to show that 

$$\P\left(\max_{\uk \in \mathfrak{K}_{M}\atop \tau_2(k)\le RM+z} \sum_{i} \widetilde{Z}_{i,k_{i-1},k_i} \ge (CM+z)Q^4_{n}\right) \le \exp(1-cz^{\theta'}).$$

Recall that it suffices to prove the lemma for $z\le n^{\delta}$. Notice that this a priori upper bound on $z$ implies that (for $\delta$ sufficiently small) that for each $\uk \in \mathfrak{K}_{M}$ with $\tau_2(k)\le RM+z$ we have $|k_{i-1}|\vee |k_i|\le n^{\delta}$ and hence $\widetilde{Z}_{i,k_{i-1},k_{i}}$ satisfy the conclusion of Lemma \ref{l:uij2aux}. Observe also that the family of random variables $\widetilde{Z}_{i,k,k'}$ are independent across different $i$. The remainder of the argument apply Proposition \ref{p:perc1} to $Q_n^{-4}\widetilde{Z}_{i,k,k'}$  identically to the proof of Lemma \ref{l:uij1} and we omit the details. 
\end{proof}

As already mentioned, the proof of Lemma \ref{l:uij3} is identical to that of Lemma \ref{l:uij1} and hence is omitted. Combining Lemmas \ref{l:uij1}, \ref{l:uij2} and \ref{l:uij3}, the proof of the second estimate in Proposition \ref{p:uij.bounds} is completed. \qed

\subsection{Proof of the third estimate} Proof of the third estimate in Proposition \ref{p:uij.bounds} is similar to the second one, in fact, slightly simpler. Recalling the definition of $S^{\pm}_{i,j}$, observe that for $i\in [1,M]$

$$\sum_{j=-M}^{M} I(\cU_{ij}) \le S^{+}_{i,J_{i-1}\vee J_{i}}-S^{-}_{i,J_{i-1}\wedge J_{i}}+1.$$ 
Since $S^{+}_{i,J_{i-1}\vee J_{i}}-S^{-}_{i,J_{i-1}}$ is a nonnegative integer it follows that 

$$\left(\sum_{j=-M}^{M} I(\cU_{ij})\right)^2\le 3(S^{+}_{i,J_{i-1}\vee J_{i}}-S^{-}_{i,J_{i-1}})^2+1.$$
Summing over $i$ and using $(x+y+z)^2\le 3(x^2+y^2+z^2)$ yields

$$\sum_i \left(\sum_{j=-M}^{M} I(\cU_{ij})\right)^2 \le M+ 9\left(\sum_i (S^{+}_{i,J_{i-1}\vee J_{i}}-J_{i-1}\vee J_{i})^2+\sum_{i} (J_{i}-J_{i-1})^2+\sum_{i} (S^{-}_{i,J_{i-1}\wedge J_{i}}-J_{i-1}\wedge J_{i})^2  \right). $$

Therefore, the third estimate in Proposition \ref{p:uij.bounds} follows from the next two lemmas together with Proposition \ref{p:tau2perc}. 

\begin{lemma}
    \label{l:uij3rd1}
     There exist constants $C,c, \theta'>0$ such that for $z\ge 0$ we have 
$$\P\left({\sum_{i} (S^{+}_{i,J_{i-1}\vee J_{i}}-J_{i-1}\vee J_{i})^2}\ge (CM+z)\right)\le \exp(1-cz^{\theta'}).$$
\end{lemma}

\begin{lemma}
    \label{l:uij3rd2}
     There exist constants $C,c, \theta'>0$ such that for $z\ge 0$ we have 
$$\P\left(\sum_{i} (S^{-}_{i,J_{i-1}\wedge J_{i}}-J_{i-1}\wedge J_{i})^2\ge (CM+z)\right)\le \exp(1-cz^{\theta'}).$$
\end{lemma}

The proofs of the two lemmas above are essentially identical so we shall only focus on the proof of the first one.

\begin{proof}[Proof of Lemma \ref{l:uij3rd1}]
Notice first that
$$\sum_i \left(\sum_{j=-M}^{M} I(\cU_{ij})\right)^2$$
is deterministically $O(M^3)$ and therefore it suffices to prove it for $z=O(M^3)$.

Define $Z_{i,k,k'}=(S^{+}_{i,k\vee k'}-k\vee k')^2$. Arguing as in the proof of the second estimate we get that it suffices to upper bound

$$\P\left(\max_{\uk \in \mathfrak{K}_{M}\atop \tau_2(\uk)\le RM+z}\sum_i Z_{i,k_{i-1},k_i} \ge CM+z\right)+\P\left(\sum_{i}(J_{i}-J_{i-1})^2\ge RM+z\right).$$
Applying Proposition \ref{p:tau2perc}, we  fix $R$ (independent of $M$ and $z$) such that the second term above is upper bounded by $\exp(1-z^{\theta_3})$ and therefore it suffices to show that 
$$\P\left(\max_{\uk \in \mathfrak{K}_{M}\atop \tau_2\uk)\le RM+z}\sum_i Z_{i,k_{i-1},k_i} \ge CM+z\right)\le \exp(1-cz^{\theta_3}).$$

Observe also that  $Z_{i,k,k'}$ is independent across $i$. Further, since $z=O(M^3)$ it follows that $\max|k_i|=O(M^3)$ for all $\uk$ as above. Therefore, Lemma
\ref{l:overhang} ensures that for all $\uk$ as above, the $Z_{i,k,k'}$ satisfy the hypothesis of Proposition~\ref{p:perc1}. By Proposition \ref{p:perc1} we therefore have
  $$\P\left(\max_{\uk \in \mathfrak{K}_{M}\atop \tau_2(\uk)\le RM+z} \sum_{i} Z_{i,k_{i-1},k_i} \ge CM+z \right)\le \exp(1-cz^{\theta_3})$$
  and this completes the proof.
\end{proof}

\subsection{Proof of the first estimate}
By exchangeability of $\rX$ and $\rX^{ij}$ it follows that 
$$\E\Bigg[ \sum_{i=1}^M \sum_{j=-M}^M  I(\cU_{ij})(\rX^{ij}_{Mn} - \rX_{Mn})^2 \Bigg] = 2\E\Bigg[ \sum_{i=1}^M \sum_{j=-M}^M  I(\cU_{ij})((\rX^{ij}_{Mn} - \rX_{Mn})^{+})^2 \Bigg].$$ 
It therefore suffices to prove the following stronger result: There exist $C,c,\theta_3>0$ such that for $z\ge 0$ we have 

\begin{equation}
    \label{e:vartail}
    \P\Bigg[ \sum_{i=1}^M \sum_{j=-M}^M  I(\cU_{ij})((\rX^{ij}_{Mn} - \rX_{Mn})^{+})^2 \ge (CM+z)Q_n^2 \Bigg] \le \exp(1-cz^{\theta'}),
\end{equation}

Proof of \eqref{e:vartail} is very similar to the proof of the second estimate in Proposition \ref{p:uij.bounds} and we omit the details to avoid repetitions. \qed

\section{Multi-scale argument for Proposition \ref{p:chaos}: Definition of events}
\label{s:events}

The next few sections are devoted to the proof of Proposition \ref{p:chaos} that was sketched in Section \ref{s:prelim}. We shall fix a constant $\kappa\in (0,1)$. The reader can easily check that if the conclusion of Proposition \ref{p:chaos} holds for some value of $\kappa$, then it holds for all larger values of $\kappa$, therefore it suffices to prove the proposition under the assumption that $\kappa^{-1}$ is an integer. We shall henceforth assume so as it will be technically convenient. 

Assume that $\log_2 \log_2 M$ is a large integer (the reader can easily check that the argument goes through for general $M$ with minor modifications, but for the purpose of the proof of Theorem \ref{t:theorem} having one $M$ is enough) and let $$\Phi=2^{\ell(\log_2 \log_2 M)^5}.$$ We will define events on a series of scales 
$$r_\ell = \Phi^\ell n= 2^{(\log_2 \log_2 M)^5}n$$ 
for integer values of $\ell \in [0,\ell_{\max}]$ where $$\ell_{\max}= \frac1{100}\bigg\lfloor \frac{\log_2 M}{(\log_2 \log_2 M)^5} \bigg\rfloor.$$ 
From now on, we shall work with a fixed $\ell$ in the above range. 

We shall work with vertical columns of width $r_{\ell}$. For $\ell=0$ these will correspond to the same columns $\Lambda_i$ we had previous defined. Recall from Section \ref{s:prelim} that we set  
\[
J_i^{n,M,\ell}=\lfloor y_{i\Phi^\ell}/W_{r_\ell} \rfloor
\]
to denote the location at which the geodesic $\gamma$ (from $\bf0$ to $(Mn,0)$) exits the $i$th column at scale $r_\ell$. 
Note again that for $\ell=0$ this equals $J_i$ considered previously. 

In the following subsection we will define a collection of events, each of which is implicitly depends on $n,M,\ell$ but for notational simplicity we will suppress this and drop the corresponding subscripts and superscripts (in particular, by a slight abuse of notation $J_i$ would refer to $J_i^{n,M,\ell}$ when it is clear from the context that we are dealing with a fixed $\ell$).  As we shall only be working with a fixed value of $\ell$ (except in Section~\ref{s:chaos.estimate} where we will be considering multiple $\ell$ simultaneously and where the dependence will be made explicit) there will be no scope for confusion.

\subsection{Elementary Events}
\label{s:elementary1}
The events we will need to prove Proposition \ref{p:chaos} are rather complicated but they are made up of a number of elementary events dealing with passage times of paths across rectangles and parallelograms. We first define these events and give estimates of their probabilities. The constants involved in the probability estimates for events in this section will not depend on $M,n$ or $\ell$. We shall not mention this explicitly each time. If not specified otherwise we shall also assume that all such stated estimates at scale $r=r_{\ell}$ work for all $i\in [1,Mn/r]$ and for all $j\in [-MW_n/W_r,MW_n/W_r]$ (in many cases the estimates will hold for a larger range of $j$). The proofs of the estimates in this subsection are provided in Section \ref{s:letter}.

\noindent
{\bf Horizontal lines with well behaved passage times to the sides.}

The first estimate asks that passage times from the side of the $i$-th column to points on the line $y=jW_r$ are not too small.  We define
\begin{align*}
\cK_{i,j,z} &= \bigg\{ \inf_{\substack{x\in[(i-1)r,ir]| \\ |y|\leq {n^{\beta}W_n}}} \inf_{\substack{\gamma'(0)=((i-1)r,y)) \\ \gamma'(1)=(x,jW_r)}} \rX_{\gamma'} - (x-(i-1)r) - \frac12\Big(\frac{|y-jW_r| }{W_r} - 1\Big)^2 Q_r \geq  -z Q_r \bigg\}\\
&\cap \bigg\{ \inf_{\substack{x\in[(i-1)r,ir]| \\ |y|\leq {n^{\beta}W_n}}} \inf_{\substack{ \gamma'(0)=(x,jW_r) \\ \gamma'(1)=(ir,y)) }} \rX_{\gamma'} - (ir-x) - \frac12\Big(\frac{|y-jW_r| }{W_r} - 1\Big)^2 Q_r \geq  -z Q_r \bigg\}
\end{align*}

and let
\[
\cK^*_{i,j,z,w} = \cK_{i,j-w,z} \cap \cK_{i,j,z} \cap \cK_{i,j+w,z}.
\]
\begin{lemma}\label{l:cK.bound}
There exists $C,\theta_5>0$, not depending on $n,M,\ell$ such that for all $i$ and {$|j|\leq MW_n/W_r$} {and all $z\ge 0$}
\[
\P[\cK_{i,j,z}] \geq 1 - \exp(-C z^{\theta_5}).
\]
\end{lemma}

\noindent
{\bf Side to side passage times for paths in a rectangle.}
To define this event it would be useful to introduce a centered version of the conforming passage times. Instead of centering by the Euclidean distance between the end point it will be more convenient (when optimizing over paths) by an approximation of the same.

For conforming paths $\gamma'$ with $\gamma'(0)=(in,y),\gamma'(1)=(i'n,y')$ with $i<i'$ integers, we define
 \[
 \hrX_{\gamma'} = \rX_{\gamma'} - (i'-i)n - \frac12 \frac{(y-y')^2}{(i'-i)n} .
 \]
 The quadratic term corresponds to the second order term from Taylor Series expansion of the Euclidean distance using Pythagoras' Theorem. For $u=\gamma'(0),v=\gamma'(1)$ we set 

 $$\hrX_{uv}=\inf_{\gamma':\gamma'(0)=u, \gamma'(1)=v} \hrX_{\gamma'}.$$

We define 

\begin{align*}
\cI^+_{i,j,j',z} &= \bigg\{\inf_{y,y'\in[jW_r,j' W_r]} \inf_{\substack{\gamma' \subset  [(i-1)r,ir]\times [jW_r,j' W_r] \\ \gamma'(0)=((i-1)r,y) \\ \gamma'(1)=(ir,y')}} \hrX_{\gamma'}  \geq  z Q_r \bigg\};\\
\cI^-_{i,j,j',z} &= \bigg\{\inf_{y,y'\in[jW_r,j' W_r]} \inf_{\substack{ \gamma' \subset  [(i-1)r,ir]\times [jW_r,j' W_r] \\ \gamma'(0)=((i-1)r,y) \\ \gamma'(1)=(ir,y')}} \hrX_{\gamma'}  \leq  z Q_r \bigg\}.
\end{align*}

For large $z$, the event $\cI^+$ says that all the paths in this rectangle have larger length than typical and hence the rectangle acts as a \emph{barrier} for the geodesic, whereas for large negative $z$ the event $\cI^-$ guarantees the existence of an unusually good path across the rectangle. Both these events occur with positive probability.

\begin{lemma}\label{l:cI.bound}
There exists $C,\theta_5>0$, not depending on $n,M,\ell$ such that for all $i$ and $|j|,|j'|\leq M\frac{W_n}{W_r}$ and $z\geq 0$
\[
\P[\cI^+_{i,j,j',-z}] \geq 1 - (1\vee|j-j'|)^2\exp(-C z^{\theta_5});
\]
\[
\P[\cI^-_{i,j,j',z}] \geq 1 -\exp(-C z^{\theta_5});
\]
For any $t$ and $z\geq 0$ there exists $\delta(t,z) >0$ such that if $j'\leq j+t$
\[
\P[\cI^+_{i,j,j',z}] \geq \delta.
\]
For any $z\geq 0$ there exists $\delta'>0$ such that for all $n$ sufficiently large and all $j'\geq j+1$,
\[
\P[\cI^-_{i,j,j',-z}] \geq \delta'.
\]

\end{lemma}
\noindent 
{\bf Lower bound on passage times of paths with vertical change.} Define
\begin{align*}
\cJ_{i,j,j',z,w} &= \bigg\{\inf_{\substack{x,x' \in [(i-1)r,ir]\\ y,y'\in[jW_r,j' W_r] \\ |y-y'|\geq w W_r}} \inf_{\substack{\gamma' \subset  [(i-1)r,ir]\times [jW_r,j' W_r] \\ \gamma'(0)=(x,y) \\ \gamma'(1)=(x',y')}} \rX_{\gamma'} - |x-x'|  \geq  z Q_r \bigg\}
\end{align*}
{where the first infimum is taken over all pairs of points except the vertical boundary pairs.}

\begin{lemma}\label{l:cJ.bound}
For any $t,w > 0$ and $z\geq 0$ there exists $\delta(t,z) >0$ not depending on $n,M,\ell$ such that if $j\leq j'\leq j+t$
\[
\P[\cJ_{i,j,j',z,w}] \geq \delta.
\]
\end{lemma}

\noindent
{\bf Typical passage times for the full column.} We set
\begin{align*}
\cA_{i,j,z}^- &= \Bigg\{\sup_{\substack{|y|,|y'| \leq MW_n \\ u=((i-1)r,y) \\ v= (ir,y')}} \hrX_{uv} - \frac{|y - jW_r| + |y' - jW_r|}{W_r}Q_r \leq z Q_r \Bigg\};\\
\cA_{i,j,z}^+ &= \Bigg\{\inf_{\substack{|y|,|y'| \leq n^\beta W_n \\ u=((i-1)rn,y) \\ v= (ir,y')}} \hrX_{uv}  + \frac{|y - jW_r| + |y' - jW_r|}{W_r}Q_r  \geq -z Q_r \Bigg\}.
\end{align*}

\begin{lemma}\label{l:cA.bound}
There exists $C,\theta_5>0$, not depending on $n,M,\ell$ such that for all $i$ and {$|j|\leq \frac{MW_n}{W_r}$}
\[
\P[\cA_{i,j,z}^- ] \geq 1 - \exp(-C z^{\theta_5}),\qquad \P[\cA_{i,j,z}^+ ] \geq 1 - \exp(-C z^{\theta_5}).
\]
\end{lemma}

\noindent
{\bf  Comparing side to side passage times across two boxes with different heights}. We set
\begin{align*}
\cM_{i,j,j',z,w} &= \bigg\{\inf_{\substack{y,y'\in[jW_r,j' W_r]\\ \gamma' \subset  [(i-1)r,ir]\times [jW_r,j' W_r] \\ \gamma'(0)=((i-1)r,y) \\ \gamma'(1)=(ir,y')}} \hrX_{\gamma'} 
-  \inf_{\substack{y,y'\in[jW_r,j' W_r]\\ \gamma' \subset  [(i-1)r,ir]\times [(j-w)W_r,(j'+w) W_r] \\ \gamma'(0)=((i-1)r,y) \\ \gamma'(1)=(ir,y')}} \hrX_{\gamma'} 
\leq  z Q_r \bigg\}.
\end{align*}

The probability estimate we need for $\cM$ will be coupled with the estimate of another event and is stated in Lemma \ref{l:gadget} below. 

\noindent
{\bf Events in the resampled environment.}
Recall the environment $\omega_{\kappa}$ obtained by resampling the randomness in each $n\times W_n$ block $\Lambda^{+}_{ij}$ with probability $\kappa$. Recall that we had used the notation $\rX'$ to denote passage times in this environment. Analogous to above we define events $\cK', \cI^{'\pm}, \cJ', \cA^{'\pm}, \cM'$ with passage times $\rX$ replaced in their respective definitions by $\rX'$.

\begin{lemma}\label{l:gadget}
There exist $C,s_0>0$ depending on $\kappa$ but not on $n,M,\ell$ such that for any $s\geq s_0$ there exists $\delta(s)>0$ independent of $n,M,\ell$ and $\alpha=\alpha(n,M,\ell,s), h=h(n,M,\ell,s)$ such that
\[
\alpha \in  [s,2\kappa^{-1} s], \qquad h\in[1,\frac32],
\]
and
\begin{equation}
    \label{e:mboundI}
    \P[\cI^+_{i,j,h,-(\alpha-\alpha^{9/10})}, \cI^{'-}_{i,j,h,-(\alpha+\alpha^{9/10})}] \geq \delta
\end{equation}
and for all $w\leq \frac14$,
\begin{equation}
    \label{e:mboundM}
    \P[\cM^c_{i,j,h,z,w}] \leq C w/z.
\end{equation}
\end{lemma}

The events in the above lemma are going to be crucial for showing separation of geodesics before and after the resampling. Observe that the event $\cI^+_{i,j,h,-(\alpha-\alpha^{9/10})} \cap \cI^{'-}_{i,j,h,-(\alpha+\alpha^{9/10})}$ guarantees that the passage time across the corresponding rectangle drops sharply after the resampling.  The $\cM$ event guarantees that shortest paths across a slightly expanded box are not too much shorter than before. These together with other events defined below will (essentially) imply that the geodesic after the resampling passes through this rectangle while the geodesic before the resampling does not come close to it.

\subsection*{Parameters}
Using these basic estimates we are now going to construct a series of events. These events will depend on a number of parameters. We summarise the interrelations between the parameters now. Recall that $\beta$ is the parameter for the construction of conforming paths and is small enough such that for $n$ sufficiently large we have that $n^{\beta}W_n \ll n$.  Meanwhile, $\kappa$ is the density of resampling chosen sufficiently small such that $\kappa^{-1}$ is an integer. The remaining parameters are chosen possibly depending on these. We shall fix $L_0$ to be some large number, and $\delta_A$ is taken to be a small number depending on $L_0$ coming from Lemma~\ref{l:gadget}. Then we choose $w\ll 1$ depending on $\delta_A$. Next $L_1\gg L_0$ is chosen depending on $L_0,\kappa$ and $\delta_A$, and we set $L_2=L_1^{100}$. There are other parameters $\delta_B$, $\delta_C$ which are chosen small depending on $L_0,L_1,L_2$. All parameters so far are chosen independent of $n,M$ and $\ell$. Fixing these we pick $M$ sufficiently large, $n$ sufficiently large depending on $M$, and then $\ell$ ranges within $\{0,\ldots,\ell_{\max}\}$. The parameters $\alpha=\alpha(n,M,\ell,L_0^{100})$ and $h=h(n,M,\ell,L_0^{100})$ are chosen from Lemma \ref{l:gadget} depending on $n,M,\ell$; however, note that by Lemma \ref{l:gadget}, $\alpha\in[L^{100},2\kappa^{-1}L_0^{100}]$ and $h\in[1,\frac32]$ remain bounded independent of $n,M,\ell$.

\begin{center}
\begin{figure}[htbp!]
\includegraphics[width=6in]{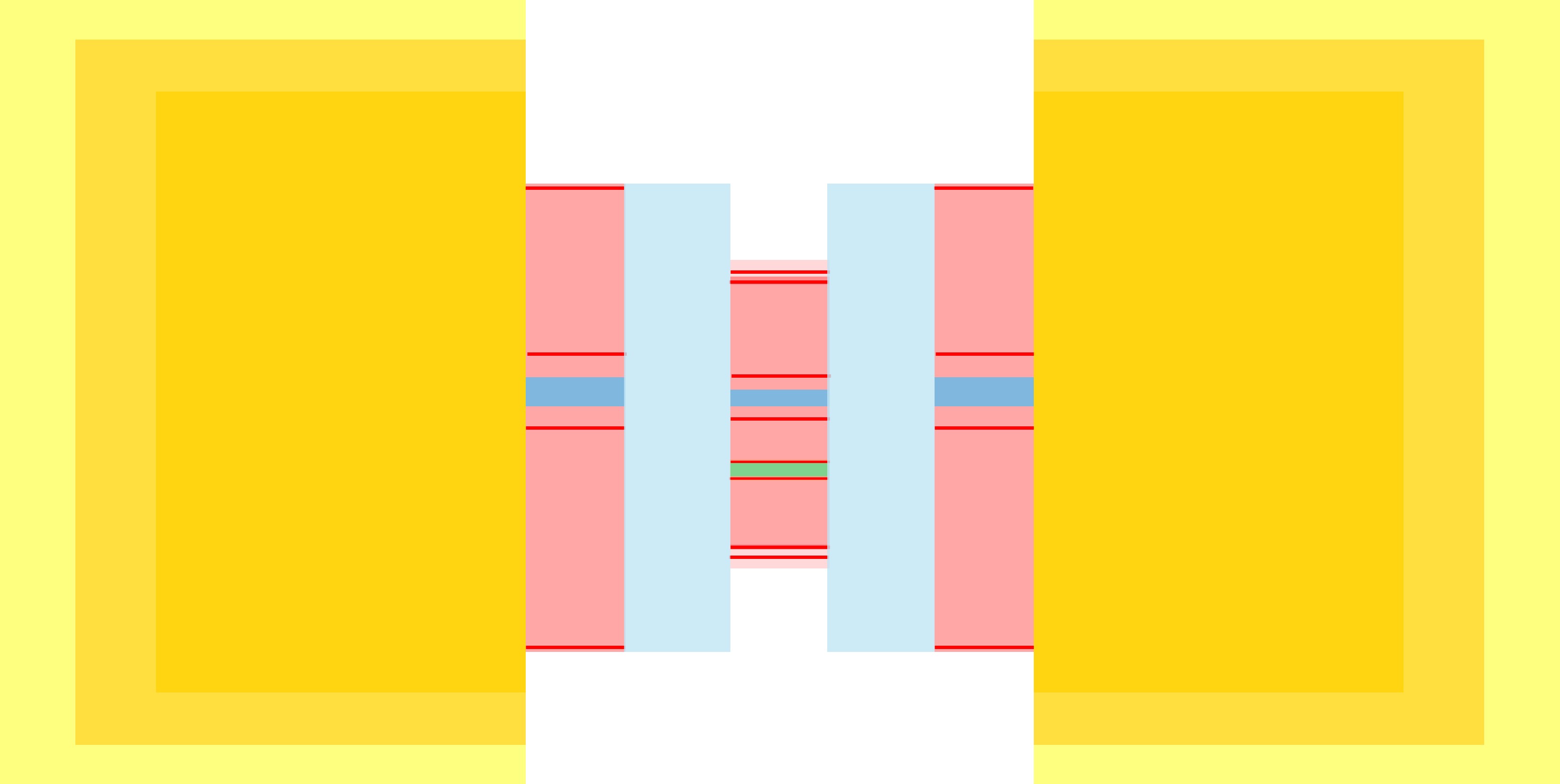}
\caption{The regions where different events are defined at coordinate $i,j$ at length scale $r$. In the vertical direction, the centres of the different columns shown in the figure is located at $jW_r$. The column at the middle (central column) corresponds to $[(i-1)r,ir]$, the columns marked in light blue are the intermediate columns ($[(i-2)r,(i-1)r]$ and $[ir,(i+1)r]$) and and the outer columns correspond to $[(i-3)r,(i-2)r]$ and $[(i+1)r, (i+2)r]$. The different colours in the central and outer columns represents regions where we ask for different type of events. The yellow columns flanking these five columns are referred to as wings on a series of different scales and we ask them to satisfy certain typical events called wing conditions. 
}
\label{f:fullevent}
\end{figure}
\end{center}

\noindent
\textbf{Columns for different events.}
We shall now construct events for the chaos estimate using the basic events defined above. These events will be indexed by $i,j$ (and several other parameters) which will indicate that they correspond to the location at height $jW_r$ at the column $[(i-1)r,ir]$. The events will be divided into four parts: $\cB$ for the central column , i.e., the column $[(i-1)r, ir]$, one for its neighbouring columns $\cC$, the intermediate columns  $[(i-2)r, (i-1)r]$ and $[ir, (i+1)r]$, the third one $\cD$ for the outer columns ($[(i-3)r,(i-2)r]$ and $[(i+1)r,(i+2)r]$) and the final one which we call the \emph{wing condition} $\cW$ for the region outside the columns $[(i-3)r,(i+2)r]$. We now move towards constructing all these events; see Figure \ref{f:fullevent} for an illustration.

\subsection{Events for the Central column}
The events for the central column covers (roughly) the region $[(i-1)r,ir]\times [(j-L_1)W_r, (j+L_1)W_r]$. These events are divided into the following seven sub-events that corresponds to different parts of the rectangle defined below. See Figure \ref{f:EventB} for an illustration of the same. 

\begin{align*}
\cB^{(1)}_{i,j} = \cI^-_{i,j-\tfrac1{40} L_0,j+\tfrac1{40} L_0,L_0}\cap \cI^{'-}_{i,j-\tfrac1{40} L_0,j+\tfrac1{40} L_0,L_0}.
\end{align*}
\begin{align*}
\cB^{(2)}_{i,j} = 
\bigg\{\inf_{y,y'\in[(j - \frac14 L_0)W_r,(j + \frac14 L_0)W_r]} \inf_{\substack{\gamma \subset  [(i-1)r,ir]\times [(j - \frac32 L_0)W_r,(j + \frac32 L_0)W_r] \\
\gamma \not\subset  [(i-1)r,ir]\times [(j - \frac12 L_0)W_r,(j + \frac12 L_0)W_r] \\ \gamma(0)=((i-1)r,y) \\ \gamma(1)=(ir,y')}} \min\{\rX_\gamma,\rX_\gamma'\}  \geq  r + \frac1{40} L_0^2 Q_r \bigg\}\\
\cap \cI^+_{i,j-\frac32 L_0,j+\frac32 L_0,-L_0}\cap \cI^{'+}_{i,j-\frac32 L_0,j+\frac32 L_0,-L_0}.
\end{align*}

\begin{align*}
\cB^{(3)}_{i,j} = \cI^+_{i,j-L_1 - 1, j-L_1 + 1,-L_0} \cap \cI^{'+}_{i,j-L_1 - 1, j-L_1 + 1,-L_0}\\
\cap \cI^+_{i,j+L_1 - 1, j+L_1 + 1,-L_0} \cap \cI^{'+}_{i,j+L_1 - 1, j+L_1 + 1,-L_0}.
\end{align*}
\begin{align*}
\cB^{(4)}_{i,j} = \cI^+_{i,j-\sqrt{\alpha},j-\sqrt{\alpha}+h,-(\alpha-\alpha^{9/10})} \cap \cI^{'-}_{i,j-\sqrt{\alpha},j-\sqrt{\alpha}+h,-(\alpha+\alpha^{9/10})}.
\end{align*}

\begin{align*}
\cB^{(5)}_{i,j} = \cM_{i,j-\sqrt{\alpha},j-\sqrt{\alpha}+h,1,w}\cap\cM'_{i,j-\sqrt{\alpha},j-\sqrt{\alpha}+h,1,w}.
\end{align*}
\begin{align*}
\cB^{(6)}_{i,j} = \bigcap_{s\in S}\cK^*_{i,j+s,{L_1},w}\cap\cK^{'*}_{i,j+s,L_1,w} \qquad \hbox{where } \\
S=\{ -L_1,-L_1+w, - \sqrt{\alpha} -\frac{2w}{3}, - \sqrt{\alpha} + h + \frac{2w}{3},-\frac32 L_0,\frac32 L_0,L_1 - w,L_1 \}.
\end{align*}

\begin{align*}
\cB^{(7)}_{i,j} = \cI^+_{i,j-L_1+\frac{2}{W_r},j-\sqrt{\alpha}-\frac{2}{W_r},L_2^2}\cap \cI^{'+}_{i,j-L_1+\frac{2}{W_r},j-\sqrt{\alpha}-\frac{2}{W_r},L_2^2}\\
\cap\cI^+_{i,j-\sqrt{\alpha}+h+\frac{2}{W_r},j-L_0-\frac{2}{W_r},L_2^2}\cap \cI^{'+}_{i,j-\sqrt{\alpha}+h+\frac{2}{W_r},j-L_0-\frac{2}{W_r},L_2^2}\\
\cap\cI^+_{i,j+L_0+\frac{2}{W_r},j+L_1-\frac{2}{W_r},L_2^2}\cap \cI^{'+}_{i,j+L_0+\frac{2}{W_r},j+L_1-\frac{2}{W_r},L_2^2}\\
\cap\cJ_{i,j-L_1+\frac{2}{W_r},j-\sqrt{\alpha}-\frac{2}{W_r},L_2^2,w/3}\cap \cJ^{}_{i,j-L_1+\frac{2}{W_r},j-\sqrt{\alpha}-\frac{2}{W_r},L_2^2,w/3}\\
\cap\cJ_{i,j-\sqrt{\alpha}+h+\frac{2}{W_r},j-L_0-\frac{2}{W_r},L_2^2,w/3}\cap \cJ^{'}_{i,j-\sqrt{\alpha}+h+\frac{2}{W_r},j-L_0-\frac{2}{W_r},L_2^2,w/3}\\
\cap\cJ_{i,j+L_0+\frac{2}{W_r},j+L_1-\frac{2}{W_r},L_2^2,w/3}\cap \cJ^{'}_{i,j+L_0+\frac{2}{W_r},j+L_1-\frac{2}{W_r},L_2^2,w/3}\\
\cB_{i,j} = \bigcap_{\ell=1}^7 \cB^{(\ell)}_{i,j}.
\end{align*}
\begin{center}
\begin{figure}[htbp!]
\includegraphics[width=6in]{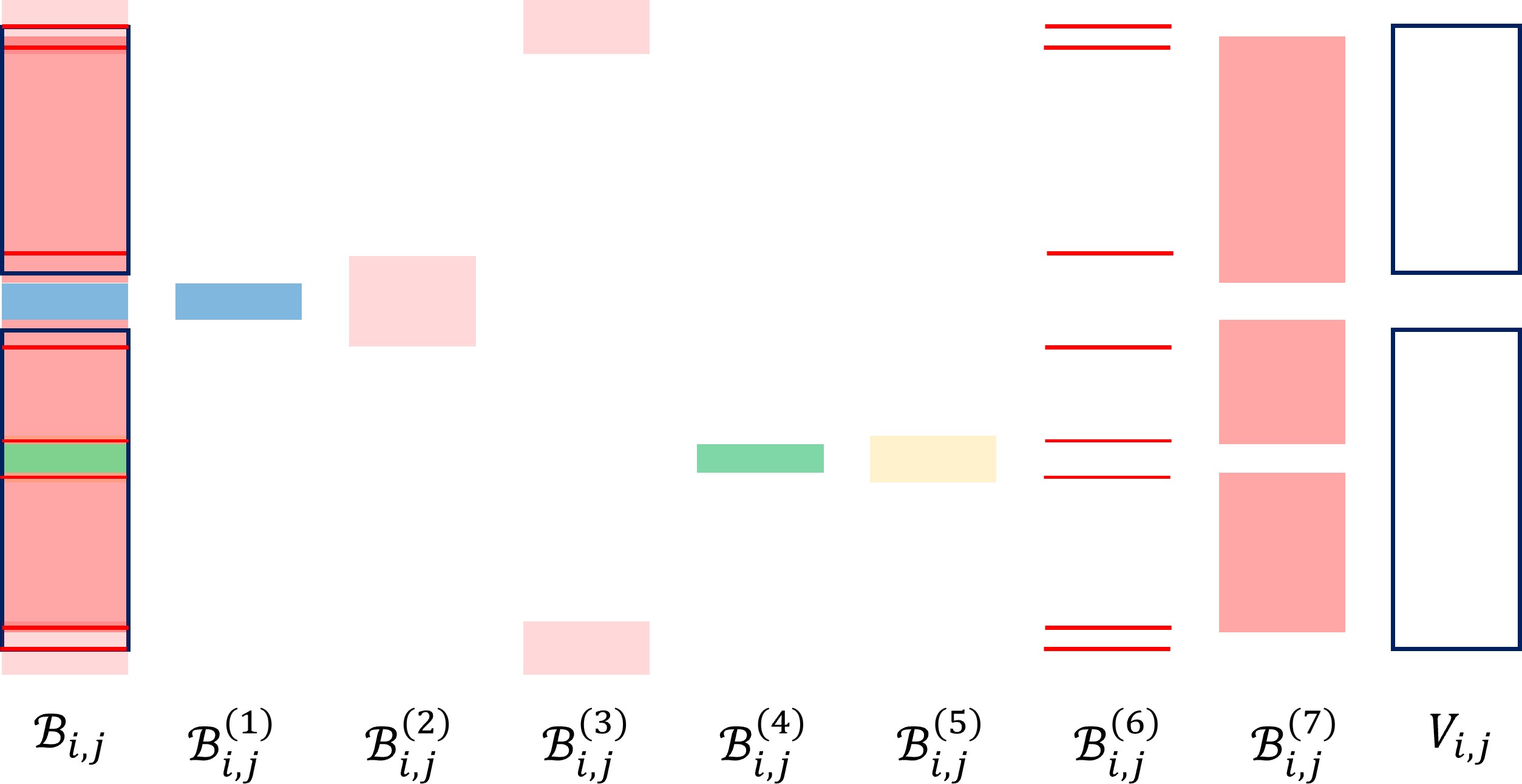}
\caption{The event $\cB_{i,j}$ for the central column $[(i-1)r,ir]$. The different panels show different subevents referring to different regions of the column and the first panel illustrates them combined. The first event $\cB_{i,j}^{(1)}$ asks that the passage time  across blue region in the middle for both $\rX$ and $\rX'$ is not too larger than typical. The second event $\cB_{i,j}^{(2)}$ asks that in the corresponding light red region, none of the paths are too short and any path across with large vertical change will have large excess length both before and after resampling. The third event $\cB_{i,j}^{(3)}$ asks that paths across the corresponding light red region are not too short both before and after the resampling. The fourth event $\cB_{i,j}^{(4)}$ asks that across the green region marked (which we shall refer to as the \emph{gadget}), none of the paths are too short before resampling, while after resampling there is a very good path. The fifth event $\cB_{i,j}^{(5)}$ asks that the $\cM$ event holds in the yellow region before and after resampling, i.e., the passage time across the yellow region and the passage time across the green region are not too different. The sixth event $\cB_{i,j}^{(6)}$ asks that the $\cK$ event holds for the horizontal lines marked in red both before and after resampling. The seventh event $\cB_{i,j}^{(7)}$ asks that the regions marked in dark red are barriers, i.e., all paths across these regions are atypically large (given by $\cI^+$ and $\cJ$ events). The regions $V_{ij}$ in the final panel will be used for conditioning and separating out the likely part of the $\cB$ events, see Section \ref{s:eventperc}. 
}
\label{f:EventB}
\end{figure}
\end{center}

We have the following estimates.

\begin{lemma}
\label{l:Bbound}
Provided $L_0$ is large enough, for all $i$ and $|j|\leq \frac{MW_n}{W_r}$, and all large enough $n$,
\[
\P[\cB_{i,j}^{(1)}],\P[\cB_{i,j}^{(2)}],\P[\cB_{i,j}^{(3)}] \geq 1-L_0^{-1}.
\]
\end{lemma}

\begin{proof}
The bounds in the above lemma for $\cB^{(1)}$ and $\cB^{(3)}$ are easy consequences of Lemma \ref{l:cI.bound} by choosing $L_0$ to be sufficiently large. The second and the third events in the definition of $\cB^{(2)}$ have probability at least $1-(3L_0)^{-1}$ by  Lemma \ref{l:cI.bound} and choosing $L_0$ sufficiently large. It therefore only remains to deal with the first event in $\cB^{(2)}$. By the exchangeability of $\rX$ and $\rX'$ it suffices to show that $\P(A)\ge 1-(6L_0)^{-1}$
where 
$$A=\bigg\{\inf_{y,y'\in[(j - \frac14 L_0)W_r,(j + \frac14 L_0)W_r]} \inf_{\substack{\gamma \subset  [(i-1)r,ir]\times [(j - \frac32 L_0)W_r,(j + \frac32 L_0)W_r] \\
\gamma \not\subset  [(i-1)r,ir]\times [(j - \frac12 L_0)W_r,(j + \frac12 L_0)W_r] \\ \gamma(0)=((i-1)r,y) \\ \gamma(1)=(ir,y')}} \rX_{\gamma}  \geq  r + \frac1{40} L_0^2 Q_r \bigg\}.$$
It suffices to show this only for the case $i=1,j=0$; it is easy to see that the same proof with minimal changes go through for general values of $i$ and $j$. Suppose that the event $A^c$ holds in which case we can find a $\gamma$ satisfying the the constraints such that $\rX_\gamma < r + \frac1{40} L_0^2 Q_r$. By Lemma~\ref{l:strongly.conforming} we can find another path $\gamma^{(a)}$, also satisfying the same set of constraints such that $\rX_{\gamma^{(a)}} < r + \frac1{40} L_0^2 Q_r$ and $\gamma^{(a)}$ has no vertical boundary pairs. By the constraints, there must be a point $u\in \gamma^{(a)}$ on the line segment $[0,r]\times\{\frac{1}{2}L_0W_r\}$ or the line segment $[0,r]\times\{-\frac{1}{2}L_0W_r\}$.  Setting $v=\gamma^{(a)}(0)=(0,y)$ and $w=\gamma^{(a)}(1)=(r,y')$ we have that  $y,y'\in [-\frac{1}{4}L_0 W_r,\frac{1}{4}L_0 W_r]$ and that neither $\{v,u\}$ or $\{u,w\}$ are a vertical boundary pair and so
\[
\rX_{vu}+\rX_{uw} \leq \rX_{\gamma^{(a)}} < r+\frac{1}{40}L_0^2Q_r.
\]
Notice, however, that on the event 
$$\cK_{1,\frac{1}{2}L_0, \frac{1}{100}L_0^2} \cap \cK_{1,-\frac{1}{2}L_0, \frac{1}{100}L_0^2}$$
we have that for $L_0$ sufficiently large and for all $u,v,w$ as above 
$$\rX_{vu}+\rX_{uw}\geq r + \frac12\Big(\frac{|y-\frac12 L_0 W_r| }{W_r} - 1\Big)^2 Q_r+\frac12\Big(\frac{|y-\frac12 L_0 W_r| }{W_r} - 1\Big)^2 Q_r> r+\frac{1}{40}L_0^2Q_r.$$
It follows from Lemma \ref{l:cK.bound} that 
$\P(A)\ge \P(\cK_{1,\frac{1}{2}L_0, \frac{1}{100}L_0^2} \cap \cK_{1,-\frac{1}{2}L_0, \frac{1}{100}L_0^2})\geq 1-(6L_0)^{-1}$ for $L_0$ sufficiently large, as required.     
\end{proof}

By Lemma~\ref{l:gadget}, there exists $\delta_A>0$ such that
\begin{equation}\label{eq:B4.bound}
\P[\cB_{i,j}^{(4)}] \geq \delta_A.
\end{equation}
With $C$ the constant in Lemma~\ref{l:gadget}, set $w=\frac{\delta_A}{200 C}\wedge \frac14$.  Then by Lemma~\ref{l:gadget},
\begin{equation}\label{eq:B5.bound}
\P[\cB_{i,j}^{(5)}] \geq 1 - \frac1{100}\delta_A.
\end{equation}
By Lemma~\ref{l:cK.bound} we can choose $L_1$ large enough so that $L_1\geq \kappa^{-2} L_0^{100}$ and
\[
\P[\cK_{i,j+s,L_1}]\geq 1 - \frac1{10000L_0}\delta_A.
\]
Then by a union bound,
\begin{equation}\label{eq:B6.bound}
\P[\cB_{i,j}^{(6)}] \geq 1 - \frac1{100L_0}\delta_A.
\end{equation}
Finally, we set $L_2 = L_1^{100}$.  The event $\cB_{i,j}^{(7)}$ is defined as the intersection of 12 events of type~$\cI^{+}$ and~$\cJ$.  By Lemmas~\ref{l:cI.bound} and~\ref{l:cJ.bound} each of the events have probability at least some $\delta>0$.  Conditional on the $T_{ij}$, which does not affect their individual probabilities, they are all increasing events and so by the FKG Inequality,
\begin{equation}\label{eq:B7.bound}
\P[\cB_{i,j}^{(7)}] \geq \delta^{12} =: \delta_B >0.
\end{equation}

\subsection{Intermediate Left and Right Columns}
We define the event
\begin{align*}
\cC_{i,j} &= \cA_{i,j,L_0}^- \cap \cA_{i,j,L_0}^+ \cap \cA_{i,j,L_0}^{'-} \cap \cA_{i,j,L_0}^{'+}.
\end{align*}
This event asks that (both before and after resampling) the passage times across the column are not too far from their expectation. 

By Lemma~\ref{l:cA.bound}
\begin{equation}\label{eq:cC.bound}
\P[\cC_{i,j}] \geq 1-L_0^{-1}.
\end{equation}
Recalling that the intermediate left and the intermediate right columns correspond to indices $(i-1)$ and $(i+1)$ respectively we shall set the event for the intermediate left and right columns to be 
$$\cC_{i-1,j} \cap \cC_{i+1,j}.$$

\subsection{Outer Left and Right Columns}
The event for the outer left and right columns roughly correspond to events in the regions $[(i-3)r,(i-2)r]\times [(j-L_2)W_r, (j+L_2)W_r]$. This event also consists of a number of sub-events corresponding to different regions of the rectangle; see Figure \ref{f:outer} for an illustration.

\begin{center}
\begin{figure}[htbp!]
\includegraphics[width=4in]{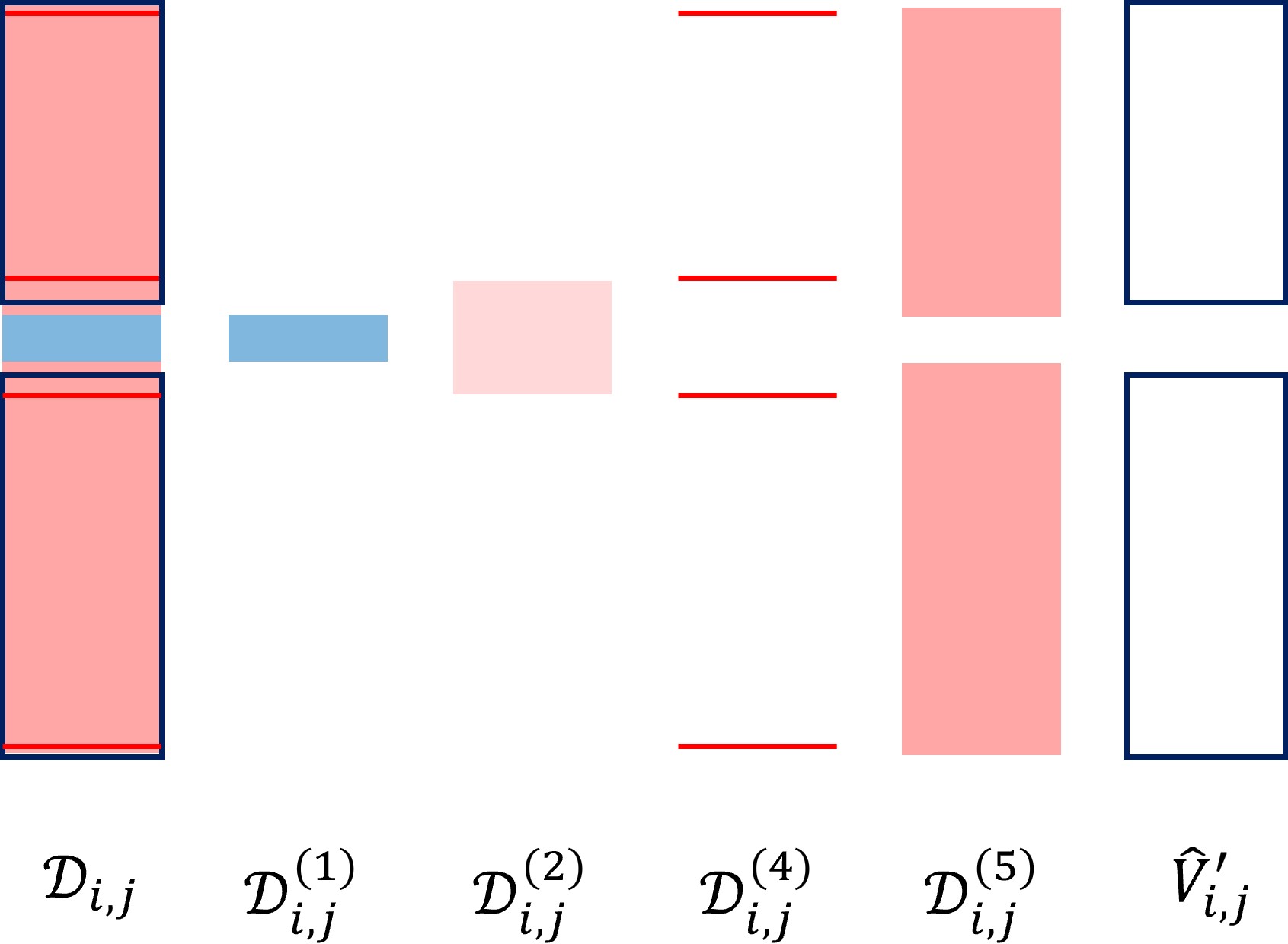}
\caption{The event $\cD_{i,j}$ for the rectangle $[(i-1)r,ir]\times [(j-L_2)W_r,(j+L_2)W_r]$. As in Figure \ref{f:EventB}, the different panels show different subevents referring to different regions of the rectangle and the first panel illustrates them combined. The first event $\cD_{i,j}^{(1)}$ asks that the passage time  across blue region in the middle for both $\rX$ and $\rX'$ is not too larger than typical. The second event $\cD_{i,j}^{(2)}$ asks that in the corresponding light red region, none of the paths are too short both before and after resampling. The third event $\cD_{i,j}^{(3)}$ (not shown) asks for not having any very good paths across the whole column (with some tolerance as we go away from the centre of the rectangle as in event $\cA$) both before and after resampling. The fourth event $\cD_{i,j}^{(4)}$ asks that the $\cK$ event holds for the horizontal lines marked in red both before and after resampling. The seventh event $\cD_{i,j}^{(3)}$ asks that the regions marked in dark red are barriers, i.e., all paths across these regions are atypically large (given by $\cI^+$ and $\cJ$ events). Finally, the regions $\widehat{V}'_{ij}$ in the final panel will be used for conditioning and separating out the likely part of the $\cD$ events, see Section \ref{s:eventperc}.}
\label{f:outer}
\end{figure}
\end{center}

We define the events
\begin{align*}
\cD^{(1)}_{i,j} = \cI^-_{i,j-\tfrac1{40} L_0,j+\tfrac1{40} L_0,{L_0}}\cap \cI^{'-}_{i,j-\tfrac1{40} L_0,j+\tfrac1{40} L_0,L_0}.
\end{align*}
\begin{align*}
\cD^{(2)}_{i,j} = \cI^+_{i,j-4L_0,j+4L_0,-L_0}\cap \cI^{'+}_{i,j-4L_0,j+4L_0,-L_0}.
\end{align*}
\begin{align*}
\cD^{(3)}_{i,j} = \cA_{i,j,L_0}^+ \cap \cA_{i,j,L_0}^{'+}.
\end{align*}
\begin{align*}
\cD^{(4)}_{i,j} = \bigcap_{s\in S} \cK_{i,j+s,L_1},\cK'_{i,j+s,L_1} \qquad \hbox{for } \\
S= \{ -L_2+w,-\frac1{20}L_0,-\frac1{20}L_0 + 1,\frac1{20}L_0-1,\frac1{20},L_2 - w \}.
\end{align*}
\begin{align*}
\cD^{(5)}_{i,j} = \cI^+_{i,j-L_2+\frac{2}{W_r},j-\tfrac1{40}L_0-\frac{2}{W_r},L_2^3}\cap \cI^{'+}_{i,j-L_2+\frac{2}{W_r},j-\tfrac1{40}L_0-\frac{2}{W_r},L_2^3}\\
\cap\cI^+_{i,j+\tfrac1{40}L_0+\frac{2}{W_r},j+L_2-\frac{2}{W_r},L_2^3}\cap \cI^{'+}_{i,j+\tfrac1{40}L_0+\frac{2}{W_r},j+L_2-\frac{2}{W_r},L_2^3}\\
\cap\cJ_{i,j-L_2+\frac{2}{W_r},j-\tfrac1{40}L_0-\frac{2}{W_r},L_2^3,w/3}\cap \cJ^{'}_{i,j-L_2+\frac{2}{W_r},j-\tfrac1{40}L_0-\frac{2}{W_r},L_2^3,w/3}\\
\cap\cJ_{i,j+\tfrac1{40}L_0+\frac{2}{W_r},j+L_2-\frac{2}{W_r},L_2^3,w/3}\cap \cJ^{'}_{i,j+\tfrac1{40}L_0+\frac{2}{W_r},j+L_2-\frac{2}{W_r},L_2^3,w/3}.\\
\end{align*}
Finally, set

\[
\cD_{i,j} = \bigcap_{\ell=1}^5 \cD^{(\ell)}_{i,j}
\]
By Lemma~\ref{l:cI.bound}
\begin{equation}\label{eq:D12.bound}
\P[\cD_{i,j}^{(1)}] \geq 1 - L_0^{-1},\qquad \P[\cD_{i,j}^{(2)}] \geq 1 - L_0^{-1}.
\end{equation}
and by Lemma~\ref{l:cA.bound},
\begin{equation}\label{eq:D3.bound}
\P[\cD_{i,j}^{(3)}] \geq 1 - L_0^{-1}.
\end{equation}
By Lemma~\ref{l:cK.bound} and a union bound
\begin{equation}\label{eq:D4.bound}
\P[\cD_{i,j}^{(4)}] \geq 1 - L_1^{-1}.
\end{equation}
and finally, similarly to equation~\eqref{eq:B7.bound} we have by the FKG inequality that for some $\delta_C>0$,
\begin{equation}\label{eq:D5.bound}
\P[\cD_{i,j}^{(5)}] \geq \delta_C.
\end{equation}

Recalling that the outer left and the outer right columns correspond to the indices $i-2$ and $i+2$ we shall define the event for the far left and far right columns as 
$$\cD_{i-2,j}\cap \cD_{i+2,j}.$$

\subsection{Wing Conditions} For passage times outside of $[(i-3)r,(i+2)r]$ we define a series of passage time estimate on both local and global scales.  The global ones will be likely enough to hold with high probability.  We set (for $\theta_2$ as in Proposition \ref{p:paraestimateconforming}) 
\begin{align*}
\cW^{*}_{i}(\cX)&= \bigcap_{i'=0}^{i-4}\bigcap_{i''=i'+1}^{i-3}\left\{\max_{\substack{|y|,|y'| \leq MW_n \\ u'=(i'r,y)\\u''=(i''r,y') }} |\hrX_{u',u''}|  \leq  \log^{\frac{100}{\theta_2}} (M) Q_{(i''-i')r}\right\}\\
&\bigcap_{i'=0}^{i-4}\bigcap_{i''=i'+1}^{i-3}\left\{\min_{\substack{|y|,|y'| \leq n^\beta W_n \\ u'=(i'r,y)\\u''=(i''r,y') }}\hrX_{u',u''}  \geq - (1\vee M^{-2}W_n^{-1}(|y|+|y'|))\log^{\frac{100}{\theta_2}} (M) Q_{(i''-i')r}\right\}\\
&\bigcap_{i'=i+2}^{Mn/r-1}\bigcap_{i''=i'+1}^{Mn/r}\left\{\max_{\substack{|y|,|y'| \leq MW_n \\ u'=(i'r,y)\\u''=(i''r,y') }} |\hrX_{u',u''}|  \leq  \log^{\frac{100}{\theta_2}} (M) Q_{(i''-i')r}\right\}\\
&\bigcap_{i'=i+2}^{Mn/r-1}\bigcap_{i''=i'+1}^{Mn/r}\left\{\min_{\substack{|y|,|y'| \leq {n^{\beta}W_n} \\ u'=(i'r,y)\\u''=(i''r,y') }}\hrX_{u',u''}  \geq - (1\vee M^{-2}W_n^{-1}(|y|+|y'|))\log^{\frac{100}{\theta_2}} (M) Q_{(i''-i')r}\right\}.
\end{align*}

Let $\cW_i^*(\cX')$ denote the event above with $\cX$ replaced by the resampled weights $\cX'$, i.e., $\hrX_{u',u''}$ in the events above are replaced by $\hrX'_{u',u''}$ and set 
$$\cW_i^*=\cW_i^*(\cX)\cap \cW_i^*(\cX').$$

Next, define 
\begin{align*}
\cZ_{i,j,k}&= \left\{\max_{\substack{|y|,|y'| \leq {MW_n} \\ u=(ir,y) \\ v= ((i+2^k L_2)r,y')}} |\hrX_{uv}| - (|\frac{y}{W_r}-j|^{\tfrac1{100}} +|\frac{y'}{W_r}-j|^{\tfrac1{100}})\frac{Q_r}{2^k L_2^2} \leq (2^k L_2)^{3/5} Q_{r}\right\}\\
&\bigcap \left\{\max_{\substack{|y|,|y'| \leq {MW_n} \\ u=(ir,y) \\ v= ((i+2^k L_2)r,y')}} |\hrX'_{uv}| - (|\frac{y}{W_r}-j|^{\tfrac1{100}} +|\frac{y'}{W_r}-j|^{\tfrac1{100}})\frac{Q_r}{2^k L_2^2} \leq (2^k L_2)^{3/5} Q_{r}\right\},\\
\cW_{i,j}^{loc}&= \bigcap_{k=1}^{(\log_2 \log_2 M)^2} \cZ_{i-2^{k+1}L_2-3,j,k}\cap \cZ_{i-2^{k}L_2-3,j,k} \cap \cZ_{i+2,j,k}\cap \cZ_{i+2^kL_2 +2 ,j,k},\\
\cW_{i,j}^{glo}&= \cW^{*}_{i} \cap \bigcap_{k=(\log_2 \log_2 M)^2}^{\lfloor\log_2 (\frac12M^{99/100}{\Phi^{-\ell}})\rfloor} \cZ_{i-2^{k+1}-3,j,k}\cap \cZ_{i-2^{k}-3,j,k} \cap \cZ_{i+2,j,k}\cap \cZ_{i+2^k +2 ,j,k}.
\end{align*}

We have the following probability bound for $\cZ_{i,j,k}$. 

\begin{lemma}
    \label{l:zbound}
    There exists $c,\theta'$ such that for $L_2$ chosen sufficiently large enough,  we have for all $i, |j|\le M$ and for all $k\le (\log_2\log_2 M)^2$
    $$ \P(\cZ_{i,j,k}) \ge 1-\exp(-c(2^{k}L_2)^{\theta'}).$$
\end{lemma}

The proof of Lemma \ref{l:zbound} is given in Section \ref{s:letter} and the bounds on $\cW^{loc}_{i,j}$ are provided in Section \ref{s:likely}. The next lemma which gives bound on $\cW^{glo}_{i,j}$ is proved in Section \ref{s:letter}.

\begin{lemma}\label{l:cW.global}
For all $M$ and $n$ large enough,
\[
\P[\cW^{glo}_{i,j}] \geq 1-M^{-200}.
\]

\end{lemma}

Finally, we let
\[
\cW_{i,j} = \cW^{loc}_{i,j} \cap \cW^{glo}_{i,j}.
\]

\subsection{Percolation Events}
\label{s:eventperc}
Not all parts of the events $\cB,\cD$ are highly likely so we separate the likely parts.  We set
\begin{align*}
\hV_j &= \Big(((j-L_1) W_r ,(j-L_0) W_r )\cup ((j + L_0) W_r ,(j+L_1)W_r)\Big)\\
V_{i,j} &= \bigcup_{i'=(i-1)\Phi^\ell+1}^{i\Phi^\ell} \{i'\}\times [(i'-1)n-1,i'n+1]\times \hV_j \\
V_{i,j}^c &= \bigcup_{i'=(i-1)\Phi^\ell+1}^{i\Phi^\ell} \{i'\}\times [(i'-1)n-1,i'n+1]\times (\R\setminus \hV_j)\\ 
\hV_{i,j} &= [(i-1)\Phi^{\ell}-1,i\Phi^{\ell}+1]\times\hV_{j}
\\
\hV_j'&= \Big(((j-L_2) W_r +2,(j-\frac1{40}L_0) W_r -2)\cup ((j + \frac1{40}L_0) W_r + 2,(j+L_2)W_r -2)\Big)\\
V'_{i,j} &= \bigcup_{i'=(i-1)\Phi^\ell+1}^{i\Phi^\ell} \{i'\}\times [(i'-1)n-1,i'n+1]\times \hV_j'\\
V^{'c}_{i,j} &= \bigcup_{i'=(i-1)\Phi^\ell+1}^{i\Phi^\ell} \{i'\}\times [(i'-1)n-1,i'n+1]\times (\R\setminus \hV_j')\\
\hV'_{i,j} &=[(i-1)\Phi^{\ell}-1,i\Phi^{\ell}+1]\times\hV'_{j}.
\end{align*}
We define the following events.
\begin{align*}
\cP_{i,j}^- &= \cB^{(1)}_{i,j}
\cap \Big \{\P\Big[\cB^{(2)}_{i,j}, \cB^{(3)}_{i,j} \mid \bomega(V_{i,j}^c)\Big] \geq \tfrac12 \Big \}
\cap \Big \{\P\Big[\cB^{(6)}_{i,j} \mid \bomega(V_{i,j}^c)\big)\Big] \geq 1-\frac{\delta_A}{100} \Big \}\\
&\qquad\bigcap_{i'\in \{i-2,i+2\}} \Bigg(\cD^{(1)}_{i',j}
\cap \Big \{\P[\cD^{(2)}_{i',j}, \cD^{(3)}_{i',j},\cD^{(4)}_{i',j}  \mid \bomega(V_{i',j}^{'c})] \geq 1/2 \Big \}\Bigg)\\
&\qquad \bigcap \cC_{i-1,j} \bigcap \cC_{i+1,j} \bigcap \cW_{i,j}^{loc}
\end{align*}
and
\[
\cP_{i,j} = \cP_{i,j}^-\cap \cD_{i-2,j}\cap \cD_{i+2,j}\cap\cB_{i,j} \cap \cW^{glo}_{i,j}.
\]
Implicitly $\cP_{i,j}$ is a function of $n,M$ and $\ell$ so when there is ambiguity we will write $\cP_{i,J_i}^{n,M,\ell}$.

The above events are local events, which in particular do not reference the optimal path (the event $\cW^{glo}_{i,j}$ is not  but it is sufficiently likely that we will employ a simple union bound for it.  

\subsection{Conclusions about the events}
We shall need two results about the events that we have defined above. The first one states that with high probability the event $\cP_{i,j}$ occurs at a positive fraction of locations in the bulk along the geodesic; see Figure \ref{f:Ppercolation}.

\begin{center}
\begin{figure}
\includegraphics[width=6in]{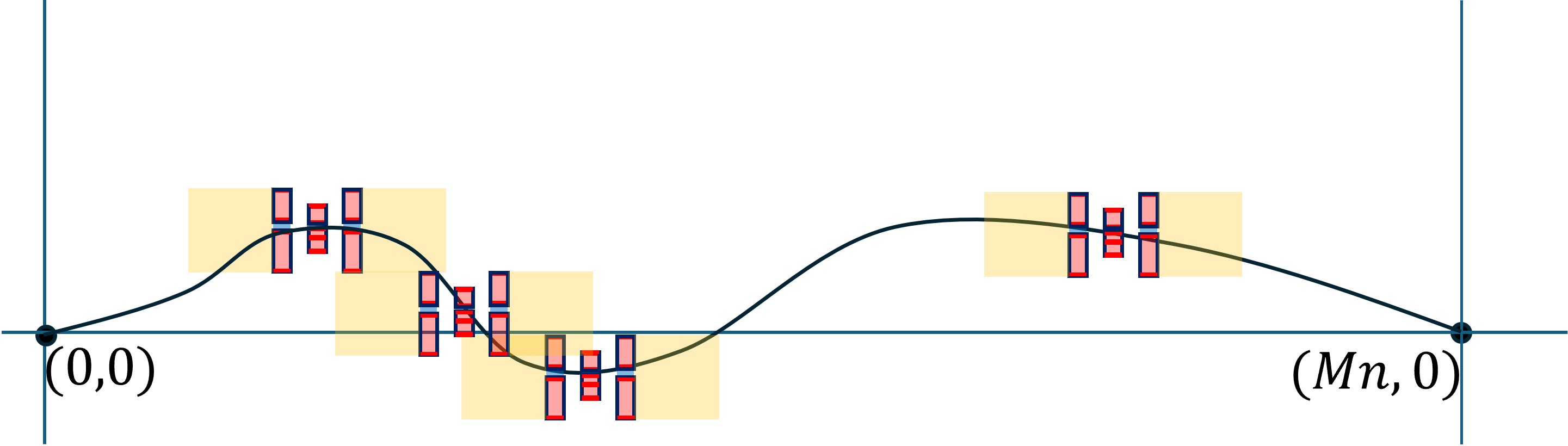}
\caption{In Theorem \ref{t:Pplus}, for each scale $r=r_{\ell}$, we check if the event $\cP_{i,j}$ occurs at location $i$ along the geodesic from $(0,0)$ to $(Mn,0)$, i.e., if $\cP_{i,J_i}$ holds. This event is local and as shown in the figure depends on the marked regions (in the figure, the red region marks the central and outer columns and the wing conditions depends on the regions marked in yellow, intermediate columns are not marked). Theorem \ref{t:Pplus} asserts that with large probability along any consecutive $\Phi$ many consecutive locations in the bulk, $\cP_{i,J_i}$ occurs at a positive (not depending on $n,M,\ell$) fraction of locations.}
\label{f:Ppercolation}
\end{figure}
\end{center}

\begin{theorem}\label{t:Pplus}
There exists $M_0$ such that for all $M\geq M_0$ and all $n$ sufficiently large and $0\leq \ell\leq \ell_{max}$ and $2M^{99/100}n\leq  ir_\ell \leq (M-2M^{99/100})n$,
\[
\P\Bigg[\sum_{i'=i}^{i+\Phi-1} I(\cP_{i',J^{n,M,\ell}_{i'}}^{n,M,\ell},|J^{n,M,\ell}_{i'}|W_r \leq M^{8/10}W_n) \leq \frac{\delta_A\delta_B\delta_C^2}{200} \Phi \Bigg] \leq M^{-90}.
\]
\end{theorem}

The second result states that on the event $\cP_{i,J_{i}}$ the optimal path $\gamma$ before the resampling and the optimal path $\gamma'$ after the resampling does not share any $n\times W_n$ block in the $i$-th column (see Corollary \ref{c:path.sepatated} for a more detailed statement).

\begin{lemma}
    \label{l:separatedfirst}
    There exists $M_0$ such that for all $M\geq M_0$ and all $n$ sufficiently large and $0\leq \ell\le \ell_{\max}$, ( i.e., $n\leq r_\ell \leq M^{1/100}n$) and $2M^{99/100}n\leq  ir_\ell \leq (M-2M^{99/100})n$ we have the following: on the event $\cP^{n,M,\ell}_{i,J_i}$, for all $i'\in [(i-1)\Phi^{\ell}+1,i\Phi^{\ell}]$ and for all $j$, $I(\cU_{i'j}\cap \cU'_{i'j})=0$. 
\end{lemma}

\section{Chaos Estimate: Proof of Proposition \ref{p:chaos}}\label{s:chaos.estimate}

Assuming Theorem \ref{t:Pplus} and Lemma \ref{l:separatedfirst} (which will be proved over the next few sections) we prove Proposition \ref{p:chaos} in this section. Recall that $\cU_{ij}$ is the event that the conforming geodesic $\gamma$ passes within distance 1 of the block $\Lambda_{ij}$ and let $\cU_{ij}'$ is the analogous event for the updated path $\gamma'$. We need to show that the expected overlap of blocks of the original and resampled paths is $o(M)$. This is done in the following lemma.
\begin{lemma}
\label{l:chaosbasic}
There exists $M_0$ such that for $M \geq M_0$ and all large enough $n$,
\[
\sum_{i=1}^M \sum_{j=-M}^M  \P[\cU_{ij},\cU_{ij}'] \leq \exp\Big(-\frac{\log M}{(\log_2 \log_2 M)^6}\Big) M.
\]
\end{lemma}
\begin{proof}
We shall treat the cases where $i$ is close to $1$ or $M$ separately from the case when $i$ is in the bulk. 
{The proof of Lemma \ref{l:chaosbasic} will follow from the following three estimates:
\begin{equation}
    \label{e:chaosbasic1}
    \sum_{i=2M^{99/100}}^{M-2M^{99/100}} \sum_{j=-M}^M  \P[\cU_{ij},\cU_{ij}'] \leq \exp\Big(-\frac{\log M}{(\log_2 \log_2 M)^6}\Big) M;
\end{equation}
\begin{equation}
    \label{e:chaosbasic2}
    \sum_{i=1}^{2M^{99/100}} \sum_{j=-M}^M  \P[\cU_{ij}] \leq M^{995/1000};
\end{equation}
\begin{equation}
    \label{e:chaosbasic3}
    \sum_{i=M-2M^{99/100}}^{M} \sum_{j=-M}^M  \P[\cU_{ij}] \leq M^{995/1000};
\end{equation}
The proof of \eqref{e:chaosbasic1} using Theorem \ref{t:Pplus} and Lemma \ref{l:separatedfirst} is the main part of the argument and is provided below. The proof of \eqref{e:chaosbasic2} is given in Lemma \ref{l:chaosbasic2} below. The proof of \eqref{e:chaosbasic3} is identical to the proof of \eqref{e:chaosbasic2} and is omitted.}

For the proof of \eqref{e:chaosbasic1}, recall that $\ell_{\max}$ is the largest $\ell$ such that $\Phi^\ell \leq M^{1/100}$.  Let $\cT$ be the event that for all $0\leq \ell \leq \ell_{\max}$ and $2M^{99/100}n\leq  ir_\ell \leq (M-2M^{99/100})n$ that
\[
\sum_{i'=i}^{i+\Phi-1} I(\cP_{i,J_i^{n,M,\ell}}^{n,M,\ell}) \geq \delta_0 \Phi
\]
where $\delta_0=\frac{\delta_A\delta_{B}\delta^2_C}{200}$ and also for all $0\leq \ell \leq \ell_{\max}$
$$\max_{i} |J^{n,M,\ell}_{i}| W_{r\_{\ell}} \le M^{8/10}W_n.$$
By Theorem~\ref{t:Pplus}, Lemma \ref{l:proxytrans:intro} and a union bound, $\P[\cT] \geq 1-M^{-90}$.

For $2M^{99/100}\leq i \leq M-2M^{99/100}$ let $\cQ_{i,\ell}$ denote the event $\cP_{i,J_i^{n,M,\ell}}^{n,M,\ell}$ at level $i$ and let
\[
\cV_i = \bigcup_{\ell=0}^{\ell_{\max}} \cQ_{\lceil i\Phi^{-\ell} \rceil,\ell}
\]
For $1\leq i < 2M^{99/100}$ and $M-2M^{99/100}< i \leq M$ let $\cV^c_i$ be the empty event. For $2M^{99/100}\leq i \leq M-2M^{99/100}$, on the event $\cV_i$, there is some $1\leq \ell \leq \ell_{\max}$ such that $\cQ_{\lceil i\Phi^{-\ell} \rceil,\ell}$ holds.  By Lemma~\ref{l:separatedfirst}, there is no event $\cU_{i'j}\cap\cU_{i'j}'$ that holds for any $i'\in [(\lceil i \Phi^{-\ell} \rceil\Phi^{\ell}-1) + 1,\lceil i\Phi^{-\ell} \rceil\Phi^{\ell}]$ since the paths are separated.  In particular $\cU_{ij}\cap\cU_{ij}'$ does not hold.

On the event $\cT$ we will show that for every $2M^{99/100}n/r_\ell \leq k \leq (M- 2M^{99/100})n/r_\ell$ that
\begin{align*}
\sum_{i=(k-1) \Phi^{\ell} + 1}^{k\Phi^{\ell}} I(\bigcup_{\ell'=0}^{\ell-1} \cQ_{\lceil i \Phi^{-\ell'}\rceil,\ell'}) \geq (1-(1-\delta_0)^\ell) \Phi^{\ell}.
\end{align*}
Assume by induction this holds for $\ell-1$.  Then
\begin{align*}
&\sum_{i=(k-1) \Phi^{\ell} + 1}^{k\Phi^{\ell}} I(\bigcup_{\ell'=0}^{\ell-1} \cQ_{\lceil i\Phi^{-\ell'} \rceil,\ell'})\\
&\qquad= \sum_{s=1}^{\Phi} \sum_{i=((k-1)\Phi+ s - 1)\Phi^{\ell-1} + 1}^{((k-1)\Phi + s)\Phi^{\ell-1}} I(\bigcup_{\ell'=0}^{\ell-1} \cQ_{\lceil i\Phi^{-\ell'} \rceil,\ell'})\\
&\qquad= \sum_{s=1}^{\Phi} I(\cQ_{((k-1)\Phi + s),\ell-1})\Phi^{\ell-1} +  I(\cQ_{((k-1)\Phi + s),\ell-1}^c)\sum_{i=((k-1)\Phi + s - 1)\Phi^{\ell-1} + 1}^{((k-1)\Phi + s)\Phi^{\ell-1}} I(\bigcup_{\ell'=0}^{\ell-2} \cQ_{\lceil i\Phi^{-\ell'} \rceil,\ell'})\\
&\qquad\geq \sum_{s=1}^{\Phi} I(\cQ_{((k-1)\Phi + s),\ell-1})\Phi^{\ell-1} +  I(\cQ_{((k-1)\Phi + s),\ell-1}^c)(1-(1-\delta_0)^{\ell-1}) \Phi^{\ell-1}\\
&\qquad\geq (1-(1-\delta_0)^\ell) \Phi^\ell.
\end{align*}
where the second equality is by considering the blocks at level $\ell-1$, the first inequality is by the induction hypothesis and the final inequality is by $\cT$.  Hence we have that
\begin{equation}\label{eq:separated.columns}
\sum_{i=1}^M \P[\cV_i^c] \leq 4M^{99/100} +M(1-\delta_0)^{\ell_{\max}} + M^2 \P[\cT^c] \leq 2M(1-\delta_0)^{\ell_{\max}}.
\end{equation}
Then, since for $2M^{99/100}\le i \le M-2M^{99/100}$, $\cU_{ij}\cap\cU_{ij}'$ can only hold on $\cV_i^c$,

\begin{align*}
\sum_{i=2M^{99/100}}^{M-2M^{99/100}} \sum_{j=-M}^M  \P[\cU_{ij},\cU_{ij}']
&\leq \E\bigg[\sum_{i=1}^M \Big(\sum_{j=-M}^M  I(\cU_{ij})\Big)I(\cV_i^c)\bigg]\\
&\leq\Bigg(\E\bigg[\sum_{i=1}^M \Big(\sum_{j=-M}^M  I(\cU_{ij})\Big)^2\bigg]\E\bigg[\sum_{i=1}^M I(\cV_i^c)\bigg]\Bigg)^{1/2}\\
&\leq\Big(CM \cdot 2M(1-\delta_1)^{\ell_{\max}}\Big)^{1/2}\\
&\leq \exp\Big(-\frac{\log M}{(\log_2 \log_2 M)^6}\Big) M,
\end{align*}
where the second inequality is by Cauchy-Schwartz, the third is by Proposition~\ref{p:uij.bounds} and equation~\eqref{eq:separated.columns} and the final inequality is because $\ell_{\max} = \frac1{100}\lfloor \frac{\log_2 M}{(\log_2 \log_2 M)^5} \rfloor$.
This completes the proof of \eqref{e:chaosbasic1}. As explained at the start of the proof, the lemma now follows from Lemma \ref{l:chaosbasic2} below. 
\end{proof}

We first complete the proof of Proposition \ref{p:chaos} before completing the remainder of the argument in Lemma \ref{l:chaosbasic2}.

\begin{proof}[Proof of Proposition \ref{p:chaos}]
Given $\kappa, \epsilon>0$ choose $M$ sufficiently large  so that 
$$\exp\Big(-\frac{\log M}{(\log_2 \log_2 M)^6}\Big)\le \epsilon.$$
By definition, $\cU_{ij}$ and $\cU'_{ij}$ are independent given $\cF_{1-\kappa}$, and therefore 
$$\P(\cU_{ij}\cap \cU'_{ij})=\E[\P[\cU_{ij}\mid \cF_{1-\kappa}]^2].$$
Summing over $i$ from $1$ to $M$ and $j$ from $-M$ to $M$ and using Lemma \ref{l:chaosbasic} completes the proof of the proposition.
\end{proof}

\begin{lemma}
    \label{l:chaosbasic2}
    There exists $M_0$ such that for $M \geq M_0$ and all large enough $n$,
\[
    \sum_{i=1}^{2M^{99/100}} \sum_{j=-M}^M  \P[\cU_{ij}] \leq M^{995/1000}.
    \]    
\end{lemma}

\begin{proof}
Since 
$$ \sum_{i=1}^{2M^{99/100}} \sum_{j=-M}^M I(\cU_{ij}) \le M^2$$
deterministically, it suffices to prove that 
\begin{equation}
    \label{e:chaosbasic4}
    \P\left[\sum_{i=1}^{2M^{99/100}} \sum_{j=-M}^M I(\cU_{ij}) \ge M^{994/1000} \right] \le M^{-100}. 
\end{equation}
Proof of \eqref{e:chaosbasic4} is a simpler version of the proof of the third estimate in Proposition \ref{p:uij.bounds}. Recall the definitions of $S^{+}_{i,j}$ and $S^{-}_{i,j}$ from Section \ref{s:1.2proof}. Recall also the definition of $J_i$. 
As argued in Section \ref{s:1.2proof}, we have for each fixed $i$,
\[
\sum_{j=-M}^{M} I(\cU_{i,j}) \le (S^{+}_{i,J_{i-1}\vee J_i}-J_{i-1}\vee J_i) + (J_{i-1}\wedge J_{i}-S^{-}_{i,J_{i-1}\wedge J_{i}})+|J_{i}-J_{i-1}|.
\]
It follows that 
$$\sum_{i=1}^{2M^{99/100}} \sum_{j=-M}^M I(\cU_{i,j})\le A+B+C$$
where 
\begin{align}
A&= \sum_{i=1}^{2M^{99/100}} |J_i-J_{i-1}|;\\
B&=\sum_{i=1}^{2M^{99/100}} \left(\max_{j\in [-M,M]} S^{+}_{i,j}-j \right);\\
C&=\sum_{i=1}^{2M^{99/100}} \left(\max_{j\in [-M,M]} j-S^{-}_{i,j} \right).
\end{align}
It follows from Lemma \ref{l:proxytrans} that for each fixed $i$ and $j$
$$\P(S^{+}_{i,j}-j\ge M^{1/1000}) \le M^{-1000}$$
for $M$ sufficiently large. Taking a union bound over $j$ between $-M$ and $M$ it follows that 
$$\P\left(\max_{j\in [-M,M]} S^{+}_{i,j}-j \ge M^{1/1000} \right) \le M^{-998}.$$
Taking a further union bound over $i$ between $1$ and $2M^{99/100}$ we finally get 
$$\P(B\ge M^{992/100})\le  M^{-997}.$$
A similar argument shows that 
$$\P(C\ge M^{992/100})\le  M^{-997}.$$
Therefore, to prove \eqref{e:chaosbasic4} it only remains to show that 
 $$\P\left(\sum_{i=1}^{2M^{99/100}} |J_i-J_{i-1}|\ge M^{9992/100} \right)\le M^{-1000}.$$

Notice that by Lemma \ref{l:localtransproxy:intro} we know that for $M$ large 
 $$\P(|J_{2M^{99/100}}|\ge M^{3/4})\le M^{-2000}$$ and therefore it suffices to prove that for each $j\in [-M^{3/4},M^{3/4}]$ we have 

 $$\P\left(\sum_{i=1}^{2M^{99/100}} |J_i-J_{i-1}|\ge M^{9992/10000}, J_{2M^{99/100}}=j \right) \le M^{-2000}.$$
 To this end, let us fix $j$ as above, and for $v\in \ell_{2M^{99/100}n, jW_n, (j+1)W_n}$ and the conforming geodesic $\gamma_{0v}$ from $\origin$ to $v$, let us define $J^v_i=\lfloor\frac{y_i}{W_n}\rfloor$ and $y_i$ is the point where $\gamma_{0v}$ intersects the line $x=in$. Define 
 $$\tau_1(\gamma_{0v})=\sum_i |J_i^v-J_{i-1}^v|.$$
 Clearly it suffices to prove that for all $j\in [-M^{3/4}, M^{3/4}]$
 \begin{equation}
     \label{e:chaosbasic5}
     \P\left(\sup_{v\in \ell_{2M^{99/100}n, jW_n, (j+1)W_n}} \tau_1(\gamma_{0v})\ge M^{9992/10000}\right) \le M^{-2000}.
 \end{equation}
 Notice that By Cauchy-Schwarz inequality
 
$$\frac{1}{2M^{99/100}}(\tau_1(\gamma_{0v}))^2 \le \tau_2(\gamma_{0v}):=\sum_{i}(J^v_i-J^v_{i-1})^2.$$
Therefore 
$$\tau_{1}(\gamma_{0v})\ge M^{9992/10000} \Rightarrow \tau_2(\gamma_{0v})\ge M^{10084/10000}.$$
Applying Lemma \ref{l:tau2percgen} with $r=n,D=2M^{99/100},s_1=0$ and $s_2=j$, \eqref{e:chaosbasic5} follows. 

 This completes the proof of the lemma. 
\end{proof}

\section{Analysis of the events: separation of paths}
\label{s:separation}

The objective of this section is to prove Lemma \ref{l:separatedfirst}. Recall the basic set up of that lemma. For $M$ sufficiently large and $n$ sufficiently large depending on $M$ we shall work with a fixed $\ell \le \ell_{\max}$ and $r=r_{\ell}$ and an index $i$ such that $2M^{99/100}n \le ir \le (M-2M^{99/100})n$. We shall show that on the event $\cP_{i,J_i}$ the geodesics before and after the resampling do not share (comes within distance $1$ of) any $n\times W_n$ block $\Lambda_{ij}$ within the column $[(i-1)r,ir]$. To this end we start with analyzing various events constituting the event $\cP_{i,j}$ (as before we shall omit the superscripts $n,M,\ell$). 

We shall need the following geometric notation. Let us set $$H_{i,j} = H_{i,j}^{n,M,\ell}= [(i-1)r,ir]\times\{jW_r\}.$$ We start with a lemma that gives a lower bound on the passage times across a column on certain barrier events.

\begin{lemma}\label{l:barrier.entry}
For $0<w\leq 1$ and {$j,j'\in [-2MW_n / W_r, 2MW_n / W_r]$} with {$j'-j\geq w$} on the event
\[
\cI^+_{i,j-w,j'+w,z_1}\cap \cJ_{i,j-w,j'+w,z_1,w} \cap \cK^*_{i,j,z_2,w} \cap \cK^*_{i,j',z_2,w}
\]
the following holds. Let $u=((i-1)r,y), u'=(ir,y')$ and let $\gamma_{u,u'}$ be the optimal conforming path joining $u$ to $u'$. If $\gamma_{u,u'}$ intersects $[(i-1)r,ir]\times [jW_r,j' W_r]$ then
\[
\rX_{u,u'} \geq (z_1-2z_2) Q_r + r + \frac12 \Big(\Big(\frac{|y-jW_r|  }{W_r} - |j'-j| - 2 \Big)^+\Big)^2 Q_r + \frac12 \Big(\Big(\frac{|y'-j W_r| }{W_r} - |j'-j| - 2 \Big)^+\Big)^2 Q_r.
\]

\end{lemma}
\begin{center}
\begin{figure}
\includegraphics[width=6in]{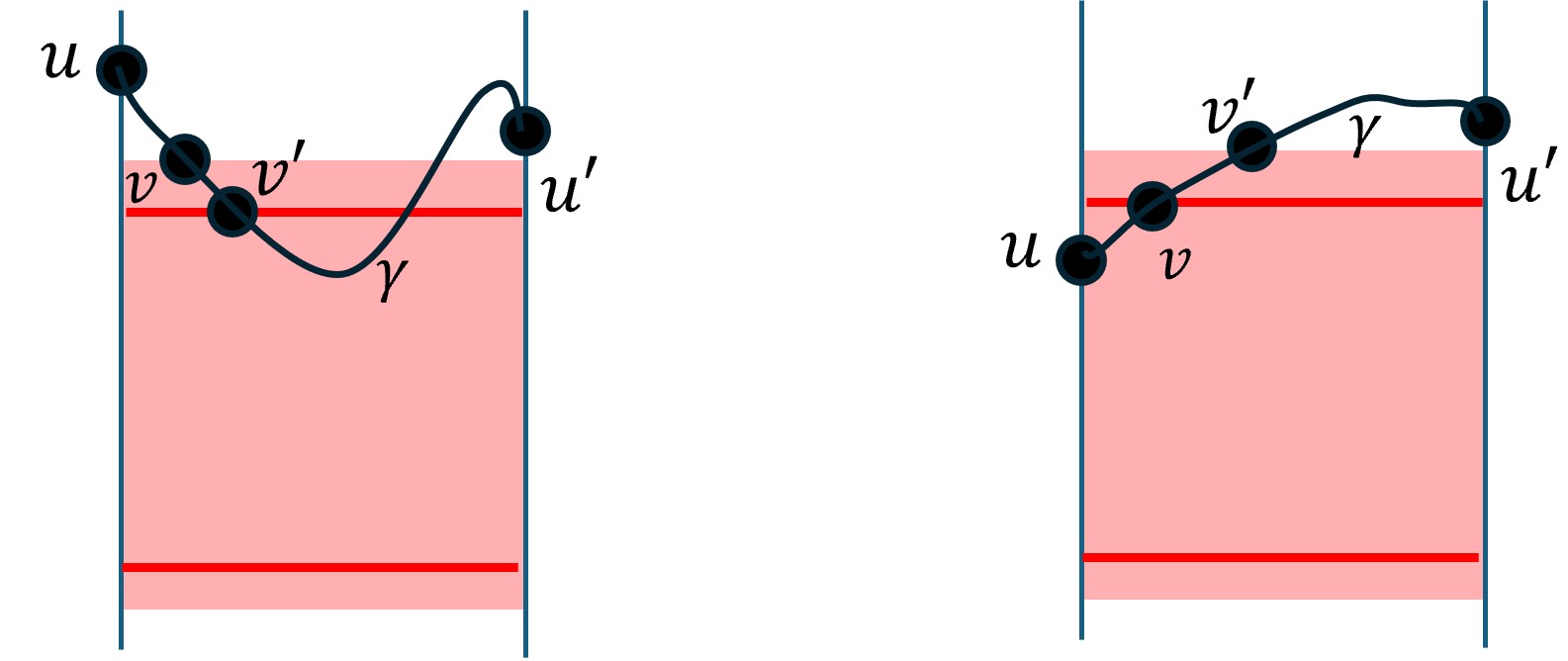}
\caption{Proof of Lemma \ref{l:barrier.entry}. This lemma gives a lower bound on passage times of paths across a column that intersects a certain rectangle (the region between the two dark red horizontal lines in the figures) provided certain barrier type events hold on a slightly larger rectangle (the region marked in light red in the figures). There are two cases: the left panel depicts the case where the starting point of the path is outside the barrier region while the right panels illustrates the case where the path starts inside the barrier region.}
\label{f:BCross}
\end{figure}
\end{center}
\begin{proof}
Suppose $\gamma_{u,u'}$ intersects $[(i-1)r,ir]\times [jW_r,j' W_r]$.  Fixing $\epsilon>0$, by Lemma~\ref{l:strongly.conforming} we can find a strongly conforming path $\zeta$ with $\zeta(0)=u$ and $\zeta(1)=u'$ that also intersects $[(i-1)r,ir]\times [jW_r,j' W_r]$ with $\rX_\zeta \leq \rX_{u,u'}+\epsilon$.  First suppose that $y\geq (j'+w)W_r$.  Then since $\zeta$ intersects $[(i-1)r,ir]\times [jW_r,j' W_r]$ we can find  $v=(x,(j'+w)W_r),v'=(x',j'W_r)$ along $\zeta$ in that order with $(i-1)r\leq x\leq x' \leq ir$ (see the left panel of Figure \ref{f:BCross}) and none of $(u,v), (v,v')$ and $(v',u')$ are vertical boundary pairs.  Then  by $\cJ_{i,j-w,j+w,z_1,w}$ and  $\cK^*_{i,j',z_2,w}$,
\begin{align*}
&\rX_{u,u'}\geq \rX_{\zeta} -\epsilon = {\rX_{u,v}} + {\rX_{v,v'}} + {\rX_{v',u'}}-\epsilon\\
&\geq  {(x-(i-1)r) + \frac12\Big(\frac{|y-(j'+w)W_r| }{W_r} - 1\Big)^2 Q_r - z_2 Q_r} +  {(x'-x) + z_1 Q_r}\\
&\qquad + {(ir-x') + \frac12\Big(\frac{|y'-j'W_r| }{W_r} - 1\Big)^2 Q_r -z_2 Q_r}-\epsilon\\
&\geq (z_1-2z_2) Q_r + r + \frac12 \Big(\Big(\frac{|y-jW_r|  }{W_r} - |j'-j| - 2 \Big)^+\Big)^2 Q_r + \frac12 \Big(\Big(\frac{|y'-j W_r| }{W_r} - |j'-j| - 2 \Big)^+\Big)^2 Q_r -\epsilon,
\end{align*}
as required, where we have also used the assumption that $0< w\le 1$. The case of $y\leq (j-w)W_r$ follows similarly.  

Suppose next that $y\in [(j-w)W_r,(j'+w)W_r]$. If $\gamma_{u,u'}$ stays within $[(i-1)r,ir]\times [(j-w)W_r,(j'+w) W_r]$ then by $\cI^+_{i,j-w,j+w,z_1}$,
\begin{align*}
&\rX_{u,u'} 
\geq z_1 Q_r + r + \frac12 \Big(\frac{|y-y'|)^+}{W_r} \Big)^2 Q_r\\
&\geq (z_1-2z_2) Q_r + r + \frac12 \Big(\Big(\frac{|y-jW_r|  }{W_r} - |j'-j| - 2 \Big)^+\Big)^2 Q_r + \frac12 \Big(\Big(\frac{|y'-j W_r| }{W_r} - |j'-j| - 2 \Big)^+\Big)^2 Q_r.
\end{align*}
The last step above follows since $z_2>0$ and from the fact that for $y\in [(j-w)W_r,(j'+w)W_r]$
$$ \frac{|y-j W_r| }{W_r} - |j'-j| - 2 \le 0.$$

If $\gamma_{u,u'}$ does not stay within $[(i-1)r,ir]\times [(j-w)W_r,(j'+w) W_r]$,  $\gamma_{u,u'}$ must intersect either $H_{i,j}$ and $H_{i,j-w}$ or $H_{i,j'}$ and $H_{i,j'+w}$ and so must $\zeta$. Suppose that it is the latter case and that $v=(x,(j'W_r),v'=(x',(j'+w)W_r)$ are the intersection points (see the right panel of Figure \ref{f:BCross}).  We will assume that $v$ is hit before $v'$ but the case of the opposite order follows similarly.  Then the same estimates as in the $y\geq (j'+w)W_r$ case hold and
\begin{align*}
\rX_{u,u'} &\geq \rX_{\zeta} -\epsilon= {\rX_{u,v}} + {\rX_{v,v'}} + {\rX_{v',u'}}-\epsilon\\
&\geq  {(x-(i-1)r) + \frac12\Big(\frac{|y-(j'+w)W_r| }{W_r} - 1\Big)^2 Q_r - z_2 Q_r} +  {(x'-x) + z_1 Q_r}\\
&\qquad + {(ir-x') + \frac12\Big(\frac{|y'-j'W_r| }{W_r} - 1\Big)^2 Q_r -z_2 Q_r}-\epsilon\\
&\geq (z_1-2z_2) Q_r + r + \frac12 \Big(\Big(\frac{|y-jW_r|  }{W_r} - |j'-j| - 2 \Big)^+\Big)^2 Q_r\\
&\qquad + \frac12 \Big(\Big(\frac{|y'-j W_r| }{W_r} - |j'-j| - 2 \Big)^+\Big)^2 Q_r -\epsilon.
\end{align*}
The case of hitting $H_{i,j}$ and $H_{i,j-w}$ follows similarly.  Taking $\epsilon\to 0$ completes the result.
\end{proof}

\subsection{Wing Passage Times}
Our next job is to analyze consequences of the wing events. We have the following lemma. 

\begin{lemma}\label{l:wing.bound}
For $i,j$ such that $M^{99/100}n\leq ir \leq (M-M^{99/100})n$ and $|jW_r|\leq M^{4/5}W_n$, on the event $\cW_{i,j}$,  
\begin{equation}\label{eq:wing.boundA}
\inf_{|y|\leq n^\beta W_n} \rX_{\origin,((i-3)r,y)} - \rX_{\origin,((i-3)r,jW_r)} +\Big(5|y/W_r - j| + 2L_2\Big)Q_r\geq 0
\end{equation}
and
\begin{equation}\label{eq:wing.boundB}
\sup_{y\in [jW_r-2L_2W_r,jW_r+2L_2 W_r]} \rX_{\origin,((i-3)r,y)} - \rX_{\origin,((i-3)r,jW_r)} - 2L_2 Q_r \leq 0.
\end{equation}
The same estimates hold for $\rX'$.
\end{lemma}

\begin{proof}
We shall only prove the statements for $\rX$, the proofs for $\rX'$ are identical. 

\noindent
\emph{Proof of \eqref{eq:wing.boundA}.}
Let $u=((i-3)r,y)$.  We will begin with the case that $|y|\geq 2M^{4/5}W_n$. Recall that 
$$\widehat{\rX}_{\origin,u}=\rX_{\origin,u}-\frac{y^2}{2(i-3)r}; \quad \widehat{\rX}_{\origin,((i-3)r,jW_r)}=\rX_{\origin,((i-3)r,jW_r)}-\frac{(jW_r)^2}{2(i-3)r}$$
and therefore 
$$\rX_{\origin,u} - \rX_{\origin,((i-3)r,jW_r)}=\widehat{\rX}_{\origin,u}-\widehat{\rX}_{\origin,((i-3)r,jW_r)}-\frac{y^2-(jW_r)^2}{2(i-3)r}.$$
Then by $\cW_{i,j}\subset\cW^{*}_{i}$, and using $|jW_r|\le M^{4/5}W_n$ and $(i-3)r\le Mn$ we get
\begin{align*}
\rX_{\origin,u} - \rX_{\origin,((i-3)r,jW_r)}&\geq \frac{y^2 - (M^{4/5}W_n)^2}{2(i-3)r}  - 2  (1\vee M^{-2}W_n^{-1}|y|)\log^{\frac{100}{\theta_2}} (M) Q_{Mn}\\
&\geq \frac{\frac{7}{9}M^{8/5}}{M}Q_n + \frac1{18M} (y/W_n)^2Q_n  - 2  (1\vee M^{-2}W_n^{-1}|y|)\log^{\frac{100}{\theta_2}} (M) Q_{Mn} > 0,
\end{align*}
where in the second inequality we have also used $y^2 \geq \frac19 y^2 + (\frac43M^{4/5}W_n)^2$, and the final inequality follows since $Q_{Mn}=O(M^{1/2}Q_n)$.

It therefore suffices to consider the case $|y|\leq 2M^{4/5}W_n$. Let 
\[
S=|\frac{y}{W_r}-j|+1
\]
and let $v_k=((i-3-2^k L_2)r,y_k')$ be the intersection of $\gamma_{\origin,u}$ with the line $x=(i-3-2^k L_2)r$; see Figure \ref{f:Wing}. Set
\[
k_{\min} = \lfloor\log_2 (\frac{S}{L_2}\vee 1)\rfloor,\qquad k_{\max} = \lfloor\log_2 (\frac12M^{99/100}{L_2^{-1}\Phi^{-\ell}})\rfloor.
\]
We will show that the events $A_k$ defined as
\[
A_k=\big\{|y-y_k'| \geq (2^k L_2)^{9/10} W_r\big\}
\]
do not hold for any $k_{\min}\leq k \leq k_{\max}$.  

\begin{center}
\begin{figure}[htbp!]
\includegraphics[width=6in]{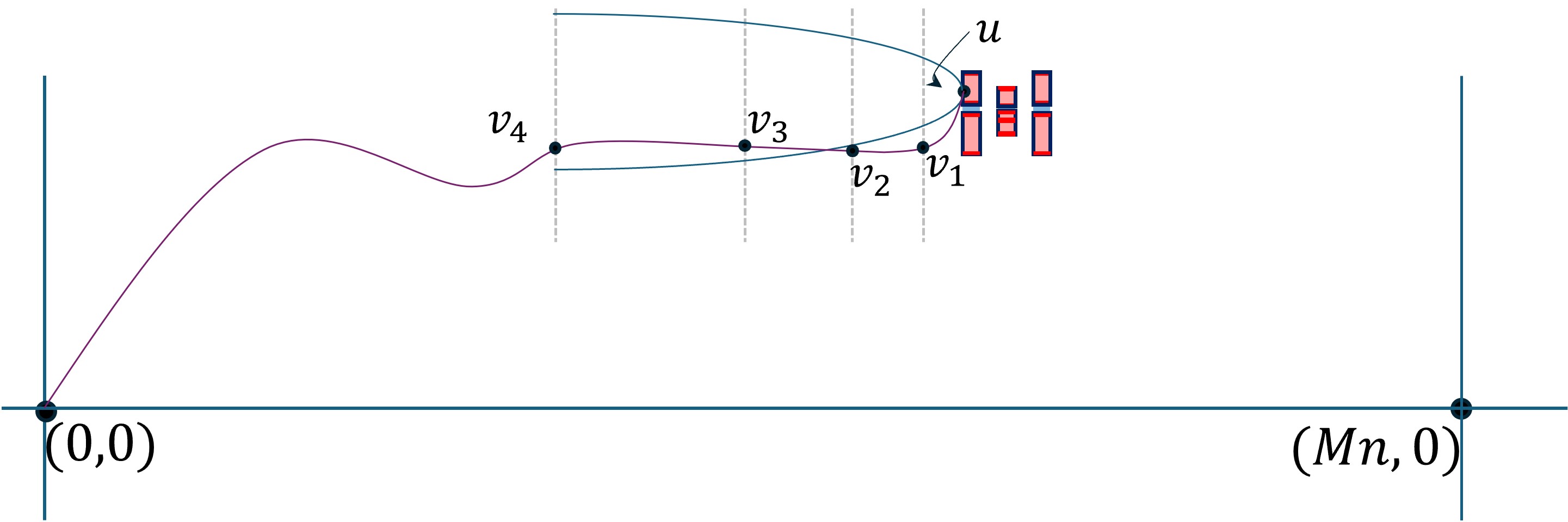}
\caption{Proof of Lemma \ref{l:wing.bound}. The path in the figure is the geodesic from the point $(0,0)$ to $u=((i-3)r,y)$ and $v_k$ denotes the intersection of this path with the vertical lines at distance $2^{k}L_2Sr$ to the left of $u$. We want to show that on the event $\cW_{i,j}$, the passage time $\cX_{\origin,u}$ cannot be too much smaller compared to the passage time  $\cX_{\origin,(i-3)r,jW_r}$. We show this by showing first that by a chaining argument on the event $\cW_{i,j}$, the vertical coordinates of the points $v_k$ are not too far from $y$, which lets us lower bound $\cX_{\origin,u}-\cX_{\origin,(i-3)r,jW_r}$ by the triangle inequality.}
\label{f:Wing}
\end{figure}
\end{center}

Suppose that such an $A_k$ does occur and let $k_\star$ be the largest such $k$ between $k_{\min}$ and $k_{\max}$. We shall treat the cases $k_*<k_{\max}$ and $k_*=k_{\max}$ separately. 

\textbf{Case 1.} If $k_\star<k_{\max}$ then
\begin{align*}
D:=\frac{(y-y_{k_\star}')^2 + (y_{k_\star}'-y_{k_\star+1}')^2 - \frac12 (y-y_{k_\star+1}')^2}{2^{k_\star+1}L_2 r} 
&= \frac{2\Big((y-y_{k_\star}') - \frac12(y-y_{k_\star+1}')\Big)^2}{2^{k_\star+1}L_2 r} \\
&\geq \frac{2\Big(|y-y_{k_\star}'| - 2^{\frac{9}{10}(k_\star+1)-1}L_2^{9/10} W_r \Big)^2}{2^{k_\star+1}L_2  r}
\end{align*}
where the inequality used that 
\[
|y-y_{k_\star}'| \geq (2^{k_\star} L_2)^{9/10}  W_r \geq \frac{21}{20} 2^{\frac{9}{10}(k_\star+1)-1}L_2^{9/10} W_r \geq \frac{21}{20}\cdot \frac12 |y-y_{k_\star+1}'|.
\]
Hence, using the above two equations, 
\begin{align}\label{eq:D.traingle.bound}
D\geq \frac{|y-y_{k_\star}'|^2}{400\cdot 2^{k_\star+1}L_2 r} + \frac{\Big(2^{\frac{9k_\star-1}{10}}L_2^{9/10} W_r \Big)^2}{400\cdot 2^{k_\star+1}L_2 r}
\end{align}
Since $A_{k_\star+1}$ does not hold,  
\begin{align*}
|y_{k_\star+1}'| &\leq |y| + |y-y_{k_\star+1}'|\leq 2 M^{\frac45}W_n + (2^{k_{\max}} L_2)^{9/10} W_r\\
&\leq 2 M^{\frac45}W_n  + (M^{\frac{99}{100}}\Phi^{-\ell})^{9/10} \Phi^{\frac34\ell}W_n \leq MW_n.
\end{align*}
If $|y_{k_\star}'|\leq MW_n$ 
\begin{align*}
&\rX_{v_{k_\star},u} + \rX_{v_{k_\star+1},v_{k_\star}}-\rX_{v_{k_\star+1},u}\\
&\qquad \geq \hrX_{v_{k_\star},u} + \hrX_{v_{k_\star+1},v_{k_\star}}-\hrX_{v_{k_\star+1},u} + \frac{|y-y_{k_\star}'|^2}{400\cdot 2^{k_\star+1}L_2 r} + \frac{\Big(2^{\frac{9k_\star-1}{10}}L_2^{9/10} W_r \Big)^2}{400\cdot 2^{k_\star+1}L_2 r}\\
&\qquad \geq -3(2^{k_\star+1} L_2)^{3/5} Q_{r} - 4\Big( S + 2^{\frac{9(k_\star+1)}{10}}\Big)^{\tfrac1{100}}\frac{Q_r}{2^{k_\star} L_2^2} 
- \Big( S + \frac{|y-y_{k_\star}'|}{W_r}\Big)^{\tfrac1{100}}\frac{Q_r}{2^{k_\star} L_2^2}\\
&\qquad \quad+ \frac{|y-y_{k_\star}'|^2}{400\cdot 2^{k_\star+1}L_2 r} + \frac{\Big(2^{\frac{9k_\star-1}{10}}L_2^{9/10} W_r \Big)^2}{400\cdot 2^{k_\star+1}L_2 r} 
\end{align*}
where the first inequality is by by equation~\eqref{eq:D.traingle.bound}, the second is by applying the event  $\cW_{i,j}^{loc}\cap \cW_{i,j}^{glo}$ to the passages times.  Now since $2^{k_\star}\geq 2^{k_{\min}}\geq S$ it follows that
\begin{align*}
\frac{\Big(2^{\frac{9k_\star-1}{10}}L_2^{9/10} W_r \Big)^2}{400\cdot 2^{k_\star+1}L_2 r} -3(2^{k_\star+1} L_2)^{3/5} Q_{r} - 4\Big( S + 2^{\frac{9(k_\star+1)}{10}}\Big)^{\tfrac1{100}}\frac{Q_r}{2^{k_\star} L_2^2} 
-  S^{\tfrac1{100}}\frac{Q_r}{2^{k_\star} L_2^2}-Q_r\\
\geq \bigg(\frac1{800}2^{\frac{4k_\star}{5}}L_2^{8/10}  - 6\cdot2^{\frac{3k_\star}{5}}L_2^{3/5} - 10\cdot2^{-99k_{\star}/100}L_2^{-2}  -1\bigg)Q_r>0
\end{align*}
and for $L_2$ large enough
\begin{align*}
\frac{|y-y_{k_\star}'|^2}{400\cdot 2^{k_\star+1}L_2 r} +Q_r
- \Big(\frac{|y-y_{k_\star}'|}{W_r}\Big)^{\tfrac1{100}}\frac{Q_r}{2^{k_\star} L_2^2}>0
\end{align*}
so adding the last two equations and using that $x^{\frac1{100}}+y^{\frac1{100}}\geq (x+y)^{\frac1{100}}$ we have that
\[
\rX_{v_{k_\star},u} + \rX_{v_{k_\star+1},v_{k_\star}}-\rX_{v_{k_\star+1},u}>0
\]
which gives a contradiction.  

If $|y_{k_\star}'|>MW_n$ then we instead apply the bounds from $\cW_i^*$ and we get
\begin{align*}
&\rX_{v_{k_\star},u} + \rX_{v_{k_\star+1},v_{k_\star}}-\rX_{v_{k_\star+1},u}\\
&\qquad \geq \hrX_{v_{k_\star},u} + \hrX_{v_{k_\star+1},v_{k_\star}}-\hrX_{v_{k_\star+1},u} + \frac{(\frac12 MW_n+\frac12|y_{k_\star}'|-2M^{4/5}W_n)^2}{2Mn}\\
&\qquad \geq \frac1{20}MQ_n + \frac1{20} \Big(\frac{|y_{k_\star}'|}{W_n}\Big)^2 Q_n -  (1+1\vee M^{-2}W_n^{-1}|y_{k_\star}|)\log^{\frac{100}{\theta_2}} (M) Q_{2^{k_\star +1}r}> 0,
\end{align*}
by using $Q_{2^{k_\star +1}r}=O(M^{1/2}Q_n)$, again giving a contradiction.

\textbf{Case 2}. Next suppose that $k_\star=k_{\max}$.  Then $2^{k_\star}L_2 \Phi^\ell \geq \frac18 M^{99/100}$ and so
\[
|y_{k_\star}-y| \geq 2^{\frac{9k_\star}{10}}W_r \geq  \Big(\frac{M^{99/100}}{8L_2 \Phi^\ell} \Big)^{\frac{9}{10}}W_r\geq M^{\frac{17}{20}}W_r.
\]
and hence
\begin{align*}
\frac{(y-y_{k_\star}')^2 }{2^{k_\star+1}L_2 r} + \frac{y_{k_\star}'^2}{2((i-3) - 2^{k_\star}L_2) r} - \frac{y^2}{2(i-3) r} &\geq \frac{M^{34/20}W_r^2+|y_{k_\star}-y|^2}{8Mn} - \frac{(2M^{4/5}W_r)^2}{\tfrac14M^{99/100}n}\\
&\geq \frac14 M^{14/20}Q_n +\frac{|y_{k_\star}-y|^2}{8Mn}.
\end{align*}
Then using the estimates from $\cW_i^*$,
\begin{align*}
\rX_{\origin,v_{k_\star}}+\rX_{v_{k_\star},u} -\rX_{\origin,u} \geq \frac14 M^{14/20}Q_n +\frac{|y_{k_\star}-y|^2}{8Mn} - 3  (1+1\vee M^{-2}W_n^{-1}|y_{k_\star}|)\log^{\frac{100}{\theta_2}} (M) Q_{2^{k_\star}r} >0
\end{align*}
and so $A_{k_{\max}}$ does not hold.  

Hence no $A_k$ holds between $k_{\min}$ and $k_{\max}$ and in particular, $A_{k_{\min}}$ does not hold so
\[
|y_{k_{\min}}' - jW_r| \leq ((2^{k_{\min}}L_2)^{9/10} +S\big) W_r \leq \Big(\Big(S\vee L_2\Big)^{9/10} +S\Big)W_r.
\]
Since the optimal path from $\origin$ to $u=((i-3)r,y)$ passes through $v_{k_{\min}}$  we have that 
\begin{align*}
&\rX_{\origin,((i-3)r,y)} - \rX_{\origin,((i-3)r,jW_r)} \\
&\geq \rX_{v_{k_{\min}},((i-3)r,y)} - \rX_{v_{k_{\min}},((i-3)r,jW_r)}\\
&\geq \hrX_{v_{k_{\min}},((i-3)r,y)} - \hrX_{v_{k_{\min}},((i-3)r,jW_r)} - \frac{\Big(((2^{k_{\min}}L_2)^{9/10} +S\big) W_r\Big)^2}{2^{k_{\min}+1}L_2 r}\\
&\geq - 4\Big((2^{k_{\min}}L_2)^{9/10} +S\Big)^{\tfrac1{100}}\frac{Q_r}{2^{k_{\min}} L_2^2} 
-(2^{k_{\min}}L_2)^{3/5} Q_{r} 
- \frac{\Big((2^{k_{\min}}L_2)^{9/10} +S\Big)^2}{2^{k_{\min}+1}L_2 }Q_r\\
&\geq - 4\Big((S\vee L_2)^{9/10} +S\Big)^{\tfrac1{100}}\frac{Q_r}{(S\vee L_2) L_2} 
-(S\vee L_2)^{3/5} Q_{r} 
- \frac{\big((S\vee L_2)^{9/10} +S\big)^2}{2(S\vee L_2) }Q_r\\
&\geq - \Big(5|y/W_r - j| + 2L_2\Big)Q_r
\end{align*}
where in the second inequality above we have used the upper bound on $|y_{k_{\min}}' - jW_r|$ and in the third inequality we have used the definition of the event $\cW_{i,j}^{loc}\cap \cW_{i,j}^{glo}$.
This establishes~\eqref{eq:wing.boundA} for $|y|\leq 2M^{4/5}W_n$ completing the proof of that equation.

\noindent
\emph{Proof of \eqref{eq:wing.boundB}.}
Choose $y=jW_r$ and set $w=v_{10}$ defined as above.  We have from the above calculations that $A_{10}$ holds and so $w=((i-3-2^{10}L_2)r,\hat{y})$ where $|\hat{y}-jW_r| \leq (2^{10}L_2)^{9/10}W_r$.
Moving onto equation~\eqref{eq:wing.boundB} for any $y\in [jW_r-2L_2W_r,jW_r+2L_2 W_r]$,
\begin{align*}
\rX_{\origin,((i-3)r,y)} - \rX_{\origin,((i-3)r,jW_r)} &\leq \rX_{w,((i-3)r,y)} - \rX_{w,((i-3)r,jW_r)}\\
&\leq \hrX_{w,((i-3)r,y)} - \hrX_{w,((i-3)r,jW_r)} + 2\frac{(2L_2W_r + (2^{10}L_2)^{9/10} W_r)^2}{2^{11}L_2 r}\\
&\leq 2(2L_2 + (2^{10}L_2)^{9/10})^{\tfrac1{100}}\frac{Q_r}{2^{10} L_2^2} + \frac{(3L_2 )^2}{2^{11}L_2}Q_r\\
&\leq 2L_2 Q_r,
\end{align*}
which completes the proof.
\end{proof}

\subsection{Outer Columns}
The next lemma is a consequence of the definition of the outer column events. 

\begin{lemma}\label{l:barrier6}
On the event $\cW_{i,j}\cap D_{i-2,j}$, for $M^{99/100}n\leq ir \leq (M-M^{99/100})n$ and $|jW_r|\leq M^{4/5}W_n$, for all $|y|\leq n^{\beta}W_n$
\begin{align}\label{eq:barrier.column}
&\rX_{\origin,((i-2)r,y)} - \inf_{y'\in [-\tfrac{L_0 W_r}{20},\tfrac{L_0 W_r}{20}]}\rX_{\origin,((i-2)r,y'+jW_r)}\nonumber\\
&\qquad\geq -\bigg( 6L_2 + \frac{7|y - jW_r|}{W_r} \bigg)Q_r + I\Big(|y-jW_r|\in (\frac1{20}L_0 W_r,(L_2 -3) W_r]\Big)\frac{L_2^3}{2} Q_r.
\end{align}
The same estimate also holds for $\cX'$. 

\end{lemma}
\begin{center}
\begin{figure}
\includegraphics[width=1.5in]{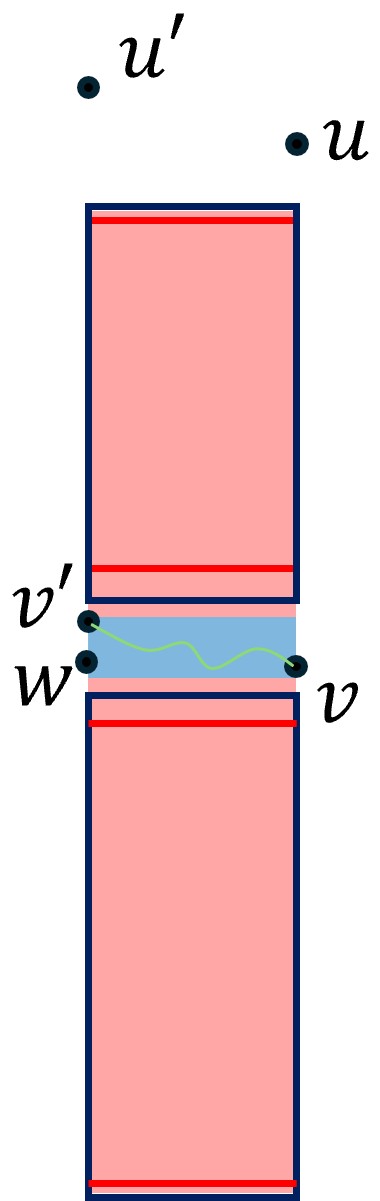}
\caption{Proof of Lemma \ref{l:barrier6}: this lemma shows that the passage time from $\bf0$ to $u=((i-3)r,y)$ in the figure cannot be too small compared to the minimum passage time from $0$ to the right side of the rectangular region marked in blue. On the outer column event $\cD_{i-2,j}$ for the column $[(i-3)r,(i-2)r]$ and the wing event implies that the minimum passage time can be compared to $\cX_{0,w}$. Next we take $u'$ to be the point where the geodesic from $\bf0$ to $u$ intersects the vertical line $x=(i-3)r$. The depicts the case when $y>(L_2-3+j)W_r$, we lower bound $\cX_{0,u}+\cX_{u,u'}-\cX_{0,w}$ by using Lemma \ref{l:wing.bound} and the event $\cD^{(3)}_{i-2,j}$. The other cases are dealt with similarly.}
\label{f:Eventouter}
\end{figure}
\end{center}
\begin{proof}

By $D_{i-2,j}$ we can find points $v'=((i-3)r,q'+jW_r),v=((i-2)r,q+jW_r)$  points such that $|q|,|q'|\leq \frac1{40}L_0 W_r$ and $\hrX_{v',v} \leq L_0 Q_r$ (we can find such points by $\cD^{(1)}_{i-2,j}$); See Figure \ref{f:Eventouter}.  Then if we set $w=((i-3)r,jW_r)$ we get
\begin{align}\label{eq:barrierA}
\inf_{y'\in [-\tfrac{L_0 W_r}{20},\tfrac{L_0 W_r}{20}]}\rX_{\origin,((i-2)r,y'+jW_r)}- \rX_{\origin,w}
&\leq\rX_{\origin,v} - \rX_{\origin,w} \nonumber\\
&\leq \rX_{\origin,v'} +\rX_{v',v} - \rX_{\origin,w} \nonumber \\
&\leq 2L_2 Q_r  + r + \frac{(q-q')^2}{2r} +L_0 Q_r \nonumber \\
&\leq 3L_2 Q_r + r.
\end{align}
where the third inequality is by Lemma~\ref{l:wing.bound}.
Equation~\eqref{eq:barrier.column} is trivially true for $y\in [jW_r-\tfrac{L_0 W_r}{20},jW_r+\tfrac{L_0 W_r}{20}]$ since the right hand side is positive. So assume $y\not \in [jW_r-\tfrac{L_0 W_r}{20},jW_r+\tfrac{L_0 W_r}{20}]$. Set $u=((i-2)r,y)$ and $u'=((i-3)r,y')$ such that  $\rX_{\origin,u} =\rX_{\origin,u'}+\rX_{u',u}$. 

For $y \geq (L_2 -3 + j) W_r$ by Lemma~\ref{l:wing.bound} and $\cD_{i-2,j}^{(3)}$ we have that
\begin{align*}
&\rX_{\origin,u} - \rX_{\origin,w}=\rX_{\origin,u'} + \rX_{u',u} - \rX_{\origin,w}\\
&\qquad \geq r + \bigg(-\frac{5|y' - jW_r|}{W_r} - 2L_2 + \frac12\Big(\frac{y-y'}{W_r}\Big)^2 - \frac{|y - jW_r| + |y' - jW_r|}{W_r} -L_0  \bigg)Q_r.
\end{align*}
Differentiating the right hand side in $y'$, it is minimized at $y'=y+6W_r$ and so
\begin{align}\label{eq:barrierB}
\rX_{\origin,u} - \rX_{\origin,w}
&\geq r + \bigg( - 2L_2 - 18 - \frac{7|y - jW_r|}{W_r} -L_0  \bigg)Q_r \nonumber\\
&\geq r +\bigg( - 3L_2 - \frac{7|y - jW_r|}{W_r} \bigg)Q_r.
\end{align}
If $y\in [(\frac1{20}L_0 + j) W_r,(L_2 -3 + j) W_r]$ then $\gamma$ must pass through $[(i-1)r,r)]\times [(\frac1{20}L_0 + j) W_r,(L_2 -3 + j) W_r]$ and so by Lemma~\ref{l:barrier.entry} and $D_{i-2,j}$,
\begin{align*}
&\rX_{\origin,u'} + \rX_{u',u} - \rX_{\origin,w}\\
&\qquad \geq r + \bigg(-\frac{5|y' - jW_r|}{W_r} - 2L_2 + (L_2^3-2L_1) + \frac12 \Big(\Big(\frac{|y'-jW_r|  }{W_r} - 2L_2 \Big)^+\Big)^2 \bigg)Q_r.
\end{align*}
Again optimizing over $y'$ we have that
\begin{align}\label{eq:barrierC}
\rX_{\origin,u} - \rX_{\origin,w}  \geq r +\frac12 L_2^3  Q_r.
\end{align}
Combining with~\eqref{eq:barrierA}, \eqref{eq:barrierB} and~\eqref{eq:barrierC}, this completes the result for $y>jW_r$ and the case $y<jW_r$ follows similarly.
\end{proof}

\subsection{Intermediate Columns}
The next lemma shows that on the intersection of the outer column, intermediate column and the wing events, passage times from $\bf0$ vary in a sufficiently regular manner along the right bounder of the left intermediate column.

\begin{lemma}\label{l:intermeadiate.bound}
On the event $\cW_{i,j}\cap \cD_{i-2,j}\cap \cC_{i-1,j}$, for $M^{99/100}n\leq ir \leq (M-M^{99/100})n$ and $|jW_r|\leq M^{4/5}W_n$ and for all $|y|\leq n^{\beta}W_n$ 
\begin{align*}
\rX_{\origin,((i-1)r,y)} - \rX_{\origin,((i-1)r,jW_r)}\geq -\bigg( 8L_2 +\frac{8|y - jW_r|}{W_r}\bigg)Q_r ,
\end{align*}
and for $y\in [(j-\frac13 L_2)W_r,(j+\frac13 L_2)W_r]$
\begin{align*}
&|\rX_{\origin,((i-1)r,y)} - \rX_{\origin,((i-1)r,jW_r)}  - \frac12\Big(\frac{|y-jW_r|}{W_r}\Big)^2 Q_r|\\
&\leq \bigg(\frac{L_0^2}{800} + \frac{21}{10}L_0  + (1+\frac{L_0}{20})\frac{|y-jW_r|}{W_r} \bigg)Q_r 
\end{align*}
The same bounds hold for $\rX'$.
\end{lemma}

\begin{proof}
We choose $y_*\in [(j-\tfrac{L_0}{20})W_r,(j+\tfrac{L_0}{20})W_r]$ such that with $v=((i-2)r,y_*)$ we have that that
\[
\rX_{\origin,v} = \inf_{y'\in [(j-\tfrac{L_0}{20})W_r,(j+\tfrac{L_0}{20})W_r]}\rX_{\origin,((i-2)r,y')}.
\]
By Lemma~\ref{l:barrier.entry} and $\cC_{i-1,j}$,
\begin{align}\label{eq:intermeadiateA}
\rX_{\origin,((i-1)r,y)}&= \inf_{|y'|\leq n^\beta W_n}  \rX_{\origin,((i-2)r,y')} + \rX_{((i-2)r,y'),((i-1)r,y)}\nonumber\\
&\geq -\bigg( 6L_2 + \frac{7|y' - jW_r|}{W_r} \bigg)Q_r + I\Big(|y'-jW_r|\in (\frac1{20}L_0 W_r,(L_2 -3) W_r]\Big)\frac{L_2^3}{2} Q_r\nonumber\\
&\qquad +r +\frac12\Big(\frac{y-y'}{W_r}\Big)^2 Q_r-\bigg( L_0 + \frac{|y' - jW_r|}{W_r}+ \frac{|y - jW_r|}{W_r} \bigg)Q_r\nonumber\\
&\geq \rX_{\origin,v} +r  -\bigg( 7L_2 +8\frac{|y-jW_r|}{W_r}\bigg)Q_r
\end{align}
and where the last inequality follows by optimizing over $y'$.
For $y\in [(j-\frac13 L_2)W_r,(j+\frac13 L_2)W_r]$, by $\cC_{i-1,j}$ we have that
\begin{align}\label{eq:intermeadiateB}
\rX_{\origin,((i-1)r,y)} &\leq  \rX_{\origin,v} +r +\frac12\Big(\frac{y-y_*}{W_r}\Big)^2  Q_r+\bigg(L_0 +\frac{|y-jW_r|}{W_r} +\frac{|y_*-jW_r|}{W_r} \bigg)Q_r \nonumber\\
&\leq  \rX_{\origin,v} +r +\frac12\Big(\frac{|y-jW_r|}{W_r}\Big)^2Q_r  + \bigg(\frac{L_0^2}{800} + \frac{21}{20}L_0  + (1+\frac{L_0}{20})\frac{|y-jW_r|}{W_r} \bigg)Q_r \nonumber\\
&\leq  \rX_{\origin,v} +r + \frac1{17} L_2^2 Q_r.
\end{align}
When $y=jW_r$ this gives
\[
\rX_{\origin,((i-1)r,y)} \leq \rX_{\origin,v} +r + \bigg(\frac{L_0^2}{800} + \frac{21}{20}L_0  \bigg)Q_r
\]
Combining this with equation~\eqref{eq:intermeadiateA} for any $|y|\leq n^\beta W_n$,
\begin{align*}
\rX_{\origin,((i-1)r,y)} - \rX_{\origin,((i-1)r,jW_r)} &\geq  -\bigg( 7L_2 +8\frac{|y-jW_r|}{W_r}\bigg)Q_r - \bigg(\frac{L_0^2}{800} + \frac{21}{20}L_0   \bigg)Q_r \\
&\geq -\bigg( 8L_2 +8\frac{|y-jW_r|}{W_r}\bigg)Q_r 
\end{align*}
which completes the proof of the first part of the lemma.  For the remainder assume that $y\in [(j-\frac13 L_2)W_r,(j+\frac13 L_2)W_r]$.  We will show that the path from the origin to $(i-1)r,y)$ must pass though $u= ((i-2)r,y')$ with $y'\in [(j-\tfrac{L_0}{20})W_r,(j+\tfrac{L_0}{20})W_r]$.  Splitting into two cases, first by~\eqref{eq:intermeadiateA} and $\cC_{i-1,j}$,
\begin{align}\label{eq:intermeadiateC}
&\inf_{\substack{u= ((i-2)r,y') \\ |y'-jW_r|\in [\frac1{20}L_0 W_r,(L_2 -3) W_r]}} \rX_{\origin,u} 
+ \rX_{u,((i-1)r,y)} \nonumber\\
&\qquad\geq \inf_{|y'|\leq n^\beta W_n} \rX_{\origin,v} +r +\frac12\Big(\frac{y-y'}{W_r}\Big)^2 Q_r -\bigg(L_0+6L_2 +\frac{|y-jW_r|}{W_r} + 8\frac{|y'-jW_r|}{W_r} \bigg)Q_r + \frac{L_2^3}{2} Q_r\nonumber\\
&\qquad\geq  \rX_{\origin,v} +r + \frac{L_2^3}{3} Q_r 
\end{align}
and secondly
\begin{align}\label{eq:intermeadiateD}
&\inf_{\substack{u= ((i-2)r,y') \\ |y'-jW_r| \geq (L_2 -3) W_r}} \rX_{\origin,u} 
+ \rX_{u,((i-1)r,y)} \nonumber\\
&\qquad\qquad\geq \inf_{|y'|\leq n^\beta W_n} \rX_{\origin,v} +r +\frac12\Big(\frac{y-y'}{W_r}\Big)^2 Q_r -\bigg(L_0+6L_2 +\frac{|y-jW_r|}{W_r} + 8\frac{|y'-jW_r|}{W_r} \bigg)Q_r \nonumber\\
&\qquad\qquad\geq  \rX_{\origin,v} +r + \frac{L_2^2}{3} Q_r 
\end{align}
which are both greater than $\rX_{\origin,v} +r + \frac1{17} L_2^2 Q_r$.  Hence we have that
\begin{align}\label{eq:intermeadiateE}
\rX_{\origin,((i-1)r,y)}&=\inf_{\substack{u= ((i-2)r,y') \\ y'\in [(j-\tfrac{L_0}{20})W_r,(j+\tfrac{L_0}{20})W_r]}} \rX_{\origin,u} 
+ \rX_{u,((i-1)r,y)} \nonumber\\
&\geq \inf_{|y'|\leq n^\beta W_n} \rX_{\origin,v} +r +\frac12\Big(\frac{y-y'}{W_r}\Big)^2 Q_r  -\bigg(L_0 +\frac{|y-jW_r|}{W_r} +\frac{|y'-jW_r|}{W_r} \bigg)Q_r \nonumber\\
&\geq  \rX_{\origin,v} +r +\frac12\Big(\frac{|y-jW_r|}{W_r}\Big)^2 Q_r  -\bigg(\frac{21}{20}L_0  + (1+\frac{L_0}{20})\frac{|y-jW_r|}{W_r} \bigg)Q_r 
\end{align}
Combining equations~\eqref{eq:intermeadiateB} and~\eqref{eq:intermeadiateE} we have that
\begin{align*}
&|\rX_{\origin,((i-1)r,y)} - \rX_{\origin,((i-1)r,jW_r)}   - \frac12\Big(\frac{|y-jW_r|}{W_r}\Big)^2 Q_r|\\
&\leq \bigg(\frac{L_0^2}{800} + \frac{21}{10}L_0  + (1+\frac{L_0}{20})\frac{|y-jW_r|}{W_r} \bigg)Q_r 
\end{align*}
which completes the proof.
\end{proof}

Notice that Lemma \ref{l:intermeadiate.bound} provides bounds on how the passage times from $(0,0)$ to the left side of the central column (the line $x=(i-1)r$) varies as the end point is varied near height $jW_r$. By symmetry, we have the following analogous bound for passage times to the right side of the central column to $(Mn,0)$.

\begin{lemma}\label{l:intermeadiate.bound.sym}
On the event $\cW_{i,j}\cap \cD_{i+2,j}\cap \cC_{i+1,j}$, for $M^{99/100}n\leq ir \leq (M-M^{99/100})n$ and $|jW_r|\leq M^{4/5}W_n$ and for all $|y|\leq n^{\beta}W_n$ 
\begin{align*}
\rX_{(ir,y),(0,Mn)} - \rX_{(ir,jW_r),(0,Mn)}\geq -\bigg( 8L_2 +8\frac{|y-jW_r|}{W_r}\bigg)Q_r ,
\end{align*}
and for $y\in [(j-\frac13 L_2)W_r,(j+\frac13 L_2)W_r]$
\begin{align*}
&|\rX_{(0,Mn),(ir,y)} - \rX_{(0,Mn),(ir,jW_r)}   - \frac12\Big(\frac{|y-jW_r|}{W_r}\Big)^2 Q_r|\\
&\leq \bigg(\frac{L_0^2}{800} + \frac{21}{10}L_0  + (1+\frac{L_0}{20})\frac{|y-jW_r|}{W_r} \bigg)Q_r 
\end{align*}
The same bounds hold for $\rX'$.
\end{lemma}

\subsection{Central Column}

Let us consider conforming paths $\zeta$ from $\origin$ to $(0,Mn)$ that pass through $((i-1)r,y)$ and $(ir,y')$.  We will divide them into 6 different types.  Let $\zeta^*$ be the segment of $\zeta$ from $((i-1)r,y)$ to $(ir,y')$; see Figure \ref{f:pathtypes}. 
\begin{itemize}
    \item Type 1: Paths with $\zeta^*\subset [(i-1)r,ir]\times [(j-\frac12 L_0)W_r,(j+ \frac12 L_0)W_r]$.
    \item Type 2: Paths not of Type 1 with $\zeta^*\subset [(i-1)r,ir]\times [(j-\frac32 L_0)W_r,(j+ \frac32 L_0)W_r]$.
    \item Type 3: Paths with $\zeta^*\subset [(i-1)r,ir]\times [(j-\alpha - w)W_r,(j-\alpha + h + w)W_r]$.
    \item Type 4: Paths  where $\zeta^*$ intersects $[(i-1)r,ir]\times [(j- L_1+w)W_r,(j+  L_1-w)W_r]$ and not of type 1-3.
    \item Type 5: Paths not of type 1-4 where $\zeta^*$ intersects $[(i-1)r,ir]\times [(j- L_1)W_r-2,(j+  L_1)W_r+2]$.
    \item Type 6: Paths not of type 1-5.
\end{itemize}
\begin{center}
\begin{figure}[htbp!]
\includegraphics[width=1.5in]{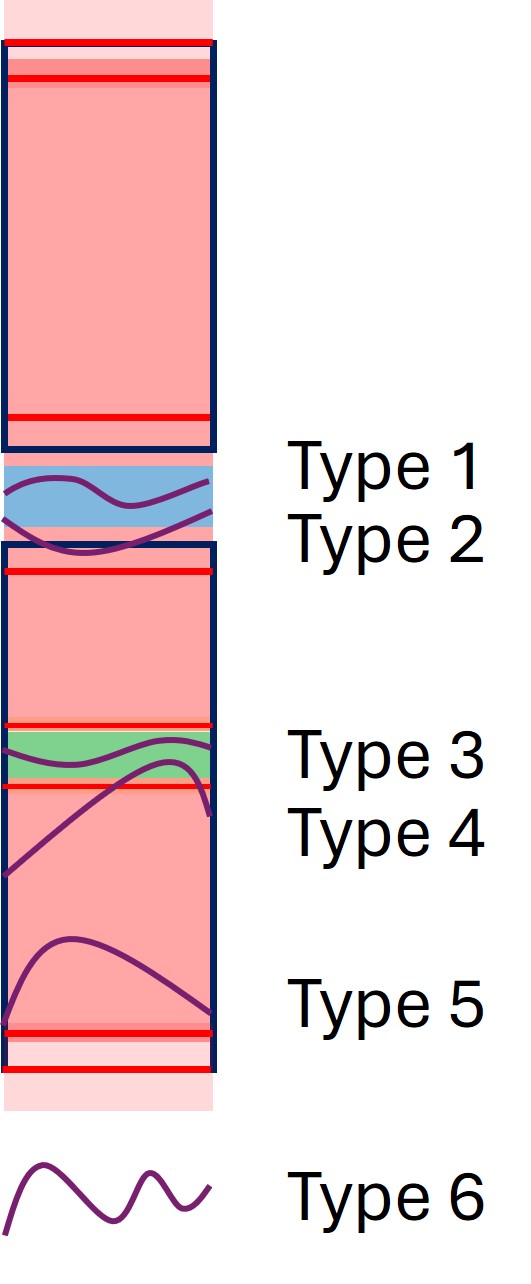}
\caption{Different types of paths through the central column. We classify the paths passing through the central column $[(i-1)r,ir]\times\R$ depending on which region of the central column it passes through. Eventually our goal is to show that on the event $\cP_{i,j}$ the geodesic $\gamma$ from $(0,0)$ to $(Mn,0)$ is either of type $1$ or type $6$ and the geodesic after the resampling is either of type $3$ or of type $6$. This will show that provided that the geodesic before the resampling is not of type $6$, its restriction to the columns $[(i-1)r,ir]$ is separated from the restriction of the resampled geodesic on $\cP_{i,j}$, which, in turn, leads to our desired chaos estimate.}
\label{f:pathtypes}
\end{figure}
\end{center}

Our objective is to show that on $\cP_{i,j}$, if the geodesic $\gamma$ is not of type $6$, (i.e., it passes the column $i$ near height $jW_r$) then it is of type $1$, and further the geodesics after a $\kappa$ fraction of the noise is resampled is either of type $3$ or type $6$. This will ensure that on the event $P_{i,J_i}$, the geodesics before and after resampling are separated in column $i$. We start with the existence of a good (but not too good) path of type $1$ on the event $\cP_{i,j}$.

\begin{lemma}
\label{l:type1}
On the event $\cP_{i,j}$ there exists a Type 1 path $\zeta$ such that
\[
\rX_\zeta \leq \rX_{\origin,((i-1)r,jW_r)} + r + \rX_{(ir,jW_r),(0,Mn)} + \frac{L_0^2}{100} Q_r.
\]
For every
Type 1 path $\zeta$,
\[
\rX_\zeta \geq \rX_{\origin,((i-1)r,jW_r)} + r + \rX_{(ir,jW_r),(0,Mn)} - \frac{L_0^2}{100} Q_r.
\]
The same bounds holds for $\rX'$.

\end{lemma}
\begin{center}
\begin{figure}[htbp!]
\includegraphics[width=6in]{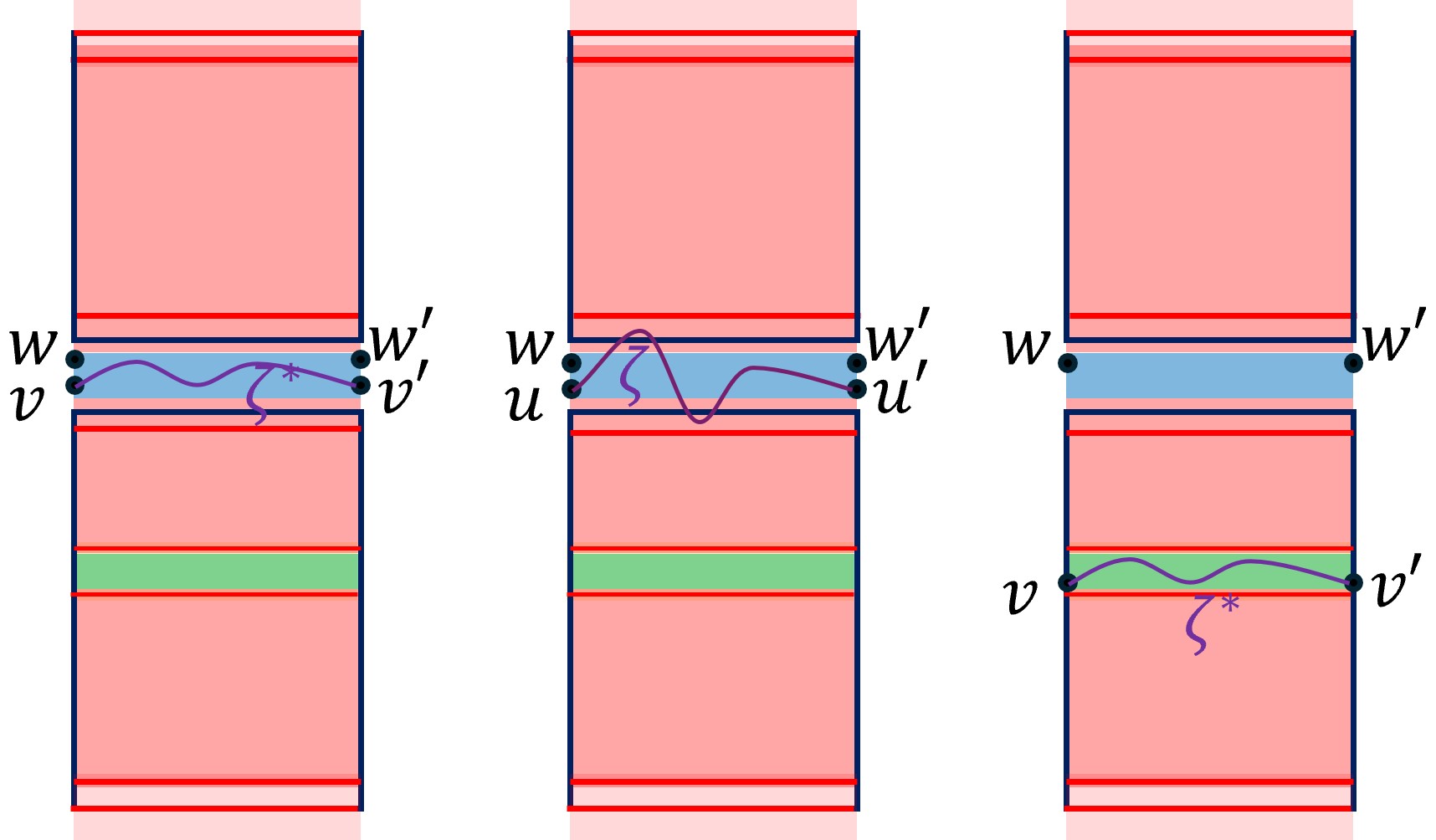}
\caption{Analysis of paths of types 1, 2 and 3. On the event $\cP_{i,j}$, there is a Type 1 path which is good but not too good. This is obtained by concatenating the path $\zeta^*$ shown in the figure with the geodesics from $0$ and $(Mn,0)$ to its endpoints respectively. This is done in Lemma \ref{l:type1} and is illustrated in the leftmost panel. On the event $\cP_{i,j}$ any type 2 path (the segment of a type 2 path $\zeta$ is shown in the middle panel above) has to be bad both before and after resampling. This is done in Lemma \ref{l:type2}. The event $\cP_{i,j}$ is designed in a way such that any type 3 path (paths passing through the region marked in green) is bad before resampling whereas there exists a type 3 path after resampling (marked in the rightmost panel of the figure figure) which is rather good. This is done in Lemma \ref{l:type3}. This, in conjunction with other consequences of the event $\cP_{i,j}$ will show that if the geodesic before the resampling passes near the point $(ir,jW_r)$ (i.e., it is not of type 6), then after resampling, the geodesic before and after resampling will be separated at location $i$ (in scale $r$).
}
\label{f:type1}
\end{figure}
\end{center}

\begin{proof}
Let $w=((i-1)r,jW_r),w'=(ir,jW_r)$.  By $\cB^{(1)}_{i,j}$ we can find $\zeta^*\subset [(i-1)r,ir]\times [(j-\frac1{40} L_0)W_r,(j+ \frac1{40} L_0)W_r]$ with $v=\zeta^*(0)=((i-1)r,y),v'=\zeta^*(1)=(ir,y')$ such that $\hrX_{\zeta^*} \leq \frac{L_0}{40}Q_r$; see the left panel of Figure \ref{f:type1}.  Let $\zeta$ be the concatenation of the optimal path from $\origin$ to $\zeta^*(0)$, then $\zeta^*$ and then the optimal path from $\zeta^*(1)$ to $(Mn,0)$.  Then by Lemmas~\ref{l:intermeadiate.bound} and~\ref{l:intermeadiate.bound.sym},
\begin{align*}
\rX_\zeta &\leq \rX_{\origin,v} + \rX_{\zeta^*} + \rX_{v',(0,Mn)}\\
&\leq \rX_{\origin,w} + r + \frac12\Big(\frac{|y-y'|}{W_r}\Big)^2 Q_r  + \rX_{w',(0,Mn)} + \frac{L_0}{40}Q_r + \frac12\Big(\frac{|y-jW_r|}{W_r}\Big)^2 Q_r + \frac12\Big(\frac{|y'-jW_r|}{W_r}\Big)^2 Q_r\\
&\qquad +\bigg(\frac{L_0^2}{800} + \frac{21}{10}L_0  + (1+\frac{L_0}{20})\frac{|y-jW_r|}{W_r} \bigg)Q_r +\bigg(\frac{L_0^2}{800} + \frac{21}{10}L_0  + (1+\frac{L_0}{20})\frac{|y'-jW_r|}{W_r} \bigg)Q_r \\
&\qquad\leq \rX_{\origin,w} + r + \rX_{w',(0,Mn)} + \frac{L_0^2}{100} Q_r
\end{align*}
Now suppose $\zeta$ is some Type 1 path {passing through} points $u=((i-1)r,y)$ and $u'=(ir,y')$.  Then by Lemmas~\ref{l:intermeadiate.bound} and~\ref{l:intermeadiate.bound.sym} and $\cB_{i,j}^{(2)}$,
\begin{align*}
\rX_\zeta &\ge \rX_{\origin,u} + X_{u,u'} + \rX_{u',(0,Mn)}\\
&\geq \rX_{\origin,w}  + \rX_{w',(0,Mn)}\\
&+r + \frac12\Big(\frac{|y-y'|}{W_r}\Big)^2 Q_r - L_0 Q_r +  \frac12\Big(\frac{|y-jW_r|}{W_r}\Big)^2 Q_r + \frac12\Big(\frac{|y'-jW_r|}{W_r}\Big)^2 Q_r\\
&-\bigg(\frac{L_0^2}{400} + \frac{21}{5}L_0  + (1+\frac{L_0}{20})(\frac{|y-jW_r|}{W_r}+\frac{|y'-jW_r|}{W_r}) \bigg)Q_r\\
&\geq \rX_{\origin,w}  + \rX_{w',(0,Mn)} + r -\frac{L_0^2}{100}Q_r
\end{align*}
where the last inequality follows we optimized over $|y-jW_r|$ and $|y'-jW_r|$ using the fact that $\inf_x \frac12x^2-\frac{L_0}{20}x = -\frac{L_0^2}{400}$.
\end{proof}

Our next lemma is about Type 2 paths. 

\begin{lemma}
\label{l:type2}
On the event $\cP_{i,j}$ for any Type 2 path $\zeta$ we have that
\[
\rX_\zeta \geq \rX_{\origin,((i-1)r,jW_r)} + r + \rX_{(ir,jW_r),(0,Mn)} + \frac{L_0^2}{50} Q_r.
\]
The same bound holds for $\rX'$.
\end{lemma}

\begin{proof}
Let $w=((i-1)r,jW_r),w'=(ir,jW_r)$. Denote $u=((i-1)r,y)$ and $u'=(ir,y')$ as points on $\zeta$; see the middle panel of Figure \ref{f:type1}. By Lemmas~\ref{l:intermeadiate.bound} and~\ref{l:intermeadiate.bound.sym} and the second part of $\cB^{(2)}_{ij}$,
\begin{align*}
{\rX_\zeta} &\ge \rX_{\origin,u} + X_{u,u'} + \rX_{u',(0,Mn)}\\
&\geq \rX_{\origin,w}  + \rX_{w',(0,Mn)}\\
&+r + \frac12\Big(\frac{|y-y'|}{W_r}\Big)^2 Q_r - L_0 Q_r +  \frac12\Big(\frac{|y-jW_r|}{W_r}\Big)^2 Q_r + \frac12\Big(\frac{|y'-jW_r|}{W_r}\Big)^2 Q_r\\
&-\bigg(\frac{L_0^2}{400} + \frac{21}{5}L_0  + (1+\frac{L_0}{20})(|\frac{y}{W_r}-j|+|\frac{y}{W_r}-j|) \bigg)Q_r
\end{align*}
Since an optimization of quadratic functions gives
\[
\inf_{x\in\mathbb{R},x'\geq L_0/4} \frac12(x-x')^2+\frac12 x^2 + \frac12 x'^2 - \frac{L_0}{20}(x+x') - \frac{L_0^2}{400}= \frac{L_0^2}{40},
\]
if $|y-jW_r|\geq \frac14 L_0 W_r$ or $|y'-jW_r|\geq \frac14 L_0 W_r$ then for large enough $L_0$,
\begin{align*}
\rX_\zeta &\geq \rX_{\origin,w}  + \rX_{w',(0,Mn)} + r + \frac{L_0^2}{50}Q_r.
\end{align*}
Otherwise if both $|y-jW_r|,|y'-jW_r|\leq \frac14 L_0 W_r$ then if the path is not Type 1 it must exit $[(i-1)r,ir]\times [(j-\frac12 L_0)W_r,(j+ \frac12 L_0)W_r]$ and so we can apply the first part of $\cB^{(2)}_{ij}$ and hence
\begin{align*}
{\rX_\zeta} &\ge \rX_{\origin,u} + X_{u,u'} + \rX_{u',(0,Mn)}\\
&\geq \rX_{\origin,w}  + \rX_{w',(0,Mn)}\\
&+ r + \frac1{40} L_0^2 Q_r +  \frac12\Big(\frac{|y-jW_r|}{W_r}\Big)^2 Q_r + \frac12\Big(\frac{|y'-jW_r|}{W_r}\Big)^2 Q_r\\
&-\bigg(\frac{L_0^2}{400} + \frac{21}{5}L_0  + (1+\frac{L_0}{20})(|\frac{y}{W_r}-j|+|\frac{y}{W_r}-j|) \bigg)Q_r\\
&\geq \rX_{\origin,w}  + \rX_{w',(0,Mn)} + r + \frac{L_0^2}{50}Q_r
\end{align*}
using the fact that $\inf_x \frac12x^2-\frac{L_0}{20}x = -\frac{L_0^2}{400}$ applied to $|\frac{y}{W_r}-j|$ and $|\frac{y'}{W_r}-j|$.
\end{proof}

\begin{lemma}
\label{l:type3}
On the event $\cP_{i,j}$ for any Type 3 path $\zeta$ we have that
\begin{align*}
\rX_\zeta &\geq \rX_{\origin,((i-1)r,jW_r)} + r + \rX_{(ir,jW_r),(0,Mn)} + \frac12 \alpha^{9/10} Q_r,
\end{align*}
while there exists a Type 3 path $\gamma$ such that
\begin{align*}
\rX_\zeta' &\leq \rX_{\origin,((i-1)r,jW_r)}' + r + \rX_{(ir,jW_r),(0,Mn)}' - \frac12 \alpha^{9/10} Q_r.
\end{align*}
\end{lemma}

\begin{proof}
Let $w=((i-1)r,jW_r),w'=(ir,jW_r)$.  By $\cB^{(4)}_{i,j}$ we can find $\zeta^*\subset [(i-1)r,ir]\times [(j-\sqrt{\alpha})W_r,(j-\sqrt{\alpha} + h)W_r]$ with $v=\zeta^*(0)=((i-1)r,y),v'=\zeta^*(1)=(ir,y')$ such that $\hrX_{\zeta^*}' \leq -(\alpha+\alpha^{9/10})Q_r$; see the right panel of Figure \ref{f:type1}.  Let $\zeta$ be the concatenation of the optimal path from $\origin$ to $v$, then $\zeta^*$ and then the optimal path from $v'$ to $(0,Mn)$.  We have that
\[
\sqrt{\alpha}-2\leq |y-jW_r|,|y'-jW_r|\leq \sqrt{\alpha}+1
\]
Then by Lemmas~\ref{l:intermeadiate.bound} and~\ref{l:intermeadiate.bound.sym},
\begin{align*}
\rX_\zeta' &\leq \rX_{\origin,v}' + \rX_{\zeta^*}' + \rX_{v',(0,Mn)}'\\
&\leq \rX_{\origin,w} + \rX_{w',(0,Mn)} +r + \frac12\Big(\frac{|y-y'|}{W_r}\Big)^2 Q_r -(\alpha+\alpha^{9/10})Q_r\\
&\qquad+ \frac12\Big(\frac{|y-jW_r|}{W_r}\Big)^2 Q_r + \frac12\Big(\frac{|y'-jW_r|}{W_r}\Big)^2 Q_r\\
&\qquad +\bigg(\frac{L_0^2}{800} + \frac{21}{10}L_0  + (1+\frac{L_0}{20})|\frac{y}{W_r}-j| \bigg)Q_r +\bigg(\frac{L_0^2}{800} + \frac{21}{10}L_0  + (1+\frac{L_0}{20})|\frac{y'}{W_r}-j| \bigg)Q_r \\
&\qquad\leq \rX_{\origin,w}' + \rX_{w',(0,Mn)}' + r +\frac92Q_r- (\alpha+\alpha^{9/10} Q_r) + (\sqrt{\alpha}+2)^2 Q_r\\
&\qquad +2\bigg(\frac{L_0^2}{800} + \frac{21}{10}L_0  + (\sqrt{\alpha}+2)(1+\frac{L_0}{20}) \bigg)Q_r\\
&\qquad\leq \rX_{\origin,w}' + \rX_{w',(0,Mn)}'   + r - \frac12 \alpha^{9/10} Q_r
\end{align*}
where we used that $\alpha \geq L_0^3$ which establishes the second part of the lemma.  Now suppose $\zeta$ is some Type 3 path passing through points $u=((i-1)r,y)$ and $u'=(ir,y')$.  By $\cB^{(4)}_{i,j}$ and  $\cB^{(5)}_{i,j}$ we have that
\[
\rX_{u,u'} \geq r -  (\alpha-\alpha^{9/10} + 1)Q_r
\]
and hence
\begin{align*}
{\rX_\zeta} &\ge \rX_{\origin,u} + \rX_{u,u'} + \rX_{u',(0,Mn)}\\
&\geq \rX_{\origin,w}  + \rX_{w',(0,Mn)}\\
& +r -  (\alpha-\alpha^{9/10} + 1)Q_r + \frac12\Big(\frac{|y-jW_r|}{W_r}\Big)^2 Q_r + \frac12\Big(\frac{|y'-jW_r|}{W_r}\Big)^2 Q_r\\
&-\bigg(\frac{L_0^2}{400} + \frac{21}{5}L_0  + (1+\frac{L_0}{20})(|\frac{y}{W_r}-j|+|\frac{y}{W_r}-j|) \bigg)Q_r\\
&\geq \rX_{\origin,w}  + \rX_{w',(0,Mn)}\\
& +r -  (\alpha-\alpha^{9/10} + 1)Q_r + \Big(\sqrt{\alpha}-2\Big)^2 Q_r \\
&-\bigg(\frac{L_0^2}{400} + \frac{21}{5}L_0  + 2(1+\frac{L_0}{20})(\sqrt{\alpha}+1) \bigg)Q_r\\
&\geq \rX_{\origin,w} + \rX_{w',(0,Mn)}   + r + \frac12 \alpha^{9/10} Q_r.
\end{align*}
\end{proof}

\begin{lemma}
\label{l:type4}
On the event $\cP_{i,j}$ for any Type 4 path $\zeta$ we have that
\begin{align*}
\rX_\zeta &\geq \rX_{\origin,((i-1)r,jW_r)} + r + \rX_{(ir,jW_r),(0,Mn)} + \frac12 L_2^3 Q_r.
\end{align*}
The same bound holds for $\rX'$.
\end{lemma}
\begin{center}
\begin{figure}[htbp!]
\includegraphics[width=3.5in]{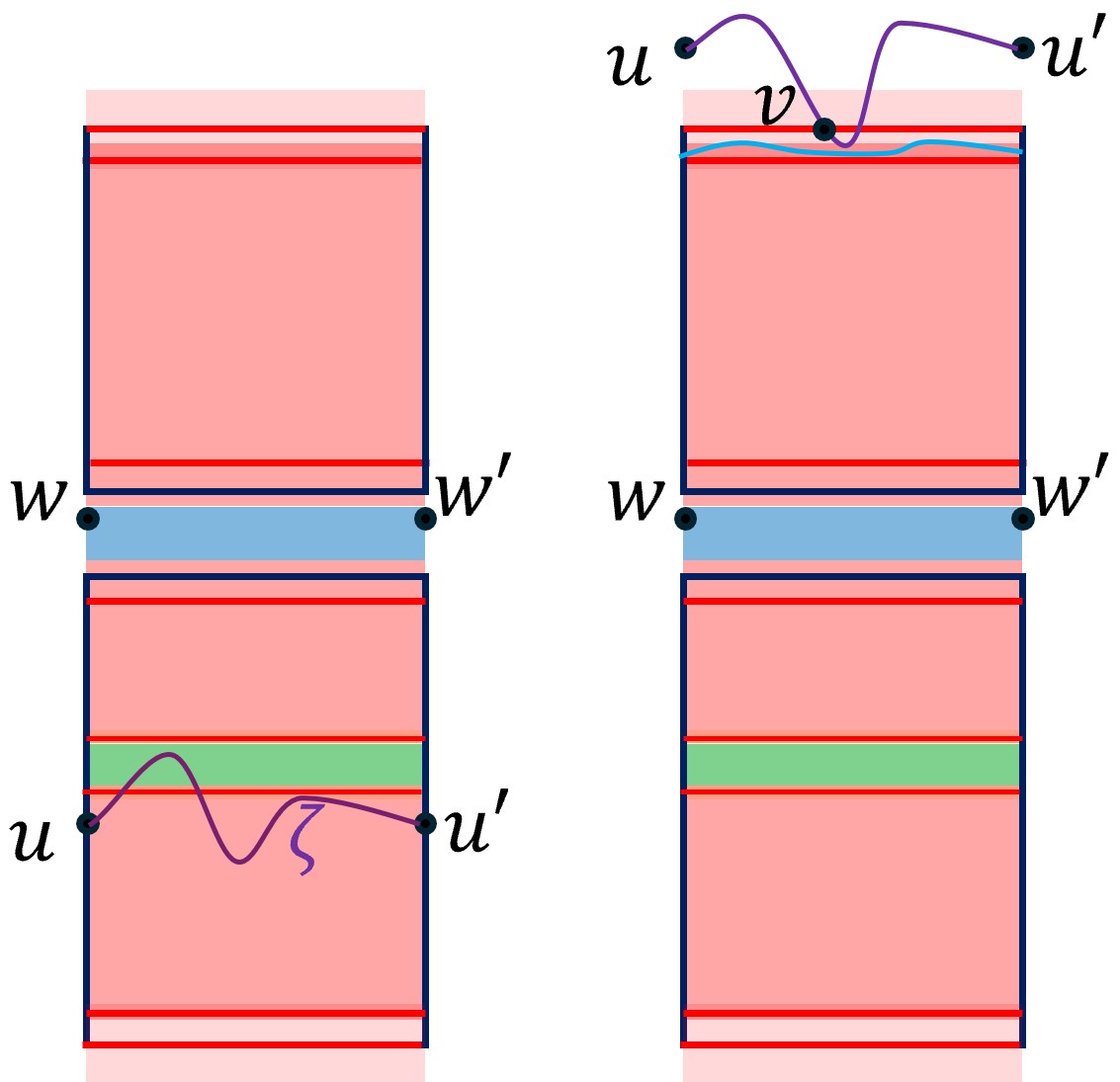}
\caption{Analysis of type 4 and type 5 paths. On the event $\cP_{i,j}$, any type 4 or type 5 path is bad both before and after resampling. A sample path of type 4 is shown in the left panel, and this case is dealt with in Lemma \ref{l:type4}. A sample path of type 5 is shown in the right panel, and this case is dealt with in Lemma \ref{l:type5}.}
\label{f:type4}
\end{figure}
\end{center}

\begin{proof}
Let $w=((i-1)r,jW_r),w'=(ir,jW_r)$. Let $u=((i-1)r,y)$ and $u'=(ir,y')$ be points along $\zeta$; see the left panel of Figure \ref{f:type4}.
Since $\zeta$ is Type 4, it must enter $[(i-1)r,ir]\times [(j- L_1+w)W_r,(j+  L_1-w)W_r]$ but not be contained in either $[(i-1)r,ir]\times [(j-\frac32 L_0)W_r,(j+ \frac32 L_0)W_r]$ or $[(i-1)r,ir]\times [(j-\alpha - w)W_r,(j-\alpha + h + w)W_r]$.  As such, $\zeta$ must enter either $[(i-1)r,ir]\times [(j+\frac32 L_0)W_r,(j+  L_1-w)W_r]$ or $[(i-1)r,ir]\times [(j-\alpha + h + w)W_r,(j-\frac32 L_0)W_r]$ or $[(i-1)r,ir]\times [(j- L_1+w)W_r, (j-\alpha - w)W_r]$.  By $\cB_{i,j}^{(6)}\cap\cB_{i,j}^{(7)}$ all three of these rectangles are barriers in the sense of Lemma~\ref{l:barrier.entry} and hence
\[
\rX_{u,u'} \geq r + (L_2^3-2 L_1)Q_r + \frac12 \Big(\Big(\frac{|y-jW_r|  }{W_r} - 3L_1 \Big)^+\Big)^2 Q_r + \frac12 \Big(\Big(\frac{|y'-j W_r| }{W_r} - 2L_1 \Big)^+\Big)^2 Q_r.
\]
By Lemmas~\ref{l:intermeadiate.bound} and~\ref{l:intermeadiate.bound.sym}
\begin{align*}
\rX_\zeta 
&\geq \rX_{\origin,w}  + \rX_{w',(0,Mn)} + r + (L_2^3-2 L_1)Q_r \\
&\qquad + \frac12 \Big(\Big(\frac{|y-jW_r|  }{W_r} - 3L_1 \Big)^+\Big)^2 Q_r + \frac12 \Big(\Big(\frac{|y'-j W_r| }{W_r} - 2L_1 \Big)^+\Big)^2 Q_r\\
 &\qquad-\bigg( 8L_2 +8|\frac{y}{W_r}-j|\bigg)Q_r-\bigg( 8L_2 +8|\frac{y'}{W_r}-j|\bigg)Q_r \\
&\geq \rX_{\origin,w} + r + \rX_{w',(0,Mn)} + \frac12 L_2^3 Q_r
\end{align*}
where the final inequality is from optimizing over the values of $|\frac{y}{W_r}-j|$ and $|\frac{y'}{W_r}-j|$.
\end{proof}

\begin{lemma}
\label{l:type5}
On the event $\cP_{i,j}$ for any Type 5 path $\zeta$ we have that
\begin{align*}
\rX_\zeta &\geq \rX_{\origin,((i-1)r,jW_r)} + r + \rX_{(ir,jW_r),(0,Mn)} + \frac1{6} L_1^2 Q_r.
\end{align*}
The same bound holds for $\rX'$.
\end{lemma}

\begin{proof}
Suppose first that $\zeta$ is strongly conforming.  Let $w=((i-1)r,jW_r),w'=(ir,jW_r)$.  We begin be showing that  for all $|y|\leq n^\beta W_n$,
\begin{equation}\label{eq:side.quadratic}
\rX_{\origin,((i-1)r,y)} - \rX_{\origin,w} + \frac12\Big(\frac{|y-(j+L_1)W_r| }{W_r} - 1\Big)^2 Q_r \geq \frac1{10} L_1^2 Q_r.
\end{equation}
If {$|y-jW_r|\leq \frac13 L_2W_r$} then by Lemma~\ref{l:intermeadiate.bound},
\begin{align*}
&\rX_{\origin,((i-1)r,y)} - \rX_{\origin,w} + \frac12\Big(\frac{|y-(j+L_1)W_r| }{W_r} - 1\Big)^2 Q_r\\
&\quad \geq  \frac12\Big(\frac{|y-(j+L_1)W_r| }{W_r} - 1\Big)^2 Q_r + \frac12\Big(\frac{|y-jW_r|}{W_r}\Big)^2 Q_r - \bigg(\frac{L_0^2}{800} + \frac{21}{10}L_0  + (1+\frac{L_0}{20})\frac{|y-jW_r|}{W_r} \bigg)Q_r \\
&\quad\geq \frac1{10} L_1^2 Q_r.
\end{align*}
where the last inequality follows by optimizing the quadratic over $y$ and that $L_1 \gg L_0^2$.  Otherwise if {$|y-jW_r| > \frac13 L_2W_r$} then by Lemma~\ref{l:intermeadiate.bound},
\begin{align*}
&\rX_{\origin,((i-1)r,y)} - \rX_{\origin,w} + \frac12\Big(\frac{|y-(j+L_1)W_r| }{W_r} - 1\Big)^2 Q_r\\
&\quad \geq  \frac12\Big(\frac{|y-(j+L_1)W_r| }{W_r} - 1\Big)^2 Q_r -\bigg( 8L_2 +8\frac{|y-jW_r|}{W_r}\bigg)Q_r \\
&\quad\geq \frac1{20} L_2^2 Q_r \geq \frac1{10} L_1^2 Q_r.
\end{align*}
which establishes~\eqref{eq:side.quadratic}.  Similarly
\begin{equation}\label{eq:side.quadratic.sym}
\rX_{(ir,y'),(Mn,0)} - \rX_{w',(Mn,0)} + \frac12\Big(\frac{|y'-(j+L_1)W_r| }{W_r} - 1\Big)^2 Q_r \geq \frac1{10} L_1^2 Q_r.
\end{equation}
Let $u=((i-1)r,y)$ and $u'=(ir,y')$ be points along $\zeta$.  We will show that 
\begin{equation}\label{eq:type.6.bound}
\rX_\zeta > \rX_{\origin,w}  + \rX_{w',(0,Mn)} + r + \frac1{6} L_1^2 Q_r
\end{equation}
First, suppose that $\zeta$ hits $H_{i,j+L_1}$ at a point $v=(x,(j+L_1)W_r$; see the right panel of Figure \ref{f:type4}.  Then by $\cB^{(6)}$,
\begin{align*}
\rX_\zeta 
&= \rX_{\origin,u} + \rX_{u,v} + \rX_{v,u'} + \rX_{u',(0,Mn)}\\
&\geq  \rX_{\origin,w}  + \rX_{w',(0,Mn)} \\
&\qquad + \rX_{\origin,u} - \rX_{\origin,w} + \frac12\Big(\frac{|y-(j+L_1)W_r| }{W_r} - 1\Big)^2 Q_r + (x - (i-1)r) - L_1 Q_r\\
&\qquad + \rX_{u',(Mn,0)} - \rX_{w',(Mn,0)} + \frac12\Big(\frac{|y'-(j+L_1)W_r| }{W_r} - 1\Big)^2 Q_r + (ir - x) - L_1 Q_r\\
&> \rX_{\origin,w}  + \rX_{w',(0,Mn)} + r + \frac1{6} L_1^2 Q_r
\end{align*}
and so~\eqref{eq:type.6.bound} holds.  If $\zeta$ hits $H_{i,j-L_1}$ then similarly~\eqref{eq:type.6.bound} holds.  

Therefore to establish~\eqref{eq:type.6.bound} we only need to consider Type 5 paths that hit neither $H_{i,j+L_1}$ or $H_{i,j-L_1}$.  Such a path must be confined in $[(i-1)r,ir]\times ((j- L_1)W_r,(j+  L_1)W_r)$ within column $i$.  But since it is not Type 4 it must avoid $[(i-1)r,ir]\times [(j- L_1+w)W_r,(j+  L_1-w)W_r]$. So between $u$ and $u'$ it is confined within $[(i-1)r,ir]\times [(j+  L_1-w)W_r,(j+  L_1)W_r]$ or $[(i-1)r,ir]\times [(j-  L_1)W_r,(j-  L_1+w)W_r]$.  In the former case, by $\cB^{(3)}_{i,j}$ and Lemmas~\ref{l:intermeadiate.bound} and~\ref{l:intermeadiate.bound.sym}
\begin{align*}
\rX_\zeta 
&\geq \rX_{\origin,((i-1)r,jW_r)}  + \rX_{(ir,jW_r),(0,Mn)} + r -L_0 Q_r \\
&+\frac12\Big(\frac{|y-jW_r|}{W_r}\Big)^2 Q_r - \bigg(\frac{L_0^2}{800} + \frac{21}{10}L_0  + (1+\frac{L_0}{20})\frac{|y-jW_r|}{W_r} \bigg)Q_r \\
&+\frac12\Big(\frac{|y'-jW_r|}{W_r}\Big)^2 Q_r - \bigg(\frac{L_0^2}{800} + \frac{21}{10}L_0  + (1+\frac{L_0}{20})\frac{|y'-jW_r|}{W_r} \bigg)Q_r \\
&\geq \rX_{\origin,((i-1)r,jW_r)}  + \rX_{(ir,jW_r),(0,Mn)} + r + \bigg(-L_0 + (L_1-w)^2+ \frac{L_0^2}{400} + \frac{21}{5}L_0  + 2(1+\frac{L_0}{20})L_1 \bigg)Q_r\\
&> \rX_{\origin,((i-1)r,jW_r)}  + \rX_{(ir,jW_r),(0,Mn)} + r + \frac1{6} L_1^2 Q_r,
\end{align*}
which completes the proof of~\eqref{eq:type.6.bound}.  If $\zeta$ is not strongly conforming then by Lemma~\ref{l:strongly.conforming} we can find an approximating path $\hat{\zeta}$ which is also Type 5 such that $\rX_\zeta \geq \rX_{\hat{\zeta}} -\epsilon$ and hence the lemma holds for all $\zeta$.
\end{proof}

\begin{lemma}\label{c:path.sepatated}
On the event $\cP_{i,j}$ the optimal path $\gamma$ is Type 1 or 6, while the optimal resampled path $\gamma'$ is Type 3 or Type 6. On the event $\{\cP_{i,J_i}\}\cap\{|J_iW_r|\leq M^{4/5}W_n\}$, the optimal path is of Type 1 and for all $i'\in [(i-1)\Phi^{\ell}+1,i\Phi^{\ell}]$ and for all $j$, $I(\cU_{i'j}\cap \cU'_{i'j})=0$. 
\end{lemma}
\begin{proof}
On the event $\cP_{i,j}$, there is a Type 1 path with passage time at most $\rX_{\origin,((i-1)r,jW_r)} + r + \rX_{(ir,jW_r),(0,Mn)} + \frac{L_0^2}{100} Q_r$ while any Type 2, 3, 4 or 5 paths have passage time at least $\rX_{\origin,((i-1)r,jW_r)} + r + \rX_{(ir,jW_r),(0,Mn)} + \frac{L_0^2}{50} Q_r$ so $\gamma$ must be Type 1 or 6.  Similarly, in the resampled environment, there is a Type 3 path with passage time at most
\[
 \rX_{\origin,((i-1)r,jW_r)}' + r + \rX_{(ir,jW_r),(0,Mn)}' - \frac12 \alpha^{9/10} Q_r \leq  \rX_{\origin,((i-1)r,jW_r)}' + r + \rX_{(ir,jW_r),(0,Mn)}' - L_0 Q_r
\]
while Type 1,2,4 or 5 paths have passage times at least $\rX_{\origin,((i-1)r,jW_r)}' + r + \rX_{(ir,jW_r),(0,Mn)}' - \frac1{100}L_0 Q_r$.

On the event $\cP_{i,J_i}$, the optimal path intersects $\{ir\}\times [(J_i-1)W_r,J_iW_r]$ and so cannot be Type 6 and hence must be Type 1.  So in the $i$th column it must stay within $[(i-1)r,ir]\times [(J_i-\frac12 L_0)W_r,(J_i+ \frac12 L_0)W_r]$.  Since the optimal resampled path $\gamma'$ is either Type 3 or 6, it does not intersect $[(i-1)r,ir]\times [(J_i- L_0)W_r,(J_i+  L_0)W_r]$.  Hence the horizontal separation between the paths is at least $\frac12 L_0 W_r$ and so for all $i'\in [(i-1)\Phi^{\ell}+1,i\Phi^{\ell}]$ and for all $j$, $I(\cU_{i'j}\cap \cU'_{i'j})=0$. 
\end{proof}

\section{Likely events occur typically along the geodesic}
\label{s:likely}
Over the next two sections we shall prove Theorem \ref{t:Pplus}. As explained in Section \ref{s:prelim}, we divide this into different parts depending on the events concerned. In this section we shall deal with the likely events. We will add one more event that requires the path to have small transversal fluctuations in the columns $\{i-2,\ldots,i+2\}$  and so define
\[
\cR_i:=\cR_i^{n,M,\ell} = \bigg\{\gamma \cap \Big(\bigcup_{i'=i-2}^{i+2} H^{n,M,\ell}_{i,J_i-\frac1{100} L_0} \cup H^{n,M,\ell}_{i,J_i+\frac1{100} L_0} \Big) =\emptyset \bigg\}.
\]
The event corresponds to $\gamma$ making a left to right crossing of the rectangle $[(i-3)r,(i+2)r]\times[(J_i-\frac1{100} L_0)W_n,(J_i+\frac1{100} L_0)]$. We shall prove the following estimate. 

\begin{lemma}\label{l:Pminus.perc}
There exists $M_0$ such that for all $M\geq M_0$ and all $n$ sufficiently large and $n\leq r_\ell=\Phi^{\ell}n \leq M^{1/100}n$ and $2M^{99/100}\leq  i\Phi^\ell \leq (M-2M^{99/100})$,
\[
\P\Bigg[\sum_{i'=i}^{i+\Phi-1} I(\cP_{i',J^{n,M,\ell}_{i'}}^{-,n,M,\ell}, \cR^{n,M,\ell}_{i'},|J^{n,M,\ell}_{i'}|W_r\leq M^{8/10}W_n) \geq \frac9{10} \Phi \Bigg] \geq 1 - M^{-100}.
\]
\end{lemma}

Let us fix $\ell$ (hence $r=r_{\ell}$) and $i$ as in the statement of the lemma. {To avoid notational overhead, for the rest of this section we shall drop the superscripts $n,M,\ell$}. 
We shall handle the three events in the statement of the lemma separately. Lemma \ref{l:Pminus.perc} will follow from the next three lemmas together with a union bound.

\begin{lemma}
    \label{l:pminusJ}
    For $M$ sufficiently large, 
    $$\P\left(\max_{1\leq i' \leq Mn/r}|J_{i'}|W_r\ge M^{8/10}W_n\right)\le M^{-200}.$$
\end{lemma}

\begin{lemma}
    \label{l:pminusR}
   In the set-up of Lemma \ref{l:Pminus.perc}, and all $2M^{99/100}\leq  i\Phi^\ell \leq (M-2M^{99/100})$,  
    $$\P\left( \sum_{i'=i}^{i+\Phi-1} I(\cR_{i'}^c, |J^{n,M,\ell}_{i'}|W_r\le M^{4/5}W_n) \ge \frac{\Phi}{20} \right)\le M^{-200}.$$
\end{lemma}

\begin{lemma}
    \label{l:pminusP}
   In the set-up of Lemma \ref{l:Pminus.perc}, and all $2M^{99/100}\leq  i\Phi^\ell \leq (M-2M^{99/100})$,  
    $$\P\left( \sum_{i'=i}^{i+\Phi-1} I((\cP_{i',J_{i'}}^{-})^c,|J_{i'}|W_r\le M^{4/5}W_n) \ge \frac{\Phi}{20} \right)\le M^{-200}.$$
\end{lemma}

Observe that Lemma \ref{l:pminusJ} follows immediately from Lemma \ref{l:proxytrans}, equation~\eqref{eq:basic.W.bounds} and a union bound. It remains to prove Lemmas \ref{l:pminusR} and \ref{l:pminusP}. For both of those proofs, we first need the following result to control the $\tau_1$ fluctuation of $\gamma$ between the lines $x=ir_{\ell}$ and $x=(i+\Phi)r_{\ell}$.

\begin{lemma}
    \label{l:tau14.1}
    There exists $H_1$ sufficiently large (not depending on $n,M,\ell$) such that
    $$\P\left( \sum_{i'=i}^{i+\Phi-1} |J_{i'+1}-J_{i'}| \le H_1\Phi, |J_{i+\Phi}-J_{i}|\le \Phi^{9/10} \right)\ge 1-M^{-1000}.$$
\end{lemma}

Proof of this lemma is given at the end of this section. We shall now assume Lemma \ref{l:tau14.1} and prove Lemmas \ref{l:pminusR} and \ref{l:pminusP}.

We shall need a {stretched exponential polymer result}. We set-up some notation first. For $k_{s},k_{e}\in \Z$ let $\fK_{i,\Phi,k_s,k_e}$ denote the set of all tuples $(k_{i},k_{i+1},\ldots, k_{i+\Phi})\in \Z^{\Phi+1}$ with $k_i=k_s$ and $k_{i+\Phi}=k_e$. As before, we shall define, for $\uk\in \fK_{i,\Phi,k_s,k_e}$, $\tau_1(\uk)=\sum_{i'} |k_{i'+1}-k_{i'}|$. 

The number of ways to sum up $A$ non-negative integers to have sum equal to at most $B$ is $\binom{B+1}{A}$.  It follows that the number of tuples in $\fK_{i,\Phi,k_s,k_e}$ with $\tau_1(\uk)\le H\Phi$ is at most $2^\Phi{H\Phi+1\choose\Phi}$ where the $2^\Phi$ comes from the choice of the signs of the {$k_{i'+1}-k_{i'}$}. There exists $c>0$, such that for all $H$ sufficiently large the number of tuples in $\fK_{i,\Phi,k_s,k_e}$ with $\tau_1(\uk)\le H\Phi$ is upper bounded by 
\begin{equation}
    \label{e:entropy}
    2^\Phi{H\Phi+1\choose\Phi}\leq\exp(c\Phi \log H). 
\end{equation}
We start with the following easy polymer lemma. 

\begin{lemma}
    \label{l:abstractperc}
    For $i'=i,i+1, \ldots, i+\Phi-1$, and $k\in \Z$, let the collection of events $B_{i',k}$ be such that the collections $\{B_{i',k}\}_{k}$ and $\{B_{i'',k}\}$ are independent if $|i'-i''|\ge K$. Given $\varepsilon>0$ small and $H>0$, there exists $\delta=\delta(\varepsilon,K,H)>0$ such that if $\P(B_{i',k})\le \delta$ for all $i',k$ then 
    $$\P\left(\max_{\uk\in \fK_{i,\Phi,k_s,k_e}: \tau_1(\uk)\le H\Phi}\sum_{i'=i}^{i+\Phi-1} I(B_{i',k_{i'}}) \ge \varepsilon \Phi\right) \le \exp(-c\Phi)$$
    for some $c>0$.
\end{lemma}

After splitting the sum into $i'$ mod $K$, the lemma follows from \cite[Lemma 12.7]{BSS23}, so we omit the proof.

\begin{center}
\begin{figure}
\includegraphics[width=5in]{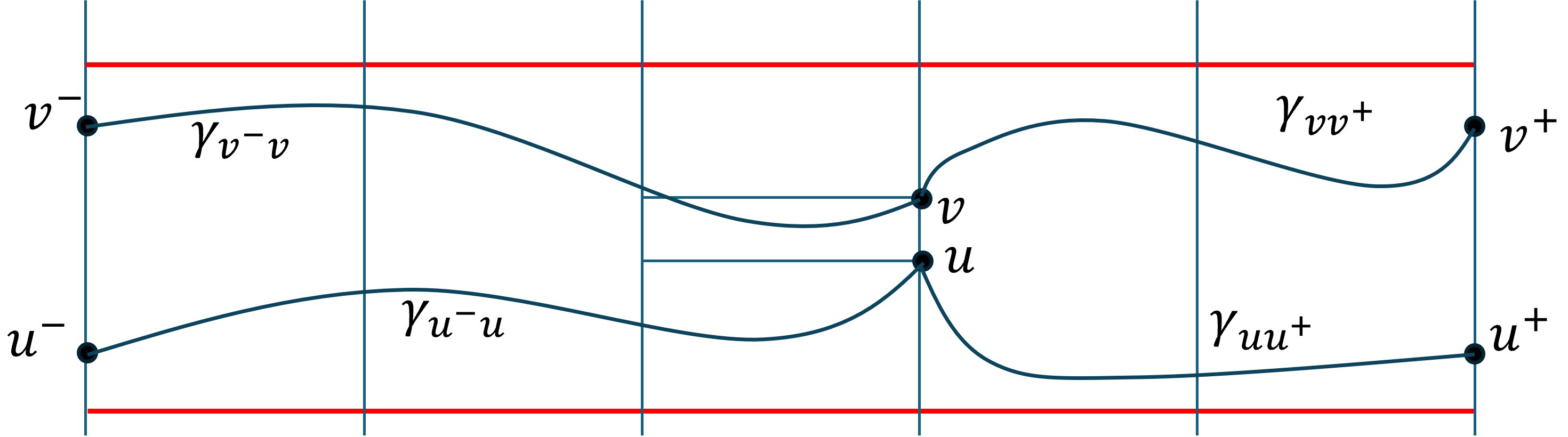}
\caption{Local event $\cR_{i,j}$ used in Lemma \ref{l:pminusRperc} to control the fluctuation of the geodesic between columns $i-2, \ldots, i+2$. The event $\cR_{i,j}$ asks that the geodesics from $v=(ir,(j+1)W_r)$ to $v^-$ and $v^+$ and the geodesics from $u=(ir,jW_r)$ to $u^-$ and $u^+$ stay within the two red lines are $\cup_{i'=i-2}^{i+2}H_{i,j\pm\frac{1}{1000}L_0}$ which are a large constant vertical distance away from the points (at scale $W_r$). Note the $\cR_{i,j}$ is a local event whose probability can be made close to $1$ by choosing $L_0$ sufficiently large and hence by a percolation argument $\cR_{i,J_i}$ holds holds for most $i$. Also, by ordering of the geodesics, if the geodesic $\gamma$ does not have large fluctuations around location $i$ $\gamma$ will also remain between the two red lines while passing though these columns, thereby showing that the non-local event $\cR_i$ also holds at most locations with good probability.}
\label{f:rij}
\end{figure}
\end{center}
Let us now move towards the proof of Lemma \ref{l:pminusR}. We shall  define a local version of the event $\cR_i$. Let $H_1$ be as in Lemma \ref{l:tau14.1}. Consider the points $u=(ir, jW_{r})$ and $v=(ir, (j+1)W_{r})$. Let us also consider points $u^{+}=((i+2)r, (j-10^4H_1)W_{r})$, $v^{+}=((i+2)r, (j
+10^4H_1)W_{r})$, $u^{-}=((i-3)r, (j-10^4H_1)W_{r})$, $v^{-}=((i-3)r, (j+10^4H_1)W_{r})$; see Figure \ref{f:rij}. Let $\cR_{i,j}$ denote the event that the geodesics $\gamma_{vv^{+}}$ and $\gamma_{vv^{-}}$ are below $\cup_{i'=i-2}^{i+2}H_{i,j+\frac{1}{1000}L_0}$ and the geodesics $\gamma_{uu^{+}}$ and $\gamma_{uu^{-}}$ are above $\cup_{i'=i-2}^{i+2}H_{i,j-\frac{1}{1000}L_0}$. Let now $A_i$ denote the event that 
\[
\sum_{i'=i-2}^{i+2} |J_{i'+1}-J_{i'}|\le 10^{4}H_1.
\]
Observe that on the event $A_{i}\cap \{J_{i}=j\}\cap \cR_{i,j}$, the geodesic $\gamma$ must pass below $\gamma_{vv^{+}}$ and $\gamma_{vv^{-}}$ and above $\gamma_{uu^{+}}$ and $\gamma_{uu^{-}}$ and so $\cR_i$ holds. Notice also that on the event $\{\sum_{i'=i}^{i+\Phi-1}|J_{i'+1}-J_{i'}|\le H_1\Phi\}$, one has (by Markov inequality) that 
$$\sum I(A_i^{c}) \le \frac{\Phi}{40}.$$ Therefore, together with Lemma \ref{l:tau14.1} the following lemma completes the proof of Lemma \ref{l:pminusR}.

\begin{lemma}
    \label{l:pminusRperc}
    In the above set-up, there exists $L_0$ sufficiently large (depending only on $H_1$) such that
    $$\P\left(\max_{{|k_{s}|\le M^{4/5}\frac{W_n}{W_r}}, |k_{e}-k_s|\le \Phi^{9/10}}\max_{\uk\in \fK_{i,\Phi,k_s,k_e}:\tau_1(\uk)\le H_1\Phi} \sum_{i'=i}^{i+\Phi-1} I(\cR^c_{i',k_{i'}}) \ge \frac{\Phi}{40}\right)\le M^{-1000}$$
\end{lemma}

\begin{proof}
    Notice that, by definition of $\Phi$, the number of choices $(k_s,k_e)$ satisfying $|k_s|\le M^{4/5}\frac{W_n}{W_r}$ and $|k_s-k_e|\le \Phi^{9/10}$
 is at most $M$. Therefore, it suffices to prove that for each such $(k_s,k_e)$ we have 
$$\P\left(\max_{\uk\in \fK_{i,\Phi,k_s,k_e}:\tau_1(\uk)\le H_1\Phi} \sum_{i'=i}^{i+\Phi-1} I(\cR^c_{i',k_{i'}}) \ge \frac{\Phi}{40}\right)\le M^{-1001}.$$
Notice that by Definition, the events $\cR_{i',k'}$ and $\cR_{i'',k''}$ are independent if $|i-i'|\ge 7$. Therefore, we can apply Lemma \ref{l:abstractperc} with $K=7$, $H=H_1$ and $\varepsilon=1/40$. Observing that by Lemma \ref{l:proxytrans} one gets $\P(\cR^c_{i',k'})\le \delta$ where $\delta$ can be made arbitrarily small by choosing $L_0$ sufficiently large depending on $H_1$. The desired result now follows from Lemma \ref{l:abstractperc}, taking a union bound over $k_s,k_e$ and observing that $\exp(-c\Phi)\le M^{-1001}$ for all  $M$ sufficiently large. 
 \end{proof}

 Next, we shall prove Lemma \ref{l:pminusP}. Arguing as in the proof of Lemma \ref{l:pminusR} and using Lemma \ref{l:tau14.1}. It suffices to prove the following lemma. 

\begin{lemma}
    \label{l:pminusPred}
    Let $k_s,k_e$ be fixed such that {$|k_s|\le M^{4/5}\frac{W_n}{W_r}$} and $|k_s-k_e|\le \Phi^{9/10}$. Then 
    $$\P\left(\max_{\uk\in \fK_{i,\Phi,k_s,k_e}:\tau_1(\uk)\le H_1\Phi}\sum_{i'=i}^{i+\Phi-1} I((\cP^{-}_{i',k_{i'}})^c)\ge \frac{\Phi}{20}\right)\le M^{-1000}$$
    for all $M$ sufficiently large. 
\end{lemma}

Recall the definition of $\cP$. We set 
\begin{align*}
\cP_{i,j}^{*} &= \cB^{(1)}_{i,j}
\cap \Big \{\P\Big[\cB^{(2)}_{i,j}, \cB^{(3)}_{i,j} \mid \bomega(V_{i,j}^c)\Big] \geq \tfrac12 \Big \}
\cap \Big \{\P\Big[\cB^{(6)}_{i,j} \mid \bomega(V_{i,j}^c)\big)\Big] \geq 1-\frac{\delta_A}{100} \Big \}\\
&\qquad\bigcap_{i'\in {i-2,i+2}} \Bigg(\cD^{(1)}_{i',j}
\cap \Big \{\P[\cD^{(2)}_{i',j}, \cD^{(3)}_{i',j},\cD^{(4)}_{i',j}  \mid \bomega(V_{i',j}^{'c})] \geq 1/2 \Big \}\Bigg)\\
&\qquad \bigcap \cC_{i-1,j} \bigcap \cC_{i+1,j}.
\end{align*}
 Recall that 
$$\cP^{-}_{i,j}=\cP_{i,j}^{*}\cap \cW_{i,j}^{loc}.$$
We shall prove Lemma \ref{l:pminusPred} by controlling the two parts separately in the following two lemmas. 

\begin{lemma}
    \label{l:pminusPred1}
    For $L_0$ sufficiently large, the following holds for all $M$ sufficiently large.  For any $k_s,k_e$ such that {$|k_s|\le M^{4/5}\frac{W_n}{W_r}$} and $|k_s-k_e|\le \Phi^{9/10}$ we have that, 
    $$\P\left(\max_{\uk\in \fK_{i,\Phi,k_s,k_e}:\tau_1(\uk)\le H_1\Phi}\sum_{i'=i}^{i+\Phi-1} I((\cP^{*}_{i',k_{i'}})^c)\ge \frac{\Phi}{40}\right)\le M^{-1001}.$$
 
\end{lemma}

\begin{lemma}
    \label{l:pminusPred2}
    For $L_0$ sufficiently large, the following holds for all $M$ sufficiently large.  For any $k_s,k_e$ such that {$|k_s|\le M^{4/5}\frac{W_n}{W_r}$} and $|k_s-k_e|\le \Phi^{9/10}$ we have that, 
    $$\P\left(\max_{\uk\in \fK_{i,\Phi,k_s,k_e}:\tau_1(\uk)\le H_1\Phi}\sum_{i'=i}^{i+\Phi-1} I((\cW_{i',k_{i'}}^{loc})^c)\ge \frac{\Phi}{40}\right)\le M^{-1001}.$$
\end{lemma}

We first prove Lemma \ref{l:pminusPred1} which again uses Lemma \ref{l:abstractperc}. 

\begin{proof}[Proof of Lemma \ref{l:pminusPred1}]
    
    {Notice that each event in the definition of $\cP^{*}_{i',k_{i'}}$ depends only on the randomness in the region $[(i-3)r,(i+3)r]\times \R$.
    Furthermore, note that by the hypothesis on $k_s,k_e$ and $\tau_1(\uk)$ it follows that all $k_{i'}$s considered in the max in the statement of the lemma satisfies $|k_{i'}|\le \frac{MW_n}{W_r}$. Therefore, by applying Lemma \ref{l:abstractperc} as in the proof of Lemma \ref{l:pminusRperc} with $K=7$, $\varepsilon=1/40$ and $H=H_1$ it suffices to prove that 
    $$\P(\cP^*_{i',j})\ge 1-\delta$$
    for all $i\le i'\le i+\Phi$ and all $j$ with $|j|\le \frac{MW_n}{W_r}$
    where $\delta$ can be made arbitrarily small by choosing $L_0$ sufficiently large.} We shall therefore, need to show that the probability of each event in the definition of $P^*_{i',j}$ can be made arbitrarily close to 1 by choosing $L_0$ sufficiently large.

   {For the event $\cB^{(1)}_{i',j}$ this follows from Lemma \ref{l:Bbound}.}
   Lemma \ref{l:Bbound} together with Markov inequality implies 
    $$\P\Big [\P\Big[\cB^{(2)}_{i',j}, \cB^{(3)}_{i',j} \mid \bomega(V_{i',j}^c)\Big] \geq \tfrac12 \Big ] \ge 1-2\P\left((\cB^{(2)}_{i',j})^c \cup (\cB^{(2)}_{i',j})^c\right)\ge 1-2L_0^{-1} $$
    and we get the desired bound. For the event
    $$\Big \{\P\Big[\cB^{(6)}_{i',j} \mid \bomega(V_{i',j}^c)\big)\Big] \geq 1-\frac{\delta_A}{100} \Big \}$$
the desired bound follows from \eqref{eq:B6.bound} together with an application of Markov's inequality as above. For the events $\cD_{i',j}$ the required bound is given in \eqref{eq:D12.bound}. For the events 
$$\Big \{\P[\cD^{(2)}_{i',j}, \cD^{(3)}_{i',j},\cD^{(4)}_{i',j}  \mid \bomega(V_{i,j}^{'c})] \geq 1/2 \Big \}$$
the desired bound follows from \eqref{eq:D12.bound} and \eqref{eq:D3.bound} together with another application of Markov's inequality. The bounds for the events $\cC_{i'-1,j}$ and $\cC_{i'+1,j}$ are given in \eqref{eq:cC.bound}. Combining all these we get 
$$\P(\cP^*_{i',j})\ge 1-\delta$$
    where $\delta$ can be made arbitrarily small by choosing $L_0$ sufficiently large, and the proof of the lemma is completed by invoking Lemma \ref{l:abstractperc} as explained above.    
\end{proof}

It remains now to prove Lemma \ref{l:pminusPred2}. Since the events $\cW_{i',k_{i'}}^{loc}$ do not have finite range of dependence in $i'$, this cannot be done by using Lemma \ref{l:abstractperc} directly. Instead we divide the events into different parts corresponding to the different scale $k$ in its definition. For convenience of notation, let us set for $1\le m \le (\log_2\log_2 M)^2$ 

$$\mathfrak{Z}_{i,j,m}=\cZ_{i-2^{m+1}L_2-3,j,m}\cap \cZ_{i-2^{m}L_2-3,j,m} \cap \cZ_{i+2,j,m}\cap \cZ_{i+2^{m}L_2 +2 ,j,m}.$$
Recall that 

$$\cW_{i,j}^{loc}=\bigcap_{m=1}^{(\log_2\log_2 M)^2} \mathfrak{Z}_{i,j,m}.$$

We shall prove the following lemma.

\begin{lemma}
    \label{l:pminuslock}
    Let $k_s,k_e$ be fixed such that {$|k_s|\le M^{4/5}\frac{W_n}{W_r}$} and $|k_s-k_e|\le \Phi^{9/10}$. For each $1\le m \le (\log_2\log_2 M)^2$ we have 
    $$\P\left(\max_{\uk\in \fK_{i,\Phi,k_s,k_e}:\tau_1(\uk)\le H_1\Phi}\sum_{i'=i}^{i+\Phi-1} I(\mathfrak{Z}_{i',k_{i'},m}^{c}) \ge \frac{m^{-100}\Phi}{1000}\right)\le M^{-1002}. $$
\end{lemma}

First we complete the proof of Lemma \ref{l:pminusPred2} using Lemma \ref{l:pminuslock}. 

\begin{proof}[Proof of Lemma \ref{l:pminusPred2}]
    Let $A_{m}$ (locally) denote the event that 
    $$\sum_{i'=i}^{i+\Phi-1} I(\mathfrak{Z}_{i',k_{i'},m}^{c}) \ge \frac{m^{-100}\Phi}{1000}.$$
    Since $\sum_{\ell=1}^{\infty} \ell^{-100}<25$ it follows that on the event $\cap_{m} A^{c}_{m}$ we have 
    $$\sum_{i'=i}^{i+\Phi-1} I((\cW_{i',k_{i'}}^{loc})^c) \le \frac{\Phi}{40}.$$
    The desired result now follows from Lemma \ref{l:pminuslock} and a union bound over $m$.
\end{proof}

Finally, we provide the proof for Lemma \ref{l:pminuslock}. This will use {Lemma \ref{l:zbound}} and will be similar to the proof Lemma \ref{l:pminusPred1} except that instead of using an abstract result like Lemma \ref{l:abstractperc}, we shall need to keep track of the range of dependence more carefully. 

\begin{proof}[Proof of Lemma \ref{l:pminuslock}]
    Let us first fix $1\le m \le (\log_2 \log_2 M)^2$, and also $\uk\in \fK_{i,\Phi,k_s,k_e}$. For convenience of notation, let us set $L^*=2^{m+3}L_2$. Observe now that by definition of $\mathfrak{Z}_{i',j,m}$ it follows that for each $s\in \{0,1,\ldots, L^*-1\}$, the events 
    $\{\mathfrak{Z}_{i+hL^*+s,k_{i+hL^*+s},m}\}_{h}$ are independent where $h$ varies over all nonnegative integers such that $i+hL^*+s\in [i,i+\Phi-1]\cap \Z$. Let $A_{s,m}$ denote (again, locally) the events that for a fixed $s$ as above we have 
    $$\max_{\uk\in \fK_{i,\Phi,k_s,k_e}:\tau_1(\uk)\le H_1\Phi}\sum_{h}  I(\mathfrak{Z}_{i+hL^*+s,k_{i+hL^*+s},m}^{c}) \ge \frac{L_*^{-1}m^{-100}\Phi}{1000}.$$
    To prove the lemma, it suffices (by a simple union bound over $s$) to prove that for each $s$, we have 
    $$\P(A_{s,m})\le L_*^{-1}M^{-1002}.$$
    From now on fix $s\in \{0,1,2,\ldots, L^*-1\}$. Notice that, in the maximum over $\uk$, we only need maximize over all possible choices of $k_{i+hL_*+s}$ as $h$ varies. It is not hard to see by using the same argument as in \eqref{e:entropy} that the number of distinct tuples of such $k_{i+hL_*+s}$'s corresponding to some $\uk\in \fK_{i,\Phi,k_s,k_e}$ with $\tau_1(\uk)\le H_1\Phi$ is upper bounded by 
    $$\exp(c\Phi \log (L_*H_1)/L_*).$$
    We now upper bound $\P(A_{s,m})$ by first fixing a choice of  $k_{i+hL_*+s}$'s as above, then bounding 
      $$\P\left(\sum_{h}  I(\mathfrak{Z}_{i+hL^*+s,k_{i+hL^*+s},\ell}^{c}) \ge \frac{L_*^{-1}m^{-100}\Phi}{1000}\right)$$ by
    using the independence of the indicators,  Lemma \ref{l:zbound} (which states that the probability of each of the indicator in the above sum is upper bounded by $\exp(c'(L_*)^{\theta'})$ for some $c',\theta'>0$, {notice that Lemma \ref{l:zbound} is applicable since all possible choices of $k_{i+hL_*+s}$ satisfies, by the hypothesis on $k_s,k_e$ and $\tau_1(\uk)$, $|k_{i+hL_*+s}|\le M$}) and a Chernoff bound, and finally taking a union bound over all choices of  $k_{i+hL_*+s}$'s. 
    This leads to
    $$\P(A_{s,m})\le \exp(c\Phi \log (L_*H_1)/L_*)\exp\left(-\frac{L_*^{-1}m^{-100}\Phi}{1000}\log \bigl(\frac{m^{-100} \exp(c'(L^*)^{\theta'})}{1000}\bigr)\right).$$
  Clearly, if $L_0$ (and hence $L_2$) is chosen sufficiently large (depending on $c',\theta'$ and {$H_1$}, but not on $m$) we get for all $m$
    $$\log \bigl(\frac{m^{-100} \exp(c'(L^*)^{\theta'})}{1000}\bigr)\ge 
     1000(L_*)^{\theta'/2}m^{1000}$$ which implies that (again using $L_2$ sufficiently large) for some $c>0$
     $$\P(A_{s,m}) \le \exp(-c\Phi/L_*) \le L_*^{-1}M^{-1002}$$
     where the final inequality follows from noticing $L^*\le L_22^{(\log_2\log_2 M)^2+3}$ and $\Phi=2^{(\log_2\log_2 M)^5}$
     and choosing $M$ sufficiently large (depending on $L_2$). This completes the proof of the lemma. 
\end{proof}

\subsection{Proof of Lemma \ref{l:tau14.1}} 

Notice first that by $Q_{\Phi r}=O(\Phi^{0.51}Q_r)$ (this follows from \cite{BSS23}; see in particular Lemma 7.3 there and the comment following it) and using the definition of $W_{\cdot}$ we have 
$$W_{\Phi r}/W_{r}\le \Phi^{4/5}.$$

We next want want to show that it suffices to further restrict the event in the lemma by asking that {$|J_{i}|\le M^{4/5}\frac{W_n}{W_r}$} and $$|J_{i+\Phi}-J_{i}|\ge \Phi^{9/10}.$$
To do this we make use of the transversal fluctuation estimate from Lemma \ref{l:proxytrans}, and some associated estimates that will be proved in Appendix \ref{s:proxytrans}. By Lemma \ref{l:proxytrans} we have that $\P(|J_{i}|\ge M^{4/5}\frac{W_n}{W_r})\le M^{-2002}$.  Also, for $H$ as in Lemma \ref{l:2.8}, let $M$ be sufficiently large so that $(\log M)^{H}\le \Phi^{1/10}$, applying this we have 
$$\P(|J_{i}|\le M^{4/5},|J_{i+\Phi}-J_{i}|\ge \Phi^{9/10}) \le M^{-2001}. $$

It therefore suffices to show that for $H_1$ sufficiently large
\begin{equation}
    \label{eq:tau1phi}
\P\left(\sum_{i'=i}^{i+\Phi-1} |J_{i'+1}-J_{i'}|\ge H_1\Phi, |J_{i}|\le M^{4/5}\frac{W_n}{W_r}, |J_{i+\Phi}-J_{i}|\le \Phi^{9/10} \right) \le M^{-2000}.
\end{equation}

Since the number of pairs $(J_{i+\Phi},J_{i})$ satisfying the constraints above is bounded by $M^2$ it suffices to prove the following general fact. As before, fix $k_s=k_{i}$ and $k_e=k_{i+\Phi}$ with {$|k_s|\le M^{4/5}\frac{W_n}{W_r}$} and $|k_s-k_e|\le \Phi^{9/10}$. For $u\in \ell_{ir,kr,(k+1)W_{r}}$ and $v\in \ell_{(i+\Phi)r, k'W_{r},(k'+1)W_r}$ and for $i\le i' \le i+\Phi$ set $J^{uv}_{i'}=\lfloor \frac{y_{i'}}{W_r} \rfloor$ where $(i'r,y_{i'})$ is the point where $\gamma_{uv}$ intersects 
the line $x=i'r$. 
 
Define
$$\tau_1(\gamma_{uv})=\sum_{i'} |J^{uv}_{i'}-J^{uv}_{i'-1}|,$$
we have the following lemma which is analogous to Proposition \ref{p:tau2perc}. 
\begin{lemma}
    \label{l:tau14.1gen}
    There exists $H_1>0$, $c,\theta'>0$ such that for all $k_s,k_e$ with $|k_s|\le M^{4/5}\frac{W_n}{W_r}, |k_s-k_e|\le  \Phi^{9/10}$ we have 
    $$\P\left(\max_{u\in \ell_{ir,k_sW_r,(k_s+1)W_r}, v\in \ell_{(i+\Phi)r, k_eW_r, (k_e+1)W_r}} \tau_1(\gamma_{uv}) \ge H_1\Phi\right)\le \exp(-c\Phi^{\theta'}).$$    
\end{lemma}

Using Lemma \ref{l:tau14.1gen} together a union bound over all possible  $(k,k')$, \eqref{eq:tau1phi} follows. This completes the proof of Lemma \ref{l:tau14.1} modulo Lemma \ref{l:tau14.1gen}, whose proof is given below. 

\begin{proof}[Proof of Lemma \ref{l:tau14.1gen}]
Fix $k_s,k_e$ as in the statement of the lemma. Observe that by Cauchy-Schwarz inequality 
$$\tau_1(\gamma_{uv})\ge H_1\Phi \Rightarrow \sum_{i'} (J^{uv}_{i'}-J^{uv}_{i'-1})^2\ge H_1^2\Phi.$$
The result now follows by applying Lemma \ref{l:tau2percgen} with $s_1=k_s,s_2=k_e$ and $D=\Phi$. 
\end{proof}

\section{Glauber resampling analysis}
\label{s:glauber}

The main result of this section is to prove Theorem~\ref{t:Pplus} which holds that (suppressing the dependence on $n,m,\ell$) a constant fraction of the events $\cP_{i',J_{i'}}$ hold for any $\Phi$ many consecutive $i'$ with large probability.  Lemma~\ref{l:Pminus.perc} already established this for a large constant fraction of the $\cP^-_{i',J_{i'}}$ which constitutes the ``likely'' part of $\cP_{i',J_{i'}}$ and omits the barrier events from the outer columns as well as events in the central column, most importantly $\cB_{i,j}^{(4)}$ that creates a channel for an alternative good path after resampling.

To prove Theorem~\ref{t:Pplus}, we will show that the for the $i'$ for which $\cP^-_{i',J_{i'}}$ hold, the events $\cP_{i',J_{i'}}$ stochastically dominate independent Bernoulli random variables with success probability bounded below. Our approach will, for each $i'$, resample the relevant regions of the field and then check if the event  $\cP_{i',J_{i'}}$ holds; see Figure \ref{f:resamplingG}. Then, since resampling does not change the distribution, we can use this to give probabilistic bounds on the number of $i'$ for which $\cP_{i',J_{i'}}$ holds. It is important to note that the resampling here is not the same resampling as when we resample a $\kappa$ fraction of blocks. It is a separate argument where we prove some properties of a measure (here the pair of fields before and after updating the $\kappa$ fraction of blocks) by doing a Glauber dynamics style resampling (with respect to which the measure is stationary) of different regions of the field. To highlight the difference we refer to this as Glauber resampling.

A key challenge here is that we resample, conditional on the geodesic remaining fixed.  We will also do the resampling separately for the outer and central columns.  For the outer columns, since conditioning on the geodesic is an increasing event away from the geodesic and the barrier events are increasing, we may make use of the FKG inequality.  The events in the central column are not exclusively increasing and so the will be more delicate making use of Lemma~\ref{c:path.sepatated} and its characterization of which paths can be optimal under $\cP$.

We will begin with a simple lemma showing that a general resampling scheme preserves the distribution.  Let $\varphi$ be a spatially independent random field on some space $U$.  Let $\cS_1,\ldots,\cS_k$ be disjoint events on $\varphi$ and let $U_j$ be subsets of $U$.  We define a {\bf resampling operator} $T=T_{\{\cS_j\},\{U_j\}}$ as as follows.  If $\varphi\in\cS_j$ then set $T\varphi(U_j)$ according to the law $\P[\varphi(U_j) \in \cdot \mid \varphi(U_j^c),\cS_j]$ and set $T\varphi(U_j^c)=\varphi(U_j^c)$.  For $\varphi\in (\bigcup_{j=1}^k \cS_j)^c$ then set $T\varphi=\varphi$.
\begin{lemma}
The pair $(\varphi,T\varphi)$ is exchangeable.
\end{lemma}
\begin{proof}
Set $\varphi' = T \varphi$ and write $\cS_0= (\bigcup_{j=1}^k \cS_j)^c$.  To show exchangeability we must show that for all events $(A,A')$ that
\begin{equation}\label{eq:exchange.defn}
\P[\varphi\in A,\varphi'\in A'] = \P[\varphi'\in A,\varphi\in A'].
\end{equation}
Note that by construction of $T$, the resampling never moves between $\cS_j$ so for $j\neq j'$,
\[
\P[\varphi\in \cS_j,\varphi'\in \cS_{j'}] = 0,
\]
and so
\[
\P[\varphi\in A,\varphi'\in A'] = \sum_{j=0}^k \P[\varphi\in A \cap \cS_j,\varphi'\in A' \cap \cS_j].
\]
Hence it is enough to prove~\eqref{eq:exchange.defn} for all $j\in\{0,\ldots,k\}$ and $A,A'\subset \cS_j$.  When $j=0$ this is trivially true because $\varphi=\varphi'$ on this event.  For $j\geq 1$,
\begin{align*}
\P[\varphi\in A,\varphi'\in A'] &= \E\Big[\P[\varphi\in A \mid \varphi(U_j^c)] \P[\varphi\in A' \mid \varphi(U_j^c)]\Big]\\
&= \P[\varphi\in A',\varphi'\in A],
\end{align*}
which completes the proof.
\end{proof}

\subsection{Outer Barriers}
We show that along $\gamma$ a constant fraction of site have $\cP_{i,j}^-\cap \cD_{i-2,j}\cap \cD_{i+2,j}$.  Define the events 
\[
\cS^{n,M,\ell}_{i,j} =\cP_{i,j}^{-,n,M,\ell} \cap \cR^{n,M,\ell}_{i} \cap \{J^{n,M,\ell}_{i} = j\}
\]
for {$-M^{8/10}W_n\leq jW_r \leq M^{8/10}W_n$} and set {$\cS^{n,M,\ell}_{i,j}$} to be the empty set when $|jW_r| > M^{8/10}W_n$.  Let  $\xi=2^{(\log_2 \log_2 M)^3}$ and $\Delta=\Phi\xi^{-1}$.

{Since in this section we are proving Theorem \ref{t:Pplus} which deals with only a fixed scale $r=r_{\ell}$, from now on we shall assume that $n,M,\ell$ are fixed and will drop the sub and superscripts. We have the following lemma.}

\begin{lemma}\label{l:Pminus.D.perc}
There exists $M_0$ such that for all $M\geq M_0$ and all $n$ sufficiently large and $0\leq \ell\leq \ell_{\max}$ and $2M^{99/100}\leq  i\Phi^\ell \leq (M-2M^{99/100})$,
\begin{align*}
\P\Bigg[&\sum_{t=1}^{\Delta} I(\cD_{i+t\xi-2,J_{i+t\xi}},\cS_{i+t\xi,J_{i+t\xi}}) + \Delta^{2/3} \geq \frac{\delta_C}{2} \sum_{t=1}^{\Delta} I(\cS_{i+t\xi,J_{i+t\xi}}) \Bigg] \geq 1 - M^{-200}.
\end{align*}
\end{lemma}

\begin{proof}
We define the resampling operator $T^{(i')}=T_{\{\cS_{i',j}\},\{U_{i',j}\}}$ on the field $\bomega$ on $\Z\times \R^2$ where $U_{i',j} = V_{i'-2,j}'$.  Define the event
\[
\cL_{i',j} = \bigg\{d(\gamma \cap [(i'-3)r,(i'-2)r] \times \R, \hV_{i'-2,j}') >1 \bigg\}.
\]

Suppose that $\bomega_A$ and $\bomega_B$ are two configurations and let $\gamma_A,\gamma_B$ be their optimal paths.

\noindent {\bf Claim:} If $\bomega_A \in \cS_{i',j}$ then
\[
\bigg\{\bomega_B:\bomega_B(V_{i'-2,j}^{'c}) = \bomega_A(V_{i'-2,j}^{'c}),\bomega_B \in \cS_{i',j} \bigg\}
=\bigg\{\bomega_B:\bomega_B(V_{i'-2,j}^{'c}) = \bomega_A(V_{i'-2,j}^{'c}),\bomega_B \in \cL_{i',j} \bigg\}.
\]
Furthermore, on this event $\gamma_A=\gamma_B$.

\noindent {\it Proof of Claim}: If $\bomega_B(V_{i'-2,j}^{'c}) = \bomega_A(V_{i'-2,j}^{'c})$ but $\bomega_B \not\in \cL_{i',j}$ then $\gamma_B$ does not pass between $H_{i'-2,j-\frac1{100}L_0}$ and $H_{i'-2,j+\frac1{100}L_0}$ which geometrically implies that $\bomega_B \not\in\cR_{i'} \cap \{J_{i'} = j\}$ and hence $\bomega_B \not\in \cS_{i',j}$.  

If $\bomega_B(V_{i'-2,j}^{'c}) = \bomega_A(V_{i'-2,j}^{'c})$ and $\bomega_B \in \cL_{i',j}$ then the path in $[(i'-3)r,(i'-2)r] \times \R$ is always distance more than 1 from $\hV_{i'-2,j}$, the region of the field that is different and so the passage time of $\gamma_B$ is the same under both fields, that is $\rX_{\gamma_B}^{\bomega_B} = \rX_{\gamma_B}^{\bomega_A}$.  By $\bomega_A \in \cS_{i',j}$ we similarly have that $\rX_{\gamma_A}^{\bomega_A} = \rX_{\gamma_A}^{\bomega_B}$.  It follows that $\gamma_A,\gamma_B$ are optimal optimal paths for both $\bomega_A$ and $\bomega_B$ and hence must be equal. {Indeed, recall that the conforming geodesics are canonically chosen to be the topmost optimal paths. Since $\gamma_A$ and $\gamma_B$ are both optimal paths in $\bomega_A$, this implies $\gamma_A$ lies above $\gamma_B$, and arguing similarly considering $\bomega_B$, it also lies below $\gamma_B$ and hence $\gamma_A=\gamma_B$}. Since the optimal path remains the same, $\bomega_B \in \cS_{i',j}$.\qed

\noindent {\bf Claim:} If $i'=i''$ or $|i'' - i'|\geq \xi$ then $T^{(i'')} \bomega_A \in \cS_{i',j}$ if and only if $\bomega_A \in \cS_{i',j}$.

\begin{center}
\begin{figure}
\includegraphics[width=4in]{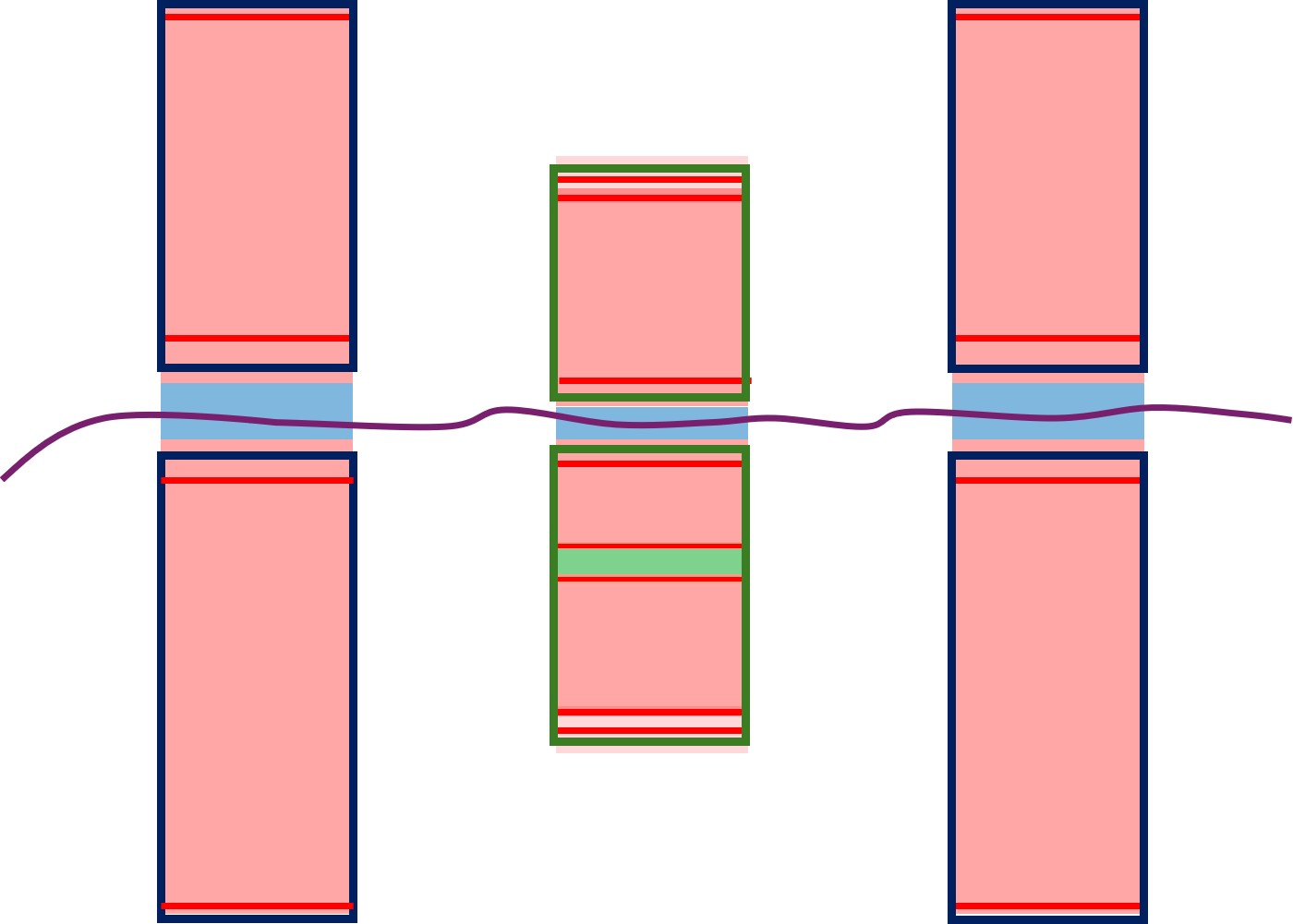}
\caption{A schematic of the resampling argument in Section \ref{s:glauber}. First we resample the outer barriers (boxes with boundaries marked in black) to show that the number of locations where both the typical and outer column barrier events occur is a constant fraction of the locations where the typical events occur. Next by resampling various parts of the central columns we show in Lemma \ref{l:Pminus.B.perc} that the number of locations where both the typical and outer column barrier event as well as the central column events occur is another constant fraction of the locations where the first two types of events occur. This argument is more delicate because it involves events that are neither monotone nor likely, namely the event $\cB^{(4)}_{i,j}$ which involves resampling the green region of the central column.}
\label{f:resamplingG}
\end{figure}
\end{center}

\noindent {\it Proof of Claim}: Suppose that $\bomega_A \in \cS_{i',j}$.  By the first claim, the optimal path is the same for $\bomega_A$ and $T^{(i'')} \bomega_A$ so each $J_{k}$ are the same as well and $T^{(i'')} \bomega_A \in \cR_{i'} \cap \{J_{i'} = j\}$. Since $\bomega_A \in\cP_{i',j}^{-}$ and $\cP_{i',j}^{-}$ only depends on the field in {$\{(i'-3-L_2 2^{(\log_2\log_2 M)^2+1})\Phi^{\ell}+1,\ldots i'+3+L_2 2^{(\log_2\log_2 M)^2+1}\Phi^{\ell}\} \times (\R^2 \setminus  V_{i'-2,j}')$} which is unaffected by  resampling  $V_{i''-2,j}'$ if $|i''-i'| \geq 2^{(\log_2\log_2 M)^3}=\xi$ or $i'=i''$ for $M$ sufficiently large. Hence $T^{(i'')} \bomega_A \in\cP_{i',j}^{-,n,M,\ell}$ and so $T^{(i'')} \bomega_A \in \cS_{i',j}$.  The other direction follows similarly. \qed

\noindent {\bf Claim:} If $\bomega_A \in \cS_{i',j}$ then 
\[
\P[T^{(i')} \bomega_A \in \cD_{i'-2,j}]\geq \frac12\delta_C
\]
where $\delta_C$ is as in \eqref{eq:D5.bound}. 

\noindent {\it Proof of Claim}:  We have that
\begin{align*}
&\P[T^{(i')} \bomega_A \in \cD_{i'-2,j}]\\
& \ = \P\bigg[  \bomega \in \cD_{i'-2,j} \mid \bomega(V_{i'-2,j}^{'c}) =  \bomega_A(V_{i'-2,j}^{'c}), \bomega \in \cS_{i',j} \bigg]\\
& \ = \P\bigg[  \bomega \in \cD^{(2)}_{i'-2,j}\cap \cD^{(3)}_{i'-2,j}\cap \cD^{(4)}_{i'-2,j}\cap \cD^{(5)}_{i'-2,j} \mid \bomega(V_{i'-2,j}^{'c}) =  \bomega_A(V_{i'-2,j}^{'c}), \bomega \in \cL_{i',j} \bigg]
\end{align*}
where the first equality follows by the definition of the resampling operator, the second is by the first claim and the fact that $\cP_{i',j} \subseteq \cD^{(1)}_{i'-2,j}$.

Now note that for $k\in\{2,3,4,5\}$ the events $\cD^{(k)}_{i'-2,j}$ are increasing events in the field.  The event $\cL_{i',j}$ is increasing in the field on $V_{i'-2,j}$ since if the optimal path is distance more than 1 from $\hV_{i'-2,j}'$ the optimal path will remain unchanged if the field in  $V_{i'-2,j}$ is increased and $\cL_{i',j}$ will still hold. So by the equation above and the FKG inequality  
\begin{align*}
&\P[T^{(i')} \bomega_A \in \cD_{i'-2,j}]\\
& \ \geq \P\bigg[  \bomega \in \cD^{(2)}_{i'-2,j}\cap \cD^{(3)}_{i'-2,j}\cap \cD^{(4)}_{i'-2,j}\cap \cD^{(5)}_{i'-2,j} \mid \bomega(V_{i'-2,j}^{'c}) =  \bomega_A(V_{i'-2,j}^{'c})\bigg]\\
& \ \geq \P\bigg[  \bomega' \in \cD^{(2)}_{i'-2,j}\cap \cD^{(3)}_{i'-2,j}\cap \cD^{(4)}_{i'-2,j}\mid \bomega(V_{i'-2,j}^{'c}) =  \bomega_A(V_{i'-2,j}^{'c}) \bigg]
\P\bigg[  \bomega' \in \cD^{(5)}_{i'-2,j}\bigg]\\
& \ \geq \frac12 \delta_C
\end{align*}
where the second inequality is another application of the FKG inequality and the fact that $\cD^{(5)}_{i'-2,j}$ only depends on the field in $V_{i'-2,j}$.  The final inequality follows by the definition of $\cP^-_{i',j}$ and equation~\eqref{eq:D5.bound}. \qed

To complete the proof we will set $\bomega^{(0)} = \bomega$ and for integers $1\leq t \leq \Delta$ define
\[
\bomega^{(t)} = T^{(i+t\xi)} \bomega^{(t-1)}.
\]
and $\bomega^\dagger = \bomega^{(\Delta)}$.  By the second claim we have that for any integers $1\leq t,t' \leq (\log_2 \log_2 M)^2$ $\bomega^{(t')} \in \cS_{i+t\xi,J_{i+t\xi}}$ if and only if $\bomega \in \cS_{i+t\xi,J_{i+t\xi}}$.  Hence we have that
\begin{align*}
&\Bigg(\sum_{t=1}^{\Delta} I(\bomega \in \cD_{i+t\xi-2,J_{i+t\xi}}\cap \cS_{i+t\xi,J_{i+t\xi}}), \sum_{t=1}^{\Delta} I(\bomega \in \cS_{i+t\xi,J_{i+t\xi}})  \Bigg)\\
& \quad \stackrel{d}{=} \Bigg(\sum_{t=1}^{\Delta} I(\bomega^\dagger \in \cD_{i+t\xi-2,J_{i+t\xi}}\cap \cS_{i+t\xi,J_{i+t\xi}}), \sum_{t=1}^{\Delta} I(\bomega^\dagger \in \cS_{i+t\xi,J_{i+t\xi}})  \Bigg)\\
& \quad = \Bigg(\sum_{t=1}^{\Delta} I(\bomega^{(t)} \in \cD_{i+t\xi-2,J_{i+t\xi}}\cap \cS_{i+t\xi,J_{i+t\xi}}), \sum_{t=1}^{\Delta} I(\bomega^{(t-1)} \in \cS_{i+t\xi,J_{i+t\xi}})  \Bigg)\\
\end{align*}
where the equality in distribution is by the exchangeability of the resampling operator and the equality is by the fact that resampling preserves the $\cS_{i+t\xi,J_{i+t\xi}}$ and only $T^{(i+t\xi)}$ changes the event $\cD_{i+t\xi-2,J_{i+t\xi}}$.  By the final claim we have the stochastic domination
\[
\sum_{t=1}^{\Delta} I(\bomega^{(t)} \in \cD_{i+t\xi-2,J_{i+t\xi}}\cap \cS_{i+t\xi,J_{i+t\xi}}) \succeq \hbox{Bin}\Bigg(\sum_{t=1}^{\Delta} I(\bomega^{(t-1)} \in \cS_{i+t\xi,J_{i+t\xi}}) ,\frac12 \delta_C\Bigg)
\]

and so by our equality in distribution,
\begin{align*}
&\P\Bigg[\sum_{t=1}^{\Delta} I(\bomega \in \cD_{i+t\xi-2,J_{i+t\xi}}\cap \cS_{i+t\xi,J_{i+t\xi}})\geq \frac12 \delta_C \sum_{t=1}^{\Delta} I(\bomega \in \cS_{i+t\xi,J_{i+t\xi}}) - \Delta^{2/3} \Bigg]\\
& \quad \geq \max_{Q\leq \Delta} \P\Bigg[\hbox{Bin}\Bigg(Q,\frac12 \delta_C\Bigg)\geq \frac12 \delta_C Q - \Delta^{2/3} \Bigg]\\
&\quad \geq 1 - M^{-200}
\end{align*}

where the last inequality follows by Azuma-Hoeffding Inequality.  This completes the proof.
\end{proof}
We have the following corollary.
\begin{corollary}\label{c:Pminus.D.perc}
There exists $M_0$ such that for all $M\geq M_0$ and all $n$ sufficiently large and $n\leq r=r_\ell \leq M^{1/100}n$ and $2M^{99/100}n\leq  ir \leq (M-2M^{99/100})n$,
\[
\P\Bigg[\sum_{i'=i}^{i+\Phi-1} I(\cD_{i'-2,J_{i'}}\cap \cP_{i',J_{i'}}^{-}, \cR_{i'},|J_{i'}|W_r\leq M^{8/10}W_n) \geq \frac{9\delta_C}{20} \Phi - \Delta^{2/3}\xi \Bigg] \geq 1 - 2M^{-100}.
\]
\end{corollary}

\begin{proof}
By Lemma~\ref{l:Pminus.perc} and Lemma~\ref{l:Pminus.D.perc},
\begin{align*}
&\P\Bigg[\sum_{i'=i}^{i+\Phi-1} I(\cD_{i'-2,J_{i'}}\cap \cP_{i',J_{i'}}^{-}, \cR_{i'},|J_{i'}|W_r\leq M^{8/10}W_n) < \frac{9\delta_C}{20} \Phi - \Delta^{2/3}\xi\Bigg]\\
&\ \leq \P\Bigg[\sum_{i'=i}^{i+\Phi-1} I(\cS_{i',j}) < \frac{\delta_C}{10} \Phi \Bigg]\\
& \ +\sum_{i'=1}^{\xi} \P\Bigg[\sum_{t=1}^{\Delta} I(\cS_{i'+t\xi,J_{i'+t\xi}}, \cD_{i'+t\xi-2,J_{i'+t\xi}}) + \Delta^{2/3} < \frac{\delta_C}{2} \sum_{t=1}^{\Delta} I(\cS_{i'+t\xi,J_{i'+t\xi}}) \Bigg]\\
& \ \leq M^{-100} + \xi M^{-200} \leq 2M^{-100}.
\end{align*}
\end{proof}
By essentially the same proof of Lemma~\ref{l:Pminus.D.perc} and Corollary~\ref{c:Pminus.D.perc} we have the following lemma.
\begin{lemma}\label{l:Pminus.D2.perc}
There exists $M_0$ such that for all $M\geq M_0$ and all $n$ sufficiently large and $0\leq \ell \leq \ell_{\max}$ and $2M^{99/100}n\leq  ir=ir_\ell \leq (M-2M^{99/100})n$,
\[
\P\Bigg[\sum_{i'=i}^{i+\Phi-1} I(\cD_{i'-2,J_{i'}}\cap\cD_{i'+2,J_{i'}}\cap \cP_{i',J_{i'}}^{-}, \cR_{i'},|J_{i'}|W_r\leq M^{8/10}W_n) \geq \frac{9\delta_C^2}{40} \Phi - 2\Delta^{2/3}\xi \Bigg] \geq 1 - 3M^{-100}.
\]
\end{lemma}

\subsection{Central Column}
The final piece to prove Theorem~\ref{t:Pplus} is to show that a constant fraction of the $i$ which have  $\cD_{i-2,J_i}\cap\cD_{i+2,J_i}\cap \cP_{i,J_i}^{-}\cap \cR_i\cap\{|J_i|W_r\leq M^{8/10}W_n\}$ also have the central column event $\cB_{i,J_i}$.  The proof is similar to the outer columns resampling scheme but differs because $\cB_{i,j}$ is not an increasing event in the field in $V_{i,j}$ so we cannot apply the FKG inequality in the same way.  In this subsection we will define the events 
\[
\cS_{i,j} = \cD_{i-2,j}\cap\cD_{i+2,j}\cap\cP_{i,j}^{-} \cap \cR_{i} \cap \{J_{i} = j\}
\]
for {$-M^{8/10}W_n\leq jW_r \leq M^{8/10}W_n$} and set {$\cS_{i',j}$} to be the empty set when {$|j|W_r > M^{8/10}W_n$}.
\begin{lemma}\label{l:Pminus.B.perc}
There exists $M_0$ such that for all $M\geq M_0$ and all $n$ sufficiently large and $0\leq \ell \leq \ell_{\max}$ and $2M^{99/100}n\leq  ir=ir_\ell \leq (M-2M^{99/100})n$,
\begin{align*}
\P\Bigg[&\sum_{t=1}^{\Delta} I(\cB_{i+t\xi,J_{i+t\xi}},\cS_{i+t\xi,J_{i+t\xi}}) + \Delta^{2/3} \geq \frac{\delta_A\delta_B}{32} \sum_{t=1}^{\Delta} I(\cS_{i+t\xi,J_{i+t\xi}}) \Bigg] \geq 1 - M^{-150},
\end{align*}
where $\delta_A$ and $\delta_B$ are as in \eqref{eq:B4.bound} and \eqref{eq:B7.bound} respectively.  
\end{lemma}

\begin{proof}
The beginning of the proof is very similar to Lemma~\ref{l:Pminus.D.perc}.
We define the resampling operator $T^{(i')}=T_{\{\cS_{i',j}\},\{U_{i',j}\}}$ on the field $\bomega$ on $\Z\times \R^2$ where $U_{i',j} = V_{i',j}$.  
Suppose that $\bomega_A$ and $\bomega_B$ are two configurations and let $\gamma_A,\gamma_B$ be their optimal paths.  The following two claims have essentially identically proofs to the corresponding claims in Lemma~\ref{l:Pminus.D.perc}.

\noindent {\bf Claim:} If $\bomega_A \in \cS_{i',j}$ then
\[
\bigg\{\bomega_B:\bomega_B(V_{i',j}^{c}) = \bomega_A(V_{i',j}^{c}),\bomega_B \in \cS_{i',j} \bigg\}
=\bigg\{\bomega_B:\bomega_B(V_{i',j}^{c}) = \bomega_A(V_{i',j}^{c}),\bomega_B \in \cL_{i',j} \bigg\}.
\]
On this event $\gamma_A=\gamma_B$.

\noindent {\bf Claim:} If $i'=i''$ or $|i'' - i'|\geq \xi$ then $T^{(i'')} \bomega_A \in \cS_{i',j}$ if and only if $\bomega_A \in \cS_{i',j}$.

The next claim is more complicated than in Lemma~\ref{l:Pminus.D.perc} because $\cB_{i',j}$ is not monotone in $\bomega(V_{i',j})$.

\noindent {\bf Claim:} If $\bomega_A \in \cS_{i',j}\cap \cW_{i',j}^{glo}$ then 
\[
\P[T^{(i')} \bomega_A \in \cB_{i',j}]\geq \frac1{32}\delta_A \delta_B.
\]
\noindent {\it Proof of Claim}:  
Suppose that $\bomega$ satisfies both $\bomega \in \cB_{i',j}$ and $\bomega(V_{i',j}^{c}) =  \bomega_A(V_{i',j}^{c})$ .
Note that the event $\cD_{i'-2,j}\cap\cD_{i'+2,j}\cap\cP_{i',j}^{-}\cap \cW_{i',j}^{glo}$ does not depend on the field in $V_{i',j}$ so
\[
\bomega \in \cB_{i',j}\cap \cD_{i'-2,j}\cap\cD_{i'+2,j}\cap\cP_{i',j}^{-}\cap \cW_{i',j}^{glo} = \cP_{i',j}.
\]
By Lemma~\ref{c:path.sepatated} the optimal path $\gamma$ must be either Type~1 or Type~6.  Both Type~1 and Type~6 paths have distance greater than 1 from $\hV_{i',j}$ and so  $\rX_\gamma^{\bomega} = \rX_\gamma^{\bomega_A}$.  By $\cR_{i'}$ the geodesic $\gamma_A$ also has distance greater than 1 from $\hV_{i',j}$ so $\rX_{\gamma_A}^{\bomega} = \rX_{\gamma_A}^{\bomega_A}$.  Since $\gamma,\gamma_A$ are both optimal paths we must have, as before, $\gamma=\gamma_A$.  Since $\gamma_A$ satisfies $\cR_{i'}\cap \{J_{i'} = j\}$, so does $\gamma$.  Hence
\begin{equation}\label{eq:bomega.in.cS}
\bomega \in \cD_{i'-2,j}\cap\cD_{i'+2,j}\cap\cP_{i',j}^{-} \cap \cR_{i'} \cap \{J_{i'} = j\} =\cS_{i',j}.
\end{equation}
By the definition of the resampling operator,
\begin{align}\label{eq:cB.resampling.equality}
\P[T^{(i')} \bomega_A \in \cB_{i',j}] & = \P\bigg[  \bomega \in \cB_{i',j} \mid \bomega(V_{i',j}^{c}) =  \bomega_A(V_{i',j}^{c}), \bomega \in \cS_{i',j} \bigg]\nonumber\\
& \geq \P\bigg[  \bomega \in \cB_{i',j} \mid \bomega(V_{i',j}^{c}) =  \bomega_A(V_{i',j}^{c}) \bigg]
\end{align}
where the inequality follows from equation~\eqref{eq:bomega.in.cS}.  

Now suppose that $\bomega$ satisfies just $\bomega(V_{i',j}^{c}) =  \bomega_A(V_{i',j}^{c})$.  Since this implies $\bomega\in \cP_{i',j}^{-,n,M,\ell}$ we automatically have $\bomega\in \cB^{(1)}_{i',j}$.  Both $\cB^{(4)}_{i',j}$ and $\cB^{(5)}_{i',j}$ are independent of $\bomega(V_{i',j}^{c})$.  Defining 
\[
\Theta_{i,j} = \Big\{\frac{(i-1)r}{n},\ldots,\frac{ir}{n}\Big\} \times \R \times \big[(j -\sqrt{\alpha}) W_r -1, (j-\sqrt{\alpha} + h) W_r + 1\big]
\]

we have that $\cB^{(4)}_{i',j}$ is $\bomega(\Theta_{i',j})$ measurable and independent of $\bomega(V_{i',j}^{c}$).  By~\eqref{eq:B5.bound} and Markov's Inequality,

\begin{align*}
\P\bigg[\P[\cB^{(5)}_{i',j}\mid \bomega(\Theta_{i',j}\cup V_{i',j}^{c})]\geq \frac12 \mid \bomega(V_{i',j}^{c}) =  \bomega_A(V_{i',j}^{c})\bigg] 
&= 1-\P\big[\P[(\cB^{(5)}_{i',j})^c\mid \bomega(\Theta_{i',j})]> \frac12\big] \\
&\geq 1 - \frac{\E\big[ \P[(\cB^{(5)}_{i',j})^c\mid \bomega(\Theta_{i',j})] \big]}{1/2}\\
&\geq 1-\frac{\delta_A}{50}.
\end{align*}
Similarly by~$\cP_{i',j}^{-}$,
\begin{align*}
&\P\bigg[\P[\cB^{(6)}_{i',j}\mid \bomega(\Theta_{i',j}\cup V_{i',j}^{c})]\geq \frac12 \mid \bomega(V_{i',j}^{c}) =  \bomega_A(V_{i',j}^{c})\bigg] \\
&\qquad= 1-\P\bigg[\P[(\cB^{(6)}_{i',j})^c\mid \bomega(\Theta_{i',j}\cup V_{i',j}^{c})]> \frac12 \mid \bomega(V_{i',j}^{c}) =  \bomega_A(V_{i',j}^{c})\bigg]  \\
&\qquad\qquad\geq 1 - \frac{\E\bigg[ I\big((\cB^{(6)}_{i',j})^c\big)\mid \bomega(V_{i',j}^{c}) =  \bomega_A(V_{i',j}^{c})\bigg] }{1/2}\\
&\qquad\geq 1-\frac{\delta_A}{50}.
\end{align*}

Altogether we have that if
\[
\cB^{(*)}_{i',j}= \bigg\{ \P[\cB^{(4)}_{i',j}\mid \bomega(\Theta_{i',j}\cup V_{i',j}^{c})] = 1, \P[\cB^{(5)}_{i',j}\mid \bomega(\Theta_{i',j}\cup V_{i',j}^{c})] \geq \frac12,\P[\cB^{(6)}_{i',j}\mid \bomega(\Theta_{i',j}\cup V_{i',j}^{c})]\geq \frac12 \bigg\}
\]

then by the above estimates and \eqref{eq:B4.bound}
\begin{align}\label{eq:bstar.lbound}
&\P\bigg[\cB^{(*)}_{i',j} \mid \bomega(V_{i',j}^{c}) =  \bomega_A(V_{i',j}^{c})\bigg] \geq \delta_A-\frac{2\delta_A}{50} \geq \frac{\delta_A}{2}.
\end{align}
Note that the event $\cB^{(*)}_{i',j}$ only depends on the configuration in $\Theta_{i',j}\cup V_{i',j}^{c}$ so we will abuse notation and view it as a configuration on just this set.  The events $\cB^{(k)}_{i',j}$ for $k\in\{1,2,3,7\}$ do not depend on the field in $\Theta_{i',j}$ and so
\begin{align}\label{eq:cond.cBk}
\P[\cB^{(k)}_{i',j}\mid \bomega(\Theta_{i',j}),\bomega(V_{i',j}^{c}) =  \bomega_A(V_{i',j}^{c})] = \P[\cB^{(k)}_{i',j}\mid\bomega(V_{i',j}^{c}) =  \bomega_A(V_{i',j}^{c})] \geq \begin{cases}
1  &\hbox{if } k=1,\\
\frac12  &\hbox{if } k=2,\\
\frac12  &\hbox{if } k=3,\\
\delta_B  &\hbox{if } k=7.
\end{cases}
\end{align}
where the bounds for $k\in\{1,2,3\}$ follows from the definition of $\cP_{i',j}^-$ while the bound in the case of $k=7$ follows from equation~\eqref{eq:B7.bound} and the fact that $\cB^{(7)}_{i',j}$ does not depend on the field in $V_{i',j}^{c}$.

Clearly from their definitions the events $\cB^{(k)}_{i',j}$ for $k\in\{2,3,6,7\}$  are increasing events in the field and in particular the field in $(\Theta_{i',j}\cup V_{i',j}^{c})^c$.  The event $\cB^{(5)}_{i',j}$ is not an increasing event as a function of the field overall but it is increasing as a function of  the field in the field in $(\Theta_{i',j}\cup V_{i',j}^{c})^c$ because the {event $\cM$ involves a difference of two infimums of passage times}, the latter of which is measurable with respect to the field in $\Theta_{i',j}$.  Hence we have that
\begin{align*}
&\P\bigg[  \bomega \in \cB_{i',j} \mid \bomega(V_{i',j}^{c}) =  \bomega_A(V_{i',j}^{c}) \bigg]\\
&=\E\bigg[  \P\Big[\bomega \in \bigcap_{k=1}^7 \cB_{i',j}^{(k)}\mid \bomega(\Theta_{i',j}\cup V_{i',j}^{c})\Big] \mid \bomega(V_{i',j}^{c}) =  \bomega_A(V_{i',j}^{c}) \bigg]\\
&\geq\E\bigg[  \P\Big[\bomega \in \bigcap_{k=1}^7 \cB_{i',j}^{(k)}\mid \bomega(\Theta_{i',j}\cup V_{i',j}^{c})\Big] I(\bomega(\Theta_{i',j}\cup V_{i',j}^{c})\in \cB^{(*)}_{i',j})\mid \bomega(V_{i',j}^{c}) =  \bomega_A(V_{i',j}^{c}) \bigg]\\
&=\E\bigg[  \P\Big[\bomega \in \bigcap_{k\in\{2,3,5,6,7\}} \cB_{i',j}^{(k)}\mid \bomega(\Theta_{i',j}\cup V_{i',j}^{c})\Big] I(\bomega(\Theta_{i',j}\cup V_{i',j}^{c})\in \cB^{(*)}_{i',j})\mid \bomega(V_{i',j}^{c}) =  \bomega_A(V_{i',j}^{c}) \bigg]\\
&\geq\E\bigg[  \prod_{k\in\{2,3,5,6,7\}} \P\Big[\bomega \in \cB_{i',j}^{(k)}\mid \bomega(\Theta_{i',j}\cup V_{i',j}^{c})\Big] I(\bomega(\Theta_{i',j}\cup V_{i',j}^{c})\in \cB^{(*)}_{i',j})\mid \bomega(V_{i',j}^{c}) =  \bomega_A(V_{i',j}^{c}) \bigg]\\
&\geq\E\bigg[  \frac12 \cdot\frac12 \cdot \frac12\cdot\frac12 \cdot \delta_B  I(\bomega(\Theta_{i',j}\cup V_{i',j}^{c})\in \cB^{(*)}_{i',j})\mid \bomega(V_{i',j}^{c}) =  \bomega_A(V_{i',j}^{c}) \bigg]\\
&=\frac{\delta_B}{16}\P\bigg[ \cB^{(*)}_{i',j}\mid \bomega(V_{i',j}^{c}) =  \bomega_A(V_{i',j}^{c}) \bigg]\\
&\geq \frac{\delta_A\delta_B}{32},
\end{align*}
where the first inequality is because we simply added an indicator, the next equality is because $\cB_{i',j}^{(1)}$ and $\cB_{i',j}^{(4)}$ hold with probability 1 on the events $\bomega(V_{i',j}^{c}) =  \bomega_A(V_{i',j}^{c})$ and $\cB^{(*)}_{i',j}$ respectively, the second inequality is by the FKG inequality noting that each of these events are increasing in the field on $(\Theta_{i',j}\cup V_{i',j}^{c})^c$, the third inequality is by equation~\eqref{eq:cond.cBk} and the definition of $\cB^{(*)}_{i',j}$ and the final inequality is by equation~\eqref{eq:bstar.lbound}. Together with equation~\eqref{eq:cB.resampling.equality} this completes the proof of the claim. \qed

To complete the proof we will set $\bomega^{(0)} = \bomega$ and for integers $1\leq t \leq \Delta$ define
\[
\bomega^{(t)} = T^{(i+t\xi)} \bomega^{(t-1)}.
\]
and $\bomega^\dagger = \bomega^{(\Delta)}$.  By the second claim we have that for any integers $1\leq t,t' \leq \Delta$ that $\bomega^{(t')} \in \cS_{i+t\xi,J_{i+t\xi}}$ if and only if $\bomega \in \cS_{i+t\xi,J_{i+t\xi}}$.  Hence we have that
\begin{align}\label{eq:glauber.resampl.equality.dist}
&\Bigg(\sum_{t=1}^{\Delta} I(\bomega \in \cB_{i+t\xi,J_{i+t\xi}}\cap \cS_{i+t\xi,J_{i+t\xi}}), \sum_{t=1}^{\Delta} I(\bomega \in \cS_{i+t\xi,J_{i+t\xi}})  \Bigg)\nonumber\\
& \quad \stackrel{d}{=} \Bigg(\sum_{t=1}^{\Delta} I({\bomega^{\dagger}} \in \cB_{i+t\xi,J_{i+t\xi}}\cap \cS_{i+t\xi,J_{i+t\xi}}), \sum_{t=1}^{\Delta} I(\bomega^{\dagger} \in \cS_{i+t\xi,J_{i+t\xi}})  \Bigg)\nonumber\\
& \quad = \Bigg(\sum_{t=1}^{\Delta} I(\bomega^{(t)} \in \cB_{i+t\xi,J_{i+t\xi}}\cap \cS_{i+t\xi,J_{i+t\xi}}), \sum_{t=1}^{\Delta} I(\bomega^{(t-1)} \in \cS_{i+t\xi,J_{i+t\xi}})  \Bigg)
\end{align}
where the equality in distribution is by the exchangeability of the resampling operator and the equality is by the fact that resampling preserves the $\cS_{i+t\xi,J_{i+t\xi}}$ and only $T^{(i+t\xi)}$ changes the event $\cB_{i+t\xi,J_{i+t\xi}}$.  By the final claim we have the stochastic domination
\[
\sum_{t=1}^{\Delta} I(\bomega^{(t)} \in \cB_{i+t\xi,J_{i+t\xi}}\cap \cS_{i+t\xi,J_{i+t\xi}}) \succeq \hbox{Bin}\Bigg(\sum_{t=1}^{\Delta} I(\bomega^{(t-1)} \in \cS_{i+t\xi,J_{i+t\xi}} \cap \cW_{i+t\xi,J_{i+t\xi}}^{glo}), \frac{\delta_A\delta_B}{32}\Bigg)
\]

and so by the Azuma-Hoeffding Inequality
\begin{align}\label{eq:binomial.domination.resamp}
&\P\Bigg[\sum_{t=1}^{\Delta} I(\bomega^{(t)} \in \cB_{i+t\xi,J_{i+t\xi}}\cap \cS_{i+t\xi,J_{i+t\xi}}) \nonumber\\
&\qquad\qquad\qquad \geq \frac{\delta_A\delta_B}{32} \sum_{t=1}^{\Delta} I(\bomega^{(t-1)} \in \cS_{i+t\xi,J_{i+t\xi}}\cap \cW_{i+t\xi,J_{i+t\xi}}^{glo}) - \Delta^{2/3} \Bigg]\\
& \quad \geq \max_{Q\leq \Delta} \P\Bigg[\hbox{Bin}\Bigg(Q,\frac{\delta_A\delta_B}{32} \Bigg)\geq {\frac{\delta_A \delta_B}{32} Q} - \Delta^{2/3} \Bigg]\nonumber\\
&\quad \geq 1 - \frac12 M^{-150}.
\end{align}
By Lemma~\ref{l:cW.global}
\[
\P[\bomega^{(t-1)} \in \cW_{i+t\xi,J_{i+t\xi}}^{glo}] = \P[\bomega \in \cW_{i+t\xi,J_{i+t\xi}}^{glo}] \geq 1-M^{-200}
\]
and so 
\[
\P\Bigg[\sum_{t=1}^{\Delta} I(\bomega^{(t-1)} \in \cS_{i+t\xi,J_{i+t\xi}}\cap \cW_{i+t\xi,J_{i+t\xi}}^{glo})=\sum_{t=1}^{\Delta} I(\bomega^{(t-1)} \in \cS_{i+t\xi,J_{i+t\xi}})\Bigg] \geq 1-M^{-199}
\]
and by combining with~\eqref{eq:binomial.domination.resamp} we have that
\begin{align*}
&\P\Bigg[\sum_{t=1}^{\Delta} I(\bomega^{(t)} \in \cB_{i+t\xi,J_{i+t\xi}}\cap \cS_{i+t\xi,J_{i+t\xi}})\\
&\qquad\qquad\qquad \geq \frac{\delta_A\delta_B}{32}  \sum_{t=1}^{\Delta} I(\bomega^{(t-1)} \in \cS_{i+t\xi,J_{i+t\xi}}) - \Delta^{2/3} \Bigg]\\
&\quad \geq 1 - \frac12 M^{-150}.
\end{align*}
Using the equality in distribution from equation~\eqref{eq:glauber.resampl.equality.dist},
\begin{align*}
&\P\Bigg[\sum_{t=1}^{\Delta} I(\bomega \in \cB_{i+t\xi,J_{i+t\xi}}\cap \cS_{i+t\xi,J_{i+t\xi}})\\
&\qquad\qquad\qquad \geq \frac{\delta_A\delta_B}{32} \sum_{t=1}^{\Delta} I(\bomega \in \cS_{i+t\xi,J_{i+t\xi}}) - \Delta^{2/3} \Bigg]\\
&\quad \geq 1 -  M^{-150}.
\end{align*}
This completes the proof.
\end{proof}

\begin{proof}[Proof of Theorem~\ref{t:Pplus}]
Since 
\[
\cP_{i,J_i} \cap \{|J_i|\leq M^{8/10}\}\cap \cR_i = \cS_{i,J_i}\cap \cB_{i,J_i}\cap \cW^{glo}_{i,J_i},
\]
by Lemma~\ref{l:Pminus.B.perc}, Lemma~\ref{l:Pminus.D2.perc} and Lemma~\ref{l:cW.global},
\begin{align*}
&\P\Bigg[\sum_{i'=i}^{i+\Phi-1} I(\cP_{i',J_{i'}},|J_{i'}|W_r\leq M^{8/10}W_n) < \frac{9\delta_A\delta_B\delta_C^2}{1280}\Phi - \Delta^{2/3}\xi \Bigg]\\
&\ \leq \P\Bigg[\sum_{i'=i}^{i+\Phi-1} I(\cS_{i',j}) < \frac{9\delta_C^2}{40} \Phi - 2\Delta^{2/3}\xi\Bigg]\\
& \ +{\sum_{i'=i}^{i+\xi-1}} \P\Bigg[\sum_{t=1}^{\Delta} I(\cS_{i'+t\xi,J_{i'+t\xi}}, \cB_{i'+t\xi,J_{i'+t\xi}}) + \Delta^{2/3} <  \frac{\delta_A\delta_B}{32} \sum_{t=1}^{\Delta} I(\cS_{i'+t\xi,J_{i'+t\xi}}) \Bigg]\\
& \ + \sum_{i'=i}^{i+\Phi-1}\P\bigg[ I\big((\cW^{glo}_{i,J_i})^c\big)\bigg]\\
& \ \leq 3M^{-100} + \xi M^{-200} +M^{-199} \leq 4 M^{-100},
\end{align*}
which completes the proof for large enough $M_0$.
\end{proof}

\section{Estimates for elementary events}
\label{s:letter}
In this section we prove the estimates for the elementary events $\cA, \cI, \cK, \cJ, \cM, \cZ$ that were stated in Section \ref{s:events}. These estimates, while technical, mostly follow from Proposition \ref{p:paraestimateconforming} and from the FKG inequality in some cases. 

Recall that these elementary events were all defined at scales $r=r_{\ell}(n,M)$ for $1\le l \le \ell_{\max}=\ell_{\max}(n,M)$. As mentioned at the beginning of Section \ref{s:elementary1} all the constants involved in probability estimates are independent of $n,M,\ell$ and the bounds work for $M\ge M_0$, $n\ge n_0(M)$ and all $\ell$. Unless otherwise specified,
all the lemmas in this section shall also has constants independent of $n,M,\ell$ and hold for the choice of $n,M,\ell$ as above, even if it might not be explicitly stated each time. The estimates will hold for an appropriate range of horizontal locations $i$  ($1\le i \le \frac{Mn}{r}$, unless otherwise specified) and an appropriate range of vertical locations $j$.

\subsection{Event $\cA$: Proof of Lemma \ref{l:cA.bound}}
Recall the definition of the events $\cA^{\pm}$. 
\begin{align*}
\cA_{i,j,z}^- &= \Bigg\{\sup_{\substack{|y|,|y'| \leq M W_n \\ u=((i-1)r,y) \\ v= (ir,y')}} \hrX_{uv} - \frac{|y - jW_r| + |y' - jW_r|}{W_r}Q_r \leq z Q_r \Bigg\};\\
\cA_{i,j,z}^+ &= \Bigg\{\inf_{\substack{|y|,|y'| \leq n^{\beta}W_n \\ u=((i-1)rn,y) \\ v= (ir,y')}} \hrX_{uv}  + \frac{|y - jW_r| + |y' - jW_r|}{W_r}Q_r  \geq -z Q_r \Bigg\}.
\end{align*}
These say that the passage times across a column near any given vertical location are typical with large probability where the tolerance for being typical increases slightly as the vertical distance of the end points from the specified location increases. We now prove Lemma \ref{l:cA.bound}. 

\begin{proof}[Proof of Lemma \ref{l:cA.bound}]
For notational brevity we shall assume without loss of generality that $i=1$. The  same proof goes through for general $i$. Fix $j$ with $|j|\le \frac{MW_n}{W_r}$. Clearly it suffices to prove the lemma for $z$ sufficiently large. 

For $j_1,j_2$  with $|j_1|, |j_2|\le \frac{MW_n}{W_r}$, let $A_{j_1,j_2}$ denote the event that for all $(0,y)\in \ell_{0,j_1W_r, (j_1+1)W_r}$ and for all $(r,y')\in \ell_{r,j_2W_r, (j_2+1)W_r}$ we have 
$$|\hrX_{(0,y),(r,y')}|\le (|j_1-j|+|j_2-j|-2+z)Q_r.$$
Clearly from the definition 
$$ \cA_{1,j,z}^- \supseteq \bigcap_{j_1,j_2} A_{j_1,j_2}.$$
It follows from Proposition \ref{p:paraestimateconforming} that for each pair $j_1,j_2$ as above
$$\P(A_{j_1,j_2}) \ge 1-\exp\left(-C (z+|j_1-j|+|j_2-j|)^{\theta_2}\right).$$
The probability bound on  $\cA_{1,j,z}^-$ follows by taking a union bound over all $j_1,j_2$. 

For the other bound we shall assume without loss of generality that $n^{\beta}\frac{W_n}{W_r}$ is an integer. For integers $-n^{\beta}\frac{W_n}{W_r}\le j_1,j_2\le n^{\beta}\frac{W_n}{W_r}-1$, let $\widetilde{A}_{j_1,j_2}$ denote the event that for all $(0,y)\in \ell_{0,j_1W_r, (j_1+1)W_r}$ and for all $(r,y')\in \ell_{r,j_2W_r, (j_2+1)W_r}$
$$\hrX_{(0,y),(r,y')} \ge -(z+|j_1-j|+|j_2-j|-2)Q_r.$$
Clearly from the definition 
$$ \cA_{1,j,z}^+ \supseteq \bigcap_{j_1,j_2} \widetilde{A}_{j_1,j_2}.$$
It follows from Proposition \ref{p:paraestimateconforming} that 
$$\P(\widetilde{A}_{j_1,j_2}) \ge 1-\exp\left(-C (z+|j_1-j|+|j_2-j|-2)^{\theta_2}\right).$$
The probability bound on  $\cA_{1,j,z}^+$ follows by taking a union bound over all $j_1,j_2$.
\end{proof}

\subsection{The event $\cI$: Proof of Lemma \ref{l:cI.bound}}
Recall the definition of the events $\cI^{\pm}$. 
\begin{align*}
\cI^+_{i,j,j',z} &= \bigg\{\inf_{y,y'\in[jW_r,j' W_r]} \inf_{\substack{\zeta' \subset  [(i-1)r,ir]\times [jW_r,j' W_r] \\ \zeta'(0)=((i-1)r,y) \\ \zeta'(1)=(ir,y')}} \hrX_{\zeta'}  \geq  z Q_r \bigg\};\\
\cI^-_{i,j,j',z} &= \bigg\{\inf_{y,y'\in[jW_r,j' W_r]} \inf_{\substack{ \zeta' \subset  [(i-1)r,ir]\times [jW_r,j' W_r] \\ \zeta'(0)=((i-1)r,y) \\ \zeta'(1)=(ir,y')}} \hrX_{\zeta'}  \leq  z Q_r \bigg\}.
\end{align*}
We shall the prove different parts of Lemma \ref{l:cI.bound} separately below.

\begin{lemma}\label{l:cI.bound1}
There exists $C,\theta_5>0$, not depending on $n,M,\ell$ such that for all $i$ and {$|j|,|j'|\leq M$} and $z\geq 0$
\[
\P[\cI^+_{i,j,j',-z}] \geq 1 - (1\vee|j-j'|)^2\exp(-C z^{\theta_5}).
\]
\end{lemma}

\begin{proof}
    Observe that for $|j|,|j'|\le 2M$ since the event only considers paths contained in $[(i-1)r,ir]\times [jW_r,j'W_r]$ the event is translation invariant in $i$ and $j$ and hence it suffices to prove the result for $i=1,j=0$ and each fixed $|j'|\le 2M$. Fix such a $j'$ and observe that since we are trying to show that passage times cannot be too small, we can ignore the condition $\zeta' \subset  [0,r]\times [0,j' W_r]$ for the proof of the lower bound of $\P(\cI^+_{1,0,j',-z})$. Without loss of generality assume $j'\ge 0$ and notice that for all {$k,k'\in [0,j'-1]\cap \Z$} we have by Proposition \ref{p:paraestimateconforming} that
    $$\P\left(\inf_{\substack{y\in [kW_r, (k+1)W_r],\\y'\in[k'W_r,(k'+1)W_r]}} \inf_{\substack{\zeta'(0)=((i-1)r,y) \\ \zeta'(1)=(ir,y')}} \hrX_{\zeta'}  \ge -zQ_r \right)\ge 1-\exp(-cz^{\theta_5}).$$
    The lemma follows from a union bound over $(1\vee|j-j'|)^2$ many possible pairs $(k,k')$. 
\end{proof}

\begin{lemma}
    \label{l:cI.bound2}
    For any $z\geq 0$ there exists $\delta'>0$ such that for all $i$ and all $j'\geq j+1$,
\[
\P[\cI^-_{i,j,j',-z}] \geq \delta'.
\]
\end{lemma}

\begin{proof}
Without loss of generality, we shall prove this result for $i=0$, and $j'=j+1$. It suffices to show that there exists $\delta'>0$ such that there with probability at least $\delta'>0$ there exists a conforming path $\zeta$ from $(0,(j+\frac{1}{2})W_r)$ to $(r,(j+\frac{1}{2})W_r)$ such that $\zeta \subset [0,r]\times [jW_r, (j+1)W_r]$ and $\rX_{\zeta}\le r-zQ_r$. Clearly, this event is independent of $j$, and without loss of generality from now on we shall work with $j=0$. 

Let us define the following three events (see Figure \ref{f:constrained}).

$$A_1=\left\{\exists \zeta\subset [0,r]\times [0,W_r], \zeta(0)=(0,\frac{1}{2}W_r), \zeta(1)=(\frac{r}{100},\frac{1}{2}W_r), \rX_{\zeta}\le \frac{r}{100}+zQ_r\right\};$$

$$A_2=\left\{\exists \zeta\subset [0,r]\times [0,W_r], \zeta(0)=(\frac{99}{100}r,\frac{1}{2}W_r), \zeta(1)=(r,\frac{1}{2}W_r), \rX_{\zeta}\le \frac{r}{100}+zQ_r\right\};$$

$$A_3=\left\{\exists \zeta\subset [0,r]\times [0,W_r], \zeta(0)=(\frac{1}{100}r,\frac{1}{2}W_r), \zeta(1)=(\frac{99}{100}r,\frac{1}{2}W_r), \rX_{\zeta}\le \frac{r}{100}-3zQ_r\right\}.$$

Clearly it suffices to show that 

\begin{equation}
    \label{e:goodsuff}
    \P(A_1\cap A_2 \cap A_3)\ge \delta'
\end{equation}
 as one can simply consider the concatenation of the paths given by these three events. 

For $z$ sufficiently large (we only need to deal with this case) it follows from Proposition \ref{p:constrainerX} that $\P(A_1), \P(A_2)\ge \frac{1}{2}$, and by the FKG inequality it suffices to show that $\P(A_3)\ge 4\delta'$. For some integers $H,L$ to be chosen sufficiently large later. Let $u_i=(\frac{1}{100}r+\frac{ir}{100H}, \frac{1}{2}W_r)$. Let $B_i$ denote the event
\begin{center}
\begin{figure}
\includegraphics[width=5in]{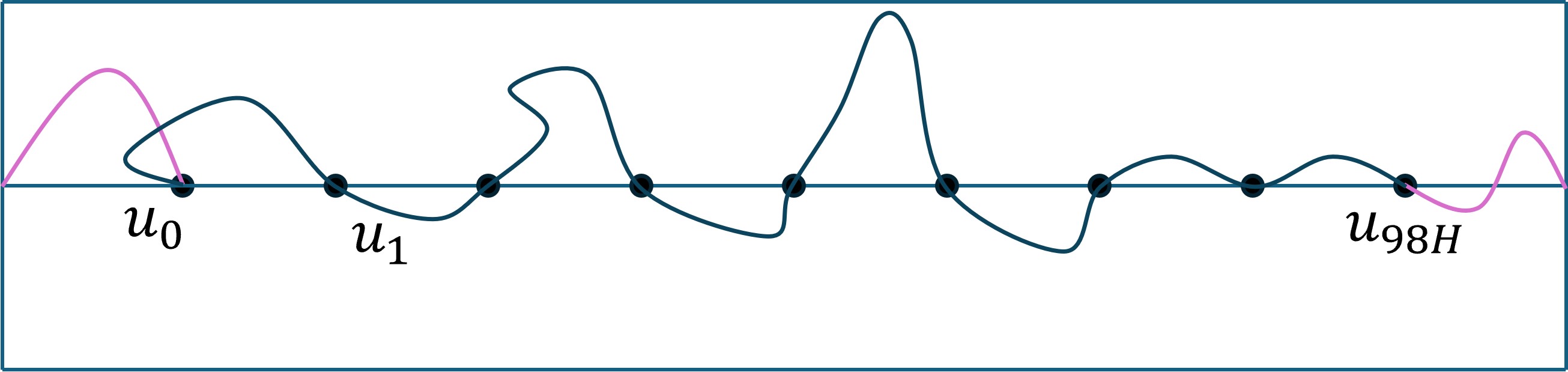}
\caption{ Proof of Lemma \ref{l:cI.bound2} where we show that with positive probability there exists a very good path (on-scale) across an on-scale rectangle contained in the rectangle. We ask that the purple paths are not too bad (events $A_1, A_2$) whereas the events $B_i$ that the paths in blue are all very good. The concatenated path will then be a very good path constrained in the rectangle.}
\label{f:constrained}
\end{figure}
\end{center}
$$B_i:=\left\{\exists \zeta\subset [0,r]\times [0,W_r], \zeta(0)=u_{i-1}, \zeta(1)=u_i, \rX_{\zeta}\le \frac{r}{100H}-LQ_{\frac{r}{100H}}\right\}$$
for $i=1,2,\ldots, 98H$. First choose $L$ sufficiently large such that $98LHQ_{\frac{r}{100H}}\ge 3zQ_r$ (note that this is possible independent of $H$ since $Q$ grows sublinearly). By \cite[Proposition 2.3]{BSS23} and Lemma \ref{l:proxy} it follows that there exists $\delta>0$ such that for each $i'$, 

$$\P\left(\rX_{(\frac{1}{100}r+(i'-1)\frac{r}{100H},\frac{1}{2}W_r),(\frac{1}{100}r+ i'\frac{r}{100H},\frac{1}{2}W_r)}\le \frac{r}{100H}-LQ_{\frac{r}{100H}}\right)\ge \delta.$$
Now choose $H$ sufficiently large depending on $L$ such that the probability that the conforming geodesic between $(\frac{1}{100}r+(i'-1)\frac{r}{100H},\frac{1}{2}W_r)$ and $(\frac{1}{100}r+ i'\frac{r}{100H},\frac{1}{2}W_r)$ exits $[0,r]\times [0,W_r]$ is at most $\delta/2$ (possible by Lemma \ref{l:proxytrans}). It therefore follows that for each $i'$, $\P(B_{i'})\ge \frac{\delta}{2}$. Since $A_3 \supset \cap_{i'} B_{i'}$, it follows by the FKG inequality that 
$$\P(A_3)\ge (\frac{\delta}{2})^{98H}.$$
By choosing $\delta'$ sufficiently small, \eqref{e:goodsuff} and hence the lemma follows. 
\end{proof}

\begin{lemma}
\label{l:cI.bound3}
For any $t$ and $z\geq 0$ there exists $\delta(t,z) >0$ such that if $j'\leq j+t$
\[
\P[\cI^+_{i,j,j',z}] \geq \delta.
\]
\end{lemma}

\begin{lemma}
    \label{l:cI.bound4}
 There exists $C,\theta_5>0$, not depending on $n,M,\ell$ such that for all $i$ and $|j|,|j'|\leq M$ and $z\geq 0$
\[
\P[\cI^-_{i,j,j',z}] \geq 1 -\exp(-C z^{\theta_5});
\]   
\end{lemma}

{The proof of Lemma \ref{l:cI.bound4} is an immediate consequence of Proposition \ref{p:constrainerX}. The proof of Lemma \ref{l:cI.bound3} is somewhat long and is done in several steps below in the next subsection.}

\subsubsection{Construction of the barrier event}

Showing that with positive probability all passage times across a rectangle are atypically large (often referred to as a barrier event), is of an independent interest, and has been useful in several other related models. Therefore we shall prove this lemma for the original passage times $X$. The same argument works for passage times $\rX$ and Lemma \ref{l:cI.bound3} will follow from the following lemma.

\begin{lemma}
    \label{l:barrier}
    For $L,L'>0$ fixed, there exists $\delta=\delta(L,L')>0$ such that for all $r$ sufficiently large 
    $$\P\left( \inf_{y,y'\in [0,L'W_r]}X_{(0,y),(r,y')}\ge r+LQ_r \right)\ge \delta.$$
    The same conclusion holds when $X$ is replaced by $\rX$. 
\end{lemma}

Observe that in the above lemma we are centering by $r$ instead of $r+\frac{1}{2r}|y-y'|^2$, but however the lemma for the second centering follows from the first by noticing that $\frac{1}{2r}|y-y'|^2\le \frac{1}{2r}(L')^2W_r^2 = \frac12 (L')^2 Q_r$. Before providing the proof of Lemma \ref{l:barrier} we record two immediate corollaries which are useful and of independent interest. 

\begin{corollary}
    \label{l:righttaillb}
    For each $L>0$ and $r\ge r_0(L)$, there exists $\delta=\delta(L)>0$ such that 
    $$\P(X_{r}\ge r+LQ_{r}) \ge \delta.$$
\end{corollary}

Observe that this complements \cite[Proposition 9.1]{BSS23} which showed the same result for the left tail. That result, however, did not require the FKG inequality.

\begin{corollary}
    \label{c:parabarrier}
    For any $L,L'>0$ there exists $\delta=\delta(L,L')>0$ such that for all $r$ sufficiently large 
    $$\P\left( \inf_{y\in [0,W_r],y'\in [L'W_r, (L'+1)W_r]}X_{(0,y),(r,y')}-\E X_{(0,y),(r,y')}\ge LQ_r \right)\ge \delta.$$
\end{corollary}
For the proof of Lemma \ref{l:barrier} we first record the following lemma which is an immediate consequence of Theorem \ref{t:all}. 

\begin{lemma}
    \label{l:righttailbasic} 
    There exists $\delta_1>0$ sufficiently small such that $\P(X_{r}\ge r+\delta_1 Q_{r})\ge \delta_1$ for all $r$ sufficiently large. 
\end{lemma}

The next step is to extend this estimate to thin rectangles. 

\begin{lemma}
    \label{l:righttailbasicrect}
    There exist $\delta_1,\delta_2>0$ such that for all $r$ sufficiently large we have 
    $$ \P\left(\inf_{y,y'\in [0,\delta_2W_r]} X_{(0,y),(r,y')} \ge r+\delta_1 Q_r \right)\ge \delta_1.$$
\end{lemma}

\begin{center}
\begin{figure}
\includegraphics[width=5in]{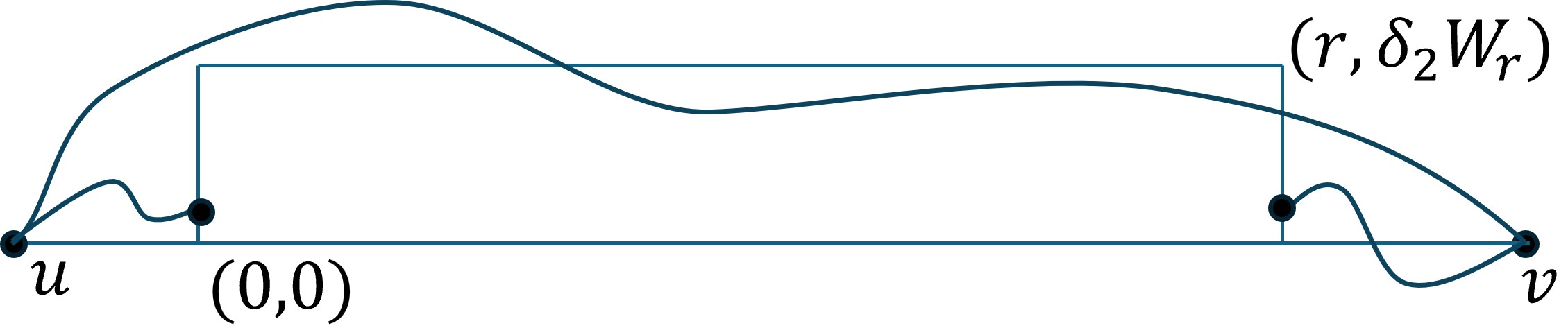}
\caption{Proof of Lemma \ref{l:righttailbasicrect} which shows that with a positive probability a thin on-scale rectangle is a barrier, that is no path across the rectangle is good. This is done by what is commonly referred to as a \emph{step back argument}. We combine the event that the point-to-point passage time from $u$ to $v$ the points obtained by stepping back a bit from the shorter boundaries of the rectangle  is large together with the event that the passage times from $u$ (resp.\ $v$) to the left (resp.\ right) side of the rectangle are not very small.}
\label{f:thin}
\end{figure}
\end{center}
\begin{proof}
    Let us define the points $u=(-\delta' r/2,0)$ and $v=((1+\delta'/2)r,0)$; see Figure \ref{f:thin}. Next, consider the following events
    $$A= \{X_{u,v}\ge (1+\delta')r+ \delta Q_{r}\}$$
    where $\delta$ is as in Lemma \ref{l:righttailbasic}. 
    $$B_1=\{\max_{y\in[0,\delta_2W_r]} X_{u,(0,y)}\le \delta' r/2+ \delta Q_{r}/3\};$$
    $$B_2=\{\max_{y\in[0,\delta_2W_r]} X_{(r,y),v}\le \delta'r/2+ \delta Q_{r}/3\}.$$
    It is clear that on $A \cap B_1 \cap B_2$ we have 
    $$\inf_{y,y'\in [0,\delta_2W_r]} X_{(0,y),(r,y')} \ge r+\delta Q_r/3.$$
    Now, by Lemma \ref{l:righttailbasic} it follows that for $\delta$ sufficiently small $\P(A)\ge \delta$. It also follows from Proposition~\ref{p:para} that first choosing $\delta'$ sufficiently small and then choosing  $\delta_2$ sufficiently small we get that $\P(B_1), \P(B_2)\ge 1-\delta/3$. This completes the proof by setting $\delta_1=\delta/3$. 
\end{proof}

\begin{lemma}
    \label{l:righttailbasicpara}
    There exist $\delta_1,\delta_2>0$ such that for all $r$ sufficiently large and all $k\le r/W_r$ we have 
    $$\P\left(\inf_{\substack{y\in [0,\delta_2W_r]\\ y'\in [k\delta_2 W_r, (k+1)\delta_2 W_r] }} X_{(0,y),(r,y')} \ge r +\delta_1 Q_{r}\right)\ge \delta_1.$$
\end{lemma}

\begin{proof}
    This follows from \cite[Corollary 3.2]{BSS23} and Lemma \ref{l:righttailbasicrect} by choosing $\delta_2$ sufficiently small.   
\end{proof}

We now prove Lemma \ref{l:barrier}. 

\begin{proof}[Proof of Lemma \ref{l:barrier}]
    Let $\delta_1, \delta_2>0$ be such that the conclusions of Lemmas \ref{l:righttailbasicrect} and \ref{l:righttailbasicpara} hold. Let $R\in\Z$ be such that $\delta_1R Q_{r/R}\ge LQ_{r}$. Such an $R$ exists independent of $r$ by Theorem \ref{t:all} (since $Q$ grows locally sublinearly). Let $L_*$ be a large fixed number depending on $L,L'$. We shall treat two cases separately: one for paths with transversal fluctuation more than $L_*W_r$ and the other for paths with smaller transversal fluctuations. 

    Let $A$ denote the event
    $$A=\left\{\inf_{y,y'\in [0,L'W_r]}\inf_{\substack{\zeta(0)=(0,y)\\ \zeta(1)=(r,y')\\ \zeta \nsubseteq \R\times [-L_*W_r,L_*W_r]}} X_{\zeta} \ge r+LQ_r\right\}. $$
    It follows from Proposition \ref{p:para} and Theorem \ref{t:tfold} that for $L_*$ sufficiently large depending on $L'$ we have 
    $\P(A)\ge 9/10$. For $i=1,2,\ldots, R$ and $j,j'\in [-\frac{L_*W_{r}}{\delta_2 W_{r/R}},\frac{L_*W_{r}}{\delta_2W_{r/R}}]$ let 
    $B_{i,j,j'}$ denote the event
    $$B_{i,j,j'}=\left\{ \inf_{\substack{y\in [jW_{r/R}, (j+\delta_2)W_{r/R}]\\ y'\in [j'W_{r/R}, (j'+\delta_2)W_{r/R}]}} X_{((i-1)r/R,y),(ir/R,y')} \ge \frac{r}{R}+\delta_1 Q_{r/R}\right\}.$$
    It follows from Lemma \ref{l:righttailbasicpara} that $\P(B_{i,j,j'})\ge \delta_1$. 
    
    Suppose that for some path $\zeta$ with $\zeta(0)=(0,y), \zeta(1)=(r,y')$ and $\zeta \nsubseteq \R\times [-L_*W_r,L_*W_r]$ we let $v_i=(ir/R,y_i)$ be its first intersection with the line $x=ir/R$.  Then on the event $\bigcap_{i,j,j'} B_{i,j,j'}$ we have that any 
    \[
    X_\zeta \geq \sum_{i=1}^R X_{v_{i-1},v_i} \geq r+R\delta_1 Q_{r/R}\geq r+L Q_r.
    \]
    Then on the event
    $$D=A\cap \bigcap_{i,j,j'} B_{i,j,j'}$$ we have that 
    $$\inf_{y,y'\in [0,L' W_r]} X_{(0,y),(r,y')} \ge r+L Q_r.$$
    Observe also that the number of triples $(i,j,j')$ as above is at most $R^3L_*^2/\delta_2^2$ (here we have again used that fact that $W$ grows locally sublinearly). Finally, observing that $A$ and $B_{i,j,j'}$ are all increasing events by the FKG inequality we get 
    $$\P(D)\ge \frac{9}{10}(\delta_1)^{R^3L_*^2/\delta_2^2}.$$
    Since $R$ and $L_*$ depend only on $L,L'$ this completes the proof of the lemma. 
\end{proof}

\subsection{Event $\cK$: Proof of Lemma \ref{l:cK.bound}}
Recall the definition of the event $\cK_{i,j,z}$. 
\begin{align*}
\cK_{i,j,z} &= \bigg\{ \inf_{\substack{x\in[(i-1)r,ir]| \\ |y|\leq n^{\beta}W_n}} \inf_{\substack{ \gamma(0)=((i-1)r,y)) \\ \gamma(1)=(x,jW_r)}} \rX_{\gamma} - (x-(i-1)r) - \frac12\Big(\frac{|y-jW_r| }{W_r} - 1\Big)^2 Q_r \geq  -z Q_r \bigg\}\\
&\cap \bigg\{ \inf_{\substack{x\in[(i-1)r,ir]| \\ |y|\leq n^{\beta}W_n}} \inf_{\substack{\gamma(0)=(x,jW_r) \\ \gamma(1)=(ir,y)) }} \rX_{\gamma} - (ir-x) - \frac12\Big(\frac{|y-jW_r| }{W_r} - 1\Big)^2 Q_r \geq  -z Q_r \bigg\}.
\end{align*}
Let us denote the first event above by $\cK^{(1)}_{i,j,z}$ and the second event by $\cK^{(2)}_{i,j,z}$. To prove Lemma \ref{l:cK.bound}, by reflection symmetry, it suffices to show that for all $i$ and $|j|\le M$ and for all $z>0$ we have
\begin{equation}
    \label{e:Kred}
    \P(\cK^{(1)}_{i,j,z})\ge 1-\exp(-Cz^{\theta_5})
\end{equation}

To reduce notational overhead we shall prove \eqref{e:Kred} for {$i=1$ and $j=0$}. It will be clear from the proof that the same argument can, with minimal changes, be used to prove the result for all required values of $i$ and $j$.

Let $\widetilde{\cK}_{z}$ denote the event 
\begin{align*}
\widetilde{\cK}_{z} = \bigg\{ \inf_{|y|\leq n^{\beta}W_n} \inf_{\substack{ \gamma(0)=(0,y)) \\ \gamma(1)=(r,0)}} \rX_{\gamma} - r - \frac12\Big(\frac{|y| }{W_r} - 1\Big)^2 Q_r \geq  - z Q_r \bigg\}.
\end{align*}

Further, let us set
$$\widehat{\cK}_z=\bigg\{\sup_{\substack{x,x'\in [0,r]}} \rX_{(x,0),(x',0)}-|x-x'|\le zQ_r \bigg\}.$$

We next claim that  
$$\widetilde{\cK}_{z/2}\cap \widehat{\cK}_{z/2} \subseteq \cK^{(1)}_{1,0,z}.$$
Indeed by triangle inequality we have for all $x\in [0,r]$ and for all $y$ with $|y|\le n^{\beta}W_n$
$$\rX_{(0,y), (x,0)}\ge \rX_{(0,y),(r,0)}-\rX_{(x,0),(r,0)}.$$
Since on the event $\widetilde{\cK}_{z/2}\cap \widehat{\cK}_{z/2}$ we have 
$$\rX_{(0,y),(r,0)}\ge r+  \frac12\Big(\frac{|y| }{W_r} - 1\Big)^2 Q_r -zQ_r/2$$
and 
$$\rX_{(x,0),(r,0)}\le (r-x)+zQ_r/2,$$ this implies
$$\rX_{(0,y), (x,0)}\ge x+ \frac12\Big(\frac{|y| }{W_r} - 1\Big)^2 Q_r-zQ_r$$
as required. 

Using the lower bounds on $\P(\widetilde{\cK}_{z/2})$ and $\P(\widehat{\cK}_{z/2})$ established in Lemma \ref{l:ktilde} and Lemma \ref{l:khat} below together with a union bound, the proof of 
\eqref{e:Kred} is complete. As explained above, Lemma \ref{l:cK.bound} follows by reflection symmetry and a further union bound. 
\qed 

\begin{lemma}
    \label{l:ktilde}
    There exist $C,\theta_5>0$ such that for all $z>0$ and for all $r$ large we have 
    we have
    $$\P(\widetilde{\cK}_{z})\ge 1-\exp(-Cz^{\theta_5}).$$
\end{lemma}

\begin{lemma}
    \label{l:khat}
    There exist $C,\theta_5>0$ such that for all $z>0$ and for all $r$ large we have
    $$\P(\widehat{\cK}_{z})\ge 1-\exp(-Cz^{\theta_5}).$$
\end{lemma}

The rest of this subsection is devoted to the proof of Lemmas \ref{l:ktilde} and \ref{l:khat}.

\begin{proof}[Proof of Lemma \ref{l:ktilde}]
    Recall the definition of $\widetilde{\cK}_{z}$. For {$j'$ with $|j'|\le n^{\beta}\frac{W_n}{W_r}$}, define 

    $$\cK^{\circ}_{j',z}=\left\{\inf_{y\in [j'W_{r} ,(j'+1)W_{r}]} \rX_{(0,y), (r,0)}-r-\frac12\Big(\frac{|y| }{W_r} - 1\Big)^2 Q_r \geq-zQ_r \right\}.$$

    Notice that for $y\in [j'W_{r} ,(j'+1)W_{r}]$
    $$\frac{|y|^2}{2r}=\frac{1}{2}\frac{|y|^2}{W^2_r}Q_r\ge \frac12\Big(\frac{|y| }{W_r} - 1\Big)^2 Q_r +\frac12\Big(\frac{|y|}{W_r} - 1\Big)Q_r\ge \frac12\Big(\frac{|y|}{W_r} - 1\Big)^2 Q_r +(|j'|-2)Q_r.$$
    Therefore, by Proposition \ref{p:paraestimateconforming} (and a Taylor expansion which shows $|(0,y)-(r,0)|\approx r+\frac{|y|^2}{2r}$) it follows that 
    $$\P(K^{\circ}_{j',z})\ge 1-\exp(-C(|j'|+z)^{\theta_5})$$
    for some $C,\theta_5>0$. The lemma now follows from a union bound over all values of $j'$. 
    \end{proof}

\begin{proof}[Proof of Lemma \ref{l:khat}]
    By Lemma~\ref{l:proxy}, {it suffices to prove} the same result for $X$ in place of $\rX$ ({indeed, for $z$ smaller than some power of $n$, one can apply Lemma \ref{l:proxy} and for larger values of $z$ the result simply follows from the fact that $\rX_{uv}|u-v|^{-1}$ is deterministically bounded above}) so let 
    \[
    \widehat{\cK}^*_{z}=\bigg\{\sup_{\substack{x,x'\in [0,r]}} X_{(x,0),(x',0)}-|x-x'|\le zQ_r \bigg\}.
    \]
    Since $\frac{X_{uv}}{|u-v|}$ is bounded uniformly away from $0$ and $\infty$ and since $Q_r>r^{\alpha}$ for some $\alpha>0$, it suffices to take union bound over points with integer coordinates. That is, it suffices to show that, on an event of probability at least $1-\exp(-Cz^{\theta_5})$ we have 
    $$\widehat{\cK}^{*\Z}_{z}=\left\{ \forall x,x'\in [0,r]\cap \Z,~~~ 
    X_{(x,0),(x',0)}-|x-x'|\le \frac{zQ_r}{2}\right\}.$$
    Let $\alpha>0$ be as in Theorem~\ref{t:all} so for all $r$ and $t\ge 1$, $Q_{rt}\ge t^{\alpha}Q_r$.  For $0\le k \le \lfloor\log_2 r\rfloor$, set
    $$H_k=\left\{ \forall  x \in 2^{k}\Z \cap [0,r],   X_{(x,0), (x+2^k,0)}- 2^k \le z \frac{(r/2^{k})^{-\alpha/2}}{4(1-2^{-\alpha/2})}Q_r\right\}.$$
    For a fixed $x$ as above it follows from Theorem~\ref{t:all},
    \begin{align}
        \P\bigg(|X_{(x,0), (x+2^k,0)}-2^k| \le   z \frac{(r/2^{k})^{-\alpha/2}}{4(1-2^{-\alpha/2})}Q_r\bigg)
        &\ge 1-\exp\bigg(-C\Big(z\big(\frac{r}{2^k}\big)^{\alpha/2}\frac{Q_r}{Q_{2^k}}\Big)^{\theta}\bigg)\\
        &\geq 1 -\exp\left(-Cz^\theta(r/2^{k})^{\theta\alpha/2}\right),
    \end{align}
    and hence
    \[
    \P[H_k] \geq 1 - \lceil r/2^k\rceil\exp\left(-Cz^\theta(r/2^{k})^{\theta\alpha/2}\right)
    \geq 1 -\exp\left(-Cz'^\theta(r/2^{k})^{\theta\alpha/2}\right).
    \]
    A union bound over $k$ gives 
    $$\P\Big(\bigcap_k H_k\Big)\ge 1-\exp(-Cz^{\theta})$$
    for some $C>0$. 
    Next we prove that on the event $\cap_k H_k$ that $\widehat{\cK}^{*\Z}_{z/2}$ holds. 
    To see this, fix $i_1<i_2\in [0,r]\cap \Z$. Consider the dyadic sequence $j_h$ between $i_1$ and $i_2$ given by Lemma \ref{l:dyadic.seq}. On the event $\cap_k H_k$ we have for each $h$, 
    $$X_{(j_{h-1},0), (j_h,0)}- (j_h-j_{h-1}) \le  z \frac{(r/2^{k_h})^{-\alpha/2}}{4(1-2^{-\alpha/2})}Q_r$$
    where $k_h = \log_2(j_h-j_{h-1})$. Since Lemma \ref{l:dyadic.seq} guarantees that each value of $k_h$ occurs at most twice in the sequence $j_{h}-j_{h-1}$ it follows by the triangle inequality and the definition of $H_k$ that on $\cap_k H_k$ we have 
    $$X_{(i_1,0), (i_2,0)}- |i_1-i_2| 
    \le  2z \frac{r^{-\alpha/2}}{4(1-2^{-\alpha/2})}Q_r \sum_{k=0}^{\log_2 (i_2-i_1)}2^{\alpha/2} 
    \le \frac{zQ_r}{2} \Big(\frac{2^{\log_2 (i_2-i_1)}}{r}\Big)^{\alpha/2}\leq \frac{zQ_r}{2}.$$
    where the last inequality follows by taking $H$ small enough. Hence
    \[
        \P[\widehat{\cK}^{*\Z}_{z/2}]\ge 1-\exp(-Cz^{\theta})
    \]
    and the proof is completed by using Lemma~\ref{l:proxy} to pass from $X$ to $\rX$.
\end{proof}

\begin{center}
\begin{figure}
\includegraphics[width=6in]{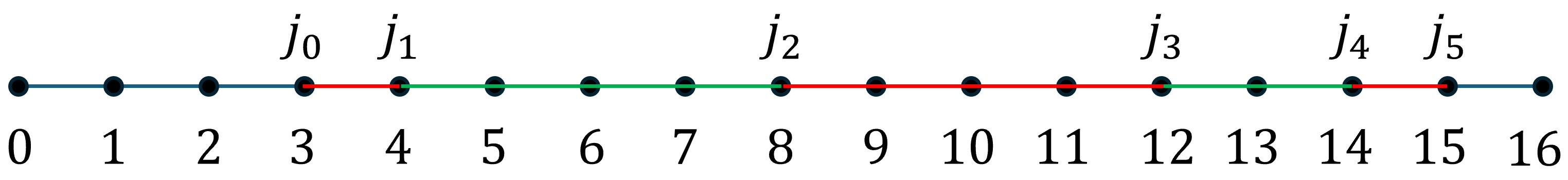}
\caption{Dyadic sequence decomposition given in Lemma \ref{l:dyadic.seq}. The decomposition is shown for the interval $[3,15]\cap \Z$ in the figure. The length of each interval is a power of 2, and both endpoint of the interval is also a multiple of the same power of $2$. There are at most two intervals of the same length.}
\label{f:dyadic}
\end{figure}
\end{center}

\begin{lemma}\label{l:dyadic.seq}
For any integers $i_1<i_2$ there exists a  sequence of integers $i_1=j_0<j_1<j_2<\cdots <j_{m}=i_2$    satisfying the following properties:
\begin{enumerate}
    \item For each $h$, $j_h-j_{h-1}=2^{k(h)}$ is a power of 2 and both $j_{h-1}$ and $j_h$ are both integer multiples of $2^{k(h)}$.
    \item For each $k$ there exists at most two different values of $h$ such that $k(h)=k$. 
    \item If $i_2-i_1\ge 2^{k+1}$ then there exists $h$ such that $j_h-j_{h-1}\ge 2^{k}$.
\end{enumerate}
\end{lemma}
\begin{proof}
The construction is simply that the pairs $(j_{h-1},j_h)$ are the set of dyadic intervals contained in $[i_1,i_2]$ that are not contained in a larger dyadic interval that is a subset; see Figure \ref{f:dyadic}.  These are the intervals,
\[
A=\Big\{[\ell 2^k,(\ell+1)2^k):[\ell 2^k,(\ell+1)2^k)\subset [i_1,i_2), [\lfloor \ell/2\rfloor 2^{k+1},(\lfloor \ell/2\rfloor+1)2^k)\not\subset[i_1,i_2) \Big\}.
\]
Each $i\in[i_1,i_2)$ must be in exactly one such interval because by definition exactly one of $[\lfloor i2^{-k}\rfloor 2^k,(\lfloor i2^{-k}\rfloor+1) 2^k)$ is in $A$.  Thus by ordering the intervals in $A$ in increasing order and writing them as 
\[
A=\{[j_0,j_1),[j_1,j_2),\ldots,[j_{m-1},j_m)\}
\]
we have constructed a sequence satisfying property (1).  

To check that this satisfies property (2) we note that of all the intervals $[\ell 2^k, (\ell+1) 2^k)$ that are subsets of $[i_1,i_2)$, only the first and last can be in $A$ as any in the middle would be part of a length $2^{k+1}$ subset and thus there are at most two $h$ with $k(h)=k$.

Finally, if $i_2-i_1\ge 2^{k+1}$ the interval $[\lceil i_1 2^{-k}\rceil2^k,( \lceil i_1 2^{-k}\rceil + 1)2^k)\subset[i_1,i_2)$ and so there is at least one interval of size at least $2^k$ in $A$ which implies property (3) is satisfied.
\end{proof}

\subsection{Event $\cJ$: Proof of Lemma \ref{l:cJ.bound}}\label{s:cJ.bound}
Recall the definition of the event $\cJ$ (see Figure \ref{f:Jevent1}):
\begin{align*}
\cJ_{i,j,j',z,s} &= \bigg\{\inf_{\substack{x,x' \in [(i-1)r,ir]\\ y,y'\in[jW_r,j' W_r] \\ |y-y'|\geq s W_r}} \inf_{\substack{ \zeta \subset  [(i-1)r,ir]\times [jW_r,j' W_r] \\ \zeta(0)=(x,y) \\ \zeta(1)=(x',y')}} \rX_\zeta - |x-x'|  \geq  z Q_r \bigg\}.
\end{align*}

\begin{center}
\begin{figure}[htbp!]
\includegraphics[width=4in]{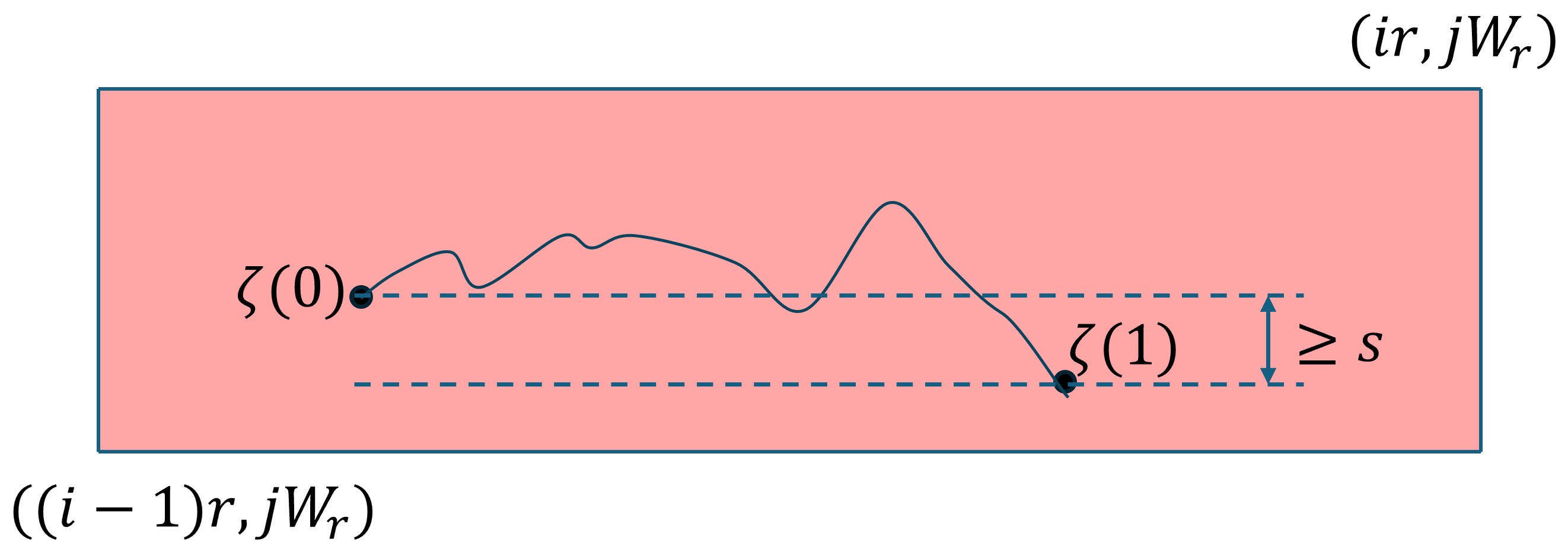}
\caption{The event $\cJ$ asks that all paths contained in an on-scale rectangle whose endpoints differ by a positive on-scale constant in the vertical co-ordinate will have large lengths. We show that this event occurs with probability bounded away from $0$.}
\label{f:Jevent1}
\end{figure}
\end{center}

It suffices to assume that $z\geq 1$. To reduce the burden of notation we shall prove this Lemma in the case $i=0$ and $j=0$. {Notice that since the paths we consider for this lemma are contained in $[(i-1)r,ir]\times [jW_r,j'W_r]$, the events are actually translation invariant in $i$ and $j$ and hence considering $i=1,j=0$ suffices. Also since the event is stronger for larger values of $j'-j$, it suffices to prove it for $j'=t$}. In particular, we shall show that there exists $\delta=\delta(s,z,t)>0$ such that 
\begin{equation}
    \label{e:Jbound1}
    \P(\cJ_{1,0,t,z,s})\ge \delta. 
\end{equation}

The proof of \eqref{e:Jbound1}, like the proof of Lemma \ref{l:cK.bound},  will subdivide passage times into segments on dyadic scales. Clearly it suffices to prove the result for $z$ sufficiently large, so this is what we shall henceforth assume. For non-negative integers $k$ and $i'\in\{0,1,2,\ldots, 2^{k}\}$, we divide the lines $x=i'2^{-k}r$ into intervals of length $W_{2^{-k}r}$ of the form $\ell_{i'2^{-k}r, hW_{2^{-k}r}, (h+1)W_{2^{-k}r}}$ for $0\leq h\le \lceil tW_r/W_{2^{-k}r}\rceil$. Now let $\cP_{k}$ denote the set of all parallelograms $P=P_{i',h,h'}$ whose left side $L_{P}$ is of the form $\ell_{(i'-1)2^{-k}r, hW_{2^{-k}r}, (h+1)W_{2^{-k}r}}$ and whose right side is of the form $\ell_{i'2^{-k}r, h'W_{2^{-k}r}, (h'+1)W_{2^{-k}r}}$ for some $h,h'$ as above and for some $i'\in \{1,2,\ldots, 2^{k}\}$. 

For the rest of Subsection~\ref{s:cJ.bound}, let $\epsilon$ take its value from Lemma~\ref{l:proxy}.
We fix some small $\varepsilon_*>0$ and set $k_*=\lceil \varepsilon_* \log_2 r \rceil$. Set
\begin{equation}\label{e:k1.defn}
    k_1=\Big\lceil\log_2\Big(\frac{240z}{s^2}\Big)\Big\rceil \vee \frac2{\alpha}\Big(1+\log_2(4(1-2^{-\alpha/2}))\Big)
\end{equation}
and for $0\leq k \leq k_1$,  define the event 
$$\widetilde{\cJ}_{k}= \bigcap_{P\in \cP_k} \bigg\{\inf_{\substack{u\in L_P\\ v\in R_P}} \rX_{uv}-|u-v| \ge 10zQ_r -n^{-\epsilon} Q_n\bigg\}.$$
For $k=k_1+1,\ldots k_*$ let us define the event 
$$\widetilde{\cJ}_{k}= \bigcap_{P\in \cP_k} \bigg\{\inf_{\substack{u\in L_P\\ v\in R_P}} \rX_{uv}-|u-v| \ge -2^{-\alpha_* k/2}Q_{r} -n^{-\epsilon} Q_n \bigg\}$$
where $\alpha_*>0$ is as in Theorem~\ref{t:all}. We then have the following lemma.

\begin{lemma}
    \label{l:jbound1}
    On the event 
    $\bigcap_{k=0}^{k_*} \widetilde{\cJ}_{k}$ we have
    \begin{equation}\label{eq:jbound1A}
        \inf_{0\le i_1<i_2\le 2^{k_*}}\inf_{\substack{y,y'\in [0,tW_r]\\ |y-y'|\ge sW_r/2}} \inf_{\substack{\zeta \subset [0,r]\times [0,tW_r]\\ \zeta(0)=(i_12^{-k_*}r,y)\\ \zeta(1)=(i_22^{-k_*}r,y')}} \rX_{\zeta}-(i_2-i_1)2^{-k_*}r \ge 5zQ_r
    \end{equation}
    and 
    \begin{equation}\label{eq:jbound1B}
    \inf_{0\le i_1<i_2\le 2^{k_*}}\inf_{y,y'\in [0,tW_r]} \inf_{\substack{\zeta \subset [0,r]\times [0,tW_r]\\ \zeta(0)=(i_12^{-k_*}r,y)\\ \zeta(1)=(i_22^{-k_*}r,y')}} \rX_{\zeta}-(i_2-i_1)2^{-k_*}r \ge -5zQ_r.
    \end{equation}
\end{lemma}

\begin{center}
\begin{figure}
\includegraphics[width=4in]{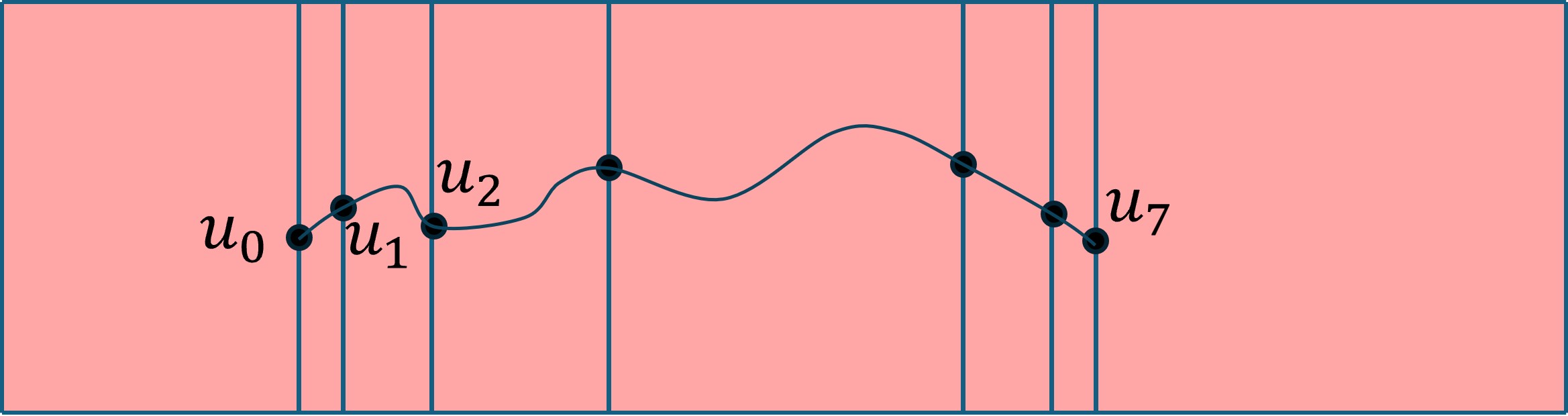}
\caption{Proof of Lemma \ref{l:jbound1}. For any fixed $i_1<i_2$, and a path $\zeta$ from $(i_12^{-k_*}r,y)$ to $(i_22^{-k_*}r,y')$ ($u_0$ and $u_7$ in the figure respectively) we decompose the path according to the dyadic scales given by Lemma \ref{l:dyadic.seq}. On the events $\cJ_k$ the segments of the path across a horizontal distance of $2^k$ are lower bounded, combining these we get the desired lower bound.}
\label{f:Jproof}
\end{figure}
\end{center}

\begin{proof} Fix $i_1,i_2$ and pick a sequence 
    \[
        i_1=j_0<j_1<j_2<\cdots <j_{m}=i_2
    \]
    satisfying the properties of Lemma~\ref{l:dyadic.seq}. For $y,y'\in [0,tW_r]$ and $\zeta$ from $(i_12^{-k_*}r,y)$ to $(i_22^{-k_*}r,y')$. For the sequence $j_h$ as above, let $u_{h}$ denote points where $\zeta$ intersects the lines $x=j_h2^{-k_*}r$; see Figure \ref{f:Jproof}. By definition of $j_h$, there exists an integer $k(h)$ with $j_h-j_{h-1}=2^{k_*-k(h)}$.  Then there exists $P\in \cP_{k(h)}$ such that $u_{h-1}\in L_P$ and $u_{h}\in R_{P}$.  Our assumption $\bigcap_{k=0}^{k_*} \widetilde{\cJ}_{k}$ implies that
    \begin{equation}\label{eq:tildeJ.implication}
    \rX_{u_{h-1}u_{h}}\ge |u_{h-1}-u_h|+10zQ_r I(k(h)\le k_1) -   2^{-\alpha_* k(h)/2}Q_rI(k(h)>k_1)-n^{-\epsilon}Q_n
    \end{equation}
    and hence
    \begin{align}\label{eq:jboundEq1}
    \rX_{\zeta}&\ge \sum_{h}\rX_{u_{h-1}u_{h}}\nonumber\\
    &\geq\sum_h \left[|u_{h-1}-u_h|+10zQ_r I(k(h)\le k_1) -   2^{-\alpha_* k(h)/2}Q_rI(k(h)> k_1)\nonumber-n^{-\epsilon} Q_n\right]\\
    &\geq |\zeta(0)-\zeta(1)| + 10zQ_rI(\min_h k(h) \leq k_1)-2n^{-\epsilon} Q_n\varepsilon_*\log_2 r -2\sum_{k> k_1} 2^{-\alpha_* k/2}Q_r\nonumber\\
    &\geq |\zeta(0)-\zeta(1)| + 10zQ_rI(\min_h k(h) \leq k_1) -2n^{-\epsilon} Q_n\varepsilon_*\log_2 r - 4z Q_r\nonumber\\
    &\geq |\zeta(0)-\zeta(1)| + 10zQ_rI(\min_h k(h) \leq k_1) - 5z Q_r,
    \end{align}
    where the second inequality is by equation~\eqref{eq:tildeJ.implication}, the third is by the triangle inequality and the fact that each $k(h)$ occurs at most twice and the third is by~\eqref{e:k1.defn}, and the final one follows by taking $n$ sufficiently large. Since $|\zeta(0)-\zeta(1)| \geq (i_2-i_1)2^{-k_*}r$ this implies~\eqref{eq:jbound1B}.  If $\min_h k(h) \leq k_1$ then we also have~\eqref{eq:jbound1A}.
    
    So to complete the lemma we need to prove~\eqref{eq:jbound1B} in the case $\min_h k(h) > k_1$ and $|y-y'|\ge sW_r/2$.  By the third property of the $j_h$ we must have that
    \begin{equation}\label{eq:k1.condition}
    (i_2-i_1)2^{-k_*}r \le 2^{1-k_1}r
    \end{equation}
    and so  
    \begin{align*}
    |\zeta(0)-\zeta(1)| &= \sqrt{((i_2-i_1)2^{-k_*}r)^2+(y-y')^2}\\
    &\ge (i_2-i_1)2^{-k_*}r+\frac{(y-y')^2}{3(i_2-i_1)2^{-k_*}r}\\
    &\ge (i_2-i_1)2^{-k_*}r+\frac{s^2}{24\cdot 2^{-k_1}}Q_r\\
    &\ge (i_2-i_1)2^{-k_*}r+10z Q_r,
    \end{align*}
    where the first inequality is by the fact that $\sqrt{x^2+y^2} \geq x + \frac{y^2}{3x}$ if $y\leq \frac12x$, the second is by equation~\eqref{eq:k1.condition} and the last is by our choice of $k_1$ which implies $\frac{s^2}{24} \geq 10z2^{-k_1}$.  Combined with equation~\eqref{eq:jboundEq1}, this competes the proof of~\eqref{eq:jbound1A}.
    \end{proof}

The next event will deal with points that are close by. Let $\cJ^{*}$ denote the event 

$${\cJ^{*}=\left\{ \inf_{\substack{x,x'\in [0,r], y,y\in [0,tW_r]\\ |(x,y)-(x',y')|\le r^{1-\varepsilon_*/2}}} \rX_{(x,y,(x',y'))}-|(x,y)-(x',y')|\ge -Q_r\right\}.}$$

\begin{lemma}
    \label{l:jbound2}
    For all $n$ sufficiently large.
    $$\cJ^{*}\cap \bigcap_{k=0}^{k_*} \widetilde{\cJ}_{k}   \subseteq \cJ_{1,0,t,z,s}.$$
       
\end{lemma}

\begin{proof}
Fix $0\le x\le x'\le r$ and $y,y'\in [0,tW_r]$ with $|y-y'|\ge sW_r$ and a conforming $\zeta$ from $(x,y)$ to $(x',y')$ contained in $[0,r]\times [0,tW_r]$. We need to show $\rX_{\zeta}\ge (x-x')+zQ_r$. We split the case into several.

\noindent
\textbf{Case 1.} Suppose $x'-x\le 5\cdot2^{-k_*}r$. Then by definition of $k_*$ and the fact that $tW_{r}\ll r^{1-\varepsilon/2}$  we get that 
$|(x,y)-(x',y')| \le r^{1-\varepsilon/2}$, therefore on the event $\cJ^*$ we get 
$${\rX_{\zeta}}\ge |(x,y)-(x',y')| -Q_r$$
It therefore suffices to prove that 
$$ |(x,y)-(x',y')|\ge x'-x +(z+1)Q_r$$ in this case. 
Observe now that 
$$\sqrt{a^2+b^2} \ge a+ \frac{b^2}{3a}\wedge \frac{|b|}{3}$$
therefore using $|y-y'|\ge sW_r$ it suffices to show that 
$$\frac{s^2 W_r^2}{3r^{1-\varepsilon/2}}\wedge \frac{sW_r}{3} \ge (z+1)Q_r $$
Clearly, $\frac{s^2 W_r^2}{3r^{1-\varepsilon/2}}\ge r^{\varepsilon/2}Q_r/3 \ge (z+1)Q_r$ for all $r$ sufficiently large. Also, since $r^{\alpha_*}<Q_r \le C\sqrt{r}$ ({by Theorem \ref{t:all} and Proposition \ref{p:uij.bounds}}) it follows that $\frac{sW_r}{3}\ge \frac{s\sqrt{rQ_r}}{3}\ge (z+1)Q_r$ for all $r$ sufficiently large. This concludes the proof in Case 1. 

\noindent
\textbf{Case 2.} Suppose $x'-x\ge 5\cdot 2^{-k_*}r$. Then there exists integers $i_1<i_2$ such that 
$x<i_12^{-k_*}r<i_22^{-k_*}r<x_2$ with $(i_12^{-k_*}r-x), (x'-i_22^{-k_*}r) \in [2^{-k_*}r, 2^{1-k_*}r]$. Let $u=(i_12^{-k_*}r,y_1)$ and $v=(i_22^{-k_*}r,y_2)$ be points on $\zeta$ and let $\zeta_1$ denote the restriction of $\zeta$ between these two points. Therefore we have
$$\rX_{\zeta}\ge \rX_{(x,y),u}+\rX_{\zeta_1}+\rX_{v,(x',y')}.$$
Notice also the by definition the event $\cJ^*$ covers the pairs of points $((x,y),u)$ and $(v,(x',y'))$.  
Now we need to consider two subcases. 

\noindent
\textbf{Case 2a.} $|y_1-y_2|\ge sW_r/2$. In this case, on the event $\bigcap_{k=0}^{k_*} \widetilde{\cJ}_{k}$, we have by Lemma \ref{l:jbound1} that 
$$X_{\zeta_1}\ge (i_2-i_1)2^{-k_*}r+5zQ_r.$$
Using this together with the fact that on $\cJ^*$ we have 
$$ \rX_{(x,y),u}\ge (i_12^{-k_*}r-x) -Q_r; \quad \rX_{v,(x',y')} \ge (x'-i_22^{-k_*}r)-Q_r$$
we get 
$$\rX_{\zeta}- (x'-x)+(5z-2)Q_r \ge zQ_r$$
as required. 

\noindent
\textbf{Case 2b.} $|y_1-y_2|\leq sW_r/2$. In this case, we have either 
$|y-y_1|\ge sW_r/4$ or $|y'-y_2|\ge sW_r/4$. We shall only deal with the first case, the proof for the second one is identical. Observe first that we have on $\bigcap_{k=0}^{k_*} \widetilde{\cJ}_{k}$, by Lemma \ref{l:jbound1} 
$$\rX_{\zeta_1}\ge (i_2-i_1)2^{-k_*}r-5zQ_r.$$
Therefore on the event $\bigcap_{k=0}^{k_*} \widetilde{\cJ}_{k}\cap \cJ^*$ we have 
$$\rX_{\zeta}\ge |u-(x,y)|+(x'-i_12^{-k_*}r)-(5z+2)Q_r$$
Therefore it suffices to show that 
$$|u-(x,y)|\ge (i_12^{-k_*}r)+(6z+2)Q_r.$$
This follows for all sufficiently large $r$ from the same argument as in the proof of Case 1 using $|y_1-y|\ge sW_r/4$. We omit the details. 

Combining all these cases completes the proof of the lemma. 
\end{proof}

So it remains to show that 
$$\P\left( \bigcap_{k=0}^{k_*} \widetilde{\cJ}_{k} \cap \cJ^*\right)\ge \delta(t,z,s)> 0.$$ This is achieved in the next three lemmas. 

\begin{lemma}
    \label{l:jstar}
    For any $t$ and {$\varepsilon_*>0$} fixed, for all $r$ sufficiently large we have 
    $$\P(\cJ^*)\ge \frac{1}{2}.$$
\end{lemma}
\begin{proof}
    Let $S$ denote
    \[
    S=\big\{(u,v)\in (\Z^2 \cap [0,r]\times[0,tW_r])^2:|u-v|\le 2r^{1-\varepsilon_*/2}\big\}
    \]
    Since the underlying noise field is bounded clearly we have 
    $$\cJ^* \supseteq \cap_{(u,v)\in S} \{\rX_{uv} \ge |u-v|-Q_r/2 \}.$$ For all $(u,v)\in S$ we have that $Q_{r}\ge (\frac12 r^{\varepsilon_{*}/2})^{\alpha_*} Q_{|u-v|}$. Therefore, by Theorem~\ref{t:all} and Lemma~\ref{l:proxy} there exists $\delta'>0$ such that for all  $(u,v)\in S$ we have $\P(\rX_{uv}\ge |u-v|-Q_r/2)\le \exp(-r^{\delta'})$. Noting that $|S|\le t^2r^{4}$ and taking a union bound for $r$ sufficiently large completes the proof of the lemma.  
\end{proof}
 
\begin{lemma}
    \label{l:jksmall}
    For any $t,s,z>0$ fixed, and for $0\le k \le k_1$, there exists $\delta_*(k)>0$ such that we have 
    $$\P(\widetilde{\cJ}_k)\ge \delta_*(k).$$
\end{lemma}

\begin{proof}
    We have by Theorem~\ref{t:all} $Q_r/Q_{2^{-k}r}\le 2^{3k/4}$ for all $k$ sufficiently large, and therefore $10zQ_r\le (10z2^{3k/4})Q_{2^{-k}r}$. Recall the definition of $\widetilde{\cJ}_{k}$. For each $P\in \cP_k$, denote by $A_P$ the event 
    $$\bigg\{\inf_{\substack{u\in L_P\\ v\in R_P}} X_{uv}-|u-v| \ge 10zQ_r \bigg\}$$
    {By Lemma~\ref{l:barrier}} there exists $\delta_1=\delta_1(k,z)>0$ such that for all $P\in \cP_k$ we have 
    $\P(A_P)\ge \delta_1$. Notice also that by Lemma \ref{l:proxy} we have 
    $$\widetilde{\cJ}_k \supseteq \cap_{P} A_{P}.$$
    Observing that $A_P$ is an increasing event for each $P$ and $|\cP_k|\le t^{2}2^{3k}$ (here we used the fact that $W_{r}\le 2^{k}W_{2^{-k}r}$ for all $r$ sufficiently large) we get by the FKG inequality that 
    $$\P(\widetilde{\cJ}_{k}) \ge \delta_1^{t^22^{3k}}.$$
    This completes the proof. 
\end{proof}

\begin{lemma}
\label{l:jklarge}
    For any $t,s,z>0$ fixed, there exists $c,\theta'>0$ such that for $k_1\le k \le k_*$ 
    $$\P(\cJ_{k}) \ge 1- t2^{3k}\exp(-c2^{k\alpha_*\theta_2}).$$
\end{lemma}

\begin{proof}
    For $P\in \cP_k$, let $A_P$ denote the same event 
    \[
     \bigg\{\inf_{\substack{u\in L_P\\ v\in R_P}} X_{uv}-|u-v| \ge -2^{-\alpha_* k/2}Q_{r} \bigg\}.
    \]
    As in the previous lemma, by Lemma \ref{l:proxy}
    $$\widetilde{\cJ}_k \supseteq \cap_{P} A_{P}.$$
    Recall the choice of the constant $\alpha_*$ in the definition of $\widetilde{\cJ}_{k}$ for $k_1\le k \le k_*$ we have that 
    $2^{-\alpha_* k/2}Q_r \ge 2^{\alpha_* k/2}Q_{2^{-k}r}$. Therefore it follows from Proposition~\ref{p:paraestimateconforming} that 
    $$\P(A_P) \ge 1-\exp(-c2^{k\alpha_*\theta_2}).$$
    Taking a union bound over all $P\in \cP_k$ ($|\cP_k|\le  t^{2}2^{3k}$ as in the proof of the previous lemma) completes the proof. 
\end{proof}

We can now complete the proof of Lemma \ref{l:cJ.bound}. 

\begin{proof}[Proof of Lemma \ref{l:cJ.bound}]
    Choose $\varepsilon$ sufficiently small and  $k_1$ sufficiently large depending on $z,t,s$ such that the conclusions of Lemmas \ref{l:jbound1}, \ref{l:jbound2} hold for all $r$ sufficiently large and further we have by a union bound as Lemma \ref{l:jklarge} that 
    $$\P\left(\bigcap_{k_1<k\le k_*}\widetilde{\cJ}_{k}\right) \ge \frac{1}{2}.$$

    Since all the events $\widetilde{\cJ}_{k}$ and $\cJ^*$ are increasing, we get using Lemma \ref{l:jbound2} and the FKG inequality together with Lemmas \ref{l:jstar} and \ref{l:jksmall} that 
    $$\P(\cJ_{1,0,t,z,s}) \ge \prod_{k=0}^{k_1}\P(\widetilde{\cJ_{k}})\P\left(\bigcap_{k_1<k\le k_*}\widetilde{\cJ}_{k}\right)\P(\cJ^{*})\ge \frac{1}{4}\prod_{i=0}^{k_1}\delta_*(k).$$
    Since $k_1$ is fixed $\min_{k\le k_1} \delta_*(k)>0$, and we get the desired lower bound on $\P(\cJ_{1,0,t,z,s})$. The same lower bound on
    $\P(\cJ_{i,j,j',z,s})$ for any $i$ and $j'\le j+t$  can be obtained by the same argument with minimal changes. This completes the proof of the lemma.
\end{proof}

\subsection{Event $\cZ$ and Wing events}
Recall the definition of $\cZ_{i,j,k}$.
    $${\cZ_{i,j,k}}= \left\{\max_{\substack{|y|,|y'| \leq MW_n \\ u=(ir,y) \\ v= ((i+2^k L_2)r,y')}} |\hrX_{uv}| - (|\frac{y}{W_r}-j|^{\tfrac1{100}} +|\frac{y'}{W_r}-j|^{\tfrac1{100}})\frac{Q_r}{2^k L_2^2} \leq (2^k L_2)^{3/5} Q_{r}\right\}.$$

\begin{proof}[Proof of Lemma \ref{l:zbound}]
    
    We will first split the choices of $y,y'$ into intervals of length $W_r$.  For integers $h_1,h_2$ with {$|h_1|, |h_2|\le \frac{MW_n}{W_r}$}, write 
    $${\cZ_{i,j,k}}=\bigcup_{h_1,h_2} \cZ_{i,j,k,h_1,h_2}$$ where 

$$\cZ_{i,j,k,h_1,h_2}= \max_{\substack{|y-h_1W_r|\le W_r/2\\ |y'-h_2W_r|\le W_r/2  \\ u=(ir,y) \\ v= ((i+2^k L_2)r,y')}} |\hrX_{uv}| - (|\frac{y}{W_r}-j|^{\tfrac1{100}} +|\frac{y'}{W_r}-j|^{\tfrac1{100}})\frac{Q_r}{2^k L_2^2} \leq (2^k L_2)^{3/5} Q_{r}.$$

Therefore, our task reduces to getting lower bounds for $\cZ_{i,j,k,h_1,h_2}$. We shall do this in two parts. First let $h_1,h_2$ be such that $|h_1|^{1/100}+|h_2|^{1/100}\le 2^{k}L_2^2 (2^{k}L_2)^{3/5}$. Let us denote the set of all such pairs of $(h_1,h_2)$ by $\mathbf{H}$. For  $(h_1,h_2)\in \mathbf{H}$ we have 

$$\P(\cZ_{i,j,k,h_1,h_2})\geq \P\left( \max_{\substack{|y-h_1W_r|\le W_r/2\\ |y'-h_2W_r|\le W_r/2}}\max_{\substack{ u=(ir,y) \\ v= ((i+2^k L_2)r,y')}} |\hrX_{uv}| \leq (2^{k}L_2)^{3/5}Q_{r}\right).$$
Recall that we know from Proposition \ref{p:uij.bounds} that $Q_{2^{k}L_2r}\le C(2^{k}L_2)^{1/2}Q_r$. It follows from this and Proposition~\ref{p:paraestimateconforming} that for all $(h_1,h_2)\in \mathbf{H}$ we get 

$$\P(\cZ_{i,j,k,h_1,h_2})\ge 1-\exp(-c(2^{k}L_2)^{\theta'})$$
for some $c,\theta'>0$. Since the number of pairs of such $h_1,h_2$ is at most $$2^{2k+2}L^4_2 (2^{k}L_2)^{6/5} \le 4L_2^2 (2^{k}L_2)^{16/5}$$ it follows by taking a union bound over $(h_1,h_2)\in \mathbf{H}$ and using the fact that $k\ge 1$ that 

\begin{equation}
    \label{eq: zbound1}
    \P\left(\bigcup_{(h_1,h_2)\in \mathbf{H}} \cZ_{i,j,k,h_1,h_2}\right) \ge 1-  4L_2^2 (2^{k}L_2)^{16/5}\exp(-c(2^{k}L_2)^{\theta'}) \ge 1-\exp(-c'(2^{k}L_2)^{\theta'})
\end{equation}

for some $c',\theta'>0$. 

To deal with the other case, let $\mathbf{H}_{s}$ denote the set of all $h_1,h_2$ such that 

$$\min_{\substack{|y-h_1W_r|\le W_r/2\\ |y'-h_2W_r|\le W_r/2  \\ u=(ir,y) \\ v= ((i+2^k L_2)r,y')}} (|\frac{y}{W_r}-j|^{\tfrac1{100}} +|\frac{y'}{W_r}-j|^{\tfrac1{100}}) \in (s,s+1] 2^{k}L^2_2 (2^{k}L_2)^{3/5}.$$

It follows that for any $(h_1,h_2)\in \mathbf{H}_{s}$, 
$$\P(\cZ_{i,j,k,h_1,h_2})\ge \P\left( \max_{\substack{|y-h_1W_r|\le W_r/2\\ |y'-h_2W_r|\le W_r/2}}\max_{\substack{ u=(ir,y) \\ v= ((i+2^k L_2)r,y')}} |\hrX_{uv}| \le s(2^{k}L_2)^{3/5}Q_{r}\right).$$
Arguing as before, we get, for $s\ge 1$, and for all $(h_1,h_2)\in \mathbf{H}_{s}$  
$$ \P(\cZ_{i,j,k,h_1,h_2})\ge 1-\exp(-c(s2^{k}L_2)^{\theta'})$$
for some $\theta',c>0$. 

Now, the cardinality of $H_s$ is upper bounded by $$4(s+2)^22^{2k}L_2^{4}(2^{k}L_2)^{6/5} \le 4(s+2)^2 (2^{k}L_2)^{26/5}$$
and therefore by taking a union bound over all $(h_1,h_2)\in H_s$ that for all $s\ge 1$ 

$$\P\left(\bigcup_{{(h_1,h_2)\in \mathbf{H}_s}} \cZ_{i,j,k,h_1,h_2}\right) \ge 1-  4(s+2)^2 (2^{k}L_2)^{26/5} \ge 1-\exp(-c'(s2^{k}L_2)^{\theta'})$$
for some $c',\theta'>0$. 

The lemma follows now by taking a union bound over all $s$ and a further union bound using \eqref{eq: zbound1}.    
\end{proof}

Finally we prove the estimate for $\cW^{glo}$ in Lemma \ref{l:cW.global}.

\begin{proof}[Proof of Lemma \ref{l:cW.global}]
    Recall that 
$$\cW_{i,j}^{glo}= \cW^{*}_{i} \cap \bigcap_{k=(\log_2 \log_2 M)^2}^{\lfloor\log_2 (\frac12M^{99/100})\rfloor} \cZ_{i-2^{k+1}-3,j,k}\cap \cZ_{i-2^{k}-3,j,k} \cap \cZ_{i+2,j,k}\cap \cZ_{i+2^k +2 ,j,k}$$
and therefore the following lemma together with Lemma~\ref{l:zbound} immediately implies Lemma \ref{l:cW.global} via a union bound. 
\end{proof}

\begin{lemma}
    \label{l:wstar}
    For $M$ sufficiently large and $n$ sufficiently large depending on $M$, we have 
    $$\P(\cW^*_{i})\ge 1-M^{-1000}.$$
\end{lemma}

\begin{proof}
    Recall that $\cW^*_{i}$ is an intersection of four events each of which is a further intersection of sub-events in the column $[i'r,i''r]$ where $i',i''$ varies over certain indices. In each of these cases, the number of pairs $(i',i'')$ is at most $M^2$, and therefore it suffices to show that each sub-event has probability at least $1-M^{-2000}$. Recall also that the four events are divided into two types depending on the vertical coordinates of the points $u',u''$ such that $\rX_{u'u''}$ are considered: one where there coordinates lie within $[-MW_n,MW_n]$ and the second where these coordinates take values in $[-n^{\beta}W_n,n^{\beta}W_n]$. We shall provide a proof for one event of each type, the proofs for the other events are identical and will be omitted. 

    For $0\le i' < i'' \le M\Phi^{-\ell}$, let $A_{i',i''}$ denote the event 
    $$A_{i',i''}=\left\{\max_{\substack{|y|,|y'| \leq MW_n \\ u'=(i'r,y)\\u''=(i''r,y') }} |\hrX_{u',u''}|  \leq  \log^{\frac{100}{\theta_2}} (M) Q_{(i''-i')r}\right\}.$$ We shall show that $\P(A_{i',i''})\ge 1-M^{-2000}$. For $k,k'\in [-M,M]$, let $B_{k,k'}$ denote the event 
    $$B_{k,k'}=\left\{\max_{\substack{|y-kW_n|\le W_n,|y'-k'W_n|\le W_n \\ u'=(i'r,y)\\u''=(i''r,y') }} |\hrX_{u',u''}|  \leq  \log^{\frac{100}{\theta_2}} (M) Q_{(i''-i')r}\right\}.$$
    It follows from Proposition \ref{p:paraestimateconforming} that for $M$ sufficiently large
    $$\P(B_{k,k'})\ge 1-M^{-10000}.$$ By taking a union bound over all $k,k'$ we get $\P(A_{i',i''})\ge 1-M^{-2000}$, as required. 

    Next, for {$0\le i'< i'' \le M\Phi^{-\ell}$}, let $C_{i',i''}$ denote the event
    $$C_{i',i''}=\left\{\max_{\substack{|y|,|y'| \leq n^\beta W_n \\ u'=(i'r,y)\\u''=(i''r,y') }}\hrX_{u',u''}  \geq - (1\vee M^{-2}W_n^{-1}(|y|+|y'|))\log^{\frac{100}{\theta_2}} (M) Q_{(i''-i')r}\right\}.$$
    We need to show that $\P(C_{i',i''})\ge 1-M^{-2000}$. For {$k,k'\in [-n^{\beta},n^{\beta}]$}, let $D_{k,k'}$ denote the event
    $$D_{k,k'}=\left\{\max_{\substack{|y-kW_n|\le W_n,|y'-k'W_n|\le W_n \\
    |y|, |y'|\le n^{\beta}W_n\\ u'=(i'r,y)\\u''=(i'' r,y') }} \hrX_{u',u''}  \geq - (1\vee M^{-2}W_n^{-1}(|y|+|y'|))\log^{\frac{100}{\theta_2}} (M) Q_{(i''-i')r}\right\}.$$
    Notice that for $|k|,|k'|\le M^2$ Proposition \ref{p:paraestimateconforming} implies that $\P(D_{k,k'})\ge 1-M^{-10000}$ for $M$ sufficiently large. While if $|k|\vee |k'|\ge M^2$ we get by Proposition \ref{p:paraestimateconforming} that 
    $$\P(D_{k,k'})\ge 1-M^{-10000(|k_1|+|k_2|)}$$
    for $M$ sufficiently large. By taking a union bound over $k,k'$, the result follows. 
\end{proof}

\subsection{Event $\cM$: proof of Lemma \ref{l:gadget}}

Conforming passage times are not in general translation invariant because of the restriction that conforming paths intersect the boundaries between columns at $|y|\leq n^\beta W_n$.  However, because the event $\cM_{i,j}$   concerns paths constrained to lie in a rectangle, when $|jW_r|\leq MW_n$ it is in fact invariant under $i$ and $j$ so without loss of generality, we shall prove the lemma for $i=1,j=0$. Define
\[
U_{r}(h) := \inf_{y,y'\in[0, h W_r]} \inf_{\substack{\zeta' \subset  [0,r]\times [0,h W_r] \\ \zeta'(0)=(0,y), \zeta'(1)=(r,y')}} \hrX_{\zeta'}
\]
Clearly $U_{r}(h)$ is decreasing in $h$.  
 We claim that there exists a constant $C_*$ such that
 \begin{equation}\label{eq:U.decreasing}
 r - C_* Q_r \leq \E  U_{r}(2) \leq \E  U_{r}(1) \leq r +C_* Q_r.
\end{equation}
We have 
$$\E U_{r}(2) \geq \E \inf_{y,y'\in[0, 2 W_r]}  \rX_{(0,y),(r,y')} - \E \sup_{y,y'\in[0, 2 W_r]} | \hrX_{(0,y),(r,y')} - \rX_{(0,y),(r,y')}| $$
and the lower bound follows from Proposition \ref{p:paraestimateconforming} and the fact that $W_r=\sqrt{rQ_r}$.  The upper bound on $\E [U_{r}(1)]$ is an easy consequence of Proposition~\ref{p:constraine}.

Now choose $h =h(n,M,\ell,t) \in [1,2]$ such that
\begin{equation}\label{eq:opt.h}
\E U_{r}(h) + 10 C_* Q_r h = \min_{h'\in[1,2]} [\E U_{r}(h') + 10 C_* Q_r h'].
\end{equation}
By equation~\eqref{eq:U.decreasing} we have that for $h'\in[\frac32,2]$,
\[
\E U_{r}(h') + 10 C_* Q_r h' \geq \E U_{r}(2) + 15 C_* Q_r > \E U_{r}(1) + 10 C_* Q_r
\]
and so $h\in[1,\frac32)$. By Markov's inequality we have that for $w\le 1/4$
\begin{align*}
\P[\cM^c_{0,0,h,z,w}] &\leq \frac1{zQ_r}\E\Bigg(\inf_{y,y'\in[0,h W_r]} \inf_{\substack{ \gamma' \subset  [0,r]\times [0,h W_r] \\ \gamma'(0)=(0,y) \\ \gamma'(1)=(r,y')}} \hrX_{\gamma'} 
- \inf_{y,y'\in[-w W_r,(h+w) W_r]} \inf_{\substack{\gamma' \subset  [0r,r]\times [-w W_r,(h+w) W_r] \\ \gamma'(0)=(0,y) \\ \gamma'(1)=(r,y')}} \hrX_{\gamma'} \Bigg) \\
&=\frac{\E U_{r}(h) - \E U_{r}(h+2w)}{zQ_r}\\
&=\frac{\E U_{r}(h) - (\E U_{r}(h+2w) + 10 C_* Q_r (h+2w)) + 10 C_* Q_r (h+2w)}{zQ_r}\\
&\leq\frac{\E U_{r}(h) - (\E U_{r}(h) + 10 C_* Q_r (h)) + 10 C_* Q_r (h+2w)}{zQ_r}=\frac{20 w C_*}{z},
\end{align*}
where the second inequality is by the optimality of the choice of $h$ in equation~\eqref{eq:opt.h}.  This establishes equation~\eqref{e:mboundM}.

Recall that we chose $\kappa$ such that $\kappa^{-1}$ is an integer.  With $\omega_s$ defined as in \eqref{eq:omega.interpolation}, for $i\in\{0,1,\ldots,\kappa^{-1}\}$ we let $\cX^{(i)}$ denote the conforming passage times with respect to the field $\omega_{i\kappa}$.  Between $\cX^{(i)}$ and $\cX^{(i+1)}$ each block is updated with probability $\kappa$ so the joint law $(\cX^{(i)},\cX^{(i+1)})$ is equal in distribution to $(\cX,\cX')$.

Define the event
\[
\cE = \bigg\{\omega_0 \in \cI^+_{0,0,h,0}, \omega_1 \in \cI^-_{0,0,h,-2\kappa^{-1} t} \bigg\}.
\]
Since $\omega_0$ and $\omega_1$ are independent,
\[
\P[\cE] = \P[\cI^+_{0,0,h,0}]\P[ \cI^-_{0,0,h,-2\kappa^{-1} t}] \geq \delta_1
\]
for some $\delta_1$ (independent of $r$) by Lemma~\ref{l:cI.bound}.  Let
\[
V^{(i)} = Q_r^{-1}\inf_{y,y'\in[0,h W_r]} \inf_{\substack{ \gamma' \subset  [0,r]\times [0,h W_r] \\ \gamma'(0)=(0,y) \\ \gamma'(1)=(r,y')}} \hrX^{(i)}_{\gamma'}.
\]
On the event $\cE$ we have that $V^{(0)}\geq 0, V^{(\kappa^{-1})}\leq -2\kappa^{-1} t$ and so we can always find $i_\star\in\{1,\ldots,\kappa^{-1}\}$ defined by
\[
i_\star = \inf\{i\geq 1:V^{(i)}  \leq - 2i t\}
\]
and have that $V^{(i_\star - 1)} \geq - 2(i_\star - 1) t$.
Hence
\begin{align*}
\sum_{i=1}^{\kappa^{-1}} \P[\cI^+_{0,0,h,-2(i - 1) t}, \cI^{-'}_{0,0,h,-2i t}] &= \sum_{i=1}^{\kappa^{-1}}  \P[ \omega_{(i-1)\kappa} \in \cI^+_{0,0,h,-2(i - 1) t}, \omega_{i\kappa} \in \cI^{-}_{0,0,h,-2i t}] \geq \P[\cE] \geq \delta_1
\end{align*}
since on $\cE$ at least one of the events in the second sum must occur.  So we can find some deterministic $i$ such that
\[
\P[\cI^+_{0,0,h,-2(i - 1) t}, \cI^{-'}_{0,0,h,-2i t}] \geq \kappa \delta_1.
\]
Setting $\alpha = 2(i - \frac12) t$ we have, provided $t$ is large enough, that
\[
2(i - 1) t\leq \alpha - \alpha^{9/10} \leq \alpha+ \alpha^{9/10} \leq 2i t
\]
and so
\[
\P[\cI^+_{0,0,h,-(\alpha - \alpha^{9/10})}, \cI^{-'}_{0,0,h,-(\alpha + \alpha^{9/10})}] \ge \delta
\]
for some $\delta>0$, completing the proof of \eqref{e:mboundI}. 
\qed

\appendix

\section{Constrained passage times and restricted distances}
\label{s:proxy}

This section provides two basic facts about conforming paths and restricted distances: Lemmas \ref{l:proxy} which states the restricted passage times $\rX$ approximate passage times $X$ sufficiently well  and \ref{l:strongly.conforming} which states the conforming geodesics can be arbitrarily well approximated by strongly conforming paths. The proof of Lemma \ref{l:proxy} will require the constrained passage time estimate Proposition~\ref{p:constraine} therefore we first provide the proof of that proposition. 

\subsection{Constrained passage times: Proof of Proposition \ref{p:constraine}}
The first step of the proof is to show that a similar estimate holds when the endpoints are slightly away from the boundary of the rectangle where the path is constrained to lie. 

\begin{center}
\begin{figure}
\includegraphics[width=5in]{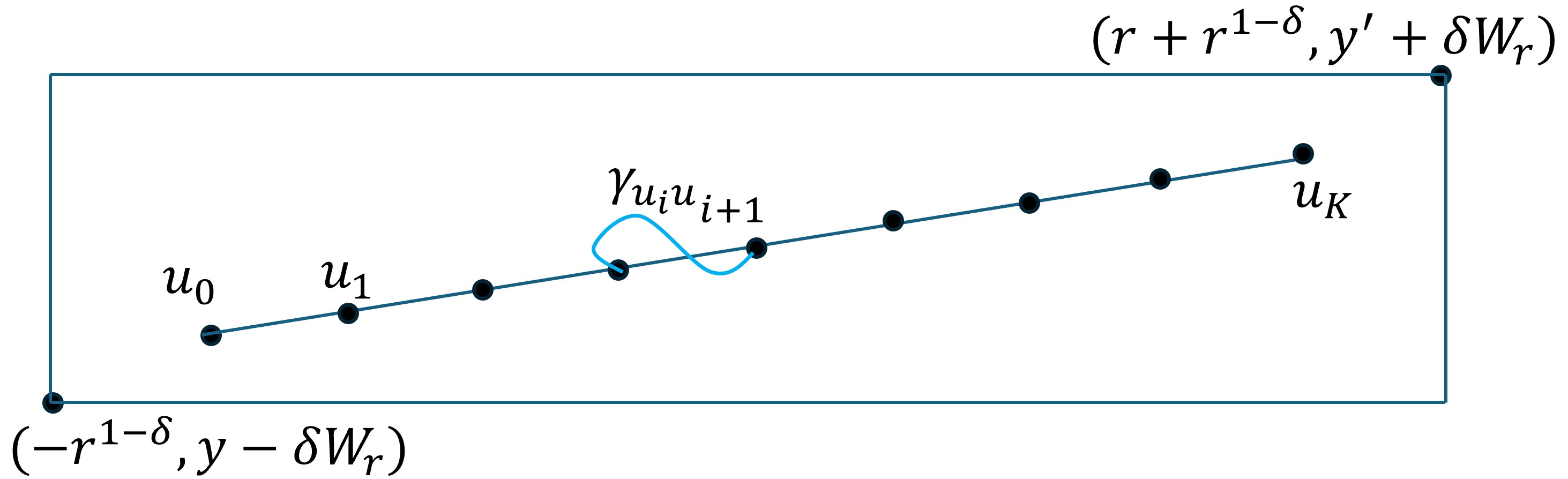}
\caption{Proof of Lemma \ref{l:consp2p}: we show that the passage time between $u_0$ and $u_k$ constrained to be in the rectangle is unlikely to be too large. This is done by considering a set of intermediate points $u_i$ on the straight line joining $u_0$ to $u_k$, and considering the concatenation of geodesics between $u_{i}$ and $u_{i+1}$. This is possible since it is unlikely that the geodesics $\gamma_{u_{i}u_{i+1}}$ are unlikely to exit the rectangle due to transversal fluctuation estimates.}
\label{f:constrained2}
\end{figure}
\end{center}
\begin{lemma}
    \label{l:consp2p}
    Let $L>0$ be fixed.  For $0\le y\le y'\le LW_r$, and $\delta$ sufficiently small let $R=R_{r,y,y',\delta}$ denote the rectangle with corners $(-r^{1-\delta}, y-\delta W_r), (-r^{1-\delta}, y'+\delta W_r), (r+r^{1-\delta}, y-\delta W_r), (r+r^{1-\delta}, y'+\delta W_r)$. There exists $C,\theta_7$ (depending on $L,\delta$) such that for all $r$ sufficiently large, all $y,y'$ as above and for all $z>0$ we have
    $$\P\left(\inf_{\gamma'(0)=(0,y), \gamma'(1)=(r,y')}\inf_{\gamma\subseteq R} X_{\gamma} \ge r+zQ_r \right) \le \exp(1-Cz^{\theta_7}).$$
\end{lemma}

\begin{proof}
    Observe that it suffices to prove the lemma for all $z\in (z_0,r^{\beta})$ for some $z_0$ sufficiently large. Let us fix $u=(0,y),v=(r,y')$ as in the statement of the lemma. Let $K=K(\delta)$ be a large integer to be fixed later. For $i=0,1,2, \ldots, K$, let us set $u_i=u+\frac{i}{K}(v-u)$; see Figure \ref{f:constrained2}. Observe now that on the event that the geodesics $\gamma_{u_{i},u_{i+1}}$ does not exit $R$, we have 
    $$\inf_{\gamma'(0)=u, \gamma'(1)=v}\inf_{\gamma\subseteq R} X_{\gamma}\le \sum_{i=0}^{K-1} X_{u_{i},u_{i+1}}.$$
    Notice that, the typical order of transversal fluctuations of $\gamma_{u_i,u_{i+1}}$ is $W_{r/K}$. Therefore using Theorem \ref{t:tfold} and the fact that $W_{r}\geq K^{1/2+\alpha_*/2}W_{r/K}$, and translation invariance it follows that for some $\theta'>0$
    $$\P\left(\inf_{\gamma'(0)=u, \gamma'(1)=v}\inf_{\gamma\subseteq R} X_{\gamma}\ge r+zQ_r\right)\le K\P(X_{u_i,u_{i+1}}\ge \frac{r}{K}+\frac{z}{K}Q_r)+K\exp(-K^{\theta'}).$$
    Using $Q_{r}\ge K^{\alpha_*}Q_{r/K}$  it follows by choosing $K=z^{1/10}$ and applying the tail estimates from Theorem \ref{t:all} that both terms on the right hand side above are upper bounded by $\exp(-Cz^{\theta_7})$ for some $C,\theta_7>0$ and this completes the proof of the lemma. 
\end{proof}

\begin{proof}[Proof of Proposition \ref{p:constraine}]
It suffices to prove the result for $z>z_0$ for some fixed $z_0$.  We shall prove 
\begin{equation}
    \label{e:halfconstrained}
    \P\left(\inf_{\substack{\gamma'(0)=(0,0)\\ \gamma'(1)=(r/2,W_r/2)\\ \gamma'\subset [0,r]\times[0,W_r]}} X_{\gamma'}\ge r/2+zQ_r\right)\le \exp(1-Cz^{\theta_6}).
\end{equation}
This together with triangle inequality and reflection symmetry will complete the proof of the lemma. To prove \eqref{e:halfconstrained} we define a sequence of points on the dyadic scale joining $(0,0)$ and $(r/2, W_r/2)$. Let $h$ denote the smallest integer such that $2^{-h}r\le Q_r$. For $k=1,2,\ldots, h$ define $u_k=(\frac{r}{2^{k}},\frac{1}{2}W_{r/2^{k-1}})$. Since $W_{r/2^{k}}/W_{r/2^{k-1}}$ is bounded above and below it follows that {for some $\delta>0$, for each pair $(u_{k+1},u_{k})$ we have $R'\subseteq [0,r]\times [0,W_r]$ where $R'$ is the image of $R_{2^{-(k+1)},\frac{1}{2}W_{r/2^k}, \frac{1}{2}W_{r/2^{k-1}},\delta}$ under the translation that takes $(0,\frac{1}{2}W_{r/2^k})$ to $u_{k+1}$ and $(2^{-(k+1)},\frac{1}{2}W_{r/2^{k-1}})$ to $u_{k}$, and hence for suitably chosen $L$ (and $\delta$ as above) one can use Lemma \ref{l:consp2p} to lower bound the probability of 
$$\inf_{\substack{\gamma'(0)=u_{k+1}\\ \gamma'(1)=u_{k}\\ \gamma'\subset [0,r]\times[0,W_r]}} X_{\gamma'}$$ being not too large.}

For $\alpha_*$ such that 
$Q_{r}\ge 2^{k\alpha_*}Q_{r/2^{k}}$, and $1\le k\le h-1$ let $A_{k}$ denote the event 
$$A_k=\left\{\inf_{\substack{\gamma'(0)=u_{k+1}\\ \gamma'(1)=u_{k}\\ \gamma'\subset [0,r]\times[0,W_r]}} X_{\gamma'}\le \frac{r}{2^{k+1}}+z\frac{2^{-\alpha_* (k+1)/2}}{2(1-2^{-\alpha_*/2})}Q_{r/2^{k+1}}\right\}.$$ Using Lemma \ref{l:consp2p}, it follows that 
\begin{equation}\label{eq:Ak.bound}
\P(A_k^c)\le \exp(-C(z2^{\alpha_* k/2})^{\theta_7}).
\end{equation}
Notice also that the straight line path joining $(0,0)$ to $u_{h}$ deterministically has length upper bounded by $C'Q_r$ for some $C$ by our choice of $h$ and set $z_0=2C$. This and the triangle inequality implies that we have on the event $\cap A_k$, 
$$\inf_{\substack{\gamma'(0)=(0,0)\\ \gamma'(1)=(r/2,W_r/2)\\ \gamma'\subset [0,r]\times[0,W_r]}} X_{\gamma'}\le r/2+(C+\frac12 z)Q_r\geq  r/2+zQ_r,$$
since $z\geq 2C$. By a union bound over $k$ and equation~\eqref{eq:Ak.bound}
\[
\P\Big[\bigcap_k A_k\Big] \geq 1-\sum_{k} \exp(-C(z2^{\alpha_* k/2})^{\theta_7}) \geq 1- \exp(-Cz^{\theta_7}) 
\]
which completes the proof of \eqref{e:halfconstrained}.  
\end{proof}

\subsection{Proof of Lemma \ref{l:proxy}}
We now move to the proof of Lemma \ref{l:proxy}. We shall first prove the second statement in the lemma, i.e., we prove that for some $\epsilon, \theta_1>0$ we have
\begin{equation}
    \label{e:proxy2}
    \P\Big[\sup_{u,v \in [0,nM]\times[- n^\beta W_n, n^\beta W_n]} X_{uv}-\rX_{uv} \geq n^{-\epsilon} Q_n\Big] \leq \exp(-n^{\theta_1}).
\end{equation}

For \eqref{e:proxy2} it suffices to consider points in the same column. We have the following lemma. 

\begin{lemma}
    \label{l:proxy2column}
    There exist $\epsilon, \theta_1>0$ such that for each $i$, we have for all $M$ and all $n\ge n(M)$
   \begin{equation*}
    {\P\Big[\sup_{u,v \in [(i-1)n,in]\times[- n^\beta W_n, n^\beta W_n]} X_{uv}-\rX_{uv} \geq n^{-2\epsilon} Q_n\Big] \leq \exp(-2n^{\theta_1})}.
\end{equation*} 
\end{lemma}

We first complete the proof of \eqref{e:proxy2} assuming Lemma \ref{l:proxy2column}. 

\begin{proof}[Proof of \eqref{e:proxy2}]
    Let $A_i$ denote the following event for $1\le i \le M$:
    $$A_i=\biggl\{\sup_{u,v \in [(i-1)n,in]\times[- n^\beta W_n, n^\beta W_n]} X_{uv}-\rX_{uv} \geq n^{-2\epsilon} Q_n\biggr\}.$$
    
    We shall show that, for $n$ large enough depending on $M$, on the event $\cap A^c_i$ we have for all $u,v\in [0,nM]\times[- n^\beta W_n, n^\beta W_n]$ (except for the vertical boundary pairs) 
    \begin{equation}
        \label{e:proxy3}
        X_{uv}-\rX_{uv} \leq n^{-\epsilon} Q_n
    \end{equation}
    Clearly, \eqref{e:proxy2} follows from \eqref{e:proxy3}, using Lemma \ref{l:proxy2column}, taking a union bound over $1\le i \le M$ and choosing $n$ sufficiently large. 

    Let us now prove \eqref{e:proxy3}. Fix $u,v\in [0,nM]\times[- n^\beta W_n, n^\beta W_n]$ with $u\in[(i_--1)n,i_-n)\times [- n^\beta W_n, n^\beta W_n]$ and $v\in((i_+-1)n,i_+n]\times [- n^\beta W_n, n^\beta W_n]$. Let $u_{i_--1}=u,u_{i_+}=v$ and for $i\in\{i_-,\ldots,i_+-1\}$ let $u_i=(ir,y_i)$ be the intersection of the canonical path with the line $x=in$ such that 
    $$\rX_{uv}=\sum_i\rX_{u_{i-1}u_i}.$$
    By the triangle inequality it follows that 
    $$X_{uv}\le \sum_{i} X_{u_{i-1}u_{i}}\le \sum_i\rX_{u_{i-1}u_i}+Mn^{-2\epsilon}Q_n$$
    on the event $\cap A^c_i$. Choosing $n$ sufficiently large completes the proof of~\eqref{e:proxy3}.
\end{proof}

Before proving Lemma \ref{l:proxy2column} we prove an elementary lemma (this proves a stronger version of model Assumption 11 in \cite{BSS23}). 

\begin{lemma}
    \label{l:resamplecompare}
    There exists $\epsilon,\theta_1>0$ such that for a fixed pair $u,v\in [(i-1)n,in]\times[-n^{\beta}W_n,n^{\beta}W_n]$ we have 
    $$\P(|X_{uv}-X^{\Lambda_i}_{uv}|\ge n^{-3\epsilon}Q_n) \le \exp(-3n^{\theta}).$$
\end{lemma}

\begin{proof}
    For $\epsilon$ sufficiently small choose $\delta<1$ such that $n^{\delta}\ge n^{2\beta}W_n$ and $Q_{n^{\delta}}\le n^{-4\epsilon}Q_n$ (this is possible by the properties of $Q_n$ and $W_n$ from Theorem \ref{t:all} and since $\beta$ is chosen small). Let us fix $u=(u_1,u_2)$ and $v=(v_1,v_2)$. If $|u-v|\le n^{\delta}$ it follows from Theorem \ref{t:all} and $Q_{n^{\delta}}\le n^{-4\epsilon}Q_n$ that 
    $$\P(|X_{uv}-X^{\Lambda_i}_{uv}|\ge n^{-3\epsilon}Q_n) \le \exp(-3n^{\theta})$$
    for some $\theta>0$. We therefore only need to deal with the case where $|u-v|\ge n^{\delta}$. By our choice of $\delta$ this also means $|v_1-u_1|\ge \frac{1}{2}n^{\delta}$. Fix such a pair $u,v$. We shall show that 
    \begin{equation}
        \label{e:exchange}
        \P(X^{\Lambda_i}_{uv}-X_{uv}\ge n^{-3\epsilon}Q_n)\le \exp(-3n^{\theta})
    \end{equation}
    for some $\theta>0$. Since $(X_{uv},X_{uv}^{\Lambda_i})$ is an exchangeable pair, this will complete the proof of the lemma by a union bound. 
    
    Let $\gamma_{uv}$ denote the geodesic attaining $X_{uv}$. Let $u'=(u'_1,u'_2)$ (resp.\ $v'=(v'_1,v'_2)$) be its last (resp.\ first) intersection with the line $x=(i-1)n+\frac{1}{4}n^{\delta}$ (resp.\ line $x=in-\frac{1}{4}n^{\delta}$). If no such intersection exists we set $u'=u$ (resp.\ $v'=v$). Without loss of generality we shall assume $u_1\le v_1$. Let $A$ denote the event that either $|u'_2|\ge 2n^{\beta}W_n$ or $|v'_2|\ge 2n^{\beta}W_n$. It follows from Theorem \ref{t:tfold} that $\P(A)\le \exp(-4n^{\theta})$ for some $\theta>0$. Let $\gamma_1,\gamma_2,\gamma_3$ denote the restrictions of $\gamma_{uv}$ between $u$ and $u'$, $u' and v'$ and $v'$ and $v$ respectively. By definition $X_{\gamma_2}=X_{\gamma_2}^{\Lambda_i}$ therefore to prove \eqref{e:exchange} it suffices to show that (just consider the path obtained by concatenating $\gamma'_1, \gamma_2, \gamma'_3$ in the environment $\omega^{\Lambda_1}$ where $\gamma'_1$ attains $X^{\Lambda_i}_{u,u'}$ and $\gamma'_3$ attains $X^{\Lambda_i}_{v,v'}$) 
    \begin{equation}
        \label{e:exchange1}
        \P(X^{\Lambda_i}_{u,u'}-X_{\gamma_1}\ge \frac{1}{2}n^{-3\epsilon}Q_n)\le \exp(-3n^{\theta});
    \end{equation}
    \begin{equation}
        \label{e:exchange2}
        \P(X^{\Lambda_i}_{\gamma_3}-X_{\gamma_3}\ge \frac{1}{2}n^{-3\epsilon}Q_n)\le \exp(-3n^{\theta}).
    \end{equation}
    We shall show \eqref{e:exchange1}. The other proof of the other one is identical. Let $B$ denote the event that for all $w,w'\in [(i-1)n, (i-1)n+\frac{1}{2}n^{\delta}]\times [-2n^{\beta}W_n,2n^{\beta}W_n]$ we have 
    $|X_{ww'}-X_{ww'}^{\Lambda_i}|\ le \frac{1}{4}n^{-3\epsilon}Q_n$. By our choice of $\delta$, and taking a union bound over all pairs of integer points $w,w'$ it follows from Theorem \ref{t:all} that 
    $\P(B)\le \exp(-3n^{\theta})$. Now, if $u_1\le (i-1)n+\frac{1}{2}n^{\delta}$, it follows that the probability on the left hand side of \eqref{e:exchange1} is upper bounded by $\P(B)$ and we are done. If $u_1\ge (i-1)n+\frac{1}{2}n^{\delta}$, the probability in \eqref{e:exchange1} is upper bounded by $\P(u\neq u')$ which is further upper bounded by $\exp(-3n^{\theta})$ by Theorem \ref{t:tfold} and we are done. 
\end{proof}

\begin{proof}[Proof of Lemma \ref{l:proxy2column}]
Using Lemma \ref{l:resamplecompare} and taking a union bound over all $u,v\in [(i-1)n,in]\times [-n^{\beta}W_n, n^{\beta} W_n]\cap \Z^2$ it follows that 

\begin{equation}
    \label{e:proxy4}
    \P\left(\max_{u,v\in [(i-1)n,in]\times [-n^{\beta}W_n, n^{\beta} W_n]} |X_{uv}-X^{\Lambda_i}_{uv}|\ge n^{-2\epsilon}Q_n\right)\le \exp(-2n^{\theta_1}).
\end{equation}

The lemma now follows from noticing that for $u,v\in [(i-1)n,in]\times [-n^{\beta}W_n, n^{\beta} W_n]$ we have 
$X^{\Lambda_i}_{uv}\le \rX_{uv}$. 
\end{proof}

Next we prove the first statement in Lemma \ref{l:proxy}. In conjunction with \eqref{e:proxy2}, it suffices to prove that 

\begin{equation}
    \label{e:proxy5}
    {\P\Big[\inf_{u,v \in [0,nM]\times[- \frac{1}{2}n^\beta W_n, \frac{1}{2}n^\beta W_n]} X_{uv}-\rX_{uv} \leq -n^{-\epsilon} Q_n\Big] \leq \exp(-n^{\theta_1})}.
\end{equation}

As before we shall first do a within column estimate. 

\begin{lemma}
    \label{l:proxy3column}
    There exist $\epsilon, \theta_1>0$ such that for each $i$, we have for all $M$ and all $n\ge n(M)$
   \begin{equation*}
    {\P\Big[\inf_{u,v \in [(i-1)n,in]\times[- n^\beta W_n, n^\beta W_n]} X_{uv}-\rX_{uv} \leq -n^{-2\epsilon} Q_n\Big] \leq \exp(-2n^{\theta_1})}.
\end{equation*} 
\end{lemma}

We shall first complete the proof of \eqref{e:proxy5} assuming Lemma \ref{l:proxy3column}.  

\begin{proof}[Proof of \eqref{e:proxy5}]
    Notice first that by discretizing space it suffices to prove that 
    \begin{equation}
    \label{e:proxy7}
    {\P\Big[\max_{u,v \in [0,nM]\times[- \frac{1}{2}n^\beta W_n, \frac{1}{2}n^\beta W_n]\cap \Z^2} X_{uv}-\rX_{uv} \leq n^{-3\epsilon/2} Q_n\Big] \leq \exp(-n^{\theta_1})}.
\end{equation}
Since the number of integer points in $[0,nM]\times[- \frac{1}{2}n^\beta W_n, \frac{1}{2}n^\beta W_n]$ is polynomial in $n$, it follows that by a union bound and by choosing $n$ sufficiently large depending on $M$, it suffices to prove the above estimate for each fixed pair of $u$ and $v$. Now, fix $u,v\in [0,nM]\times[- \frac12 n^\beta W_n, \frac12 n^\beta W_n]$ with $u\in[(i_--1)n,i_-n)\times [- \frac12 n^\beta W_n,\frac12 n^\beta W_n]$ and $v\in((i_+-1)n,i_+n]\times [- \frac12 n^\beta W_n, \frac12 n^\beta W_n]$. Let $A_{uv}$ denote the event that the geodesic (in the original model $X$) from $u$ to $v$ does not exit the strip $\R\times [-n^{\beta}W_n, n^{\beta} W_n]$. It follows from Theorem \ref{t:tfold} that $\P(A_{uv})\ge 1-\exp(-2n^{\theta_1})$ for some $\theta_1>0$. Therefore it suffices to restrict ourselves to the set $A_{uv}$.

Let $u_{i_--1}=u,u_{i_+}=v$ and for $i\in\{i_-,\ldots,i_+-1\}$ let $u_i=(ir,y_i)$ be the intersection of the optimal $X$ path with the line $x=in$ such that 
\[
X_{uv}=\sum_i X_{u_{i-1}u_i}.
\]
Denoting the event in Lemma \ref{l:proxy3column} by $B_i$, we have on $A_{uv}\cap \bigcap B^c_i$
$$X_{uv}=\sum_{i} X_{u_{i-1}u_i} \ge \sum_{i} \rX_{u_{i-1}u_i}+Mn^{-2\epsilon}Q_n\ge \rX_{uv}+Mn^{-2\epsilon}Q_n.$$
By choosing $n$ sufficiently large and using Lemma \ref{l:proxy3column}, and a union bound over $u,v$ we have that~\eqref{e:proxy5} follows. 
\end{proof}

It remains to prove Lemma \ref{l:proxy3column}. 

\begin{proof}[Proof of Lemma \ref{l:proxy3column}]
Let $i$ be fixed. As before, by a union bound it suffices to prove the lemma for each fixed $u,v\in [(i-1)n,in]\times [-n^{\beta}W_n,n^{\beta}W_n]\cap \Z^2$. Consider $u,v$ fixed as above. Also, by \eqref{e:proxy4} it suffices to replace $X_{uv}$ by $X_{uv}^{\Lambda_i}$, i.e., it suffices to show that

\begin{equation}
    \label{e:proxy10}
    \P(X^{\Lambda_i}_{uv}-\rX_{uv}\le -n^{-2\epsilon}Q_n)\le \exp(-3n^{\theta_1}).
\end{equation}

Without loss of generality we shall write the proof of \eqref{e:proxy10} only for the case $i=1$. We shall treat three cases separately. We prove that there exist $\beta,\varepsilon, \theta_1>0$ such that 

\begin{equation}
    \label{e:proxy8}
    \P\left(\sup_{u,v\in [0,n^{\beta}W_n]\times [-2n^{\beta}W_n,2n^{\beta}W_n]} X^{\Lambda_1}_{uv}-\rX_{uv}\le -n^{-3\epsilon}Q_n\right)\le \exp(-4n^{\theta_1});
\end{equation}

\begin{equation}
    \label{e:proxy9}
    \P\left(\sup_{u,v\in [n^{\beta}W_n, n-n^{\beta}W_n]\times [-2n^{\beta}W_n,2n^{\beta}W_n]} X^{\Lambda_1}_{uv}-\rX_{uv}\le -n^{-3\epsilon}Q_n\right)\le \exp(-4n^{\theta_1});
\end{equation}

\begin{equation}
    \label{e:proxy11}
    \P\left(\sup_{u,v\in [n-n^{\beta}W_n,n]\times [-2n^{\beta}W_n,2n^{\beta}W_n]} X^{\Lambda_1}_{uv}-\rX_{uv}\le -n^{-3\epsilon}Q_n\right)\le \exp(-4n^{\theta_1}).
\end{equation}

Let us first explain how to prove \eqref{e:proxy10} using \eqref{e:proxy8}, \eqref{e:proxy9}, \eqref{e:proxy11}. Observe that \eqref{e:proxy8}, \eqref{e:proxy9}, \eqref{e:proxy11} consider three regions (disjoint except at the boundary) whose union is $[0,n]\times [-2n^{\beta}W_n,2n^{\beta}W_n]$. Now if the points $u,v$ (in \eqref{e:proxy10}) belong to the same (out of the three) region, \eqref{e:proxy10} is a consequence of \eqref{e:proxy8} or \eqref{e:proxy9} or \eqref{e:proxy11}. So we need to show \eqref{e:proxy10} when $u$ and $v$ belong to different regions. Without loss of generality, we shall assume that $u\in [0,n^{\beta}W_n]\times [-n^{\beta}W_n,n^{\beta}W_n]$ and $v\in [n-n^{\beta}W_n,n]\times [-n^{\beta}W_n,n^{\beta}W_n]$; the other cases can be dealt with minor variations of the same argument. 

Let $A,B,C$ denote the events in \eqref{e:proxy8}, \eqref{e:proxy9}, \eqref{e:proxy11} respectively. Let $D$ denote the event that there exist points $u_1\in \{n^{\beta}W_n\}\times [-2n^{\beta}W_n, 2n^{\beta}W_n]$ and $v_1\in \{n-n^{\beta}W_n\}\times [-2n^{\beta}W_n, 2n^{\beta}W_n]$ such that 

$$X_{uv}^{\Lambda_1}=X_{uu_1}^{\Lambda_1}+X_{u_1v_1}^{\Lambda_1}+X_{v_1v}^{\Lambda_1}.$$
As $D^c$ implies that the geodesic from $u$ to $v$ has large transversal fluctuation it follows from Theorem \ref{t:tfold} that $P(D^c)\le \exp(-4n^{\theta_1})$ for some $\theta_1>0$ (depending on $\beta$). Observe next that on $A^c\cap B^c \cap C^c \cap D$ we have 
$$X_{uv}^{\Lambda_1}=X_{uu_1}^{\Lambda_1}+X_{u_1v_1}^{\Lambda_1}+X_{v_1v}^{\Lambda_1}\ge \rX_{uu_1}+\rX_{u_1v_1}+\rX_{v_1v}-3n^{-3\epsilon}Q_n\ge \rX_{uv}-n^{-2\epsilon}Q_n,$$
which, together with \eqref{e:proxy8}, \eqref{e:proxy9}, \eqref{e:proxy11} completes the proof since
\[
\P[A^c\cap B^c \cap C^c \cap D] \geq 1-\exp(-4n^{\theta_1})-3\exp(-4n^{\theta_1}).
\]

It remains to prove \eqref{e:proxy8}, \eqref{e:proxy9}, \eqref{e:proxy11}. By the underlying symmetries of the model the proof of \eqref{e:proxy11} is identical to that of \eqref{e:proxy8}, hence we shall only prove the first two. 

\emph{Proof of \eqref{e:proxy9}.}
As before, note that the result will follow by a union bound if we prove the same bound for every pair $u,v$ of integer points in $[n^{\beta}W_n, n-n^{\beta}W_n]\times [-2n^{\beta}W_n,2n^{\beta}W_n]$, so it suffices to prove the bound for a fixed $u,v$ as above. Fixing $u,v\in [n^{\beta}W_n, n-n^{\beta}W_n]\times [-2n^{\beta}W_n,2n^{\beta}W_n]$, it follows (for $\beta$ small) from Theorem \ref{t:tfold} that the probability that the geodesic $\gamma_{uv}$ attaining $X^{\Lambda_1}_{uv}$ exits $\Lambda_1$ is upper bounded by $\exp(-5n^{\theta_1})$ for some $\theta_1>0$. Noticing that, on the event above, we have $X^{\Lambda_1}_{uv}=\rX_{uv}$ the desired result follows.

\begin{center}
\begin{figure}
\includegraphics[width=5in]{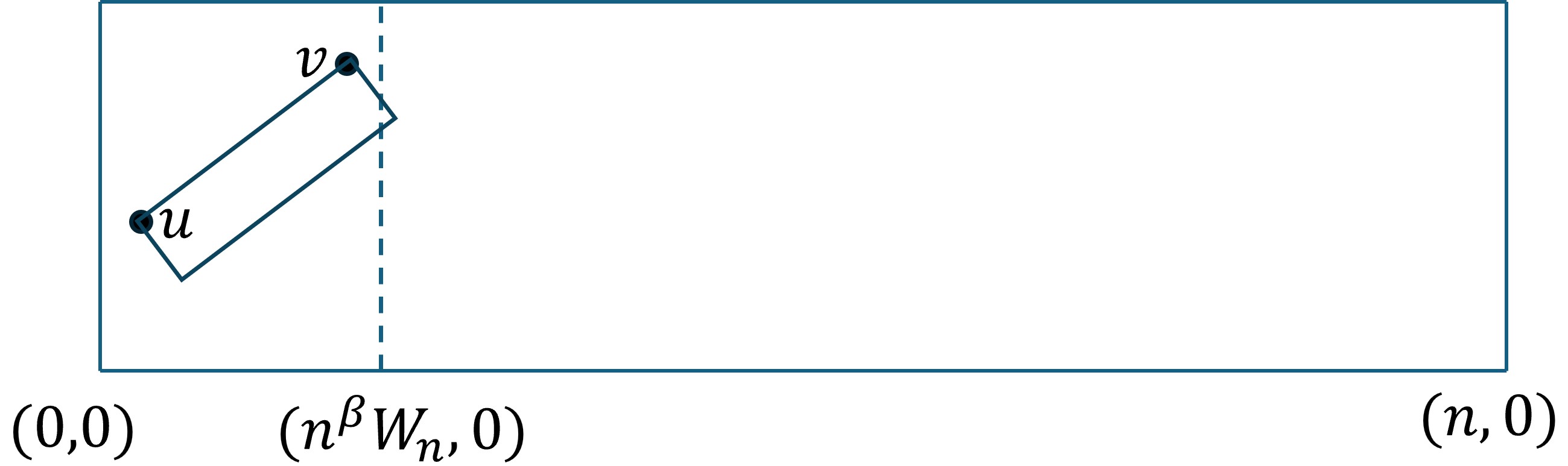}
\caption{The rectangle constructed in the proof of \eqref{e:proxy8}.}
\label{f:internal}
\end{figure}
\end{center}

\emph{Proof of \eqref{e:proxy8}.}
Again, it suffices to prove the bound for fixed $u,v\in [0,n^{\beta}W_n]\times [-2n^{\beta}W_n,2n^{\beta}W_n]$. For such $u,v$ notice that there exists an $|u-v|\times W_{|u-v|}$ rectangle $R$ contained in $\Lambda_1$ such that one of the sides of $R$ is the straight line segment joining $u$ and $v$; see Figure \ref{f:internal}. Let $A$ denote the event that
$\rX_{uv}\le |u-v|+n^{\delta}Q_{|u-v|}$ and let $B$ denote the event that 
$X^{\Lambda_1}_{uv}\ge |u-v|-n^{\delta}Q_{|u-v|}$. If $\beta,\delta>0$ are sufficiently small we have $n^{\delta}Q_{|u-v|}\ll n^{-3\epsilon}Q_n$ for some $\epsilon$ small enough and on the event $A\cap B$ we have 
$$X^{\Lambda_1}_{uv}-\rX_{uv}\ge -n^{-3\epsilon}Q_n.$$
So it only remains to upper bound $\P(A^c)$ and $\P(B^c)$. It follows from Proposition \ref{p:para} that $\P(B^c)\le \exp(-5n^{\theta_1})$ for some $\theta_1>0$. Notice that on the event $A^c$, every path $\gamma$ from $u$ to $v$ contained in $R$, must satisfy
 $$X_{\gamma}^{\Lambda_1}\ge |u-v|+n^{\delta}Q_{|u-v|}.$$
 It follows from Proposition~\ref{p:constraine} that $\P(A^c)\le \exp(-5n^{\theta_1})$ for some $\theta_1>0$. This completes the proof of \eqref{e:proxy8}.    
\end{proof}

\subsection{Proof of Lemma \ref{l:strongly.conforming}}
Finally we provide the proof of Lemma \ref{l:strongly.conforming}

\begin{proof}[Proof of Lemma \ref{l:strongly.conforming}]
Recall that the length of the path is given by
\[
\rX_\gamma = \inf_{\{t_i\}} \sum_{i} X_{\gamma([t_i,t_{i+1}])}^{\Lambda_i}
\]
By treating each segment of the path $\gamma([t_i,t_{i+1}])$ separately, we can reduce the problem to the case where $(i-1)n\leq \gamma_1(0)<\gamma_1(1)\leq in$ for some $n$.  Assume without loss of generality assume that $|\dot{\gamma}(t)|=c$ is constant.  
\[
\gamma^{(\delta)}_1(t)= (i-\frac12)n + (\gamma_1(t)-(i-\frac12)n)(1-\delta +2\delta|t-\frac12|)
\]
Then $\gamma^{(\delta)}=(\gamma^{(\delta)}_1,\gamma_2)$ is strongly conforming for all $\delta>0$ and 
\[
X_{\gamma^{(\delta)}}=   \int_{0}^{1} \Psi(\gamma^{(\delta)}(t))|\dot{\gamma}^{(\delta)}(t)| dt \to X_{\gamma}
\]
as $\delta\to 0$, so there exists $\delta>0$ with $\cX_\gamma\geq \cX_{\gamma^{(\delta)}}-\epsilon$.
\end{proof}

\section{Transversal fluctuations for conforming geodesics}
\label{s:proxytrans}
The purpose of this section is to prove results about transversal fluctuations conforming geodesics. We prove global transversal fluctuation bounds for conforming geodesics (Lemma \ref{l:proxytrans:intro}), local transversal fluctuation bounds (Lemma \ref{l:localtransproxy:intro}) and use the transversal fluctuation estimates to prove bounds on constrained passage time estimates for the restricted distances (Proposition \ref{p:constrainerX}). 

\subsection{Global transversal fluctuations} 
{We now prove Lemma \ref{l:proxytrans:intro}. Recall that Lemma \ref{l:proxytrans:intro} asks for bounding transversal fluctuations of geodesics $\gamma_{v_1,v_2}$ where $v_1\in \ell_{0,w,W_n}$ and $v_2\in \ell_{Mn,w,W_n}$ for some $w\in [-\frac{1}{2}n^{\beta}W_n, \frac{1}{2}n^{\beta}W_n]$. To reduce notational overhead we shall prove this result in the special case $w=0$. It is easy to see that the general case follows by the same arguments.}

We need to control the transversal fluctuation for conforming paths attaining the distance $\rX_{uv}$. To this end, let $\gamma_{uv}$ denote the canonical choice for the conforming geodesic attaining $\rX_{uv}$. Let us set 
$$\mathcal{W}_{n,M}=\sup_{u\in \{0\}\times [0,W_n], v\in \{Mn\}\times [0,W_{n}]} \sup_{t} \inf_{x\in [0,Mn]} |\gamma_{uv}(t)-(x,0)|;$$
that is, $\mathcal{W}_{n,M}$ denotes the maximal transversal fluctuation for all conforming geodesics $\gamma_{uv}$ from $u\in \{0\}\times [0,W_n]$ to $v\in \{Mn\}\times [0,W_{n}]$. 

The next result shows that typically $\mathcal{W}_{n,M}$ is of the order $W_{Mn}$ and proves Lemma \ref{l:proxytrans:intro} (for the special case $w=0$ as described above). 

\begin{lemma}
\label{l:proxytrans}
There exist $\epsilon>0, \theta_1>0, z_0>0$ such that for all integers $M\geq 1$ and all $n\geq n(M)$ and $z\in [z_0,n^{\epsilon}]$ we have that
\[
\P\Big[\mathcal{W}_{n,M} \ge zW_{Mn}\Big] \leq \exp(-z^{\theta_1}).
\]
\end{lemma}

Recall also that, by definition, the canonical conforming geodesic from $u$ to $v$ where $u$ and $v$ are as above is a concatenation of paths $\gamma_i, i=1,2,\ldots, M$ where $\gamma_i$ is a path contained in $\Lambda_i$ between points $u_{i-1}$ and $u_i$ where $u_i\in \{in\}\times [-n^{\beta}W_n, n^{\beta} W_n]$ and $\gamma_i$ is a conforming geodesic between $u_{i-1}$ to $u_i$ (i.e., $\gamma_i$ minimises the length of all paths between $u_{i-1}$ to $u_i$ in the environment $\omega^{\Lambda_i}$).  Also, by definition 
$$\rX_{uv}=\sum_{i} X^{\Lambda_i}_{\gamma_i}.$$
For $i=1,2,\ldots, M$, and $\delta>0$ sufficiently small, let {$A_{i, \delta}$} denote the event that for each $u_{i-1}\in \{in\}\times [-n^{\beta}W_n, n^{\beta} W_n]$ and each $u_i\in \{in\}\times [-n^{\beta}W_n, n^{\beta} W_n]$ and the conforming geodesic $\gamma_{u_{i-1},u_{i}}$ from $u_{i-1}$ to $u_i$ we have 
$$\rX_{\gamma_{u_{i-1},u_{i}}}\ge X_{\gamma_{u_{i-1},u_{i}}}-n^{-\delta}Q_{n}.$$ 
We have the following result. 

\begin{lemma}
    \label{l:proxylocal1}
    There exists $\delta,\theta_1>0$ sufficiently small such that for all $n$ sufficiently large and for all $i$ we have
    $$\P(A_{i,\delta})\ge 1-\exp(-n^{\theta_1}).$$
\end{lemma}

Before proving Lemma \ref{l:proxylocal1}, let us complete the proof of Lemma \ref{l:proxytrans}. 

\begin{proof}[Proof of Lemma \ref{l:proxytrans}]
    Suppose that, for $u\in \ell_{0,0,W_n}$ and $v\in \ell_{Mn,0,W_n}$ the conforming geodesic from $u$ to $v$, $\gamma_{uv}\in \Upsilon_{Mn,u,v,z}$. Let the canonical points where this geodesic intersects the lines $x=in$ be denoted $u_i$. It follows that 
    $$\rX_{uv}=\rX_{\gamma_{uv}}=\sum_{i} \rX_{\gamma_{u_{i-1},u_{i}}}=\sum_{i} X^{\Lambda_i}_{\gamma_{u_{i-1},u_{i}}}.$$
    It follows from Lemma \ref{l:proxylocal1} that on an event of probability at least $1-Me^{-n^{\theta_1}}$ we have 
    $$\rX_{uv}\ge \sum_{i} X_{\gamma_{u_{i-1},u_i}}-Mn^{-\delta}Q_n.$$
    Since the concatenation of $\gamma_{u_{i-1},u_{i}}$'s ($=\gamma_{uv}$) belong to $\Upsilon_{Mn,u,v,z}$ it follows from \cite[Lemma 5.3]{BSS23} (a strengthening of Theorem \ref{t:tfold}) that on an event of probability $1-\exp(-z^{\theta_1})$ we have 
    $$\sum_{i} X_{\gamma_{u_{i-1},u_i}} \ge X_{uv}+zQ_{Mn}.$$ Combining these two estimates it follows that on an event $A$ with least $\P(A)\ge 1-\exp(-z^{\theta_1})$ we have 
    $$\rX_{uv}\ge X_{uv}+zQ_{Mn}-Mn^{-\delta}Q_n.$$
    Also, by Lemma \ref{l:proxy} we have on an event $B$ with $\P(B)\ge 1-\exp(-n^{\theta_1})$ we have 
    $$\rX_{uv}\le X_{uv}+n^{-\delta}Q_n.$$
    We therefore have a contradiction on the event $A\cap B$, completing the proof of the lemma.  
\end{proof}

\begin{center}
\begin{figure}
\includegraphics[width=1.5in]{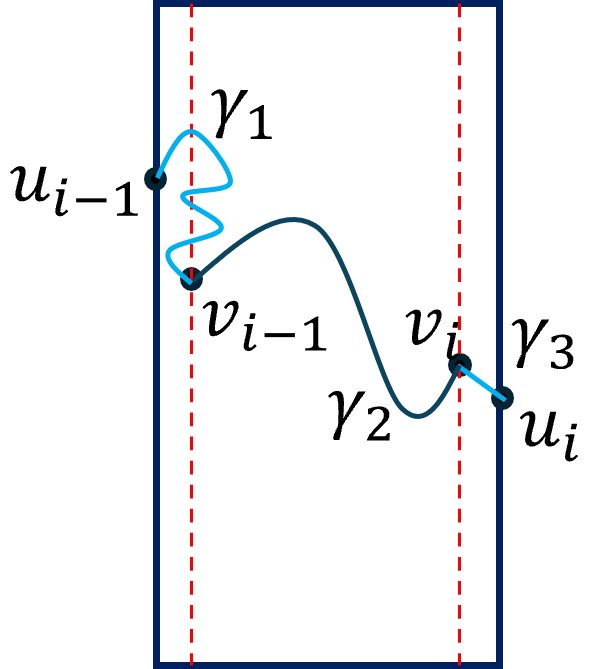}
\caption{Proof of Lemma \ref{l:proxylocal1}. The geodesic between $u_{i-1}$ and $u_i$ is divided into three parts $\gamma_1$ and $\gamma_3$ being the parts intersecting the boundary regions of width $n^\delta$, while $\gamma_2$ is the part in the middle. Since $\gamma_2$ has the same length in both the environments we need to show only that the lengths of $\gamma_1$ and $\gamma_3$ do not change much. This follows from showing that it is unlikely that the paths $\gamma_1$ or $\gamma_3$ are unlikely to travel much in the vertical direction.}
\label{f:transversal}
\end{figure}
\end{center}

\begin{proof}[Proof of Lemma \ref{l:proxylocal1}]
Let us fix $\delta>0$ sufficiently small such that $n^{3\delta}\le Q_{n}$.
The proof is done in a few steps. 

\noindent
\textbf{Step 1:} For the conforming geodesic $\gamma'=\gamma_{u_{i-1},u_i}$, let $v_{i-1}$ and $v_i$ denote points on $\gamma'\cap \{x=(i-1)n+n^{\delta}\}$ and $\gamma'\cap \{x=in-n^{\delta}\}$ such that the the restriction of $\gamma'$ between $v_{i-1}$ and $v_{i}$ is contained in the region $x\in [(i-1)n+n^{\delta}, in-n^{\delta}]$. Let $B_{i}$ denote the event that for all such $u_{i-1}, v_{i-1}$ , $u_{i}, v_{i}$ we have 
    $$|(u_{i-1}-v_{i-1})\cdot e_2|\le n^{\delta}; \quad |(u_{i}-v_{i})\cdot e_2|\le n^{\delta}.$$
    The first step is to show that $B_i\subseteq A_{i,\delta}$. 

    Indeed, observe that for each $u_{i-1},u_{i}$ as in the definition of $A_{i,\delta}$, denoting by $\gamma_1,\gamma_2, \gamma_3$ the parts of $\gamma'$ between $u_{i-1}$ to $v_{i-1}$, $v_{i-1}$ to $v_{i}$ and $v_{i}$ to $u_{i}$ (see Figure \ref{f:transversal}) respectively one has 
    $$\rX_{\gamma'}= X^{\Lambda_i}_{\gamma_1}+X^{\Lambda_i}_{\gamma_2}+X^{\Lambda_i}_{\gamma_3}.$$ 
    By definition of $v_{i-1}$ and $v_{i}$ we have that $X^{\Lambda_i}_{\gamma_2}=X_{\gamma_2}$. So it suffices to show that on the event $B_i$
    \begin{equation}
        \label{e:compare1}
        |X^{\Lambda_i}_{\gamma_1}+X^{\Lambda_i}_{\gamma_3}-X_{\gamma_1}-X_{\gamma_3}|\le n^{-\delta}Q_{n}.
    \end{equation}

 Observe that by the definition of our Riemannian FPP model, for any path $\zeta$, we have for constants $C_1,C_2\in (0,\infty)$
    $$C_1\ell(\zeta)\le X_{\zeta} \le C_2 \ell(\zeta)$$
    where $\ell(\zeta)$ denotes the Euclidean length of the curve $\zeta$
    and the same inequalities hold for $X_{\gamma}^{\Lambda_i}$ as well. By definition of $\gamma'$ (specifically the fact that $\gamma_1$ and $\gamma_3$ are geodesics between their respective endpoints) it therefore follows that there exists a constant $C_3$ such that $\ell(\gamma_1)\le C_3|u_{i-1}-v_{i-1}|, \ell(\gamma_3)\le C_3|u_{i}-v_{i}|$. Consequently we get for some $C_4>0$
     $$|X^{\Lambda_i}_{\gamma_1}-X_{\gamma_1}|\le C_4\ell(\gamma_1); \qquad |X^{\Lambda_i}_{\gamma_3}-X_{\gamma_3}|\le C_4\ell(\gamma_3).$$
    Observe now that by the definition of $B_i$ and the choice of $\delta$ we have 
    $|u_{i-1}-v_{i-1}|, |u_{i}-v_{i}|=O(n^{\delta})\ll n^{-\delta}Q_{n}$. This completes the proof of \eqref{e:compare1} and shows $B_i\subseteq A_{i,\delta}$. 
\medskip
    
    \noindent 
    \textbf{Step 2:} It remains to show that $\P(B_i)\ge 1-\exp(-n^{\theta_1})$. To this end we do a discretisation. Let $L_{i}$ denote the set $\{in\}\times ([-n^{\beta} W_n, n^{\beta} W_n] \cap \Z)$. For $u_{i-1}\in L_{i-1}, u_{i}\in L_{i}$, let $B_{u_{i-1},u_{i}}$ denote the event that $$|(u_{i-1}-v_{i-1})\cdot e_2|\le n^{\delta}/2; \quad |(u_{i}-v_{i})\cdot e_2|\le n^{\delta}/2$$
    where $v_{i-1}$ and $v_{i}$ are as defined above. Let 
    $$B'_{i}=\bigcap_{u_{i-1}\in L_{i-1}, u_{i}\in L_{i}} B_{u_{i-1},u_{i}} .$$
    Observe that for points $u_{i-1}, u_{i}$ and $u'_{i-1}, u'_{i}$ as in the statement of the lemma such that $u_{i-1}$ lies above $u'_{i-1}$ and $u_i$ lies above $u'_i$, planarity (and the fact that conforming geodesics are contained in $\Lambda_i$) implies that $\gamma_{u_{i-1}u_{i}}$ lies above $\gamma_{u'_{i-1}u'_{i}}$ (this is true because we always take the topmost geodesic as the canonical choice). Therefore, it follows that for $n$ sufficiently large,  $B'_{i}\subseteq B_{i}$.

    \medskip

   \noindent
   \textbf{Step 3:} We need to show $\P(B'_{i})\ge 1-\exp(-n^{\theta_1})$. Since the number of pairs $(u_{i-1},u_{i})$ in the definition of $B'_i$ is $o(n^2)$ it suffices to show that 
   $$\P(B_{u_{i-1},u_{i}})\ge 1-\exp(-n^{\theta_1})$$
   for each fixed such pair.  
   Fix $u_{i-1}=((i-1)n,y)\in L_{i-1}, u_{i}=(in,y')\in L_{i}$ for the rest of the proof. 
    Let {$H^+$} (resp.\ $H^{-}$) denote the 
   interval $\{(i-1)n+n^{\delta}\}\times [y-\frac{n^{\delta}}{2},y+\frac{n^{\delta}}{2}]$ (resp.\ $\{in-n^{\delta}\}\times [y-\frac{n^{\delta}}{2},y+\frac{n^{\delta}}{2}]$); see Figure \ref{f:transversal2}. Let $\upsilon_{u_{i-1},u_{i}}$ denote the set of all paths from $u_{i-1}$ to $u_{i}$ those intersect the lines $x=(i-1)n+n^{\delta}$ and $x=in-n^{\delta}$ only in the intervals $H^+$ and $H^{-}$ respectively. We shall show that on an event of probability at least $1-\exp(-n^{\theta_1})$ we have:
   \begin{enumerate}
       \item[(i)] $\inf_{\zeta: \zeta(0)=u_{i-1}, \zeta(1)=u_{i}:\zeta\notin \upsilon_{u_{i-1},u_{i},n^{\delta}, n^{\delta}/2}}X^{\Lambda_i}_{\zeta}> X^{\Lambda_i}_{u_{i-1},u_{i}}+ n^{\delta/100}Q_{n^{\delta}}$. 
       \item[(ii)] There is a conforming path $\zeta\in\upsilon_{u_{i-1},u_{i}}$ such that $X^{\Lambda_i}_{\zeta}<  X^{\Lambda_i}_{u_{i-1},u_{i}}+ n^{\delta/200}Q_{n^{\delta}}$. 
   \end{enumerate}

\begin{center}
\begin{figure}[htbp!]
\includegraphics[width=2in]{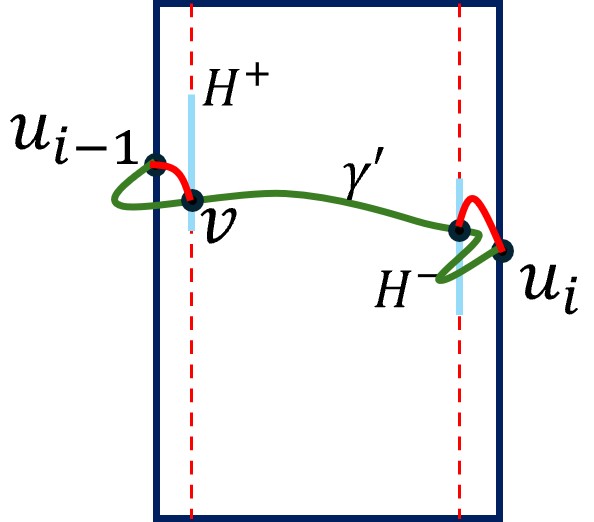}
\caption{Proof of Step 3 in Lemma \ref{l:proxylocal1}. We want to show that with large probability the conforming from $u_{i-1}$ to $u_{i}$ does not have too much of change in height within $n^{\delta}$ distance of either endpoint. This is shown by observing that any path from $u_{i-1}$ to $u_{i}$ which intersects the dotted red line outside the blue intervals will typically be much larger than the optimal path by local transversal fluctuation estimates. Then we also show that there exist with large probability a good conforming path from $u_{i-1}$ to $u_i$ without too much height change, completing the proof.} 
\label{f:transversal2}
\end{figure}
\end{center}
   
    Let us denote the event in (i) by $\widetilde{B}$.
    Since $W_{n^{\delta}}=O(n^{3\delta/4})$ it follows from \cite[Lemma 8.2]{BSS23} that $\P(\widetilde{B})\ge 1-\exp(-n^{\theta_1})$ for some $\theta_1>0$ (depending on $\delta$). 
    For (ii), let us denote by $\widetilde{H}^+$, (resp.\ $\widetilde{H}^-$) denote the subsets of $H^+$ (resp.\ $H^{-}$) with integer second co-ordinate. Let us consider the events 
    $$\widetilde{B}_1=\left\{\forall v\in \widetilde{H}^+: X^{\Lambda_i}_{u_{i-1},v} \ge |u_{i-1}-v|-n^{\delta/300}Q_{n^{\delta}}\right\};$$
    $$\widetilde{B}_2= \left\{\forall v\in \widetilde{H}^+, \exists \zeta\subset \Lambda_i, \gamma(0)=u_{i-1}, \gamma(1)=v~ \text{such that}~ X^{\Lambda_i}_{\zeta}\le |u_{i-1}-v|+ n^{\delta/300}Q_{n^{\delta}}\right\};$$
    $$\widetilde{B}_3=\left\{\forall v\in \widetilde{H}^-: X^{\Lambda_i}_{v,u_i} \ge |u_{i}-v|-n^{\delta/300}Q_{n^{\delta}}\right\};$$
    $$\widetilde{B}_4= \left\{\forall v\in \widetilde{H}^-, \exists \zeta\subset \Lambda_i, \gamma(0)=v, \gamma(1)=u_i~ \text{such that}~ X^{\Lambda_i}_{\zeta}\le |u_{i}-v|+ n^{\delta/300}Q_{n^{\delta}}\right\}.$$
    Observe that on $\widetilde{B}\cap \widetilde{B}_1\cap\widetilde{B}_2\cap\widetilde{B}_3\cap\widetilde{B}_4$ the event in (ii) happens. Indeed, $\widetilde{B}$ implies that the geodesic $\gamma'$ from $u_{i-1}$ to $u_{i}$ in the environment $\omega^{\Lambda_i}$ is in $\upsilon_{u_{i-1},u_{i}}$. Suppose its last intersection with $H^+$ is $v_1$ and its first intersection with $H^-$ is $v_2$. Now consider the conforming path obtained by concatenating the conforming path from $u_{i-1}$ to $v_1$ given by $\widetilde{B}_2$ (if needed, followed by a vertical segment of length at most 1), the restriction of $\gamma'$ between $v_1$ and $v_2$, and the conforming path from $v_2$ to $u_{i}$ given by $\widetilde{B}_4$. It is clear that on $\widetilde{B}\cap \widetilde{B}_1\cap\widetilde{B}_2\cap\widetilde{B}_3\cap\widetilde{B}_4$ this conforming path satisfies (ii). So it only remains to show that 
    $\P(\widetilde{B}_i)\ge 1-\exp(-n^{\theta_1})$ for each $i$. For $\widetilde{B}_1$ and $\widetilde{B}_3$ this follows from Theorem \ref{t:all} together with a union bound. For $\widetilde{B}_2$ and $\widetilde{B}_4$ this follows from Proposition \ref{p:constraine}. Combining these by a union bound the proof of the lemma is complete.  
\end{proof}

Observe that the proof above also shows the following. For any $1\le i <j \le M$, let $\zeta_{i}$ denote a conforming path between a point in $\ell_{(i-1)n,-n^{\beta}W_n, n^{\beta}W_n}$ and a point in $\ell_{in,-n^{\beta}W_n, n^{\beta}W_n}$. Let $\zeta$ denote a concatenation of such paths. Then by the proof above we have   
$$\rX_{\zeta}\ge X_{\zeta}-Mn^{-\delta}Q_{n}.$$
Since $Mn^{-\delta}Q_{n}\ll Q_{n}\le Q_{r}$ for all $r\ge n$ (and $n\gg M$) the same argument together Theorem \ref{t:tfold} gives the following corollary which gives a variant of Lemma \ref{l:proxytrans}. 

\begin{corollary}
    \label{c:proxytrans}
    For $k\in N$, $1\le i \le M-k$, $r=kn$, the interval $I_{k}= [-W_r, W_r]$ and for any $s\in [-n^{\beta} W_n/2, n^{\beta}W_n/2]$ let 
    $$\mathcal{W}_{n,r,s}:= \sup _{u\in \{in\}\times (s+I_{k}), u\in \{in+r\}\times (s+I_{k})}\sup_{t} |\gamma_{uv}(t)\cdot e_2-s|.$$
  There exists $\epsilon>0, \theta_1>0, z_0>0$ such that for all integers $M\geq 1$ and all $n\geq n(M)$ and $z\in [z_0,n^{\epsilon}]$ we have that
\[
\P\Big[\mathcal{W}_{n,r,s} \ge zW_{r}\Big] \leq \exp(-z^{\theta_1}).
\]
\end{corollary}

\subsection{Local transversal fluctuations}
Next we prove Lemma \ref{l:localtransproxy:intro}, the local transversal fluctuation estimate for the conforming geodesic. As mentioned earlier we shall prove a stronger estimate. For an integer $i\in [\frac{M}{2},M]$ and $j\in [-M^{4/5},M^{4/5}]$, let $\gamma_{0,i,j}$ denote the conforming geodesic from $0$ to $(in,jW_n)$. For $r\le M^{1/100}n$ define the events 
$$A^{R,i,j}_{r,z}=\{\{\exists (r,s)\in \gamma_{0,i,j}: |s|\ge zW_r\}; $$
$$A^{L,i,j}_{r,z}=\{\{\exists (in-r,s)\in \gamma_{0,i,j}: |s-jW_n|\ge zW_r\}.$$
Similarly, for $i\in [0,\frac{M}{2}]$, and $j$ as before let $\gamma_{i,j,Mn}$ denote the conforming geodesic from $(in,jW_n)$ to $(Mn,0)$. For $r\le M^{1/100}n$ define the events 
$$\tilde{A}^{R,i,j}_{r,z}=\{\{\exists (in+r,s)\in \gamma_{i,j,Mn}: |s-jW_n|\ge zW_r\}; $$
$$\tilde{A}^{L,i,j}_{r,z}=\{\{\exists (Mn-r,s)\in \gamma_{0,i,j}: |s|\ge zW_r\}.$$
We have the following lemma. 

\begin{lemma}
    \label{c:localproxytrans}
    There exist $\epsilon'>0, \theta_2>0, z_0>0$ such that for all integers $M\geq 1$ and all $n\geq n(M)$ and $z\in [z_0,n^{\epsilon'}]$  we have 
    \begin{itemize}
        \item[(i)] For all $i\in [\frac{M}{2},M]$ and $j\in [-M^{4/5},M^{4/5}]$  
        $$\P(A^{R,i,j}_{r,z}), \P(A^{L,i,j}_{r,z})\le \exp(-z^{\theta_2}).$$
        \item[(ii)] For all $i\in [0,\frac{M}{2}]$ and $j\in [-M^{4/5},M^{4/5}]$  
        $$\P(\tilde{A}^{R,i,j}_{r,z}), \P(\tilde{A}^{L,i,j}_{r,z})\le \exp(-z^{\theta_2}).$$
    \end{itemize}
\end{lemma}

Notice that Lemma \ref{l:localtransproxy:intro} is a special case of Lemma \ref{c:localproxytrans} so the former lemma follows from the latter. 

\begin{proof}[Proof of Lemma \ref{c:localproxytrans}]
    Observe that by the reflection symmetry of the model (and since $M$ is an integer) it suffices to prove only $(i)$. Although due to the definition of conforming paths one cannot quite obtain the estimate for $\P(A^{L,i,j}_{r,z})$ from that of $\P(A^{R,i,j}_{r,z})$ by symmetry, we shall only provide the bound for $\P(A^{R,i,j}_{r,z})$ and it will be clear that the same argument would prove the upper bound for $\P(A^{L,i,j}_{r,z})$ as well. 

    Fix $i,j$ as in the definition of $A^{R,i,j}_{r,z}$. It follows from \cite[Lemma 8.2]{BSS23} that for $z$ sufficiently large, on an event of probability at least $1-\exp(-z^{\theta_2})$ we have for all paths $\zeta$ with $\zeta(0)=0$ and $\zeta(1)=(in,jW_n)$ such that there exists $(r,s)$ on $\zeta$ with $|s|\ge zW_r$ we have 
    $$X_{\zeta}\ge X_{0,(in,jW_n)}+z^{1/4}Q_r.$$ 
    Observe that the geodesic $\gamma_{0,i,j}$ is a union of paths $\{\gamma_{k}\}_{1\le k\le i}$ where each $\gamma_k$ is a conforming geodesic from a point on $(k-1)n\times \R$ to a point on $kn\times \R$. 
    It follows from Lemma \ref{l:proxylocal1} (and taking a union bound over all $k$) that on an event of probability $1-M\exp(-n^{\theta_2})$ one has that 
    $\rX_{\gamma_{0,i,j}}\ge X_{\gamma_{0,i,j}}-Mn^{-\epsilon}Q_n$. By taking a union bound (and adjusting the value of $\theta_2$) if necessary we get $\P(A^{R,i,j}_{r,z})\le \exp(-z^{\theta_2})$, as desired.
\end{proof}

Combining Lemma \ref{l:proxytrans} and Corollary \ref{c:localproxytrans} we get the following result which will be needed in the proof of Theorem \ref{t:Pplus}. 

\begin{lemma}
    \label{l:2.8}
    Fix $r=r_{\ell}=2^{\ell (\log_2\log_2 M)^5}n$ for $1\le \ell \le \ell_{\max}$. Recall the definition of $J_{i}^{n,M,\ell}$. Then there exists $H$ such that we have with probability at least $1-M^{-10000}$, for all $i$ with $2M^{99/100}n\le ir_{\ell}\le (M-2M^{99/100})n$, we have 
    $$W_{r_{\ell}}|J^{n,M,\ell}_{i}-J^{n,M,\ell}_{i+\Phi}| \le (\log M)^{H}W_{\Phi r_\ell}.$$
\end{lemma}

\begin{proof}
    Observe first that Lemma \ref{l:proxytrans} implies that on an event $A$ of probability at least $1-M^{-100002}$ we have 
    $W_{r_\ell}|J^{n,M,\ell}_i|\le M^{4/5}W_n$ for all $i$ (here we have used that $W_{Mn}=O(M^{3/4}W_n)$ and $M$ is chosen sufficiently large). We next claim that for each $i$ as in the statement of the lemma we have for $z$ sufficiently large
    \begin{equation}
        \label{e:phitrans}
        \P\left(\{W_{r_{\ell}}|J^{n,M,\ell}_{i}-J^{n,M,\ell}_{i+\Phi}| \ge zW_{\Phi r_\ell}\} \cap A\right) \le 2M^{4/5}e^{-z^{\theta_2}}
    \end{equation}
    where $\theta_2$ is as in Lemma \ref{c:localproxytrans}.  
    Clearly, by choosing $H$ sufficiently large setting $z=(\log M)^{H}$ in \eqref{e:phitrans} and taking a union bound over all $i$ the lemma follows. 

    It therefore only remains to prove \eqref{e:phitrans}. Let us assume that $(i+\Phi)r_{\ell}\ge \frac{Mn}{2}$ (the other case can be treated by essentially the same argument, we shall omit the details).

    For $j\in [-M^{4/5},M^{4/5}]$, consider the event 
    $$B_j=A^{L,(i+\Phi)r_{\ell}/n,j}_{\Phi r_{\ell}, z+1}\cap A^{L,(i+\Phi)r_{\ell}/n,j+1}_{\Phi r_{\ell}, z+1}.$$
    Since the conforming geodesic from $0$ to any point in the interval $\{(i+\Phi)r_{\ell}\}\times [jW_n,(j+1)W_n]$ is sandwiched between the geodesics $\gamma_{0,(i+\Phi)r_{\ell}/n,j}$ and $\gamma_{0,(i+\Phi)r_{\ell}/n,j+1}$ it follows that 
    $$\{W_{r_{\ell}}|J^{n,M,\ell}_{i}-J^{n,M,\ell}_{i+\Phi}| \ge zW_{\Phi r_\ell}\} \cap A \subseteq \bigcup_{j=-M^{4/5}}^{M^{4/5}} B_j.$$
    The equation \eqref{e:phitrans} now follows from Lemma \ref{c:localproxytrans} by taking a union bound over all $j$. This completes the proof of the lemma. 
\end{proof}

Finally we prove Proposition \ref{p:constrainerX}. 

\begin{proof}[Proof of Proposition \ref{p:constrainerX}]

{Proof of this proposition follows by verbatim repeating the proof of Proposition \ref{p:constraine} with replacing $X$ by $\rX$. Each concentration estimate for $X$ (Theorem \ref{t:all}) is replaced by Proposition \ref{p:paraestimateconforming} and each transversal fluctuation estimate for $X$ (Theorem \ref{t:tfold}) is replaced by Lemma \ref{l:proxytrans:intro}. We omit the details.}
\end{proof}

\section{Stretched exponential polymer estimates}
\label{s:perc}
In this section we shall prove the polymer estimate Proposition \ref{p:perc1} and also use it to establish the $\tau_2$ fluctuation bound for the geodesic, Proposition \ref{p:tau2perc}.

Clearly it suffices to prove this lemma for $R$ and $z$ sufficiently large, we shall henceforth assume that. The first step is to  truncate $\mathcal{V}_{i,k_{i-1},k_{i}}$ into dyadic intervals. For $j\ge 1$, set $$\mathcal{V}^{j}_{i,k,k'}=2^{j+1}I(\mathcal{V}_{i,k,k'}\in [2^{j},2^{j+1}]).$$ Therefore we have,
\begin{equation}
    \label{e:zsum}
    \max_{\substack{\uk\in \mathfrak{K}_{M}\\ \tau_2(\uk)\leq R M}}\sum_{i=1}^{M} \mathcal{V}_{i,k_{i-1},k_i}\le M+\sum_{j\ge 1} \max_{\substack{\uk\in \mathfrak{K}_{M}\\ \tau_2(\uk)\leq R M}}\sum_{i} \mathcal{V}^{j}_{i,k_{i-1},k_i}.
\end{equation}
We shall bound the terms 
$$\max_{\substack{\uk\in \mathfrak{K}_{M}\\ \tau_2(\uk)\leq R M}}\sum_{i} \mathcal{V}^{j}_{i,k_{i-1},k_i}$$
separately for three different regimes of $j$. Let us set 
$$j_{\max}=\lceil \log_2 (\log^C(RM)+z)\rceil~\text {   and  }~ j_{\min}=\lfloor \log_2 (R^{3/4}) \rfloor$$
where $C$ is chosen sufficiently large (depending on $C_1,C_2,\xi$).
Notice that, we have, deterministically, for some $C_4>0$ 
\begin{equation}
    \label{e:jmin}
    \sum_{j<j_{\min}} \max_{\substack{\uk\in \mathfrak{K}_{M}\\ \tau_2(\uk)\leq R M}} \sum_{i} \mathcal{V}^{j}_{i,k_{i-1},k_i}\le C_4R^{3/4}M
\end{equation}
for all $R$ sufficiently large. 

For $j>j_{\max}$ we have the following lemma. 

\begin{lemma}
    \label{l:jmax}
    There exist $C_5,C_6>0$ depending only on $C_1$, $C_2$ and $\xi$ such that 
    $$\P\left(\sum_{j> j_{\max}} \max_{\underline{k}:\tau_2(\uk)\le RM}\sum_{i} \mathcal{V}^{j}_{i,k_{i-1},k_i}> R^{3/4}M\right)\le C_5\exp(-C_6z^{\xi/4}).$$
\end{lemma}

\begin{proof}
    First let us observe that for $\uk\in \fK$ with $\tau_2(\uk)\le RM$
    we have
    $$\sum_i\sum_{j>j_{\max}} \mathcal{V}^{j}_{i,k_{i-1},k_i}\le \sum_i |k_i-k_{i-1}|^{3/2}+\sum_{i} \sum_{j>j_{\max}} \mathcal{V}^{j}_{i,k_{i-1},k_i} I(\mathcal{V}_{i,k_{i-1},k_i} \ge \frac{1}{2}|k_i-k_{i-1}|^{3/2}).$$
    Further, notice that we have 
    $$\frac{1}{M}\sum_i |k_i-k_{i-1}|^{3/2} \leq \Big(\frac{1}{M}\sum_i |k_i-k_{i-1}|^{2}\Big)^{3/4}\le R^{3/4}$$
    since $\tau_2(\uk)\le RM$, and therefore we have 
    $$\sum_i\sum_{j>j_{\max}} \mathcal{V}^{j}_{i,k_{i-1},k_i}\le R^{3/4}M+\sum_{i} \sum_{j>j_{\max}} \widetilde{\mathcal{V}}^{j}_{i,k_{i-1},k_i}$$
    where 
    $$ \widetilde{\mathcal{V}}^{j}_{i,k_{i-1},k_i}=\mathcal{V}^{j}_{i,k_{i-1},k_i} I(\mathcal{V}_{i,k_{i-1},k_i} \ge \frac{1}{2}|k_i-k_{i-1}|^{3/2}).$$
    It therefore suffices to show that 
    $$\P\left(\sum_{j> j_{\max}} \max_{\substack{\uk\in \mathfrak{K}_{M}\\ \tau_2(\uk)\leq R M}} \sum_{i} \widetilde{\mathcal{V}}^{j}_{i,k_{i-1},k_i}> 0\right)\le C_5\exp(-C_6z^{\xi/4}).$$
    Now clearly, for $j>j_{\max}$,  we have 
    $$\P(\widetilde{\mathcal{V}}^{j}_{i,k_{i-1},k_i}>0)\le \P(\mathcal{V}_{i,k_{i-1},k_i}\ge 2^{j} \vee \frac{1}{2}|k_i-k_{i-1}|^{3/2}).$$
    Observe next that 
    $$2^{j} \vee \frac{1}{2}|k_i-k_{i-1}|^{3/2} \ge 2^{j/4}(1+|k_i-k_{i-1}|^{1+\delta})$$
    since $\delta<\frac{1}{100}$ and $R$ is sufficiently large. Indeed, if $2^{j}>\frac{1}{2}|k_i-k_{i-1}|^{3/2}$ then we have 
    $2^{3j/4}\ge (1+|k_i-k_{i-1}|^{1+\delta})$ since $\delta<1/100$ and $j_{\max}$ is sufficiently large (since $R$ is sufficiently large). On the other hand if $\frac{1}{2}|k_i-k_{i-1}|^{3/2}\ge 2^{j}$ we have 
    $$\frac{1}{2^{1/4}}|k_i-k_{i-1}|^{3/8}\ge 2^{j/4}$$ and the claim follows by observing 
    $$ \frac{1}{2^{3/4}}|k_i-k_{i-1}|^{9/8} \ge 1+|k_i-k_{i-1}|^{\delta}$$
    since $\delta<1/100$ and $j_{\max}$ is sufficiently large. It therefore follows from the hypothesis on the tails of $\mathcal{V}_{i,k,k'}$ that 

    $$\P(\widetilde{\mathcal{V}}^{j}_{i,k_{i-1},k_i}>0) \le C_1\exp(-C_22^{j\xi/4}).$$

    Since $k_0=0$, and $\tau_2(\uk)\le RM$, it is clear that $|k_i|\le M\sqrt{RM}$ for each $i$. So it is easy to see that there exists a deterministic set $\mathsf{A}=\mathsf{A}_{R,M}$ of triples $(i,k_{i-1},k_{i})$ of size at most $4R^{1/2}M^{5/2}$ such that for any $\underline{k}\in \mathfrak{K}_{M}$ with $\tau_2(\underline{k})\le RM$ we have $(i,k_{i-1},k_{i})\in \mathsf{A}$. Taking now a union bound over the elements of the set $\mathsf{A}$ it follows that for $j>j_{\max}$, 
    $$\P\left(\max_{\substack{\uk\in \mathfrak{K}_{M}\\ \tau_2(\uk)\leq R M}} \sum_{i} \widetilde{\mathcal{V}}^{j}_{i,k_{i-1},k_i}>0\right)\le \P\left(\max_{(i,k_{i-1},k_{i})\in \mathsf{A}} \widetilde{\mathcal{V}}^{j}_{i,k_{i-1},k_i}>0\right)\le 4R^{1/2}M^{5/2}C_1\exp(-C_22^{j\xi/4}).$$
    Summing over all $j>j_{\max}$ we get 
    $$\P\left(\sum_{j> j_{\max}} \max_{\substack{\uk\in \mathfrak{K}_{M}\\ \tau_2(\uk)\leq R M}}\sum_{i} \widetilde{\mathcal{V}}^{j}_{i,k_{i-1},k_i}> 0\right)\le 4R^{1/2}M^{5/2}C'_1\exp(-C'_22^{j_{\max}\xi/4}).$$
    Since $2^{j_{\max}}\ge \log^C (RM)$ where $C$ is sufficiently large (depending on $\xi$) it follows that 
    $$\log (4R^{1/2}M^{5/2}) \le \frac{1}{2}C'_22^{j_{\max}\xi/4}$$
    and consequently
    $$\P\left(\sum_{j> j_{\max}} \max_{\substack{\uk\in \mathfrak{K}_{M}\\ \tau_2(\uk)\leq R M}} \sum_{i} \widetilde{\mathcal{V}}^{j}_{i,k_{i-1},k_i}> 0\right)\le C_5\exp(-C_62^{j_{\max}\xi/4}).$$ 
    The lemma follows from observing that $2^{j_{\max}}\ge z$ from the definition of $j_{\max}$.
\end{proof}

The next lemma deals with the case $j_{\min} \le j \le j_{\max}$.

\begin{lemma}
    \label{l:minmax}
    There exist constant $C_3,C_4,C_5,C_6>0$ depending only on $C_1,C_2$ and $\xi$ such that 
    $$\P\left( \sum_{j=j_{\min}}^{j_{\max}} \max_{\underline{k}:\tau_2(\uk)\le RM} \sum_i \mathcal{V}^{j}_{i,k_{i-1},k_{i}} \ge (C_3+C_4R^{3/4})M+z \right)\le C_5\exp(-C_6z^{\xi/4}).$$
\end{lemma}

Postponing the proof of Lemma \ref{l:minmax} momentarily, let us complete the proof of Proposition \ref{p:perc1}. 

\begin{proof}[Proof of Proposition \ref{p:perc1}]
The proof follows from \eqref{e:zsum}, together with \eqref{e:jmin}, Lemma \ref{l:jmax} and Lemma \ref{l:minmax} by taking $R$ sufficiently large. 
\end{proof}

\subsection{Proof of Lemma \ref{l:minmax}}
The proof of this lemma similar to the proof of \cite[Lemma 13.2]{BSS23}, but is somewhat more involved. We shall divide the proof into a number of steps. 

\noindent
\textbf{Step 1: Mesoscopic coarsening of $\underline{k}$:}
To reduce the entropy of the the number of sequences $\underline{k}$ we shall consider the following discretisation which we call the mesoscopic coarsening (or $j$-mesoscopic coarsening). Note that the notion here is slightly different from the one in \cite{BSS23}.   
Fix $j_{\min}\leq j\leq j_{\max}$. Let $w=w_j$ be a fixed and large integer, to be fixed later. 

For $\underline{k}\in \mathfrak{K}_{M}$ define 
$\underline{k}^{w}$ by $k^{w}_{i}= \lfloor \frac{k_{iw^2}}{w^3}\rfloor$ for $i=1,2,\ldots \lfloor\frac{M}{w^2}\rfloor$. Define $\underline{k}_{\rm large}$ by 
$$\underline{k}_{\rm large}=\{i: 1\le i \le M: |k_{i}-k_{i-1}|\ge w\},$$
i.e., the set of locations where $\underline{k}$ has a large jump. Finally, let 
$$\underline{k}^1_{\rm large}:=\{k_{i}:i\in \underline{k}_{\rm large}~\text{or}~(i+1)\in \underline{k}_{\rm large}\}.$$ 
The mesoscopic coarsening of $\underline{k}$ is given by the triple 
$$\underline{k}_{\rm M}^{w}=(\underline{k}^w,\underline{k}_{\rm large},\underline{k}^1_{\rm large}).$$ 

We need the following estimate to count the number of distinct $\underline{k}_{\rm M}^{w}$ as $\underline{k}$ varies over all $\underline{k}\in \mathfrak{K}_{M}$ with $\tau_2(\underline{k})\le RM$. 

\begin{lemma}
\label{l:mesocount}
Let $\mathsf{M}^{w}_{R}$ denote the set of all $\underline{k}_{\rm M}^{w}$
for $\underline{k}\in \mathfrak{K}_{M}$ with $\tau_2(\underline{k})\le RM$. Then there exists a constant $c>0$ such that 
$$|\mathsf{M}^w_{R}|\leq  \exp \left(cRM\log w/w^2\right).$$
\end{lemma}

\begin{proof}
First, let us fix the size of $\underline{k}_{\rm large}$ to be $s$; denote the corresponding subset of $\mathsf{M}^{w}_{R}$ by $\mathsf{M}^{w}_{R}(s)$. Clearly, $s\le RM/w^2$ since $\tau_2(\underline{k})\le RM$.

To bound $|\mathsf{M}^w_{R}(s)|$ for $s\le RM/w^{2}$, first fix the elements of $\underline{k}_{\rm large}$; clearly there are $\binom{M}{s}$ many choices. Then we fix the sizes of big jumps, i.e., ${k}_i-k_{i-1}$ for all $i\in \underline{k}_{\rm large}$. Since the sum total of the absolute values of these jumps can be at most $RM$ (actually one could improve upon this by Cauchy-Schwarz but we do not need it), by a standard counting argument, the number of choices here is at most $2^s \binom {RM +s}{s}$ (the factor $2^{s}$ comes from the fact that each jump can be either positive or negative). 
    
It remains to determine $k_{i-1}$'s for $i\in \underline{k}_{\rm large}$ and $\underline{k}^w$; we determine the number of choices for them together. Notice first that $k^{w}_{0}=0$ and if $k^{w}_{i_0}$ is fixed for some $i_0$, then we know $k_{i_0w^2}$ up to an error of $\pm w^{3}$. Now traversing the block $[i_0w^2,(i_0+1)w^2]$ from left to right, for the first $i\in \underline{k}_{\rm large}$, such that $i-1\in [i_0w^2,(i_0+1)w^2]$, we can determine $k_{i-1}$ up to an error of $\pm 2w^3$. Since we have already fixed the sizes of the large jumps, fixing $k_{i-1}$ as above also fixes $k_{i}$, and we can continue traversing the block $[i_0w^2,(i_0+1)w^2]$. Continuing this way we can determine $k_{(i_0+1)w^2}$ up to an error of $\pm 2w^3$ and hence $k^{w}_{i_0+1}$ is determined up to an error of $\pm 2$. Therefore, the total number of choices we have for this is at most $5^{M/w^2}(4w^3)^s$.   

Putting all these together, we get 
\begin{eqnarray*}
|\mathsf{M}^{w}_{R}(s)|  &\leq & 8^s5^{M/w^2}\binom{M}{s} \binom{RM+s}{s} w^{3s}\\
& \leq & \exp \biggl( c(s+M/w^2 +s\log (M/s)+s\log (RM/s +1)+3s\log w)\biggr)\\
& \leq & \exp \left(cRM \log w/w^2\right)
\end{eqnarray*}
for some constant $c>0$ where in the last inequality we have used the fact that $s\le RM/w^2$. Summing over $s$ from $0$ to $RM/w^2$ we get the required result. 
\end{proof}

\noindent
\textbf{Step 2: Bounding the tails for a fixed $j\in [j_{\min},j_{\max}]$:}
The usefulness of defining the mesoscopic coarsening is shown in the following lemma. 

\begin{lemma}
\label{l:fixedj}
Let $j\in [j_{\min}, j_{\max}]$ be fixed. Fix $\underline{k}^*\in \mathsf{M}^w_{R}$. Then, 
$$\max_{\underline{k}:\underline{k}_{\rm M}^{w}=\underline{k}^*} \sum_{i} \mathcal{V}^j_{i,k_{i-1},k_{i}}$$ 
is stochastically dominated by 
$$2^{j+1}\left( \frac{RM}{w^2}+\mathrm{Bin}(M,8w^4p(w))\right)$$
where $p(w)=\max_{|k-k'|\le w}\P(\mathcal{V}_{i,k,k'}\ge 2^{j})$.
\end{lemma}

\begin{proof}
Once $\underline{k}^*\in \mathsf{M}^w_{R}$ is fixed, we know that set $I$ of the locations of jumps larger than $w$. For each $i\in I$ we upper bound 
$\mathcal{V}^j_{i,k_{i-1},k_{i}}$ trivially by $2^{j+1}$. Since $|I|\le RM/w^2$ this gives the first term in the statement of the lemma. 
Clearly, it suffices to show now that for $i=1,2,\ldots, M$, $i\notin I$ 
$\max_{\underline{k}:\underline{k}_{\rm M}^{w}=\underline{k}^*}I(\mathcal{V}^j_{i,k_{i-1},k_{i}}\ne 0)$ are stochastically dominated by independent $\mbox{Ber}(8w^4p(w))$ variables.

Once $\underline{k}_{\rm M}^{w}=\underline{k}^*$, is fixed, for each $i\notin \underline{k}_{\rm large}$, the choice of $(k_{i-1},k_{i})$ is fixed in a deterministic set of size at most $8w^4$. To see this, if $i\notin \underline{k}_{\rm large}$, then given $\underline{k}^{w}$, $\underline{k}_{\rm large}$ and 
    $\underline{k}^1_{\rm large}$, $k_{i-1}$ is determined up to an error of $4w^3$ as in the proof of Lemma \ref{l:mesocount}, and fixing this, $k_{i}$ can take one of the $2w$ possible values since $|k_{i}-k_{i-1}|\le w$. 

Therefore, using the tail hypothesis on $\mathcal{V}_{i,k_{i-1},i}$ for each $i\notin I$, and the definition of $p(w)$ together with a union bound we get 
    $$\P\left(\max_{\underline{k}:\underline{k}_{\rm M}^{w}=\underline{k}^*}\mathcal{V}^j_{i,k_{i-1},k_{i}}>0\right)\le 8w^4p(w).$$
    Since $\mathcal{V}_{i,k_{i-1},i}$ are independent across $i$, the claim and hence the lemma follows. 
 \end{proof}

We shall next prove tail bounds for 
$$\max_{\underline{k}:\tau_2(\underline{k})\le RM} \sum_{i} \mathcal{V}^j_{i,k_{i-1},k_{i}}$$
for a fixed $j\in [j_{\min},j_{\max}]$. For this we need to specify the value of $w_j$. For $j_{\min}\le j\le j_{\max}$ we set 
$$w=w_j=2^{7j/10}.$$
Also, for $z>0$ we set 
$$z_j=\frac{2^{j+1} RM}{w_j^2}+ 2^{-j/100}R^{3/4}M + (z+\log^C(RM))2^{\varepsilon(j-j_{\max})}$$
for some $\varepsilon>0$ to be fixed later and where $C$ is as in the definition of $j_{\max}$. For notational convenience we shall write $\widetilde{z}=(z+\log^C(RM))$ and $\widetilde{z}_j=(z+\log^C(RM))2^{\varepsilon(j-j_{\max})}$. 
 Notice also that for this choice of $w$, we have, using the hypothesis on $\mathcal{V}_{i,k,k'}$ and the fact that $\delta<1/100$, that 
 $$8w_j^{4}p(w_j)\le C_1\exp(-C_22^{j\xi/4})$$
 for $j$ sufficiently large (which is ensured by taking $R$ sufficiently large and recalling $j\ge j_{\min}$.)

\begin{lemma}
    \label{l:minmaxj}
    There exists $c>0$, such that for $j\in [j_{\min},j_{\max}]$,
    $$\P\left(\max_{\underline{k}:\tau_2(\underline{k})\le RM} \sum_{i} \mathcal{V}^j_{i,k_{i-1},k_{i}} \ge z_j\right)\le \exp\left(-c2^{j_{\max}\xi/4} \right).$$ 
\end{lemma}

Postponing the proof of Lemma \ref{l:minmaxj} for now, we first complete the proof of Lemma \ref{l:minmax}. 

\noindent
\textbf{Step 3: Completing the proof of Lemma \ref{l:minmax}:}
First note the it suffices to prove the lemma for $z,R$ and $M$ sufficiently large. Notice that for $R$ and $M$ sufficiently large
$$\sum_{j=j_{\min}}^{j_{\max}} z_{j} \le C_4R^{3/4}M+C_7z$$
for some $C_4,C_7\in (0,\infty)$.
Therefore, it suffices to upper bound 
$$\sum_{j=j_{\min}}^{j_{\max}} \P\left(\max_{\underline{k}:\tau_1(\underline{k})\le RM} \sum_{i} \mathcal{V}^j_{i,k_{i-1},k_{i}} \ge z_j\right)$$
at which point we use Lemma \ref{l:minmaxj}.

Notice that by definition 
$$2\widetilde{z} \ge 2^{j_{\max}} \ge \widetilde{z}.$$ 
Therefore, we have 
$$\exp\left(-c2^{j_{\max}\xi/4} \right)\le \exp(-c\widetilde{z}^{\xi/4}).$$
Notice also that $j_{\max}\le \log_2 (2\widetilde{z})$ and we get by Lemma \ref{l:minmaxj}
$$\sum_{j=j_{\min}}^{j_{\max}} \P\left(\max_{\underline{k}:\tau_1(\underline{k})\le RM} \sum_{i} \mathcal{V}^j_{i,k_{i-1},k_{i}} \ge z_j\right) \le \log_2 (2\widetilde{z})\exp(-c\widetilde{z}^{\xi/4}).$$
The conclusion of the lemma follows from noting $\widetilde{z}\ge z$ and $z$ is sufficiently large. \qed

We finish by proving Lemma \ref{l:minmaxj}.

\noindent 
\textbf{Step 4: Proof of Lemma \ref{l:minmaxj}:}
Using Lemma \ref{l:fixedj}, definition of $z_j$, and the choice of  $w_j$,  the upper bound on $8w^{4}p(w)$, and taking a union bound over all choices of $\underline{k}^{w}_{\rm M}$, and using Lemma \ref{l:mesocount}, it follows that 
\begin{align*}
    \P\left(\max_{\underline{k}:\tau_2(\underline{k})\le RM} \sum_{i} \mathcal{V}^j_{i,k_{i-1},k_{i}} \ge z_j\right) & \le \exp(cRM\log w/w^2)\\ &\times \P\left(\mbox{Bin}(M,C_1\exp(-C_22^{j\xi/4}))\ge R^{3/4}M2^{-(1.01j+1)}+\widetilde{z}_j2^{-(j+1)}\right).
\end{align*}

Using a Chernoff bound we get that this is further upper bounded by

$$\exp\left(\frac{cRM\log w}{w^2}-\frac{R^{3/4}M\log (C_1^{-1}2^{-1.01j-1}\exp(C_22^{j\xi/4}))}{2^{1.01j+1}}-\frac{\widetilde{z}_j\log (C_1^{-1}2^{-1.01j-1}\exp(C_22^{j\xi/4}))}{2^{j+1}}\right).$$

Our first claim is that for $j$ sufficiently large (which is ensured by noting $j>j_{\min}$ and choosing $R$ sufficiently large) we have 
$$ \frac{R^{3/4}M\log (C_1^{-1}2^{-1.01j-1}\exp(C_22^{j\xi/4}))}{2^{1.01j+1}} \ge \frac{cRM \log w}{w^2}.$$
Indeed, note that by our choice of $w$, $\log w$ is linear in $j$, whereas 
$\log (C_1^{-1}2^{-1.01j-1}\exp(C_22^{j\xi/4}))$ grows exponentially in $j$. Therefore it suffices to verify that 
$$\frac{R^{1/4}}{w^2}\le \frac{1}{2^{1.01j+1}}.$$
Substituting the value of $w_j$ this reduces to 
$$R^{1/4}\le 2^{0.39j-1} $$
and this follows since $R^{1/4}\le 2^{(j_{\min}+1)/3}\le 2^{(j+1)/3}$ and $j$ is sufficiently large. 

Therefore the required probability is upper bounded by 

$$\exp\left(-\frac{\widetilde{z}_j\log (C_1^{-1}2^{-1.01j-1}\exp(C_22^{j\xi/4}))}{2^{j+1}}\right).$$
Again, for $j$ sufficiently large 
$$ \log (C_1^{-1}2^{-1.01j-1}\exp(C_22^{j\xi/4})) \ge c2^{j\xi/4}$$
for some $c>0$ and hence this is further upper bounded by 
$$\exp\left(-\frac{c\widetilde{z}2^{\varepsilon(j-j_{\max})}2^{j\xi/4}}{2^{j+1}}\right).$$
Using $\widetilde{z}\ge 2^{j_{\max}-1}$ the above expression is upper bounded by 

$$\exp\left(-\frac{c2^{(j-j_{\max})(\xi/4+\varepsilon-1)}2^{j_{\max}\xi/4}}{4}\right).$$

Now choose $\varepsilon$ sufficiently small so that $\varepsilon+\xi/4-1<0$, since $j\le j_{\max}$ this implies that the above expression is upper bounded by 
$$\exp\left(-\frac{c}{4}2^{j_{\max}\xi/4}\right)$$
and the proof is complete. 
\qed

\subsection{Bounding $\tau_2$ fluctuation of the conforming geodesic}
Recall that for the conforming geodesic $\gamma$ from $(0,0)$ to $(Mn,0)$ we defined $k_i=k_i(\gamma):=\lfloor \frac{y_i}{W_n} \rfloor$ where $(in,y_i)$ denote the {canonical} points where $\gamma$ intersects the line $x=in$. Recall that Proposition \ref{p:tau2perc} requires us to show that setting
$$\tau_2(\gamma):=\tau_2(\uk)= \sum_i (k_i-k_{i-1})^2$$ there exists some constant $C_7$ such that $\tau_2(\gamma)-C_7M$ has stretched exponential tails.

To prove Proposition \ref{p:tau2perc} we need the following corollary of Proposition \ref{p:perc1}. 

\begin{corollary}
\label{c:tau2perc}
For any $\cZ_{i,k,k'}$ satisfying the hypothesis of Proposition~\ref{p:perc1} and for $\lambda>0$, there exist $C'_7, C_8, C_9$ such that
\[
\P\left[\max_{\uk\in \mathfrak{K}_{M}} \bigg( \sum_{i=1}^M \mathcal{V}_{i,k_{i-1}, k_{i}} -  \lambda\tau_2(\uk)\bigg)\geq C'_7(1+\lambda^{-5})M  + z \right] \leq C_8\exp\left(-C_9 z^{\xi/4}\right ).
\]
\end{corollary}

\begin{proof}
Let $C_3,C_4$ be as in the statement of Proposition \ref{p:perc1}. If $\lambda\geq 1$ then for large enough $C'_7$ we have that for all $x\geq 1$,
\[
C'_7 + \frac14 \lambda x \geq C'_7 + \frac14  x > C_3 +C_4x^{3/4}.
\]
For $0<\lambda<1$, if $x\geq \lambda^{-5}>1$ then provided $C'_7$ is large enough,
\[
C'_7 + \frac14 \lambda x \geq C_3 +  C_4x^{3/4}
\]
while if $1\leq x\leq \lambda^{-5}$ then again provided $C'_7$ is large enough,
\[
C'_7 +\lambda^{-5} > C_3 + C_4x^{3/4}.
\]
Thus, by using these equations with $x=2^j$, we may pick $C'_7>0$ large enough depending only on $C_3$ and $C_4$ such that for all $\lambda>0$ and $j\geq 0$,
\[
C'_7(1+\lambda^{-5}) + \lambda 2^{j-2} > C_3+ C_4 2^{3j/4}.
\]
For $j\geq 1$, let $\cH_j$ be the event that there exists $\uk\in \mathfrak{K}_{M}$ with $2^{j-1} M \leq \tau_2(\uk)\leq 2^j M$ such that
\[
\sum_{i=1}^M \mathcal{V}_{i,k_{i-1}, k_{i}} > (C_3+C_42^{3j/4})M + \lambda 2^{j-2}M + z
\]
and let $\cH_0$ be the event that there exists $\uk\in \mathfrak{K}_{M}$ with $0\leq \tau_2(\uk)\leq M$ such that
\[
\sum_{i=1}^M \mathcal{V}_{i,k_{i-1}, k_{i}} > (C_3+C_4) M + z.
\]
By our choice of $C'_7$ sufficiently large, if
\[
\sum_{i=1}^M \mathcal{V}_{i,k_{i-1}, k_{i}} - \lambda \tau_2(\uk)\geq C'_7(1+\lambda^{-5})M +z
\]
holds for some $\uk$ then $\cH_j$ holds for $j=\lceil\log_2 (\tau_1(\uk)/M)\rceil\vee 0$. Therefore, the required probability is upper bounded by $\P[\cup_{j=0}^\infty \cH_j]$. By Proposition~\ref{p:perc1},
\begin{align*}
\P[\cup_{j=0}^\infty \cH_j] &\leq \sum_{j=0}^\infty  C_5\exp\left(-C_6 (\lambda 2^{j-2}M + z)^{\xi/4}\right )\\
&\leq C_8\exp\left(-C_9 z^{\xi/4}\right )
\end{align*}
for some $C_8,C_9$ (depending on $\lambda$),
completing the proof.
\end{proof}

\subsubsection{Proof of Proposition \ref{p:tau2perc}}
We shall prove the following statement which implies Proposition~\ref{p:tau2perc} and will be useful elsewhere in the paper. Let $M$ be a large integer and let $r$ be an integer multiple of $n$ with $r\le M^{1/100}n$. Let $D$ be an integer with $2^{\lfloor\log_2\log_2 M\rfloor^{5}}\le D \le M$.  

Let $s_1,s_2$ be fixed with $|s_1|\le \frac{1}{2}\frac{n^{\beta}W_n}{W_r}, |s_1-s_2|\le D^{9/10}$. For $u\in \ell_{0,s_1W_r,(s_1+1)W_r}$ and $v\in \ell_{Dr,s_2W_r, (s_2+1)W_r}$ let $\gamma_{uv}$ denote the conforming geodesic from $u$ to $v$. For $1\le i \le D$, set $J^{uv}_{i,r}=\lfloor \frac{y_{i}}{W_r} \rfloor$ where $(ir,y_{i})$ is the point where $\gamma_{uv}$ intersects the line $x=ir$. Set 
$$\tau_{2,r}(\gamma_{uv}):=\sum_{i} (J^{uv}_{i,r}-J^{uv}_{i-1,r})^2.$$
We have the following lemma. 

\begin{lemma}
    \label{l:tau2percgen}
    There exist $C_7,c,\theta_4>0$ such that for all $M$ sufficiently large, $n=n(M)$ sufficiently large and for all $s_1,s_2, D,r$ as above
    we have
    $$\P\left(\max_{u\in \ell_{0,s_1W_r,(s_1+1)W_r}} \max_{v\in \ell_{Dr,s_2W_r,(s_2+1)W_r}} \tau_{2,r}(\gamma_{uv}) \ge C_7D+z\right) \le \exp(1-cz^{\theta_4}).$$
\end{lemma}

\begin{proof}
    Fix $M,s_1,s_2,D,r$ as in the statement of the lemma. To avoid notational clutter, we shall assume $s_1=0$ and $s_2=s$. It will be clear from the proof that the same argument works in general. We shall also drop the subscript $r$ from $J$ and $\tau_2$.  
     Observe first that by definition $\max |J^{uv}_i|\le n^{\beta}$ and hence $\max_{u,v} \tau_2(\gamma_{uv})\le 2Dn^{2\beta}$ and therefore it suffices to prove the lemma for values of $z\ll n^{2\delta}$ where $\delta\ll \beta$. Therefore it suffices to only consider paths $\gamma$ which do not exit that strip $\R\times [-\frac{1}{2}n^{\beta}W_n, \frac{1}{2}n^{\beta}W_n]$. 

    For $i=1,2,\ldots, D$ and $k,k'\in [-\frac{1}{2}n^{\beta}\frac{W_n}{W_r},\frac{1}{2}n^{\beta}\frac{W_n}{W_r}]$, let us set 

$$Z_{i,k,k'}:=\inf_{\substack{u\in \ell_{(i-1)r,kW_{r}, (k+1)W_r},\\ v\in \ell_{ir,k'W_{r}, (k'+1)W_r}}} \rX_{uv}.$$

We shall show that on a set of probability at least $1-\exp(1-cz^{\theta_4})$, for all $\uk=(k_0,\ldots, k_{M})$ with $k_0=0$ and $k_{D}=s$, $|k_i|\le \frac{1}{2}n^{\beta}\frac{W_n}{W_r}$ with $\tau_2(\uk)\ge C_7M+z$ we have 
$$\sum_{i} Z_{i,k_{i-1},k_i} > \max_{u\in \ell_{0,0,W_r}} \max_{v\in \ell_{Dr,sW_r,(s+1)W_r}} \rX_{uv};$$
clearly this suffices. Observe now that by Theorem \ref{t:all}, Lemma \ref{l:proxy}, the assumption on $s$ and the fact that $Q_{Dr}\le D^{3/4}Q_{r}$ we have that for some $\theta_4>0$
$$\P\left(\max_{u\in \ell_{0,0,W_r}} \max_{v\in \ell_{Dr,sW_r,(s+1)W_r}} \rX_{uv}\ge Dr+D^{4/5}Q_r+z^{1/5}D^{4/5}Q_r\right) \leq \exp(1-cz^{\theta_4}).$$

Observe also that it suffices to prove the above statement for $z$ sufficiently large. It therefore suffices to show that for $C_7$ chosen large enough

\begin{equation}
    \label{e:tau2suffice}
    \P\bigg(\min_{\substack{\uk: \tau_2(\uk)\ge C_7+z\\ |k_i|\le \frac{1}{2}n^{\beta}\frac{W_n}{W_r}}} \sum_{i=1}^{D}Z_{i,k_{i-1},k_i}\le Dr+\frac{z^{1/5}D^{4/5}Q_r}{10000} \bigg) \le \exp(1-cz^{\theta_4}). 
\end{equation}
To prove \eqref{e:tau2suffice}, set
    \[
\mathcal{V}_{i,k,k'}:=\left(-(Z_{i,k,k'} -r)/Q_r + \frac{(k-k')^2}{32}\right).
\]
We know by Proposition~\ref{p:paraestimateconforming} that $\mathcal{V}_{i,k,k'}$ satisfy the hypothesis of Proposition \ref{p:perc1} for $\xi=\theta_2$ (in fact it satisfies stronger tail estimates).
Plugging in the formula for $\mathcal{V}$, notice that 
$$\sum_{i} Z_{i,k_{i-1},k_i}=Dr -Q_r\left(\sum_{i}\mathcal{V}_{i,k_{i-1},k_i}-\frac{1}{32}\tau_2(\uk)\right).$$

Therefore, for $\ell\ge 1$, we get that 
 \[
 \P\left[\min_{\substack{\uk\in \mathfrak{K}_{D},|k_i|\le \frac{1}{2}n^{\beta}\frac{W_n}{W_r}\\C_7D+2^{\ell}z\ge \tau_2(\uk)\ge C_7D+2^{\ell-1}z}} \sum_{i=1}^D Z_{i,k_{i-1}, k_{i}} \le Dr+\frac{z^{1/5}D^{4/5}Q_{r}}{10000} \right]
 \]
equals 
 \[
 \P\left[\max_{\substack{\uk\in \mathfrak{K}_{D}, |k_i|\le \frac{1}{2}n^{\beta}\frac{W_n}{W_r}\\C_7D+2^{\ell}z\ge \tau_2(\uk)\ge C_7D+2^{\ell-1}z}} \sum_{i=1}^D \mathcal{V}_{i,k_{i-1}, k_{i}}-\frac{1}{32}\tau_2(\uk) \ge -\frac{z^{1/5}D^{4/5}}{10000}\right]
 \]
 which, using the range of $\tau_2(\uk)$ is further upper bounded by 
\[
 \P\left[\max_{\substack{\uk\in \mathfrak{K}_{D}, |k_i|\le \frac{1}{2}n^{\beta}\frac{W_n}{W_r},\\C_7D+2^{\ell}z\ge \tau_2(\uk)\ge C_7D+2^{\ell-1}z}} \sum_{i=1}^M \mathcal{V}_{i,k_{i-1}, k_{i}}-\frac{1}{64}\tau_2(\uk) \ge \frac{C_7D+2^{\ell-1}z}{64}-\frac{z^{1/5}D^{4/5}}{10000}\right].
 \]

 Observe now that by choosing $C_7$ sufficiently large and using H\H{o}lder's inequality it follows that 
$$\frac{C_7D+2^{\ell-1}z}{64}-\frac{z^{1/5}D^{4/5}}{10000}\ge C'_7(1+64^{5})D +\frac{2^{\ell-1}z}{10000}$$
where $C'_7$ is as in Corollary \ref{c:tau2perc}. Applying Corollary \ref{c:tau2perc} with $\lambda=1/64$ it follows that 
 \[
 \P\left[\min_{\substack{\uk\in \mathfrak{K}_{D} |k_i|\le \frac{1}{2}n^{\beta}\frac{W_n}{W_r},\\C_7D+2^{\ell}z\ge \tau_2(\uk)\ge C_7D+2^{\ell-1}z}} \sum_{i=1}^D Z_{i,k_{i-1}, k_{i}} \le Dr+\frac{z^{1/5}D^{4/5}Q_{r}}{10000} \right] \le \exp(1-c(z2^{\ell-1})^{\theta_2/4}).
 \]
Taking a union bound over all $\ell$ and reducing the value of $\theta_4$ if necessary, \eqref{e:tau2suffice} follows. This completes the proof of the lemma. 
\end{proof}

\begin{proof}[Proof of Proposition \ref{p:tau2perc}]
The proposition follows from applying Lemma \ref{l:tau2percgen} with $r=n$, $D=M$, $s_1=s_2=0$.     
\end{proof}

\bibliography{fpp}

\begin{thebibliography}{10}

\bibitem{ADS23}
Daniel Ahlberg, Maria Deijfen, and Matteo Sfragara.
\newblock Chaos, concentration and multiple valleys in first-passage percolation, 2024.

\bibitem{AZ23}
Kenneth Alexander and Nikolaos Zygouras.
\newblock {Subgaussian concentration and rates of convergence in directed polymers}.
\newblock {\em Electronic Journal of Probability}, 18(none):1 -- 28, 2013.

\bibitem{Ale20}
Kenneth.~S. Alexander.
\newblock Geodesics, bigeodesics, and coalescence in first passage percolation in general dimension.
\newblock {\em arXiv preprint arXiv:2001.08736}, 2020.

\bibitem{Ale21}
Kenneth.~S. Alexander.
\newblock Uniform fluctuation and wandering bounds in first passage percolation.
\newblock {\em arXiv preprint arXiv:2011.07223}, 2020.

\bibitem{ADH15}
Antonio Auffinger, Michael Damron, and Jack Hanson.
\newblock {\em 50 years of first-passage percolation}, volume~68.
\newblock American Mathematical Soc., 2017.

\bibitem{Bar05}
David Barbato.
\newblock {FKG Inequality for Brownian Motion and Stochastic Differential Equations}.
\newblock {\em Electronic Communications in Probability}, 10(none):7 -- 16, 2005.

\bibitem{BSS14}
Riddhipratim Basu, Vladas Sidoravicius, and Allan Sly.
\newblock Last passage percolation with a defect line and the solution of the {S}low {B}ond {P}roblem.
\newblock Preprint arXiv 1408.3464.

\bibitem{BSS23}
Riddhipratim Basu, Vladas Sidoravicius, and Allan Sly.
\newblock Rotationally invariant first passage percolation: Concentration and scaling relations.
\newblock {\em arXiv preprint arXiv:2312.14143}, 2023.

\bibitem{BR08}
Michel Bena{\"\i}m and Rapha{\"e}l Rossignol.
\newblock Exponential concentration for first passage percolation through modified poincar{\'e}inequalities.
\newblock {\em Annales de l'Institut Henri Poincar{\'e}, Probabilit{\'e}s et Statistiques}, 44(3):544--573, 6 2008.

\bibitem{BKS04}
Itai Benjamini, Gil Kalai, and Oded Schramm.
\newblock First passage percolation has sublinear distance variance.
\newblock {\em Ann. Probab.}, 31(4):1970--1978, 10 2003.

\bibitem{BT15}
Itai Benjamini and Romain Tessera.
\newblock {First passage percolation on nilpotent Cayley graphs and beyond}.
\newblock {\em Electronic Journal of Probability}, 20(none):1 -- 20, 2015.

\bibitem{CN19}
Van~Hao Can and Shuta Nakajima.
\newblock {First passage time of the frog model has a sublinear variance}.
\newblock {\em Electronic Journal of Probability}, 24(none):1 -- 27, 2019.

\bibitem{Cha08}
Sourav Chatterjee.
\newblock Chaos, concentration, and multiple valleys.
\newblock Arxiv 0810.4221, 2008.

\bibitem{Cha11}
Sourav Chatterjee.
\newblock The universal relation between scaling exponents in first-passage percolation.
\newblock {\em Ann. Math. (2)}, 177(2):663--697, 2013.

\bibitem{Chabook}
Sourav Chatterjee.
\newblock {\em Superconcentration and related topics}, volume~15.
\newblock Springer, 2014.

\bibitem{CL24}
Wei-Kuo Chen and Wai-Kit Lam.
\newblock Universality of superconcentration in the sherrington–kirkpatrick model.
\newblock {\em Random Structures \& Algorithms}, 64(2):267--286, 2024.

\bibitem{CD81}
J.~Theodore Cox and Richard Durrett.
\newblock {Some Limit Theorems for Percolation Processes with Necessary and Sufficient Conditions}.
\newblock {\em The Annals of Probability}, 9(4):583 -- 603, 1981.

\bibitem{DHS14}
Michael Damron, Jack Hanson, and Philippe Sosoe.
\newblock {Subdiffusive concentration in first passage percolation}.
\newblock {\em Electronic Journal of Probability}, 19:1 -- 27, 2014.

\bibitem{DHS13}
Michael Damron, Jack Hanson, and Philippe Sosoe.
\newblock Sublinear variance in first-passage percolation for general distributions.
\newblock {\em Probability Theory and Related Fields}, 163(1):223--258, 2015.

\bibitem{Dem24}
Barbara Dembin.
\newblock The variance of the graph distance in the infinite cluster of percolation is sublinear.
\newblock {\em ALEA}, 21:307--320, 2024.

\bibitem{DEP23}
Barbara Dembin, Dor Elboim, and Ron Peled.
\newblock On the influence of edges in first-passage percolation on $\mathbb{Z}^d$, 2023.
\newblock \url{https://arxiv.org/abs/2307.01162}.

\bibitem{DG24}
Barbara Dembin and Christophe Garban.
\newblock Superconcentration for minimal surfaces in first passage percolation and disordered ising ferromagnets.
\newblock {\em Probability Theory and Related Fields}, 190(3):675--702, 2024.

\bibitem{G12}
B.~T. Graham.
\newblock Sublinear variance for directed last-passage percolation.
\newblock {\em Journal of Theoretical Probability}, 25(3):687--702, 2012.

\bibitem{HW65}
J.~M. Hammersley and D.~J.~A. Welsh.
\newblock First-passage percolation, subadditive processes, stochastic networks, and generalized renewal theory.
\newblock In Jerzy Neyman and Lucien~M. Le~Cam, editors, {\em Bernoulli 1713 Bayes 1763 Laplace 1813: Anniversary Volume}, pages 61--110, 1965.

\bibitem{HN97}
C.~Douglas Howard and Charles~M. Newman.
\newblock Euclidean models of first-passage percolation.
\newblock {\em Probability Theory and Related Fields}, 108(2):153--170, 1997.

\bibitem{HN98}
C.~Douglas Howard and Charles~M. Newman.
\newblock From greedy lattice animals to euclidean first-passage percolation.
\newblock In Maury Bramson and Rick Durrett, editors, {\em Perplexing Problems in Probability: Festschrift in Honor of Harry Kesten}, pages 107--119. Birkh{\"a}user Boston, Boston, MA, 1999.

\bibitem{HN01}
C.~Douglas Howard and Charles~M. Newman.
\newblock {Special Invited Paper: Geodesics And Spanning Tees For Euclidean First-Passage Percolaton}.
\newblock {\em The Annals of Probability}, 29(2):577 -- 623, 2001.

\bibitem{J84}
S.~Janson.
\newblock Bounds on the distributions of extremal values of a scanning process.
\newblock {\em Stoc. Proc. Appl.}, 18(2):313--328, 1984.

\bibitem{KPZ86}
Mehran Kardar, Giorgio Parisi, and Yi-Cheng Zhang.
\newblock Dynamic scaling of growing interfaces.
\newblock {\em Phys. Rev. Lett.}, 56:889--892, 1986.

\bibitem{Kes86}
Harry Kesten.
\newblock Aspects of first passage percolation.
\newblock In P.~L. Hennequin, editor, {\em {\'E}cole d'{\'E}t{\'e} de Probabilit{\'e}s de Saint Flour XIV - 1984}, pages 125--264, Berlin, Heidelberg, 1986. Springer Berlin Heidelberg.

\bibitem{Kes93}
Harry Kesten.
\newblock On the speed of convergence in first-passage percolation.
\newblock {\em Ann. Appl. Probab.}, 3(2):296--338, 1993.

\bibitem{K73}
J.~F.~C. Kingman.
\newblock Subadditive ergodic theory.
\newblock {\em Ann. Probab.}, 1(6):883--899, 12 1973.

\bibitem{LW10}
T.~LaGatta and J.~Wehr.
\newblock A shape theorem for riemannian first-passage percolation.
\newblock {\em Journal of Mathematical Physics}, 51(5), 2010.

\bibitem{LW14}
Tom LaGatta and Jan Wehr.
\newblock Geodesics of random riemannian metrics.
\newblock {\em Communications in Mathematical Physics}, 327(1):181--241, 2014.

\bibitem{LNP96}
C.~Licea, C.~M. Newman, and M.~S.~T. Piza.
\newblock Superdiffusivity in first-passage percolation.
\newblock {\em Probability Theory and Related Fields}, 106(4):559--591, 1996.

\bibitem{New95}
Charles~M. Newman.
\newblock A surface view of first-passage percolation.
\newblock In S.~D. Chatterji, editor, {\em Proceedings of the International Congress of Mathematicians}, pages 1017--1023, Basel, 1995. Birkh{\"a}user Basel.

\bibitem{NP95}
Charles~M. Newman and Marcelo S.~T. Piza.
\newblock {Divergence of Shape Fluctuations in Two Dimensions}.
\newblock {\em The Annals of Probability}, 23(3):977 -- 1005, 1995.

\bibitem{Pit82}
L.~D. Pitt.
\newblock Positively correlated normal variables are associated.
\newblock {\em Ann. Probab.}, 10(2):496--499, 1982.

\bibitem{R73}
Daniel Richardson.
\newblock Random growth in a tessellation.
\newblock {\em Mathematical Proceedings of the Cambridge Philosophical Society}, 74(3):515--528, 1973.

\bibitem{VW90}
Mohammad~Q Vahidi-Asl and John~C Wierman.
\newblock First-passage percolation on the voronoi tessellation and delaunay triangulation.
\newblock In {\em Random graphs}, volume~87, pages 341--359, 1990.

\bibitem{VW92}
Mohammad~Q Vahidi-Asl and John~C Wierman.
\newblock A shape result for first-passage percolation on the voronoi tessellation and delaunay triangulation.
\newblock In {\em Random graphs}, volume~2, pages 247--262, 1992.

\bibitem{VW93}
MQ~Vahidi-Asl and JC~Wierman.
\newblock Upper and lower bounds for the route length of first-passage percolation in voronoi tessellations.
\newblock {\em Bull. Iranian Math. Soc}, 19(1):15--28, 1993.

\bibitem{Z08}
Yu~Zhang.
\newblock {Shape fluctuations are different in different directions}.
\newblock {\em The Annals of Probability}, 36(1):331 -- 362, 2008.

\end{thebibliography}
\bibliographystyle{plain}

\end{document}